\documentclass{article}
\usepackage{amsmath}
\usepackage{amssymb}
\usepackage{amsthm}
\usepackage{graphicx}
\usepackage{cite}
\usepackage[arrow,curve,matrix,arc,2cell]{xy}
\UseAllTwocells
\usepackage{makeidx}
\usepackage[refpage,intoc]{nomencl}

\makenomenclature
\makeindex
\usepackage[utf8]{inputenc}
\usepackage[unicode]{hyperref}
\begin{document}
%%%%%%%%%%%%%%%%%%%%%%%%%%%%%%%%%%%%%%%%%%%%%%%%%%%%%%%%%%%%%%%%%%%%%%%%
%%%%%%%%%%%%%%%%  Commands for glossary and index %%%%%%%%%%%%%%%%%%%%%%
%%%%%%%%%%%%%%%%%%%%%%%%%%%%%%%%%%%%%%%%%%%%%%%%%%%%%%%%%%%%%%%%%%%%%%%%
\renewcommand{\nomname}{Glossary of Notation}
\renewcommand{\pagedeclaration}[1]{, #1}
\def\G[#1]#2#3{\nomenclature[#1]{#2}{#3}}
\def\I#1{\index{#1}}
%\def\G[#1]#2#3{}
%\def\I#1{}
%%%%%%%%%%%%%%%%%%%%%%%%%%%%%%%%%%%%%%%%%%%%%%%%%%%%%%%%%%%%%%%%%%%%%%%%
%%%%%%%%%%%%%%%%%%%%%%%%%%     Macros      %%%%%%%%%%%%%%%%%%%%%%%%%%%%%
%%%%%%%%%%%%%%%%%%%%%%%%%%%%%%%%%%%%%%%%%%%%%%%%%%%%%%%%%%%%%%%%%%%%%%%%
\def\e#1\e{\begin{equation}#1\end{equation}}
\def\ea#1\ea{\begin{align}#1\end{align}}
\def\eq#1{{\rm(\ref{#1})}}
\theoremstyle{plain}% default
\newtheorem{thm}{Theorem}[section]
\newtheorem{lem}[thm]{Lemma}
\newtheorem{prop}[thm]{Proposition}
\newtheorem{cor}[thm]{Corollary}
\theoremstyle{definition}
\newtheorem{dfn}[thm]{Definition}
\newtheorem{ex}[thm]{Example}
\newtheorem{rem}[thm]{Remark}
\numberwithin{equation}{section}
\def\dim{\mathop{\rm dim}\nolimits}
\def\inc{\mathop{\rm inc}\nolimits}
\def\Ker{\mathop{\rm Ker}}
\def\Coker{\mathop{\rm Coker}}
\def\Im{\mathop{\rm Im}}
\def\rank{\mathop{\rm rank}}
\def\Obj{{\rm Obj}}
\def\Hom{\mathop{\rm Hom}\nolimits}
\def\Aut{\mathop{\rm Aut}}
\def\coh{\mathop{\rm coh}}
\def\qcoh{\mathop{\rm qcoh}}
\def\supp{\mathop{\rm supp}}
\def\Sh{\mathop{\rm Sh}}
\def\id{\mathop{\rm id}\nolimits}
\def\top{{\kern.05em\rm top}}
\def\Iso{{\rm Iso}}
\def\tr{{\rm tr}}
\def\nt{{\rm nt}}
\def\Spec{\mathop{\rm Spec}\nolimits}
\def\MSpec{\mathop{\rm MSpec}\nolimits}
\def\Ho{{\mathop{\rm Ho}}}
\def\Top{{\mathop{\bf Top}}}
\def\CRings{{\mathop{\bf C^{\bs\iy}Rings}}}
\def\CRingsco{{\mathop{\bf C^{\bs\iy}Rings^{co}}}}
\def\CRingsfg{{\mathop{\bf C^{\bs\iy}Rings^{fg}}}}
\def\CRingsfp{{\mathop{\bf C^{\bs\iy}Rings^{fp}}}}
\def\CRingsfa{{\mathop{\bf C^{\bs\iy}Rings^{fa}}}}
\def\CRS{{\mathop{\bf C^{\bs\iy}RS}}}
\def\LCRS{{\mathop{\bf LC^{\bs\iy}RS}}}
\def\ACSch{{\mathop{\bf AC^{\bs\iy}Sch}}}
\def\ACSchfp{{\mathop{\bf AC^{\bs\iy}Sch^{fp}}}}
\def\ACSchfa{{\mathop{\bf AC^{\bs\iy}Sch^{fa}}}}
\def\CSch{{\mathop{\bf C^{\bs\iy}Sch}}}
\def\CSchlfp{{\mathop{\bf C^{\bs\iy}Sch^{lfp}}}}
\def\CSchlf{{\mathop{\bf C^{\bs\iy}Sch^{lf}}}}
\def\bCSch{{\mathop{\bf{\bar C}^{\bs\iy}Sch}}}
\def\bCSchlfp{{\mathop{\bf{\bar C}^{\bs\iy}Sch^{lfp}}}}
\def\bCSchlf{{\mathop{\bf{\bar C}^{\bs\iy}Sch^{lf}}}}
\def\CSta{{\mathop{\bf C^{\bs\iy}Sta}}}
\def\CStalfp{{\mathop{\bf C^{\bs\iy}Sta^{lfp}}}} 
\def\DMCSta{{\mathop{\bf DMC^{\bs\iy}Sta}}}
\def\DMCStalf{{\mathop{\bf DMC^{\bs\iy}Sta^{lf}}}}
\def\DMCStalfp{{\mathop{\bf DMC^{\bs\iy}Sta^{lfp}}}}
\def\Presta{{\mathop{\bf Presta}\nolimits}}
\def\Sta{{\mathop{\bf Sta}\nolimits}}
\def\GSta{{\mathop{\bf GSta}\nolimits}}
\def\Man{{\mathop{\bf Man}}}
\def\Manb{{\mathop{\bf Man^b}}}
\def\Manc{{\mathop{\bf Man^c}}}
\def\ManSta{{\mathop{\bf ManSta}}}
\def\Orb{{\mathop{\bf Orb}}}
\def\Euc{{\mathop{\bf Euc}}}
\def\Sets{{\mathop{\bf Sets}}}
\def\ul{\underline}
\def\bs{\boldsymbol}
\def\ge{\geqslant}
\def\le{\leqslant\nobreak}
\def\O{{\mathcal O}}
\def\N{{\mathbin{\mathbb N}}}
\def\R{{\mathbin{\mathbb R}}}
\def\Z{{\mathbin{\mathbb Z}}}
\def\C{{\mathbin{\mathbb C}}}
\def\CP{{\mathbin{\mathbb{CP}}}}
\def\fC{{\mathbin{\mathfrak C}\kern.05em}}
\def\fD{{\mathbin{\mathfrak D}}}
\def\fE{{\mathbin{\mathfrak E}}}
\def\fF{{\mathbin{\mathfrak F}}}
\def\fm{{\mathbin{\mathfrak m}}}
\def\cB{{\mathbin{\cal B}}}
\def\cC{{\mathbin{\cal C}}}
\def\cD{{\mathbin{\cal D}}}
\def\cE{{\mathbin{\cal E}}}
\def\cF{{\mathbin{\cal F}}}
\def\cG{{\mathbin{\cal G}}}
\def\cH{{\mathbin{\cal H}}}
\def\cI{{\mathbin{\cal I}}}
\def\cJ{{\mathbin{\cal J}}}
\def\cH{{\mathbin{\cal H}}}
\def\cP{{\mathbin{\cal P}}}
\def\cS{{\mathbin{\cal S}}}
\def\cT{{\cal T}}
\def\cU{{\mathbin{\cal U}}}
\def\cV{{\mathbin{\cal V}}}
\def\cW{{\mathbin{\cal W}}}
\def\cX{{\mathbin{\cal X}}}
\def\hcX{{\mathbin{\hat{\cal X}}}{}}
\def\tcX{{\mathbin{\tilde{\cal X}}}{}}
\def\cY{{\mathbin{\cal Y}}}
\def\cZ{{\mathbin{\cal Z}}}
\def\ucU{{\mathbin{\ul{\cal U\kern -0.2em}\kern .2em}{}}}
\def\ucX{{\mathbin{\ul{\cal X\kern -0.2em}\kern .1em}{}}}
\def\bucX{{\mathbin{\ul{\bar{\cal X}\kern -0.2em}\kern .1em}{}}}
\def\oM{{\mathbin{\smash{\,\,\overline{\!\!\mathcal M\!}\,}}}}
\def\fCmod{{\mathbin{{\mathfrak C}\text{\rm -mod}}}}
\def\fCmodco{{\mathbin{{\mathfrak C}\text{\rm -mod}^{\rm co}}}}
\def\fCmodfp{{\mathbin{{\mathfrak C}\text{\rm -mod}^{\rm fp}}}}
\def\fDmod{{\mathbin{{\mathfrak D}\text{\rm -mod}}}}
\def\fDmodco{{\mathbin{{\mathfrak D}\text{\rm -mod}^{\rm co}}}}
\def\fDmodfp{{\mathbin{{\mathfrak D}\text{\rm -mod}^{\rm fp}}}}
\def\Rmod{{\mathbin{R\text{\rm -mod}}}}
\def\OWmod{{\mathbin{\O_W\text{\rm -mod}}}}
\def\OXmod{{\mathbin{\O_X\text{\rm -mod}}}}
\def\OYmod{{\mathbin{\O_Y\text{\rm -mod}}}}
\def\OZmod{{\mathbin{\O_Z\text{\rm -mod}}}}
\def\m{{\mathfrak m}}
\def\ua{{\underline{a}}{}}
\def\ub{{\underline{b\kern -0.1em}\kern 0.1em}{}}
\def\uc{{\underline{c}}{}}
\def\ud{{\underline{d}}{}}
\def\ue{{\underline{e}}{}}
\def\uf{{\underline{f\!}\,}{}}
\def\buf{{\underline{\bar f\!}\,}{}}
\def\ug{{\underline{g\!}\,}{}}
\def\uh{{\underline{h\!}\,}{}}
\def\ui{{\underline{i\kern -0.07em}\kern 0.07em}{}}
\def\uj{{\underline{j\kern -0.1em}\kern 0.1em}{}}
\def\uk{{\underline{k\kern -0.1em}\kern 0.1em}{}}
\def\um{{\underline{m\kern -0.1em}\kern 0.1em}{}}
\def\un{{\underline{n\kern -0.1em}\kern 0.1em}{}}
\def\ulp{{\underline{p\kern -0.15em}\kern 0.15em}{}}
\def\uq{{\underline{q\kern -0.15em}\kern 0.15em}{}}
\def\ur{{\underline{r\kern -0.15em}\kern 0.15em}{}}
\def\us{{\underline{s\kern -0.15em}\kern 0.15em}{}}
\def\ut{{\underline{t\kern -0.1em}\kern 0.1em}{}}
\def\uu{{\underline{u\kern -0.1em}\kern 0.1em}{}}
\def\uv{{\underline{v\!}\,}{}}
\def\uw{{\underline{w\!}\,}{}}
\def\uF{{\underline{F\!}\,}{}}
\def\uG{{\underline{G\!}\,}{}}
\def\uH{{\underline{H\!}\,}{}}
\def\uP{{\underline{P\!}\,}{}}
\def\uQ{{\underline{Q\!}\,}{}}
\def\uS{{\underline{S\!}\,}{}}
\def\uT{{\underline{T\!}\,}{}}
\def\uU{{\underline{U\kern -0.25em}\kern 0.2em}{}}
\def\utU{{{\underline{\ti U\kern -0.25em}\kern 0.2em}{}}}
\def\utV{{{\underline{\ti V\kern -0.25em}\kern 0.2em}{}}}
\def\uV{{\underline{V\kern -0.25em}\kern 0.2em}{}}
\def\uW{{\underline{W\!\!}\,\,}{}}
\def\uX{{\underline{X\!}\,}{}}
\def\uY{{\underline{Y\!\!}\,\,}{}}
\def\uZ{{\underline{Z\!}\,}{}}
\def\bW{{\bs W}}
\def\bX{{\bs X}}
\def\bY{{\bs Y}}
\def\bZ{{\bs Z}}
\def\upi{{\underline{\pi\!}\,}}
\def\uid{{\underline{\id\kern -0.1em}\kern 0.1em}{}}
\def\urho{{\underline{\rho\!}\,}}
\def\uphi{{\underline{\phi\!}\,}}
\def\umu{{\underline{\smash{\mu}}\kern 0.1em}}
\def\uDe{{\underline{\De\kern -0.1em}\kern 0.1em}{}}
\def\al{\alpha}
\def\be{\beta}
\def\ga{\gamma}
\def\de{\delta}
\def\io{\iota}
\def\ep{\epsilon}
\def\la{\lambda}
\def\ka{\kappa}
\def\th{\theta}
\def\ze{\zeta}
\def\up{\upsilon}
\def\vp{\varphi}
\def\si{\sigma}
\def\om{\omega}
\def\De{\Delta}
\def\La{\Lambda}
\def\Om{\Omega}
\def\Ga{\Gamma}
\def\Th{\Theta}
\def\pd{\partial}
\def\ts{\textstyle}
\def\st{\scriptstyle}
\def\sst{\scriptscriptstyle}
\def\w{\wedge}
\def\sm{\setminus}
\def\sh{\sharp}
\def\op{\oplus}
\def\od{\odot}
\def\op{\oplus}
\def\ot{\otimes}
\def\ov{\overline}
\def\bigop{\bigoplus}
\def\bigot{\bigotimes}
\def\iy{\infty}
\def\es{\emptyset}
\def\ra{\rightarrow}
\def\rra{\rightrightarrows}
\def\Ra{\Rightarrow}
\def\ab{\allowbreak}
\def\longra{\longrightarrow}
\def\Longra{\Longrightarrow}
\def\hookra{\hookrightarrow}
\def\dashra{\dashrightarrow}
\def\t{\times}
\def\lt{\ltimes}
\def\ci{\circ}
\def\ti{\tilde}
\def\d{{\rm d}}
\def\ha{{\ts\frac{1}{2}}}
\def\md#1{\vert #1 \vert}
\def\bmd#1{\big\vert #1 \big\vert}
%%%%%%%%%%%%%%%%%%%%%%%%%%%%%%%%%%%%%%%%%%%%%%%%%%%%%%%%%%%%%%%%%%%%%%%%
%%%%%%%%%%%%%%%%%%%%%%%%    Text of book    %%%%%%%%%%%%%%%%%%%%%%%%%%%%
%%%%%%%%%%%%%%%%%%%%%%%%%%%%%%%%%%%%%%%%%%%%%%%%%%%%%%%%%%%%%%%%%%%%%%%%
\title{Algebraic Geometry over $C^\iy$-rings}
\author{Dominic Joyce}
\date{}
\maketitle

\begin{abstract}
If $X$ is a manifold then the $\R$-algebra $C^\iy(X)$ of smooth
functions $c:X\ra\R$ is a $C^\iy$-{\it ring}. That is, for each
smooth function $f:\R^n\ra\R$ there is an $n$-fold operation
$\Phi_f:C^\iy(X)^n\ra C^\iy(X)$ acting by
$\Phi_f:(c_1,\ldots,c_n)\mapsto f(c_1,\ldots,c_n)$, and these
operations $\Phi_f$ satisfy many natural identities. Thus,
$C^\iy(X)$ actually has a far richer structure than the obvious
$\R$-algebra structure.

We explain the foundations of a version of algebraic geometry in
which rings or algebras are replaced by $C^\iy$-rings. As schemes
are the basic objects in algebraic geometry, the new basic objects
are $C^\iy$-{\it schemes}, a category of geometric objects which
generalize manifolds, and whose morphisms generalize smooth maps. We
also study {\it quasicoherent sheaves\/} on $C^\iy$-schemes, and $C^\iy$-{\it stacks}, in particular {\it Deligne--Mumford\/ $C^\iy$-stacks}, a 2-category of geometric objects generalizing orbifolds.

Many of these ideas are not new: $C^\iy$-rings and $C^\iy$-schemes have long been part of synthetic differential geometry. But we develop them in new directions. In \cite{Joyc2,Joyc3,Joyc4}, the author uses these tools to define {\it d-manifolds\/} and {\it d-orbifolds}, `derived' versions of manifolds and orbifolds related to Spivak's `derived manifolds'~\cite{Spiv}.
\end{abstract}

\setcounter{tocdepth}{2}
\tableofcontents

\section{Introduction}
\label{ag1}

Let $X$ be a smooth manifold, and write $C^\iy(X)$ for the set of
smooth functions $c:X\ra\R$. Then $C^\iy(X)$ is a commutative
$\R$-algebra, with operations of addition, multiplication, and
scalar multiplication defined pointwise. However, $C^\iy(X)$ has
much more structure than this. For example, if $c:X\ra\R$ is smooth
then $\exp(c):X\ra\R$ is smooth, and this defines an operation
$\exp:C^\iy(X)\ra C^\iy(X)$ which cannot be expressed algebraically
in terms of the $\R$-algebra structure. More generally, if $n\ge 0$
and $f:\R^n\ra\R$ is smooth, define an $n$-fold operation
$\Phi_f:C^\iy(X)^n\ra C^\iy(X)$ by
\begin{equation*}
\bigl(\Phi_f(c_1,\ldots,c_n)\bigr)(x)=f\bigl(c_1(x),\ldots,c_n(x)\bigr),
\end{equation*}
for all $c_1,\ldots,c_n\in C^\iy(X)$ and $x\in X$. These operations
satisfy many identities: suppose $m,n\ge 0$, and $f_i:\R^n\ra\R$ for
$i=1,\ldots,m$ and $g:\R^m\ra\R$ are smooth functions. Define a
smooth function $h:\R^n\ra\R$ by
\begin{equation*}
h(x_1,\ldots,x_n)=g\bigl(f_1(x_1,\ldots,x_n),\ldots,f_m(x_1
\ldots,x_n)\bigr),
\end{equation*}
for all $(x_1,\ldots,x_n)\in\R^n$. Then for all $c_1,\ldots,c_n\in
C^\iy(X)$ we have
\e
\Phi_h(c_1,\ldots,c_n)=\Phi_g\bigl(\Phi_{f_1}(c_1,\ldots,c_n),
\ldots,\Phi_{f_m}(c_1,\ldots,c_n)\bigr).
\label{ag1eq1}
\e

A $C^\iy$-{\it ring\/}\I{C-ring@$C^\iy$-ring}
$\bigl(\fC,(\Phi_f)_{f:\R^n\ra\R\,\,C^\iy} \bigr)$ is a set $\fC$
with operations $\Phi_f:\fC^n\ra\fC$ for all $f:\R^n\ra\R$ smooth
satisfying identities \eq{ag1eq1}, and one other condition. For example
$C^\iy(X)$ is a $C^\iy$-ring for any manifold $X$, but there are
also many $C^\iy$-rings which do not come from manifolds, and can be
thought of as representing geometric objects which generalize
manifolds.

The most basic objects in conventional algebraic geometry are
commutative rings $R$, or commutative $\mathbb K$-algebras $R$ for
some field $\mathbb K$. The
`spectrum'\I{C-ring@$C^\iy$-ring!spectrum functor $\Spec$} $\Spec R$
of $R$ is an affine scheme, and $R$ is interpreted as an algebra of
functions on $\Spec R$. More general kinds of spaces in algebraic
geometry --- schemes and stacks --- are locally modelled on affine
schemes $\Spec R$. This book lays down the foundations of {\it
Algebraic Geometry over\/ $C^\iy$-rings}, in which we replace
commutative rings in algebraic geometry by $C^\iy$-rings. It
includes the study of $C^\iy$-{\it schemes} and {\it
Deligne--Mumford\/ $C^\iy$-stacks}, two classes of geometric spaces
generalizing manifolds and orbifolds,\I{orbifold} respectively.

This is not a new idea, but was studied years ago as part of {\it
synthetic differential geometry},\I{synthetic differential
geometry|(} which grew out of ideas of Lawvere in the 1960s; see for
instance Dubuc \cite{Dubu3} on $C^\iy$-schemes, and the books by
Moerdijk and Reyes \cite{MoRe2} and Kock \cite{Kock}. However, we
have new things to say, as we are motivated by different problems
(see below), and so are asking different questions.

Following Dubuc's discussion of `models of synthetic differential
geometry' \cite{Dubu1} and oversimplifying a bit, synthetic
differential geometers are interested in $C^\iy$-schemes as they
provide a category $\CSch$ of geometric objects which includes
smooth manifolds and certain `infinitesimal' objects, and all fibre
products exist in $\CSch$, and $\CSch$ has some other nice
properties to do with open covers, and exponentials of
infinitesimals.

Synthetic differential geometry concerns proving theorems about
manifolds using synthetic reasoning involving `infinitesimals'. But
one needs to check these methods of synthetic reasoning are valid.
To do this you need a `model', some category of geometric spaces
including manifolds and infinitesimals, in which you can think of
your synthetic arguments as happening. Once you know there exists at
least one model with the properties you want, then as far as
synthetic differential geometry is concerned the job is done. For
this reason $C^\iy$-schemes have not been developed very far in
synthetic differential geometry.\I{synthetic differential
geometry|)}

Recently, $C^\iy$-rings and $C^\iy$-ringed spaces appeared in a very
different context, in the theory of {\it derived differential geometry}, the differential-geometric analogue of the derived algebraic geometry of Lurie \cite{Luri} and To\"en--Vezzosi \cite{Toen,ToVe}, which studies {\it derived smooth manifolds\/} and {\it derived smooth orbifolds}. This began with a short section in Lurie \cite[\S 4.5]{Luri}, where he sketched how to define an $\iy$-category of {\it derived\/ $C^\iy$-schemes}, including derived manifolds. 

Lurie's student David Spivak \cite{Spiv} worked out the details of this, defining an $\iy$-category of derived manifolds. Simplifications and extensions of Spivak's theory were given by Borisov and Noel \cite{Bori,BoNo} and the author \cite{Joyc2,Joyc3,Joyc4}. An alternative approach to the foundations of derived differential geometry involving differential graded $C^\iy$-rings is proposed by Carchedi and Roytenberg~\cite{CaRo1,CaRo2}.

The author's notion of derived manifolds \cite{Joyc2,Joyc3,Joyc4} are called {\it d-manifolds},\I{d-manifold|(} and are built using our theory of $C^\iy$-schemes and quasicoherent sheaves upon them below. They form a 2-category. We also study orbifold versions, {\it d-orbifolds},\I{d-orbifold|(} which are built using our theory of Deligne--Mumford $C^\iy$-stacks and their quasicoherent sheaves below.

Many areas of symplectic geometry\I{symplectic geometry} involve studying moduli spaces of $J$-holomorphic curves in a symplectic manifold, which are made into {\it Kuranishi spaces\/} in the framework of Fukaya, Oh, Ohta and Ono \cite{FOOO,FuOn}. The author argues that {\it Kuranishi spaces are really derived orbifolds}, and has given a new definition \cite{Joyc5,Joyc7} of a 2-category of Kuranishi spaces $\bf Kur$ which is equivalent to the 2-category of d-orbifolds $\bf dOrb$ from \cite{Joyc2,Joyc3,Joyc4}. Because of this, derived differential geometry will have important applications in symplectic geometry.

To set up our theory of d-manifolds and d-orbifolds requires a lot
of preliminary work on $C^\iy$-schemes and $C^\iy$-stacks, and
quasicoherent sheaves upon them. That is the purpose of this book.
We have tried to present a complete, self-contained account which
should be understandable to readers with a reasonable background in
algebraic geometry, and we assume no familiarity with synthetic
differential geometry.\I{synthetic differential geometry} We expect
this material may have other applications quite different to those
the author has in mind in~\cite{Joyc2,Joyc3,Joyc4}.\I{d-orbifold|)}\I{d-manifold|)}

Section \ref{ag2} explains $C^\iy$-rings. The archetypal examples of
$C^\iy$-rings, $C^\iy(X)$ for manifolds $X$, are discussed in
\S\ref{ag3}. Section \ref{ag4} studies $C^\iy$-schemes, and \S\ref{ag5} modules over $C^\iy$-rings and sheaves of modules over $C^\iy$-schemes.

Sections \ref{ag6}--\ref{ag9} discuss $C^\iy$-stacks. Section \ref{ag6} defines the 2-category $\CSta$ of $C^\iy$-{\it stacks}, analogues of Artin
stacks in algebraic geometry, and \S\ref{ag7} the 2-subcategory
$\DMCSta$ of {\it Deligne--Mumford\/ $C^\iy$-stacks}, which are
$C^\iy$-stacks locally modelled on $[\uU/G]$ for $\uU$ an affine
$C^\iy$-scheme and $G$ a finite group acting on $\uU$, and are
analogues of Deligne--Mumford stacks in algebraic geometry. We show
that orbifolds\I{orbifold} $\Orb$ may be regarded as a
2-subcategory of $\DMCSta$. Section \ref{ag8} studies quasicoherent sheaves on Deligne--Mumford $C^\iy$-stacks, generalizing \S\ref{ag5}, and
\S\ref{ag9} orbifold strata of Deligne--Mumford $C^\iy$-stacks.

Appendix \ref{agA} summarizes background on stacks from
\cite{Arti,BEFF,Gome,LaMo,Metz,Nooh}, for use in \S\ref{ag6}--\S\ref{ag9}. Stacks are a very technical area, and \S\ref{agA} is too terse to help a beginner learn the subject, it is intended only to establish notation and definitions for those already familiar with stacks. Readers with no experience of stacks are advised to first consult an introductory text such as Vistoli \cite{Vist}, Gomez \cite{Gome}, Laumon and Moret-Bailly \cite{LaMo}, or the online `Stacks Project'~\cite{Jong}.

Much of \S\ref{ag2}--\S\ref{ag4} is already understood in synthetic differential geometry,\I{synthetic differential geometry} such as in the work of Dubuc \cite{Dubu3} and Moerdijk and Reyes \cite{MoRe2}. But we believe it is worthwhile giving a detailed and self-contained exposition, from our own point of view. Sections \ref{ag5}--\ref{ag9} are new, so far as the author knows, though \S\ref{ag5}--\S\ref{ag8} are based on well known material in algebraic geometry.
\medskip

\noindent{\it Acknowledgements.} I would like to thank Omar Antolin, Eduardo Dubuc, Kelli Francis-Staite, Jacob Gross, Jacob Lurie, and Ieke Moerdijk for helpful conversations, and a referee for many useful comments. This research was supported by EPSRC grants EP/H035303/1 and EP/J016950/1.

\section{\texorpdfstring{$C^\iy$-rings}{C∞-rings}}
\label{ag2}
\I{C-ring@$C^\iy$-ring|(}

We begin by explaining the basic objects out of which our theories
are built, $C^\iy$-{\it rings}, or {\it smooth rings}. The
archetypal example of a $C^\iy$-ring is the vector space $C^\iy(X)$
of smooth functions $c:X\ra\R$ for a manifold $X$. Everything in this
section is known to experts in synthetic differential
geometry,\I{synthetic differential geometry} and much of it can be
found in Moerdijk and Reyes \cite[Ch.~I]{MoRe2}, Dubuc
\cite{Dubu1,Dubu2,Dubu3,Dubu4} or Kock \cite[\S III]{Kock}. We
introduce some new notation for brevity, in particular, our {\it
fair\/}\I{C-ring@$C^\iy$-ring!fair} $C^\iy$-rings are known in the
literature as `finitely generated\I{C-ring@$C^\iy$-ring!finitely
generated} and germ determined\I{C-ring@$C^\iy$-ring!germ
determined} $C^\iy$-rings'.

\subsection{\texorpdfstring{Two definitions of $C^\iy$-ring}{Two definitions of C∞-ring}}
\label{ag21}
\I{C-ring@$C^\iy$-ring!definition|(}

We first define $C^\iy$-rings in the style of classical algebra.

\begin{dfn} A $C^\iy$-{\it ring\/} is a set
$\fC$\G[CDE]{$\fC,{\mathfrak D},{\mathfrak
E},\ldots$}{$C^\iy$-rings}\G[Phif]{$\Phi_f:\fC^n\ra\fC$}{operations
on $C^\iy$-ring $\fC$, for smooth $f:\R^n\ra\R$} together with
operations
\begin{equation*}
\smash{\Phi_f:\fC^n={\buildrel {\!\ulcorner\,\text{$n$ copies
}\,\urcorner\!} \over {\fC\t\cdots \t \fC}}\longra \fC}
\end{equation*}
for all $n\ge 0$ and smooth maps $f:\R^n\ra\R$, where by convention
when $n=0$ we define $\fC^0$ to be the single point $\{\es\}$. These
operations must satisfy the following relations: suppose $m,n\ge 0$,
and $f_i:\R^n\ra\R$ for $i=1,\ldots,m$ and $g:\R^m\ra\R$ are smooth
functions. Define a smooth function $h:\R^n\ra\R$ by
\begin{equation*}
h(x_1,\ldots,x_n)=g\bigl(f_1(x_1,\ldots,x_n),\ldots,f_m(x_1
\ldots,x_n)\bigr),
\end{equation*}
for all $(x_1,\ldots,x_n)\in\R^n$. Then for all
$(c_1,\ldots,c_n)\in\fC^n$ we have
\begin{equation*}
\Phi_h(c_1,\ldots,c_n)=\Phi_g\bigl(\Phi_{f_1}(c_1,\ldots,c_n),
\ldots,\Phi_{f_m}(c_1,\ldots,c_n)\bigr).
\end{equation*}
We also require that for all $1\le j\le n$, defining
$\pi_j:\R^n\ra\R$ by $\pi_j:(x_1,\ldots,x_n)\mapsto x_j$, we have
$\Phi_{\pi_j}(c_1,\ldots,c_n)=c_j$ for
all~$(c_1,\ldots,c_n)\in\fC^n$.

Usually we refer to $\fC$ as the $C^\iy$-ring, leaving the
operations $\Phi_f$ implicit.

A {\it morphism\/} between $C^\iy$-rings $\bigl(\fC,(\Phi_f)_{
f:\R^n\ra\R\,\,C^\iy}\bigr)$, $\bigl({\mathfrak
D},(\Psi_f)_{f:\R^n\ra\R\,\,C^\iy}\bigr)$ is a map
$\phi:\fC\ra{\mathfrak D}$ such that $\Psi_f\bigl(\phi
(c_1),\ldots,\phi(c_n)\bigr)=\phi\ci\Phi_f(c_1,\ldots,c_n)$ for all
smooth $f:\R^n\ra\R$ and $c_1,\ldots,c_n\in\fC$. We will write
$\CRings$\G[CRings]{$\CRings$}{category of $C^\iy$-rings} for the
category of $C^\iy$-rings.
\label{ag2def1}
\end{dfn}

Here is the motivating example, which we will study at greater
length in~\S\ref{ag3}:

\begin{ex} Let $X$ be a manifold, which may be without boundary, or
with boundary, or with corners.\I{manifold!C-ring of@$C^\iy$-ring
of}\I{manifold with corners!C-ring of@$C^\iy$-ring of} Write
$C^\iy(X)$ for the set of smooth functions $c:X\ra\R$. For $n\ge 0$
and $f:\R^n\ra\R$ smooth, define $\Phi_f:C^\iy(X)^n\ra C^\iy(X)$ by
\e
\bigl(\Phi_f(c_1,\ldots,c_n)\bigr)(x)=f\bigl(c_1(x),\ldots,c_n(x)\bigr),
\label{ag2eq1}
\e
for all $c_1,\ldots,c_n\in C^\iy(X)$ and $x\in X$. It is easy to see
that $C^\iy(X)$ and the operations $\Phi_f$ form a $C^\iy$-ring.
\label{ag2ex1}
\end{ex}

\begin{ex} Take $X$ to be the point $*$ in Example \ref{ag2ex1}. Then
$C^\iy(*)=\R$, with operations $\Phi_f:\R^n\ra\R$ given by
$\Phi_f(x_1,\ldots,x_n)=f(x_1,\ldots,x_n)$. This makes $\R$ into the
simplest nonzero example of a $C^\iy$-ring, the initial object in~$\CRings$.

\label{ag2ex2}
\end{ex}

Note that $C^\iy$-rings are far more general than those coming from
manifolds. For example, if $X$ is any topological space we could
define a $C^\iy$-ring $C^0(X)$ to be the set of {\it continuous\/}
$c:X\ra\R$ with operations $\Phi_f$ defined as in \eq{ag2eq1}. For
$X$ a manifold with $\dim X>0$, the $C^\iy$-rings $C^\iy(X)$ and
$C^0(X)$ are different.

There is a more succinct definition of $C^\iy$-rings using category
theory:

\begin{dfn} Write $\Man$\G[Man]{$\Man$}{category of manifolds} for
the category of manifolds, and $\Euc$\G[Euc]{$\Euc$}{category of
Euclidean spaces $\R^n$ and smooth maps} for the full subcategory of
$\Man$ with objects the Euclidean spaces $\R^n$. That is, the
objects of $\Euc$ are $\R^n$ for $n=0,1,2,\ldots,$ and the morphisms
in $\Euc$ are smooth maps $f:\R^m\ra\R^n$. Write
$\Sets$\G[Sets]{$\Sets$}{category of sets} for the category of sets.
In both $\Euc$ and $\Sets$ we have notions of (finite) products of
objects (that is, $\R^{m+n}=\R^m\t\R^n$, and products $S\t T$ of
sets $S,T$), and products of morphisms. 

Define a ({\it category-theoretic\/}) $C^\iy$-{\it ring\/} to be a
product-preserving functor $F:\Euc\ra\Sets$. Here $F$ should also preserve the empty product, that is, it maps $\R^0$ in $\Euc$ to the terminal object in $\Sets$, the point $*$.

$C^\iy$-rings in this sense are an example of an {\it algebraic theory\/} in the sense of Ad\'amek, Rosick\'y and Vitale \cite{ARV}, and many of the basic categorical properties of $C^\iy$-rings follow from this.
\label{ag2def2}
\end{dfn}

Here is how this relates to Definition \ref{ag2def1}. Suppose
$F:\Euc\ra\Sets$ is a product-preserving functor. Define
$\fC=F(\R)$. Then $\fC$ is an object in $\Sets$, that is, a set.
Suppose $n\ge 0$ and $f:\R^n\ra\R$ is smooth. Then $f$ is a morphism
in $\Euc$, so $F(f):F(\R^n)\ra F(\R)=\fC$ is a morphism in $\Sets$.
Since $F$ preserves products $F(\R^n)=F(\R)\t\cdots\t F(\R)=\fC^n$,
so $F(f)$ maps $\fC^n\ra\fC$. We define $\Phi_f:\fC^n\ra\fC$ by
$\Phi_f=F(f)$. The fact that $F$ is a functor implies that the
$\Phi_f$ satisfy the relations in Definition \ref{ag2def1}, so
$\bigl(\fC,(\Phi_f)_{f:\R^n\ra\R \,\,C^\iy}\bigr)$ is a $C^\iy$
ring.

Conversely, if $\bigl(\fC,(\Phi_f)_{f:\R^n\ra\R\,\,C^\iy}\bigr)$ is
a $C^\iy$-ring then we define $F:\Euc\ra\Sets$ by $F(\R^n)=\fC^n$,
and if $f:\R^n\ra\R^m$ is smooth then $f=(f_1,\ldots,f_m)$ for
$f_i:\R^n\ra\R$ smooth, and we define $F(f):\fC^n\ra\fC^m$ by
$F(f):(c_1,\ldots,c_n)\mapsto\bigl(
\Phi_{f_1}(c_1,\ldots,c_n),\ldots,\Phi_{f_m}(c_1,\ldots,c_n)\bigr)$.
Then $F$ is a product-preserving functor. This defines a 1-1
correspondence between $C^\iy$-rings in the sense of Definition
\ref{ag2def1}, and category-theoretic $C^\iy$-rings in the sense of
Definition~\ref{ag2def2}.\I{C-ring@$C^\iy$-ring!definition|)}

As in Moerdijk and Reyes \cite[p.~21--22]{MoRe2} we have:

\begin{prop} In the category $\CRings$ of\/ $C^\iy$-rings, all limits and all filtered colimits exist, and regarding $C^\iy$-rings as functors $F:\Euc\ra\Sets$ as in Definition\/ {\rm\ref{ag2def2},} they may be computed objectwise in $\Euc$ by taking the corresponding limits/filtered colimits in $\Sets$.

Also, all small colimits\I{C-ring@$C^\iy$-ring!colimit}\I{colimit}\I{category!colimit}
exist, though in general they are not computed objectwise in $\Euc$ by taking colimits in $\Sets$. In particular, pushouts\I{category!pushout}\I{pushout}
and all finite colimits exist in~$\CRings$.
\label{ag2prop1}
\end{prop}

We will write $\fD\amalg_{\phi,\fC,\psi}\fE$ or
$\fD\amalg_\fC\fE$\G[CDEb]{$\fC\amalg_\fD\fE$}{pushout of
$C^\iy$-rings $\fC,\fD,\fE$} for the pushout of morphisms
$\phi:\fC\!\ra\!\fD$, $\psi:\fC\!\ra\!\fE$ in $\CRings$. When
$\fC\!=\!\R$, the initial object in $\CRings$, pushouts
$\fD\amalg_\R\fE$ are called {\it
coproducts\/}\I{C-ring@$C^\iy$-ring!coproduct}\I{coproduct} and are
usually written
$\fD\ot_\iy\nobreak\fE$.\G[CDEc]{$\fC\ot_\iy\fD$}{coproduct of
$C^\iy$-rings $\fC,\fD$} For $\R$-algebras $A,B$ the coproduct is the tensor product $A\ot B$. But the coproduct $\fD\ot_\iy\fE$ of $C^\iy$-rings $\fD,\fE$ is generally different from their coproduct $\fD\ot\fE$ as $\R$-algebras. For example we have $C^\iy(\R^m)\ot_\iy C^\iy(\R^n)\cong C^\iy(\R^{m+n})$, which contains but is much larger than the tensor product $C^\iy(\R^m)\ot C^\iy(\R^n)$ for~$m,n>0$.

\subsection{\texorpdfstring{$C^\iy$-rings as commutative $\R$-algebras, and ideals}{C∞-rings as commutative R-algebras, and ideals}}
\label{ag22}
\I{C-ring@$C^\iy$-ring!as commutative $\R$-algebra|(}%
\I{C-ring@$C^\iy$-ring!ideal|see{ideal in $C^\iy$-ring}}\I{ideal in
C-ring@ideal in $C^\iy$-ring|(}

Every $C^\iy$-ring $\fC$ has an underlying commutative $\R$-algebra:

\begin{dfn} Let $\fC$ be a $C^\iy$-ring. Then we may give $\fC$ the
structure of a {\it commutative\/ $\R$-algebra}. Define addition
`$+$' on $\fC$ by $c+c'=\Phi_f(c,c')$ for $c,c'\in\fC$, where
$f:\R^2\ra\R$ is $f(x,y)=x+y$. Define multiplication `$\,\cdot\,$'
on $\fC$ by $c\cdot c'=\Phi_g(c,c')$, where $g:\R^2\ra\R$ is
$f(x,y)=xy$. Define scalar multiplication by $\la\in\R$ by $\la
c=\Phi_{\la'}(c)$, where $\la':\R\ra\R$ is $\la'(x)=\la x$. Define
elements 0 and 1 in $\fC$ by $0=\Phi_{0'}(\es)$ and
$1=\Phi_{1'}(\es)$, where $0':\R^0\ra\R$ and $1':\R^0\ra\R$ are the
maps $0':\es\mapsto 0$ and $1':\es\mapsto 1$. The relations on the
$\Phi_f$ imply that all the axioms of a commutative $\R$-algebra are
satisfied. In Example \ref{ag2ex1}, this yields the obvious
$\R$-algebra structure on the smooth functions~$c:X\ra\R$.

Here is another way to say this. In an $\R$-algebra $A$, the
$n$-fold `operations' $\Phi:A^n\ra A$, that is, all the maps $A^n\ra
A$ we can construct using only addition, multiplication, scalar
multiplication, and the elements $0,1\in A$, correspond exactly to
polynomials $p:\R^n\ra\R$. Since polynomials are smooth, the
operations of an $\R$-algebra are a subset of those of a
$C^\iy$-ring, and we can truncate from $C^\iy$-rings to
$\R$-algebras. As there are many more smooth functions $f:\R^n\ra\R$
than there are polynomials, a $C^\iy$-ring has far more structure
and operations than a commutative $\R$-algebra.
\label{ag2def3}
\end{dfn}

\begin{dfn} An {\it ideal\/} $I$ in $\fC$ is an ideal $I\subset\fC$
in $\fC$ regarded as a commutative $\R$-algebra. Then we make the
quotient $\fC/I$ into a $C^\iy$-ring as follows. If $f:\R^n\ra\R$ is
smooth, define $\Phi_f^I:(\fC/I)^n\ra\fC/I$ by
\begin{equation*}
\Phi_f^I(c_1+I,\ldots,c_n+I)=\Phi_f(c_1,\ldots,c_n)+I.
\end{equation*}
To show this is well-defined, we must show it is independent of the
choice of representatives $c_1,\ldots,c_n$ in $\fC$ for
$c_1+I,\ldots,c_n+I$ in $\fC/I$. By Hadamard's Lemma\I{Hadamard's
Lemma} there exist smooth functions $g_i:\R^{2n}\ra\R$ for
$i=1,\ldots,n$ with
\begin{equation*}
f(y_1,\ldots,y_n)-f(x_1,\ldots,x_n)=\ts\sum_{i=1}^n(y_i-x_i)
g_i(x_1,\ldots,x_n,y_1,\ldots,y_n)
\end{equation*}
for all $x_1,\ldots,x_n,y_1,\ldots,y_n\in\R$. If $c_1',\ldots,c_n'$
are alternative choices for $c_1,\ab\ldots,\ab c_n$, so that
$c_i'+I=c_i+I$ for $i=1,\ldots,n$ and $c_i'-c_i\in I$, we have
\begin{equation*}
\Phi_f(c_1',\ldots,c_n')-\Phi_f(c_1,\ldots,c_n)=\ts\sum_{i=1}^n(c_i'-c_i)\Phi_{g_i}(c'_1,\ldots,c'_n,c_1,\ldots,c_n).
\end{equation*}
The second line lies in $I$ as $c_i'-c_i\in I$ and $I$ is an ideal,
so $\Phi_f^I$ is well-defined, and clearly $\bigl(\fC/I,(\Phi_f^I)_{
f:\R^n\ra\R\,\,C^\iy}\bigr)$ is a $C^\iy$-ring.

If $\fC$ is a $C^\iy$-ring, we will use the notation $(f_a:a\in A)$
to denote the ideal in $\fC$ generated by a collection of elements
$f_a$, $a\in A$ in $\fC$, in the sense of commutative $\R$-algebras.
That is,
\begin{equation*}
(f_a:a\in A)=\bigl\{\ts\sum_{i=1}^nf_{a_i}\cdot c_i:\text{$n\ge 0$,
$a_1,\ldots,a_n\in A$, $c_1,\ldots,c_n\in\fC$}\bigr\}.
\end{equation*}

\label{ag2def4}
\end{dfn}

\begin{dfn} A $C^\iy$-ring $\fC$ is called {\it finitely
generated\/}\I{C-ring@$C^\iy$-ring!finitely generated|(} if there
exist $c_1,\ldots,c_n$ in $\fC$ which generate $\fC$ over all
$C^\iy$-operations. That is, for each $c\in\fC$ there exists a smooth
map $f:\R^n\ra\R$ with $c=\Phi_f(c_1,\ldots,c_n)$. (This is a
much weaker condition than $\fC$ being finitely generated as a
commutative $\R$-algebra.)

By Kock \cite[Prop.~III.5.1]{Kock}, $C^\iy(\R^n)$ is the free
$C^\iy$-ring with $n$ generators. Given such $\fC,c_1,\ldots,c_n$,
define $\phi:C^\iy(\R^n)\ra\fC$ by $\phi(f)=\Phi_f(c_1,\ldots,c_n)$
for smooth $f:\R^n\ra\R$, where $C^\iy(\R^n)$ is as in Example
\ref{ag2ex1} with $X=\R^n$. Then $\phi$ is a surjective morphism of
$C^\iy$-rings, so $I=\Ker\phi$ is an ideal in $C^\iy(\R^n)$, and
$\fC\cong C^\iy(\R^n)/I$ as a $C^\iy$-ring. Thus, $\fC$ is finitely
generated if and only if $\fC\cong C^\iy(\R^n)/I$ for some $n\ge 0$
and ideal $I$ in~$C^\iy(\R^n)$.

An ideal $I$ in a $C^\iy$-ring $\fC$ is called {\it finitely
generated\/}\I{ideal in C-ring@ideal in $C^\iy$-ring!finitely
generated} if $I$ is a finitely generated ideal of the underlying
commutative $\R$-algebra of $\fC$ in Definition \ref{ag2def3}, that
is, $I=(i_1,\ldots,i_k)$ for some $i_1,\ldots,i_k\in\fC$. A
$C^\iy$-ring $\fC$ is called {\it finitely
presented\/}\I{C-ring@$C^\iy$-ring!finitely presented|(} if
$\fC\cong C^\iy(\R^n)/I$ for some $n\ge 0$, where $I$ is a finitely
generated ideal in~$C^\iy(\R^n)$.

A difference with conventional algebraic geometry is that
$C^\iy(\R^n)$ is not noetherian,\I{C-ring@$C^\iy$-ring!not
noetherian} so ideals in $C^\iy(\R^n)$ may not be finitely
generated, and $\fC$ finitely generated does not imply $\fC$
finitely presented.
\label{ag2def5}
\end{dfn}

Write $\CRingsfg$\G[CRingsfg]{$\CRingsfg$}{category of finitely
generated $C^\iy$-rings} and
$\CRingsfp$\G[CRingsfp]{$\CRingsfp$}{category of finitely presented
$C^\iy$-rings} for the full subcategories of finitely generated and
finitely presented $C^\iy$-rings
in~$\CRings$.\I{C-ring@$C^\iy$-ring!finitely generated|)}

\begin{ex} A {\it Weil algebra\/}\I{Weil algebra}
\cite[Def.~1.4]{Dubu1} is a finite-dimensional commutative
$\R$-algebra $W$ which has a maximal ideal $\fm$ with $W/\fm\cong\R$
and $\fm^n=0$ for some $n>0$. Then by Dubuc \cite[Prop.~1.5]{Dubu1}
or Kock \cite[Th.~III.5.3]{Kock}, there is a unique way to make $W$
into a $C^\iy$-ring compatible with the given underlying commutative
$\R$-algebra. This $C^\iy$-ring is finitely presented
\cite[Prop.~III.5.11]{Kock}. $C^\iy$-rings from Weil algebras are
important in synthetic differential geometry,\I{synthetic
differential geometry} in arguments involving infinitesimals. See
\cite[\S 2]{BuDu} for a detailed study of this.\I{ideal in
C-ring@ideal in $C^\iy$-ring|)}\I{C-ring@$C^\iy$-ring!finitely
presented|)}\I{C-ring@$C^\iy$-ring!as commutative $\R$-algebra|)}
\label{ag2ex3}
\end{ex}

\subsection{\texorpdfstring{Local $C^\iy$-rings, and localization}{Local C∞-rings, and localization}}
\label{ag23}
\I{C-ring@$C^\iy$-ring!local|(}
\begin{dfn} A $C^\iy$-ring $\fC$ is called {\it local\/} if regarded as an $\R$-algebra, as in Definition \ref{ag2def3}, $\fC$ is a local $\R$-algebra with residue field $\R$. That is, $\fC$ has a unique maximal ideal $\m_\fC$ with~$\fC/\m_\fC\cong\R$.

If $\fC,\fD$ are local $C^\iy$-rings with maximal ideals
$\m_\fC,\m_\fD$, and $\phi:\fC\ra\fD$ is a morphism of $C^\iy$
rings, then using the fact that $\fC/\m_\fC\cong\R\cong\fD/\m_\fD$
we see that $\phi^{-1}(\m_\fD)=\m_\fC$, that is, $\phi$ is a {\it
local\/} morphism of local $C^\iy$-rings. Thus, there is no
difference between morphisms and local morphisms.
\label{ag2def6}
\end{dfn}

\begin{rem} We use the term `local $C^\iy$-ring' following Dubuc \cite[Def.~4]{Dubu3}. They are also called $C^\iy$-{\it local rings\/} in Dubuc \cite[Def.~2.13]{Dubu2}, {\it pointed local\/ $C^\iy$-rings\/} in \cite[\S I.3]{MoRe2} and {\it Archimedean local\/ $C^\iy$-rings\/} in \cite[\S 3]{MQR}.

Moerdijk and Reyes \cite{MQR,MoRe1,MoRe2} use the term `local $C^\iy$-ring' to mean a $C^\iy$-ring which is a local $\R$-algebra, but which need not have residue field~$\R$.
\label{ag2rem}
\end{rem}

The next example is taken from Moerdijk and Reyes~\cite[\S I.3]{MoRe2}.

\begin{ex} Write $C^\iy(\N)$ for the $\R$-algebra of all functions $f:\N\ra\R$. It is a finitely generated $C^\iy$-ring isomorphic to $C^\iy(\R)/\{f\in C^\iy(\R):f\vert_\N=0\}$. Let $\cF$ be a {\it non-principal ultrafilter\/} on $\N$, in the sense of Comfort and Negrepontis \cite{CoNe}, and let $I\subset\fC$ be the prime ideal of $f:\N\ra\R$ such that $\{n\in\N:f(n)=0\}$ lies in $\cF$. Then $\fC=C^\iy(\N)/I$ is a finitely generated $C^\iy$-ring which is a local $\R$-algebra by \cite[Ex.~I.3.2]{MoRe2}, that is, it has a unique maximal ideal $\fm_\fC$, but its residue field is not $\R$ by \cite[Cor.~I.3.4]{MoRe2}. Hence $\fC$ is a local $C^\iy$-ring in the sense of \cite{MQR,MoRe1,MoRe2}, but not in our sense.
\label{ag2ex4}
\end{ex}

{\it Localizations\/}\I{C-ring@$C^\iy$-ring!localization|(} of
$C^\iy$-rings are studied in \cite{Dubu2,Dubu3,MQR,MoRe1},
see~\cite[p.~23]{MoRe2}.

\begin{dfn} Let $\fC$ be a $C^\iy$-ring and $S$ a subset of $\fC$. A
{\it localization\/} $\fC[s^{-1}:s\in S]$ of $\fC$ at $S$ is a
$C^\iy$-ring $\fD=\fC[s^{-1}:s\in S]$ and a morphism $\pi:\fC\ra\fD$
such that $\pi(s)$ is invertible in $\fD$ for all $s\in S$, with the
universal property that if $\fE$ is a $C^\iy$-ring and
$\phi:\fC\ra\fE$ a morphism with $\phi(s)$ invertible in $\fE$ for
all $s\in S$, then there is a unique morphism $\psi:\fD\ra\fE$ with
$\phi=\psi\ci\pi$.

A localization $\fC[s^{-1}:s\in S]$ always exists --- it can be
constructed by adjoining an extra generator $s^{-1}$ and an extra
relation $s\cdot s^{-1}-1=0$ for each $s\in S$ --- and is unique up
to unique isomorphism. When $S=\{c\}$ we have an exact sequence
$0\ra I\ra \fC\ot_\iy C^\iy(\R)\,{\buildrel\pi \over\longra}\,
\fC[c^{-1}]\ra 0$, where $\fC\ot_\iy C^\iy(\R)$ is the
coproduct\I{C-ring@$C^\iy$-ring!coproduct}\I{coproduct} of
$\fC,C^\iy(\R)$ as in \S\ref{ag21}, with pushout morphisms
$\io_1:\fC\ra\fC\ot_\iy C^\iy(\R)$, $\io_2:C^\iy(\R)\ra\fC\ot_\iy
C^\iy(\R)$, and $I$ is the ideal in $\fC\ot_\iy C^\iy(\R)$ generated
by $\io_1(c)\cdot\io_2(x)-1$, where $x$ is the generator
of~$C^\iy(\R)$.

An $\R$-{\it point\/}\I{C-ring@$C^\iy$-ring!R-point@$\R$-point} $x$
of a $C^\iy$-ring $\fC$ is a $C^\iy$-ring morphism $x:\fC\ra\R$,
where $\R$ is regarded as a $C^\iy$-ring as in Example \ref{ag2ex2}.
By \cite[Prop.~I.3.6]{MoRe2}, a map $x:\fC\ra\R$ is a morphism of
$C^\iy$-rings if and only if it is a morphism of the underlying
$\R$-algebras, as in Definition \ref{ag2def3}. Define $\fC_x$ to be
the localization $\fC_x=\fC[s^{-1}:s\in\fC$, $x(s)\ne 0]$, with
projection $\pi_x:\fC\ra\fC_x$. Then $\fC_x$ is a local $C^\iy$-ring
by \cite[Lem.~1.1]{MoRe1}. The $\R$-points of $C^\iy(\R^n)$ are just
evaluation at points $x\in\R^n$. This also holds for $C^\iy(X)$ for any manifold~$X$.
\label{ag2def7}
\end{dfn}

In a new result, we can describe these local $C^\iy$-rings $\fC_x$ explicitly. Note that the surjectivity of $\pi_x:\fC\ra\fC_x$ in the next proposition is surprising. It does not hold for general localizations of $C^\iy$-rings --- for instance, $\pi:C^\iy(\R)\ra C^\iy(\R)[x^{-1}]$ is injective but not surjective, as $x^{-1}\notin\Im\pi$ --- or for localizations $\pi_x:A\ra A_x$ of rings or $\mathbb K$-algebras in conventional algebraic geometry.

\begin{prop} Let\/ $\fC$ be a $C^\iy$-ring, $x:\fC\ra\R$ an $\R$-point of\/ $\fC,$ and $\fC_x$ the localization, with projection $\pi_x:\fC\ra\fC_x$. Then $\pi_x$ is surjective with kernel an ideal\/ $I\subset\fC,$ so that\/ $\fC_x\cong\fC/I,$ where
\e
I=\bigl\{\text{$c\in\fC:$ there exists\/ $d\in\fC$ with\/ $x(d)\ne 0$ in $\R$ and\/ $c\cdot d=0$ in $\fC$}\bigr\}.
\label{ag2eq2}
\e

\label{ag2prop2}
\end{prop}

\begin{proof} Clearly $I$ in \eq{ag2eq2} is closed under multiplication by elements of $\fC$. Let $c_1,c_2\in I$, so there exist $d_1,d_2\in\fC$ with $x(d_1)\ne 0\ne x(d_2)$ and $c_1d_1=0=c_2d_2$. Then $d_1d_2\in\fC$ with $x(d_1d_2)=x(d_1)x(d_2)\ne 0$, and $(c_1+c_2)(d_1\cdot d_2)=d_2(c_1d_1)+d_1(c_2d_2)=0$, so $c_1+c_2\in I$. Hence $I$ is an ideal, and $\fC/I$ a $C^\iy$-ring.

Suppose $c\in I$, so there exists $d\in\fC$ with $x(d)\ne 0$ and $cd=0$. Then $\pi_x(d)$ is invertible in $\fC_x$ by definition. Thus
\begin{equation*}
\pi_x(c)=\pi_x(c)\pi_x(d)\pi_x(d)^{-1}=\pi_x(cd)\pi_x(d)^{-1}=\pi_x(0)\pi_x(d)^{-1}=0.
\end{equation*}
Therefore $I\subseteq\Ker\pi_x$. So $\pi_x:\fC\ra\fC_x$ factorizes uniquely as $\pi_x=\imath\ci\pi$, where $\pi:\fC\ra\fC/I$ is the projection and $\imath:\fC/I\ra\fC_x$ is a $C^\iy$-ring morphism.

Suppose $c\in\fC$ with $x(c)\ne 0$, and write $\ep=\ha\md{x(c)}$. Choose smooth functions $\eta:\R\ra\R\sm\{0\}$, so that $\eta^{-1}:\R\ra\R\sm\{0\}$ is also smooth, such that $\eta(t)=t$ for all $t\in(x(c)-\ep,x(c)+\ep)$, and $\ze:\R\ra\R$ such that $\ze(t)=0$ for all $t\in\R\sm(x(c)-\ep,x(c)+\ep)$, so that $(\eta-\id_\R)\cdot\ze=0$, and~$\ze(x(c))=1$. 

Set $c_1=\Phi_\eta(c)$, $c_2=\Phi_{\eta^{-1}}(c)$ and $d=\Phi_\ze(c)$ in $\fC$, using the $C^\iy$-ring operations from $\eta,\eta^{-1},\ze$. Then $c_1c_2=1$ in $\fC$, as $\eta\cdot \eta^{-1}=1$, and $x(d)=x(\Phi_\ze(c))=\ze(x(c))=1$, as $x:\fC\ra\R$ is a $C^\iy$-ring morphism. Also
\begin{equation*}
(c_1-c)\cdot d=\bigl(\Phi_\eta(c)-\Phi_{\id_\R}(c)\bigr)\Phi_\ze(c)=\Phi_{(\eta-\id_\R)\ze}(c)=\Phi_0(c)=0.
\end{equation*}
Hence $c_1-c\in I$ as $x(d)\ne 0$, so $c+I=c_1+I$. But then $(c+I)(c_2+I)=(c_1+I)(c_2+I)=c_1c_2+I=1+I$ in $\fC/I$, so $\pi(c)=c+I$ is invertible in $\fC/I$. 

As this holds for all $c\in\fC$ with $x(c)\ne 0$, by the universal property of $\fC_x$ there exists a unique $C^\iy$-ring morphism $\jmath:\fC_x\ra\fC/I$ with $\pi=\jmath\ci\pi_x$. Since $\pi_x,\pi$ are surjective, $\pi_x=\imath\ci\pi$ and $\pi=\jmath\ci\pi_x$ imply that $\imath:\fC/I\ra\fC_x$ and $\jmath:\fC_x\ra\fC/I$ are inverse, so both are isomorphisms.
\end{proof}

\begin{ex} For $n\ge 0$ and $p\in\R^n$, define $C^\iy_p(\R^n)$ to
be the set of germs of smooth functions $c:\R^n\ra\R$ at $p\in\R^n$,
made into a $C^\iy$-ring in the obvious way. Then $C^\iy_p(\R^n)$ is
a local $C^\iy$-ring in the sense of Definition \ref{ag2def6}. Here
are three different ways to define $C^\iy_p(\R^n)$, which yield
isomorphic $C^\iy$-rings:
\begin{itemize}
\setlength{\itemsep}{0pt}
\setlength{\parsep}{0pt}
\item[(a)] Defining $C^\iy_p(\R^n)$ as the germs of functions
of smooth functions at $p$ means that points of $C^\iy_p(\R^n)$
are $\sim$-equivalence classes $[(U,c)]$ of pairs $(U,c)$, where
$U\subseteq\R^n$ is open with $p\in U$ and $c:U\ra\R$ is smooth,
and $(U,c)\sim(U',c')$ if there exists $p\in V\subseteq U\cap
U'$ open with $c\vert_V\equiv c'\vert_V$.
\item[(b)] As the localization $(C^\iy(\R^n))_p=
C^\iy(\R^n)[g\in C^\iy(\R^n):g(p)\ne 0]$. Then points of
$(C^\iy(\R^n))_p$ are equivalence classes $[f/g]$ of
fractions $f/g$ for $f,g\in C^\iy(\R^n)$ with $g(p)\ne 0$, and
fractions $f/g$, $f'/g'$ are equivalent if there exists $h\in
C^\iy(\R^n)$ with $h(p)\ne 0$ and $h(fg'-f'g)\equiv 0$.
\item[(c)] As the quotient $C^\iy(\R^n)/I$, where $I$ is the ideal
of $f\in C^\iy(\R^n)$ with $f\equiv 0$ near $p\in\R^n$.
\end{itemize}
One can show (a)--(c) are isomorphic using the fact that if $U$ is
any open neighbourhood of $p$ in $\R^n$ then there exists smooth
$\eta:\R^n\ra[0,1]$ such that $\eta\equiv 0$ on an open
neighbourhood of $\R^n\sm U$ in $\R^n$ and $\eta\equiv 1$ on an open
neighbourhood of $p$ in $U$. By Moerdijk and Reyes \cite[Prop.~I.3.9]{MoRe2}, any finitely generated local $C^\iy$-ring is a quotient of some~$C^\iy_p(\R^n)$.\I{C-ring@$C^\iy$-ring!local|)}\I{C-ring@$C^\iy$-ring!localization|)}
\label{ag2ex5}
\end{ex}

\subsection{\texorpdfstring{Fair $C^\iy$-rings}{Fair C∞-rings}}
\label{ag24}
\I{C-ring@$C^\iy$-ring!fair|(}

We now discuss an important class of $C^\iy$-rings, which we call
{\it fair\/} $C^\iy$-rings, for brevity. Although our term `fair' is
new, we stress that the idea is already well-known, being originally
introduced by Dubuc \cite{Dubu2}, \cite[Def.~11]{Dubu3}, who first
recognized their significance, under the name `$C^\iy$-rings of
finite type presented by an ideal of local character', and in more
recent works would be referred to as `finitely
generated\I{C-ring@$C^\iy$-ring!finitely generated} and
germ-determined $C^\iy$-rings'.

\begin{dfn} An ideal $I$ in $C^\iy(\R^n)$ is called {\it
fair\/}\I{ideal in C-ring@ideal in $C^\iy$-ring!fair|(} if for each
$f\in C^\iy(\R^n)$, $f$ lies in $I$ if and only if $\pi_p(f)$ lies
in $\pi_p(I)\subseteq C^\iy_p(\R^n)$ for all $p\in\R^n$, where
$C^\iy_p(\R^n)$ is as in Example \ref{ag2ex5} and
$\pi_p:C^\iy(\R^n)\ra C^\iy_p(\R^n)$ is the natural projection
$\pi_p:c\mapsto[(\R^n,c)]$. A $C^\iy$-ring $\fC$ is called {\it
fair\/} if it is isomorphic to $C^\iy(\R^n)/I$, where $I$ is a fair
ideal. Equivalently, $\fC$ is fair if it is finitely generated and
whenever $c\in\fC$ with $\pi_p(c)=0$ in $\fC_p$ for all $\R$-points
$p:\fC\ra\R$ then $c=0$, using the notation of
Definition~\ref{ag2def7}.

Dubuc \cite{Dubu2}, \cite[Def.~11]{Dubu3} calls fair ideals {\it
ideals of local character},\I{ideal in C-ring@ideal in
$C^\iy$-ring!of local character} and Moerdijk and Reyes
\cite[I.4]{MoRe2} call them {\it germ determined},\I{ideal in
C-ring@ideal in $C^\iy$-ring!germ-determined} which has now become
the accepted term. Fair $C^\iy$-rings are also sometimes called {\it
germ determined\/ $C^\iy$-rings},\I{C-ring@$C^\iy$-ring!germ
determined} a more descriptive term than `fair', but the definition
of germ determined $C^\iy$-rings $\fC$ in \cite[Def.~I.4.1]{MoRe2}
does not require $\fC$ finitely generated, so does not equate
exactly to our fair $C^\iy$-rings. By Dubuc \cite[Prop.~1.8]{Dubu2},
\cite[Prop.~12]{Dubu3} any finitely generated ideal $I$ is fair, so
$\fC$ finitely presented\I{C-ring@$C^\iy$-ring!finitely presented}
implies $\fC$ fair. We write
$\CRingsfa$\G[CRingsfa]{$\CRingsfa$}{category of fair $C^\iy$-rings}
for the full subcategory of fair $C^\iy$-rings in~$\CRings$.
\label{ag2def8}
\end{dfn}

\begin{prop} Suppose\/ $I\subset C^\iy(\R^m)$ and\/ $J\subset
C^\iy(\R^n)$ are ideals with\/ $C^\iy(\R^m)/I\cong C^\iy(\R^n)/J$ as
$C^\iy$-rings. Then $I$ is finitely generated, or fair, if and only
if\/ $J$ is finitely generated, or fair, respectively.
\label{ag2prop3}
\end{prop}

\begin{proof} Write $\phi:C^\iy(\R^m)/I\ra C^\iy(\R^n)/J$ for the
isomorphism, and $x_1,\ldots,x_m$ for the generators of
$C^\iy(\R^m)$, and $y_1,\ldots,y_n$ for the generators of
$C^\iy(\R^n)$. Since $\phi$ is an isomorphism we can choose
$f_1,\ldots,f_m\in C^\iy(\R^n)$ with $\phi(x_i+I)=f_i+J$ for
$i=1,\ldots,m$ and $g_1,\ldots,g_n\in C^\iy(\R^m)$ with
$\phi(g_i+I)=y_i+J$ for $i=1,\ldots,n$. It is now easy to show that
\begin{align*}
I&=\bigl(x_i-f_i\bigl(g_1(x_1,\ldots,x_m),\ldots,g_n(x_1,\ldots,x_m)
\bigr),\; i=1,\ldots,m, \\
&\quad\text{and}\quad h\bigl(g_1(x_1,\ldots,x_m),\ldots,
g_n(x_1,\ldots,x_m)\bigr) ,\; h\in J\bigr).
\end{align*}

Hence, if $J$ is generated by $h_1,\ldots,h_k$ then $I$ is generated
by $x_i-f_i(g_1,\ldots,g_n)$ for $i=1,\ldots,m$ and
$h_j(g_1,\ldots,g_n)$ for $j=1,\ldots,k$, so $J$ finitely generated
implies $I$ finitely generated. Applying the same argument to
$\phi^{-1}:C^\iy(\R^n)/J\ra C^\iy(\R^m)/I$, we see that $I$ is
finitely generated if and only if $J$ is.

Suppose $I$ is fair, and let $f\in C^\iy(\R^n)$ with
$\pi_q(f)\in\pi_q(J)\subseteq C^\iy_q(\R^n)$ for all $q\in\R^n$. We
will show that $f\in J$, so that $J$ is fair. Consider the function
$f'=f(g_1,\ldots,g_n)\in C^\iy(\R^m)$. If $p=(p_1,\ldots,p_m)$ in
$\R^m$ and $q=(q_1,\ldots,q_n)=\bigl(g_1
(p_1,\ldots,p_m),\ldots,g_n(p_1,\ldots,p_m)\bigr)$ then
$\phi:C^\iy(\R^m)/I\ra C^\iy(\R^n)/J$ localizes to an isomorphism
$\phi_p:C^\iy_p(\R^m)/\pi_p(I)\ra C^\iy_q(\R^n)/\pi_q(J)$ which maps
$\phi_p:\pi_p(f')+\pi_p(I)\mapsto\pi_q(f)+\pi_q(J)$. Since
$\pi_q(f)\in\pi_q(J)$, this gives $\pi_p(f')\in\pi_p(I)$ for all
$p\in\R^m$, so $f'\in I$ as $I$ is fair. But $\phi(f'+I)=f+J$, so
$f'\in I$ implies $f\in J$. Therefore $J$ is fair. Conversely, $J$
is fair implies $I$ is fair.
\end{proof}

\begin{ex} The local $C^\iy$-ring\I{C-ring@$C^\iy$-ring!local} $C^\iy_p(\R^n)$ of Example \ref{ag2ex5} is the quotient of
$C^\iy(\R^n)$ by the ideal $I$ of functions $f$ with $f\equiv 0$
near $p\in\R^n$. For $n>0$ this $I$ is fair, but not finitely
generated. So $C^\iy_p(\R^n)$ is fair, but not finitely
presented,\I{C-ring@$C^\iy$-ring!finitely presented} by
Proposition~\ref{ag2prop3}.
\label{ag2ex6}
\end{ex}

The following example taken from Dubuc \cite[Ex.~7.2]{Dubu4} shows
that localizations\I{C-ring@$C^\iy$-ring!localization} of fair
$C^\iy$-rings need not be fair:

\begin{ex} Let $\fC$ be the local $C^\iy$-ring\I{C-ring@$C^\iy$-ring!local} $C^\iy_0(\R)$, as in Example \ref{ag2ex5}. Then $\fC\cong C^\iy(\R)/I$, where $I$ is the ideal of all $f\in C^\iy(\R)$ with $f\equiv 0$ near 0 in $\R$. This $I$ is fair, so $\fC$ is fair. Let $c=[(x,\R)]\in\fC$. Then the localization $\fC[c^{-1}]$ is the $C^\iy$-ring of germs at 0 in $\R$ of smooth functions $\R\sm\{0\}\ra\R$. Taking $y=x^{-1}$ as a generator of $\fC[c^{-1}]$, we see that $\fC[c^{-1}]\cong C^\iy(\R)/J$, where $J$ is the ideal of compactly supported
functions in $C^\iy(\R)$. This $J$ is not fair, so by Proposition
\ref{ag2prop3}, $\fC[c^{-1}]$ is not fair.
\label{ag2ex7}
\end{ex}

Recall from category theory that if $\cC$ is a subcategory of a
category $\cD$, a {\it reflection\/}\I{functor!reflection}\I{reflection functor} $R:\cD\ra\cC$ is a left adjoint\I{functor!adjoint}\I{adjoint
functor} to the inclusion $\cC\hookra\cD$. That is, $R:\cD\ra\cC$ is
a functor with natural isomorphisms
$\Hom_\cC(R(D),C)\cong\Hom_\cD(D,C)$ for all $C\in\cC$ and
$D\in\cD$. We will define a reflection for
$\CRingsfa\subset\CRingsfg$, following Moerdijk and Reyes
\cite[p.~48--49]{MoRe2} (see also Dubuc~\cite[Th.~13]{Dubu3}).

\begin{dfn} Let $\fC$ be a finitely
generated\I{C-ring@$C^\iy$-ring!finitely generated} $C^\iy$-ring.
Let $I_\fC$ be the ideal of all $c\in\fC$ such that $\pi_p(c)=0$ in
$\fC_p$ for all $\R$-points $p:\fC\ra\R$. Then
$\fC/I_\fC$ is a finitely generated $C^\iy$-ring, with projection
$\pi:\fC\ra\fC/I_\fC$. It has the same $\R$-points as $\fC$, that
is, morphisms $p:\fC/I_\fC\ra\R$ are in 1-1 correspondence with
morphisms $p':\fC\ra\R$ by $p'=p\ci\pi$, and the local rings
$(\fC/I_\fC)_p$ and $\fC_{p'}$ are naturally isomorphic. It follows
that $\fC/I_\fC$ is fair. Define a functor $R_{\rm fg}^{\rm
fa}:\CRingsfg\ra\CRingsfa$\G[Rfg]{$R_{\rm fg}^{\rm
fa}:\CRingsfg\ra\CRingsfa$}{reflection functor} by $R_{\rm fg}^{\rm
fa}(\fC)=\fC/I_\fC$ on objects, and if $\phi:\fC\ra\fD$ is a
morphism then $\phi(I_\fC)\subseteq I_\fD$, so $\phi$ induces a
morphism $\phi_*:\fC/I_\fC\ra\fD/I_\fD$, and we set $R_{\rm fg}^{\rm
fa}(\phi)=\phi_*$. It is easy to see $\smash{R_{\rm fg}^{\rm fa}}$
is a reflection.\I{functor!reflection}\I{reflection functor}

If $I$ is an ideal in $C^\iy(\R^n)$, write $\bar I$ for the set of
$f\in C^\iy(\R^n)$ with $\pi_p(f)\in\pi_p(I)$ for all $p\in\R^n$.
Then $\bar I$ is the smallest fair ideal in $C^\iy(\R^n)$ containing
$I$, the {\it germ-determined closure\/} of $I$, and~$R_{\rm
fg}^{\rm fa}\bigl(C^\iy(\R^n)/I\bigr)\cong C^\iy(\R^n)/\bar I$.
\label{ag2def9}
\end{dfn}

\begin{ex} Let $\eta:\R\ra[0,\iy)$ be smooth with $\eta(x)>0$ for
$x\in(0,1)$ and $\eta(x)=0$ for $x\notin(0,1)$. Define $I\subseteq
C^\iy(\R)$ by
\begin{equation*}
I=\bigl\{\ts\sum_{a\in A}g_a(x)\eta(x-a):\text{$A\subset\Z$ is
finite, $g_a\in C^\iy(\R)$, $a\in A$}\bigr\}.
\end{equation*}
Then $I$ is an ideal in $C^\iy(\R)$, so $\fC=C^\iy(\R)/I$ is a
$C^\iy$-ring. The set of $f\in C^\iy(\R)$ such that $\pi_p(f)$ lies
in $\pi_p(I)\subseteq C^\iy_p(\R)$ for all $p\in\R$ is
\begin{equation*}
\bar I=\bigl\{\ts\sum_{a\in\Z}g_a(x)\eta(x-a):g_a\in C^\iy(\R),\;
a\in \Z\bigr\},
\end{equation*}
where the sum $\sum_{a\in\Z}g_a(x)\eta(x-a)$ makes sense as at most
one term is nonzero at any point $x\in\R$. Since $\bar I\ne I$, we
see that $I$ is {\it not fair}, so $\fC=C^\iy(\R)/I$ is {\it not a
fair\/ $C^\iy$-ring}. In fact $\bar I$ is the smallest fair ideal
containing $I$. We have $I_{C^\iy(\R)/I}=\bar I/I$, and~$R_{\rm
fg}^{\rm fa}\bigl(C^\iy(\R)/I)=C^\iy(\R)/\bar I$.\I{ideal in
C-ring@ideal in $C^\iy$-ring!fair|)}
\label{ag2ex8}
\end{ex}

\begin{prop} Let\/ $\fC$ be a $C^\iy$-ring, and\/ $G$ a finite
group acting on $\fC$ by automorphisms. Then the fixed subset\/
$\fC^G$\G[CDEd]{$\fC^G$}{$C^\iy$-subring fixed by finite group $G$
acting on $C^\iy$-ring $\fC$} of\/ $G$ in $\fC$ has the structure of
a $C^\iy$-ring in a unique way, such that the inclusion
$\fC^G\hookra\fC$ is a $C^\iy$-ring morphism. If\/ $\fC$ is fair, or
finitely presented,\I{C-ring@$C^\iy$-ring!finitely presented} then
$\fC^G$ is also fair, or finitely presented.
\label{ag2prop4}
\end{prop}

\begin{proof} For the first part, let $f:\R^n\ra\R$ be smooth, and
$c_1,\ldots,c_n\in\fC^G$. Then $\ga\cdot\Phi_f(c_1,\ldots,c_n)=
\Phi_f(\ga\cdot c_1,\ldots, \ga\cdot c_n)=\Phi_f(c_1,\ldots,c_n)$
for each $\ga\in G$, so $\Phi_f(c_1,\ldots,c_n)\in\fC^G$. Define
$\Phi_f^G:(\fC^G)^n\ra\fC^G$ by $\Phi_f^G=\Phi_f\vert_{(\fC^G)^n}$.
It is now trivial to check that the operations $\Phi_f^G$ for smooth
$f:\R^n\ra\R$ make $\smash{\fC^G}$ into a $C^\iy$-ring, uniquely
such that $\fC^G\hookra\fC$ is a $C^\iy$-ring morphism.

Suppose now that $\fC$ is finitely
generated.\I{C-ring@$C^\iy$-ring!finitely generated} Choose a finite
set of generators for $\fC$, and by adding the images of these
generators under $G$, extend to a set of (not necessarily distinct)
generators $x_1,\ldots,x_n$ for $\fC$, on which $G$ acts freely by
permutation. This gives an exact sequence $0\hookra I\ra
C^\iy(\R^n)\ra\fC\ra 0$, where $C^\iy(\R^n)$ is freely generated by
$x_1,\ldots,x_n$. Here $\R^n$ is a direct sum of copies of the
regular representation of $G$, and $C^\iy(\R^n)\ra\fC$ is
$G$-equivariant. Hence $I$ is a $G$-invariant ideal in
$C^\iy(\R^n)$, which is fair, or finitely generated, respectively.
Taking $G$-invariant parts gives an exact sequence $0\hookra I^G\ra
C^\iy(\R^n)^G\,{\buildrel\pi\over\longra}\,\fC^G\ra 0$, where
$C^\iy(\R^n)^G,\fC^G$ are clearly $C^\iy$-rings.

As $G$ acts linearly on $\R^n$ it acts by automorphisms on the
polynomial ring $\R[x_1,\ldots,x_n]$. By a classical theorem of Hilbert \cite[p.~274]{Weyl}, $\R[x_1,\ldots,x_n]^G$ is a finitely presented
$\R$-algebra, so we can choose generators $p_1,\ldots,p_l$ for
$\R[x_1,\ldots,x_n]^G$, which induce a surjective $\R$-algebra
morphism $\R[p_1,\ldots,p_l]\ra \R[x_1,\ldots,x_n]^G$ with kernel generated by~$q_1,\ldots,q_m\in\R[p_1,\ldots,p_l]$.

By results of Bierstone \cite{Bier} for $G$ a finite group and Schwarz \cite{Schw} for $G$ a compact Lie group, any $G$-invariant smooth function on $\R^n$ may be written as a smooth function of the generators $p_1,\ldots,p_l$ of $\R[x_1,\ldots,x_n]^G$, giving a surjective morphism
$C^\iy(\R^l)\ra C^\iy(\R^n)^G$, whose kernel is the ideal in $C^\iy(\R^l)$ generated by $q_1,\ldots,q_m$. Thus $C^\iy(\R^n)^G$ is finitely presented.

Also $\fC^G$ is generated by $\pi(p_1),\ldots,\pi(p_l)$, so $\fC^G$
is finitely generated, and we have an exact sequence $0\hookra J\ra
C^\iy(\R^l)\,{\buildrel\pi \over\longra}\,\fC^G\ra 0$, where $J$ is
the ideal in $C^\iy(\R^l)$ generated by $q_1,\ldots,q_m$ and the
lifts to $C^\iy(\R^l)$ of a generating set for the ideal $I^G$
in~$C^\iy(\R^n)^G\cong C^\iy(\R^l)/(q_1,\ldots,q_m)$.

Suppose now that $I$ is fair. Then for $f\in C^\iy(\R^n)^G$, $f$
lies in $I^G$ if and only if $\pi_p(f)\in\pi_p(I)\subseteq
C^\iy_p(\R^n)$ for all $p\in\R^n$. If $H$ is the subgroup of $G$
fixing $p$ then $H$ acts on $C^\iy_p(\R^n)$, and $\pi_p(f)$ is
$H$-invariant as $f$ is $G$-invariant, and $\pi_p(I)^H=\pi_p(I^G)$.
Thus we may rewrite the condition as $f$ lies in $I^G$ if and only
if $\pi_p(f)\in\pi_p(I^G)\subseteq C^\iy_p(\R^n)$ for all $p\in\R^n$.
Projecting from $\R^n$ to $\R^n/G$, this says that $f$ lies in $I^G$
if and only if $\pi_p(f)$ lies in $\pi_p(I^G)\subseteq
\bigl(C^\iy(\R^n)^G\bigr){}_p$ for all $p\in\R^n/G$. Since $C^\iy(\R^n)^G$ is finitely presented, it follows as in \cite[Cor.~I.4.9]{MoRe2} that $J$ is fair, so $\fC^G$ is fair.

Suppose $I$ is finitely generated in $C^\iy(\R^n)$, with generators
$f_1,\ldots,f_k$. As $\R^n$ is a sum of copies of the regular
representation of $G$, so that every irreducible representation of
$G$ occurs as a summand of $\R^n$, one can show that $I^G$ is
generated as an ideal in $C^\iy(\R^n/G)$ by the $n(k+1)$ elements
$f_i^G$ and $(f_ix_j)^G$ for $i=1,\ldots,k$ and $j=1,\ldots,n$,
where $f^G=\frac{1}{\md{G}}\sum_{\ga\in G}f\ci\ga$ is the
$G$-invariant part of $f\in C^\iy(\R^n)$. Therefore $J$ is finitely
generated by $q_1,\ldots,q_m$ and lifts of $f_i^G,(f_ix_j)^G$. Hence
if $\fC$ is finitely presented then $\fC^G$ is finitely
presented.\I{C-ring@$C^\iy$-ring!fair|)}
\end{proof}

\subsection{\texorpdfstring{Pushouts of $C^\iy$-rings}{Pushouts of C∞-rings}}
\label{ag25}
\I{category!pushout|(}\I{pushout|(}\I{C-ring@$C^\iy$-ring!pushout|(}

Proposition \ref{ag2prop1} shows that pushouts of $C^\iy$-rings
exist. For finitely generated $C^\iy$-rings, we can describe these
pushouts explicitly.

\begin{ex} Suppose the following is a pushout diagram of $C^\iy$-rings:
\begin{equation*}
\xymatrix@C=70pt@R=14pt{ \fC \ar[r]_\be \ar[d]^\al & \fE \ar[d]_\de \\
\fD \ar[r]^\ga & \fF,}
\end{equation*}
so that $\fF=\fD\amalg_\fC\fE$, with $\fC,\fD,\fE$ finitely
generated.\I{C-ring@$C^\iy$-ring!finitely generated} Then we have
exact sequences
\e
\begin{gathered}
0\ra I\hookra C^\iy(\R^l)\,{\buildrel\phi\over \longra}\,\fC\ra
0,\quad 0\ra J\hookra C^\iy(\R^m)\,{\buildrel\psi\over
\longra}\,\fD\ra 0,\\
\text{and}\quad 0\ra K\hookra C^\iy(\R^n)\,{\buildrel\chi\over
\longra}\,\fE\ra 0,
\end{gathered}
\label{ag2eq3}
\e
where $\phi,\psi,\chi$ are morphisms of $C^\iy$-rings, and $I,J,K$
are ideals in $C^\iy(\R^l),\ab C^\iy(\R^m),\ab C^\iy(\R^n)$. Write
$x_1,\ldots,x_l$ and $y_1,\ldots,y_m$ and $z_1,\ldots,z_n$ for the
generators of $C^\iy(\R^l),C^\iy(\R^m),C^\iy(\R^n)$ respectively.
Then $\phi(x_1),\ldots,\phi(x_l)$ generate $\fC$, and
$\al\ci\phi(x_1),\ldots,\al\ci\phi(x_l)$ lie in $\fD$, so we may
write $\al\ci\phi(x_i)=\psi(f_i)$ for $i=1,\ldots,l$ as $\psi$ is
surjective, where $f_i:\R^m\ra\R$ is smooth. Similarly
$\be\ci\phi(x_1),\ldots,\be\ci\phi(x_l)$ lie in $\fE$, so we may
write $\be\ci\phi(x_i)=\chi(g_i)$ for $i=1,\ldots,l$, where
$g_i:\R^n\ra\R$ is smooth.

Then from the explicit construction of pushouts of $C^\iy$-rings we
obtain an exact sequence with $\xi$ a morphism of $C^\iy$-rings
\e
\smash{\xymatrix@C=30pt{ 0 \ar[r] & L \ar[r] & C^\iy(\R^{m+n}) \ar[r]^(0.63)\xi & \fF \ar[r] & 0, }}
\label{ag2eq4}
\e
where we write the generators of $C^\iy(\R^{m+n})$ as
$y_1,\ldots,y_m,z_1,\ldots,z_n$, and then $L$ is the ideal in
$C^\iy(\R^{m+n})$ generated by the elements $d(y_1,\ldots,y_m)$ for
$d\in J\subseteq C^\iy(\R^m)$, and $e(z_1,\ldots,z_n)$ for $e\in
K\subseteq C^\iy(\R^n)$, and $f_i(y_1,\ldots,y_m)-
g_i(z_1,\ldots,z_n)$ for $i=1,\ldots,l$.

For the case of {\it
coproducts\/}\I{C-ring@$C^\iy$-ring!coproduct}\I{coproduct}
$\fD\ot_\iy\fE$, with $\fC=\R$, $l=0$ and $I=\{0\}$, we have
\begin{equation*}
\bigl(C^\iy(\R^m)/J\bigr)\ot_\iy\bigl(C^\iy(\R^n)/K\bigr)\cong
C^\iy(\R^{m+n})/(J,K).
\end{equation*}

\label{ag2ex9}
\end{ex}

\begin{prop} The subcategories\/ $\CRingsfg$ and\/ $\CRingsfp$ are
closed under pushouts and all finite
colimits\I{C-ring@$C^\iy$-ring!colimit}\I{colimit}\I{category!colimit}
in\/~$\CRings$.
\label{ag2prop5}
\end{prop}

\begin{proof} First we show $\CRingsfg,\CRingsfp$ are closed under
pushouts. Suppose $\fC,\fD,\fE$ are finitely
generated,\I{C-ring@$C^\iy$-ring!finitely generated} and use the
notation of Example \ref{ag2ex9}. Then $\fF$ is finitely generated
with generators $y_1,\ldots,y_m,z_1,\ldots,z_n$, so $\CRingsfg$ is
closed under pushouts. If $\fC,\fD,\fE$ are finitely
presented\I{C-ring@$C^\iy$-ring!finitely presented} then we can take
$J=(d_1,\ldots,d_j)$ and $K=(e_1,\ldots,e_k)$, and then Example
\ref{ag2ex9} gives
\e
\begin{split}
L=\bigl(&d_p(y_1,\ldots,y_m),\; p=1,\ldots,j,\;
e_p(z_1,\ldots,z_n),\; p=1,\ldots,k,\\ &
f_p(y_1,\ldots,y_m)-g_p(z_1,\ldots,z_n),\; p=1,\ldots,l\bigr).
\end{split}
\label{ag2eq5}
\e
So $L$ is finitely generated, and $\fF\cong C^\iy(\R^{m+n})/L$ is
finitely presented. Thus $\CRingsfp$ is closed under pushouts.

Now $\R$ is an initial object in $\CRingsfg,\CRingsfp
\subset\CRings$, and all finite colimits may be constructed by
repeated pushouts involving the initial object. Hence
$\CRingsfg,\CRingsfp$ are closed under finite colimits.
\end{proof}

Here is an example from Dubuc \cite[Ex.~7.1]{Dubu4}, Moerdijk and
Reyes~\cite[p.~49]{MoRe2}.

\begin{ex} Consider the
coproduct\I{C-ring@$C^\iy$-ring!coproduct}\I{coproduct}
$C^\iy(\R)\ot_\iy C^\iy_0(\R)$, where $C^\iy_0(\R)$ is the
$C^\iy$-ring of germs of smooth functions at 0 in $\R$ as in Example
\ref{ag2ex5}. Then $C^\iy(\R),\ab C^\iy_0(\R)$ are fair
$C^\iy$-rings,\I{C-ring@$C^\iy$-ring!fair} but $C^\iy_0(\R)$ is not
finitely presented. By Example \ref{ag2ex9}, $C^\iy(\R)\ot_\iy
C^\iy_0(\R)=C^\iy(\R)\amalg_\R C^\iy_0(\R)\cong C^\iy(\R^2)/L$,
where $L$ is the ideal in $C^\iy(\R^2)$ generated by functions
$f(x,y)=g(y)$ for $g\in C^\iy(\R)$ with $g\equiv 0$ near $0\in\R$.
This ideal $L$ is not fair, since for example one can find $f\in
C^\iy(\R^2)$ with $f(x,y)=0$ if and only if $\md{xy}\le 1$, and then
$f\notin L$ but $\pi_p(f)\in\pi_p(L)\subseteq C^\iy_p(\R^2)$ for all
$p\in\R^2$. Hence $C^\iy(\R)\ot_\iy C^\iy_0(\R)$ is not a fair
$C^\iy$-ring, by Proposition \ref{ag2prop3}, and pushouts of fair
$C^\iy$-rings need not be fair.
\label{ag2ex10}
\end{ex}

Our next result is referred to in the last part of
Dubuc~\cite[Th.~13]{Dubu3}.

\begin{prop} $\CRingsfa$\I{C-ring@$C^\iy$-ring!fair} is not closed
under pushouts in\/ $\CRings$. Nonetheless, pushouts and all finite
colimits\I{C-ring@$C^\iy$-ring!colimit}\I{colimit}\I{category!colimit}
exist in\/ $\CRingsfa,$ although they may not coincide with pushouts
and finite colimits in~$\CRings$.
\label{ag2prop6}
\end{prop}

\begin{proof} Example \ref{ag2ex10} shows that $\CRingsfa$ is not
closed under pushouts in $\CRings$. To construct finite colimits in
$\CRingsfa$, we first take the colimit in $\CRingsfg$, which exists
by Propositions \ref{ag2prop1} and \ref{ag2prop5}, and then apply
the reflection functor\I{functor!reflection}\I{reflection functor}
$R_{\rm fg}^{\rm fa}$. By the universal properties of colimits and
reflection functors, the result is a
colimit in~$\CRingsfa$.\I{category!pushout|)}\I{pushout|)}%
\I{C-ring@$C^\iy$-ring!pushout|)}
\end{proof}

\subsection{Flat ideals}
\label{ag26}
\I{ideal in C-ring@ideal in $C^\iy$-ring!flat|(}

The following class of ideals in $C^\iy(\R^n)$ is defined by
Moerdijk and Reyes \cite[p.~47, p.~49]{MoRe2} (see also Dubuc
\cite[\S 1.7(a)]{Dubu2}), who call them {\it flat ideals\/}:

\begin{dfn} Let $X$ be a closed subset of $\R^n$. Define
$\fm_X^\iy$\G[mXi]{$\fm_X^\iy$}{flat ideal of $f\in C^\iy(\R^n)$
vanishing to all orders on $X\subseteq\R^n$} to be the ideal of all
functions $g\in C^\iy(\R^n)$ such that $\pd^kg\vert_X\equiv 0$ for
all $k\ge 0$, that is, $g$ and all its derivatives vanish at each
$x\in X$. If the interior $X^\ci$ of $X$ in $\R^n$ is dense in $X$,
that is $\smash{\ov{(X^\ci)}}=X$, then $\pd^kg\vert_X\equiv 0$ for
all $k\ge 0$ if and only if $g\vert_X\equiv 0$. In this case
$C^\iy(\R^n)/\fm_X^\iy\cong C^\iy(X):=\bigl\{f\vert_X:f\in
C^\iy(\R^n)\bigr\}$.
\label{ag2def10}
\end{dfn}

Flat ideals are always fair.\I{ideal in C-ring@ideal in
$C^\iy$-ring!fair} Here is an example from~\cite[Th.~I.1.3]{MoRe2}.

\begin{ex} Take $X$ to be the point $\{0\}$. If $f,f'\in
C^\iy(\R^n)$ then $f-f'$ lies in $\fm_{\{0\}}^\iy$ if and only if
$f,f'$ have the same Taylor series at 0. Thus
$C^\iy(\R^n)/\fm_{\{0\}}^\iy$ is the $C^\iy$-ring of Taylor series
at 0 of $f\in C^\iy(\R^n)$. Since any formal power series in
$x_1,\ldots,x_n$ is the Taylor series of some $f\in C^\iy(\R^n)$, we
have $C^\iy(\R^n)/\fm_{\{0\}}^\iy\cong\R[[x_1,\ldots,x_n]]$. Thus
the $\R$-algebra of formal power series $\R[[x_1,\ldots,x_n]]$ can
be made into a $C^\iy$-ring.
\label{ag2ex11}
\end{ex}

The following nontrivial result is proved by Reyes and van Qu\^e
\cite[Th.~1]{ReVa}, generalizing an unpublished result of A.P.
Calder\'on in the case $X=Y=\{0\}$. It can also be found in Moerdijk
and Reyes~\cite[Cor.~I.4.12]{MoRe2}.

\begin{prop} Let\/ $X\subseteq\R^m$ and\/ $Y\subseteq\R^n$ be
closed. Then as ideals in $C^\iy(\R^{m+n})$ we
have\/~$(\fm_X^\iy,\fm_Y^\iy)=\fm_{X\t Y}^\iy$.
\label{ag2prop7}
\end{prop}

Moerdijk and Reyes \cite[Cor.~I.4.19]{MoRe2} prove:

\begin{prop} Let\/ $X\subseteq\R^n$ be closed with\/ $X\ne\es,\R^n$.
Then the ideal\/ $\fm_X^\iy$ in $C^\iy(\R^n)$ is not countably
generated.
\label{ag2prop8}
\end{prop}

We can use these to study $C^\iy$-rings of manifolds with
corners.\I{manifold with corners!C-ring of@$C^\iy$-ring of}

\begin{ex} Let $0<k\le n$, and consider the closed subset
$\R^n_k=[0,\iy)^k\t\R^{n-k}$ in $\R^n$, the local model for
manifolds with corners. Write $C^\iy(\R^n_k)$ for the $C^\iy$-ring
$\bigl\{f\vert_{\R^n_k}:f\in C^\iy(\R^n)\bigr\}$. Since the interior
$(\R^n_k)^\ci=(0,\iy)^k\t\R^{n-k}$ of $\R^n_k$ is dense in $\R^n_k$,
as in Definition \ref{ag2def10} we have
$C^\iy(\R^n_k)=C^\iy(\R^n)/\smash{\fm_{\R^n_k}^\iy}$. As
$\smash{\fm_{\R^n_k}^\iy}$ is not countably generated by Proposition
\ref{ag2prop8}, it is not finitely generated, and thus
$C^\iy(\R^n_k)$ is not a finitely presented
$C^\iy$-ring,\I{C-ring@$C^\iy$-ring!finitely presented} by
Proposition~\ref{ag2prop3}.

Consider the coproduct\I{C-ring@$C^\iy$-ring!coproduct}\I{coproduct}
$C^\iy(\R^m_k)\ot_\iy C^\iy(\R^n_l)$ in $\CRings$, that is, the
pushout\I{category!pushout}\I{pushout} $C^\iy(\R^m_k)\amalg_\R
C^\iy(\R^n_l)$ over the trivial $C^\iy$-ring $\R$. By Example
\ref{ag2ex9} and Proposition \ref{ag2prop7} we have
\begin{align*}
C^\iy(\R^m_k)\ot_\iy C^\iy(\R^n_l)&\cong C^\iy(\R^{m+n})/
(\fm_{\R^m_k}^\iy,\fm_{\R^n_l}^\iy)=
C^\iy(\R^{m+n})/\fm_{\R^m_k\t\R^n_l}^\iy\\
&=C^\iy(\R^m_k\t \R^n_l)\cong C^\iy(\R^{m+n}_{k+l}).
\end{align*}
This is an example of Theorem \ref{ag3thm} below, with
$X\!=\!\R^m_k$, $Y\!=\!\R^n_l$ and~$Z\!=\!*$.\I{ideal in
C-ring@ideal in $C^\iy$-ring!flat|)}
\label{ag2ex12}
\end{ex}

\section{\texorpdfstring{The $C^\iy$-ring $C^\iy(X)$ of a manifold $X$}{The C∞-ring of a manifold}}
\label{ag3}
\I{C-ring@$C^\iy$-ring!of a manifold $X$|(}\I{manifold|(}\I{manifold
with corners|(}\I{manifold!with boundary|see{manifold with
corners}}\I{manifold!with corners|see{manifold with corners}}

We now study the $C^\iy$-rings $C^\iy(X)$ of manifolds $X$ defined
in Example \ref{ag2ex1}. We are interested in {\it manifolds without
boundary\/} (locally modelled on $\R^n$), and in {\it manifolds with
boundary\/} (locally modelled on $[0,\iy)\t\R^{n-1}$), and in {\it
manifolds with corners\/} (locally modelled on
$[0,\iy)^k\t\R^{n-k}$). Manifolds with corners were considered by
the author \cite{Joyc1,Joyc6}, and we use the conventions of those papers.

The $C^\iy$-rings of manifolds with boundary are discussed by Reyes
\cite{Reye} and Kock \cite[\S III.9]{Kock}, but Kock appears to have
been unaware of Proposition \ref{ag2prop7}, which makes
$C^\iy$-rings of manifolds with boundary easier to understand.

If $X,Y$ are manifolds with corners of dimensions $m,n$, then
\cite[\S 2.1]{Joyc6} defined $f:X\ra Y$ to be {\it weakly
smooth\/}\I{manifold with corners!weakly smooth map} if $f$ is
continuous and whenever $(U,\phi),(V,\psi)$ are charts on $X,Y$ then
$\psi^{-1}\ci f\ci\phi:(f\ci\phi)^{-1}(\psi(V))\ra V$ is a smooth
map from $(f\ci\phi)^{-1}(\psi(V))\subset\R^m$ to $V\subset\R^n$. A
{\it smooth map\/}\I{manifold with corners!smooth map} is a weakly
smooth map $f$ satisfying some extra conditions over $\pd^kX,\pd^lY$ in \cite[\S 2.1]{Joyc6}. If $\pd Y=\es$ these conditions are vacuous, so for manifolds without boundary, weakly smooth maps and smooth maps coincide. Write $\Man,\Manb,\Manc$\G[Manb]{$\Manb$}{category of manifolds with boundary}\G[Manc]{$\Manc$}{category of manifolds with corners} for the categories of manifolds without boundary, and with boundary, and with corners, respectively, with morphisms smooth maps.

\begin{prop}{\bf(a)} If\/ $X$ is a manifold without boundary then
the $C^\iy$-ring $C^\iy(X)$ of Example\/ {\rm\ref{ag2ex1}} is
finitely presented.\I{C-ring@$C^\iy$-ring!finitely presented|(}
\smallskip

\noindent{\bf(b)} If\/ $X$ is a manifold with boundary, or with
corners, and\/ $\pd X\ne\es,$ then the $C^\iy$-ring $C^\iy(X)$ of
Example\/ {\rm\ref{ag2ex1}} is fair,\I{C-ring@$C^\iy$-ring!fair} but
is not finitely presented.
\label{ag3prop1}
\end{prop}

\begin{proof} Part (a) is proved in Dubuc \cite[p.~687]{Dubu3} and
Moerdijk and Reyes \cite[Th.~I.2.3]{MoRe2} following an observation
of Lawvere, that if $X$ is a manifold without boundary then we can
choose a closed embedding $i:X\hookra\R^N$ for $N\gg 0$, and then
$X$ is a retract of an open neighbourhood $U$ of $i(X)$, so we have
an exact sequence $0\ra I\ra C^\iy(\R^N)\,\smash{{\buildrel
i^*\over\longra}} \,C^\iy(X)\ra 0$ in which the ideal $I$ is
finitely generated, and thus the $C^\iy$-ring $C^\iy(X)$ is finitely
presented.

For (b), if $X$ is an $n$-manifold with boundary, or with corners,
then roughly by gluing on a `collar' $\pd X\t(-\ep,0]$ to $X$ along $\pd X$ for small $\ep>0$, we can embed $X$ as a closed subset in an $n$-manifold $X'$ without boundary, such that the inclusion $X\hookra X'$ is locally
modelled on the inclusion of $\R^n_k=[0,\iy)^k\t\R^{n-k}$ in $(-\ep,\iy)^k\t\R^{n-k}$ for $k\le n$. Choose a closed embedding $i:X'\hookra\R^N$ for $N\gg 0$ as above, giving $0\ra I'\ra C^\iy(\R^N)\,{\buildrel i^*\over\longra}
\,C^\iy(X')\ra 0$ with $I'$ generated by $f_1,\ldots,f_k\in
C^\iy(\R^N)$. Let $i(X')\subset T\subset\R^N$ be an open tubular neighbourhood of $i(X')$ in $\R^N$, with projection $\pi:T\ra i(X')$. Set $U=\pi^{-1}(i(X^\ci))\subset T\subset\R^N$, where $X^\ci$ is the interior of $X$. Then $U$ is open in $\R^N$ with $i(X^\ci)=U\cap i(X')$, and the closure $\bar U$ of $U$ in $\R^N$ has~$i(X)=\bar U\cap i(X')$.

Let $I$ be the ideal $(f_1,\ldots,f_k,\fm_{\bar U}^\iy)$ in
$C^\iy(\R^N)$. Then $I$ is fair, as $(f_1,\ldots,f_k)$ and
$\fm_{\bar U}^\iy$ are fair. Since $U$ is open in $\R^N$ and dense
in $\bar U$, as in Definition \ref{ag2def10} we have $g\in\fm_{\bar
U}^\iy$ if and only if $g\vert_{\bar U}\equiv 0$. Therefore the
isomorphism $(i_*)_*:C^\iy(\R^N)/I'\ra C^\iy(X')$ identifies the
ideal $I/I'$ in $C^\iy(X')$ with the ideal of $f\in C^\iy(X')$ such
that $f\vert_X\equiv 0$, since $X=i^{-1}(\bar U)$. Hence
\begin{equation*}
C^\iy(\R^N)/I\!\cong\!C^\iy(X')/\bigl\{f\!\in\! C^\iy(X'):f\vert_X
\!\equiv\!0\bigr\}\!\cong\!\bigl\{f\vert_X\!:\!f\!\in\! C^\iy(X')
\bigr\}\!\cong\! C^\iy(X).
\end{equation*}
As $I$ is a fair ideal, this implies that $C^\iy(X)$ is a fair
$C^\iy$-ring.\I{C-ring@$C^\iy$-ring!fair} If $\pd X\ne\es$ then
using Proposition \ref{ag2prop8} we can show $I$ is not countably
generated, so $C^\iy(X)$ is not finitely presented by
Proposition~\ref{ag2prop3}.\I{C-ring@$C^\iy$-ring!finitely
presented|)}
\end{proof}

Next we consider the transformation $X\mapsto C^\iy(X)$ as a
functor.

\begin{dfn} Write $\CRings^{\bf op}$, $(\CRingsfp)^{\bf op}$,
$(\CRingsfa)^{\bf op}$ for the opposite categories of
$\CRings,\CRingsfp,\CRingsfa$ (i.e.\ directions of morphisms are
reversed). Define functors
\begin{align*}
F_\Man^\CRings:\Man&\longra(\CRingsfp)^{\bf op}\subset
\CRings^{\bf op},\\
F_\Manb^\CRings:\Manb&\longra(\CRingsfa)^{\bf op}\subset
\CRings^{\bf op},\\
F_\Manc^\CRings:\Manc&\longra(\CRingsfa)^{\bf op}\subset \CRings^{\bf op}
\end{align*}
as follows. On objects the functors $F_{\Man^*}^\CRings$ map
$X\mapsto C^\iy(X)$, where $C^\iy(X)$ is a $C^\iy$-ring as in
Example \ref{ag2ex1}. On morphisms, if $f:X\ra Y$ is a smooth map of
manifolds then $f^*:C^\iy(Y)\ra C^\iy(X)$ mapping $c\mapsto c\ci f$
is a morphism of $C^\iy$-rings, so that $f^*:C^\iy(Y)\ra C^\iy(X)$
is a morphism in $\CRings$, and $f^*:C^\iy(X)\ra C^\iy(Y)$ a
morphism in $\CRings^{\bf op}$, and $F_{\Man^*}^\CRings$ maps
$f\mapsto f^*$. Clearly $F_\Man^\CRings,F_\Manb^\CRings,
F_\Manc^\CRings$ are functors.
\label{ag3def}
\end{dfn}

If $f:X\ra Y$ is only {\it weakly smooth\/}\I{manifold with
corners!weakly smooth map|(} then $f^*:C^\iy(Y)\ra C^\iy(X)$ in
Definition \ref{ag3def} is still a morphism of $C^\iy$-rings. From
\cite[Prop.~I.1.5]{MoRe2} we deduce:

\begin{prop} Let\/ $X,Y$ be manifolds with corners. Then the map
$f\mapsto f^*$ from weakly smooth maps\/ $f:X\ra Y$ to morphisms
of\/ $C^\iy$-rings $\phi:C^\iy(Y)\ra C^\iy(X)$ is a $1$-$1$
correspondence.
\label{ag3prop2}
\end{prop}

In the category of manifolds $\Man$, the morphisms are weakly smooth maps. So $F_\Man^\CRings$ is both injective on morphisms (faithful),\I{functor!faithful|(} and surjective on morphisms (full),\I{functor!full|(} as in Moerdijk and Reyes
\cite[Th.~I.2.8]{MoRe2}. But in $\Manb,\Manc$ the morphisms are
smooth maps, a proper subset of weakly smooth maps, so the functors
are injective but not surjective on morphisms. That is:

\begin{cor} The functor $F_\Man^\CRings:\Man\ra(\CRingsfp)^{\rm
op}$ is full and faithful. However, the functors $F_\Manb^\CRings:
\Manb\ra(\CRingsfa)^{\bf op}$ and\/ $F_\Manc^\CRings:
\Manc\ra(\CRingsfa)^{\bf op}$ are faithful, but not full.
\label{ag3cor}
\end{cor}

Of course, if we defined $\Manb,\Manc$ to have morphisms weakly
smooth maps, then $F_\Manb^\CRings,F_\Manc^\CRings$ would be full
and faithful.\I{manifold with corners!weakly smooth
map|)}\I{functor!faithful|)}\I{functor!full|)}

Let $X,Y,Z$ be manifolds and $f:X\ra Z$, $g:Y\ra Z$ be smooth maps.
If $X,Y,Z$ are without boundary then $f,g$ are called {\it
transverse\/}\I{manifold!transverse fibre product|(} if whenever
$x\in X$ and $y\in Y$ with $f(x)=g(y)=z\in Z$ we have $T_zZ=\d
f(T_xX)+\d g(T_yY)$. If $f,g$ are transverse then a fibre
product\I{fibre product}\I{category!fibre product} $X\t_ZY$ exists
in~$\Man$.

For manifolds with boundary, or with corners, the situation is more
complicated, as explained in \cite[\S 6]{Joyc1}, \cite[\S 4.3]{Joyc6}. In the definition of {\it smooth\/}\I{manifold with corners!smooth map} $f:X\ra Y$ we impose extra conditions over $\pd^jX,\pd^kY$, and in the definition of transverse\I{manifold with corners!transverse fibre product|(} $f,g$ we impose extra conditions over $\pd^jX,\pd^kY,\pd^lZ$. With these more restrictive definitions of smooth and transverse maps,
transverse fibre products exist in $\Manc$ by \cite[Th.~6.3]{Joyc1} (see also \cite[Th.~4.27]{Joyc6}). The na\"\i ve definition of transversality is not a sufficient condition for fibre products to exist. Note too that a fibre product of manifolds with boundary may be a manifold with corners, so fibre products work best in $\Man$ or $\Manc$ rather than~$\Manb$.

Our next theorem is given in \cite[Th.~16]{Dubu3} and
\cite[Prop.~I.2.6]{MoRe2} for manifolds without boundary, and the
special case of products $\Man\t\Manb\ra\Manb$ follows from Reyes
\cite[Th.~2.5]{Reye}, see also Kock \cite[\S III.9]{Kock}. It can be
proved by combining the usual proof in the without boundary case,
the proof of \cite[Th.~6.3]{Joyc1}, and Proposition~\ref{ag2prop7}.

\begin{thm} The functors\/ $F_\Man^\CRings,F_\Manc^\CRings$ preserve
transverse fibre products in $\Man,\Manc,$ in the sense of\/
{\rm\cite[\S 6]{Joyc1}}. That is, if the following is a Cartesian
square\I{Cartesian square} of manifolds with\/ $g,h$ transverse
\e
\begin{gathered}
\xymatrix@C=70pt@R=14pt{ W \ar[r]_f \ar[d]^e & Y \ar[d]_h \\
X \ar[r]^g & Z,}
\end{gathered}
\label{ag3eq1}
\e
so that\/ $W=X\t_{g,Z,h}Y,$ then we have a
pushout\I{category!pushout}\I{pushout} square of\/ $C^\iy$-rings
\e
\begin{gathered}
\xymatrix@C=100pt@R=14pt{ *+[r]{C^\iy(Z)} \ar[r]_{h^*} \ar[d]^{g^*} &
*+[l]{C^\iy(Y)} \ar[d]_{f^*} \\ *+[r]{C^\iy(X)} \ar[r]^{e^*} & *+[l]{C^\iy(W),\!} }
\end{gathered}
\label{ag3eq2}
\e
so that\/ $C^\iy(W)=C^\iy(X)\amalg_{g^*,C^\iy(Z),h^*}
C^\iy(Y)$.\I{C-ring@$C^\iy$-ring|)}\I{C-ring@$C^\iy$-ring!of a
manifold $X$|)}\I{manifold|)}\I{manifold with corners|)}\I{manifold
with corners!transverse fibre product|)}\I{manifold!transverse fibre
product|)}
\label{ag3thm}
\end{thm}

\section{\texorpdfstring{$C^\iy$-ringed spaces and $C^\iy$-schemes}{C∞-ringed spaces and C∞-schemes}}
\label{ag4}

In algebraic geometry, if $A$ is an affine scheme and $R$ the ring
of regular functions on $A$, then we can recover $A$ as the spectrum
of the ring $R$, $A\cong\Spec R$. One of the ideas of synthetic
differential geometry,\I{synthetic differential geometry} as in
\cite[\S I]{MoRe2}, is to regard a manifold $X$ as the `spectrum' of
the $C^\iy$-ring $C^\iy(X)$ in Example \ref{ag2ex1}. So we can try
to develop analogues of the tools of scheme theory for smooth
manifolds, replacing rings by $C^\iy$-rings throughout. This was
done by Dubuc \cite{Dubu2,Dubu3}. The analogues of the algebraic
geometry notions \cite[\S II.2]{Hart} of ringed spaces, locally
ringed spaces, and schemes, are called $C^\iy$-ringed spaces, local
$C^\iy$-ringed spaces and $C^\iy$-schemes. The material of \S\ref{ag46}--\S\ref{ag49} is new.

\subsection{Some basic topology}
\label{ag41}

Later we will use several properties of topological spaces, e.g. second countable, metrizable, Lindel\"of, \ldots, so we now recall their definitions and some relationships between them. Let $X$ be a topological space, with topology $\cT$. Then:
\begin{itemize}
\setlength{\itemsep}{0pt}
\setlength{\parsep}{0pt}
\item A {\it basis\/} for $\cT$ is a family $\cB\subseteq\cT$ such that every open set in $X$ is a union of sets in $\cB$. We call $X$ {\it second countable\/}\I{topological space!second countable} if $\cT$ has a countable basis.
\item An open cover $\{U_i:i\in I\}$ of $X$ is {\it locally finite\/} if every $x\in X$ has an open neighbourhood $W$ with $W\cap U_i\ne\es$ for only finitely many $i\in I$.

An open cover $\{V_j:j\in J\}$ of $X$ is a {\it refinement\/} of another open cover $\{U_i:i\in I\}$ if for all $j\in J$ there exists $i\in I$ with $V_j\subseteq U_i\subseteq X$.

We call $X$ {\it paracompact\/}\I{topological space!paracompact} if every open cover $\{U_i:i\in I\}$ of $X$ admits a locally finite refinement $\{V_j:j\in J\}$.
\item We call $X$ {\it Hausdorff\/}\I{topological space!Hausdorff} if for all $x,y\in X$ with $x\ne y$ there exist open $U,V\subseteq X$ with $x\in U$, $y\in V$ and $U\cap V=\es$.
\item We call $X$ {\it metrizable\/}\I{topological space!metrizable} if there exists a metric on $X$ inducing topology $\cT$.
\item We call $X$ {\it regular\/}\I{topological space!regular} if for every closed subset $C\subseteq X$ and each $x\in X\sm C$ there exist disjoint open sets $U,V\subseteq X$ with $C\subseteq U$ and~$x\in V$.
\item We call $X$ {\it completely regular\/}\I{topological space!completely regular} if for every closed $C\subseteq X$ and $x\in X\sm C$ there exists a continuous $f:X\ra[0,1]$ with $f\vert_C=0$ and $f(x)=1$.
\item We call $X$ {\it separable\/}\I{topological space!separable} if it has a countable dense subset $S\subseteq X$.
\item We call $X$ {\it locally compact\/}\I{topological space!locally compact} if for all $x\in X$ there exist $x\in U\subseteq C\subseteq X$ with $U$ open and $C$ compact.
\item We call $X$ {\it Lindel\"of\/}\I{topological space!Lindel\"of} if every open cover of $X$ has a countable subcover.
\end{itemize}

By well known results in topology, including Urysohn's metrization theorem, the following are equivalent:
\begin{itemize}
\setlength{\itemsep}{0pt}
\setlength{\parsep}{0pt}
\item[(i)] $X$ is Hausdorff, second countable and regular.
\item[(ii)] $X$ is second countable and metrizable.
\item[(iii)] $X$ is separable and metrizable.
\end{itemize}

Here are some useful implications:
\begin{itemize}
\setlength{\itemsep}{0pt}
\setlength{\parsep}{0pt}
\item $X$ Hausdorff and locally compact imply $X$ is regular.
\item $X$ metrizable implies $X$ is Hausdorff, paracompact, and regular.
\item $X$ second countable implies $X$ is Lindel\"of.
\item $X$ Lindel\"of and regular imply $X$ is paracompact.
\end{itemize}

\subsection{Sheaves on topological spaces}
\label{ag42}
\I{sheaf!on topological space|(}\I{sheaf|(}

Sheaves are a fundamental concept in algebraic geometry. They are
necessary even to define schemes, since a scheme is a topological
space $X$ equipped with a sheaf of rings $\O_X$. In this book,
sheaves of $C^\iy$-rings, and sheaves of modules over a sheaf of
$C^\iy$-rings, play a fundamental r\^ole.

We now summarize some basics of sheaf theory, following Hartshorne
\cite[\S II.1]{Hart}. A more detailed reference is Godement
\cite{Gode}. We concentrate on sheaves of abelian groups; to define
sheaves of $C^\iy$-rings, etc., one replaces abelian groups with
$C^\iy$-rings, etc., throughout. This is justified since limits in all these categories (including abelian groups) are computed at the level of underlying sets, because they are all algebras for algebraic theories.\I{sheaf!definition|(}

\begin{dfn} Let $X$ be a topological space. A {\it presheaf of
abelian groups\/}\I{sheaf!presheaf}\I{presheaf} $\cE$ on $X$
consists of the data of an abelian group $\cE(U)$ for every open set
$U\subseteq X$, and a morphism of abelian groups
$\rho_{UV}:\cE(U)\ra\cE(V)$ called the {\it restriction map\/} for
every inclusion $V\subseteq U\subseteq X$ of open sets, satisfying
the conditions that
\begin{itemize}
\setlength{\itemsep}{0pt}
\setlength{\parsep}{0pt}
\item[(i)] $\cE(\es)=0$;
\item[(ii)] $\rho_{UU}=\id_{\cE(U)}:\cE(U)\ra\cE(U)$ for all open
$U\subseteq X$; and
\item[(iii)] $\rho_{UW}=\rho_{VW}\ci\rho_{UV}:\cE(U)\ra\cE(W)$ for all
open~$W\subseteq V\subseteq U\subseteq X$.
\end{itemize}
That is, a presheaf is a functor $\cE:\mathop{\bf Open}(X)^{\bf op}\ra\mathop{\bf AbGp}$, where $\mathop{\bf Open}(X)$ is the category of open subsets of $X$ with morphisms inclusions, and $\mathop{\bf AbGp}$ is the category of abelian groups.

A presheaf of abelian groups $\cE$ on $X$ is called a {\it
sheaf\/}\I{sheaf!of abelian groups} if it also satisfies
\begin{itemize}
\setlength{\itemsep}{0pt}
\setlength{\parsep}{0pt}
\item[(iv)] If $U\subseteq X$ is open, $\{V_i:i\in I\}$ is an open
cover of $U$, and $s\in\cE(U)$ has $\rho_{UV_i}(s)=0$ in
$\cE(V_i)$ for all $i\in I$, then $s=0$ in $\cE(U)$; and
\item[(v)] If $U\subseteq X$ is open, $\{V_i:i\in I\}$ is an open cover of
$U$, and we are given elements $s_i\in\cE(V_i)$ for all $i\in I$
such that $\rho_{V_i(V_i\cap V_j)}(s_i)=\rho_{V_j(V_i\cap
V_j)}(s_j)$ in $\cE(V_i\cap V_j)$ for all $i,j\in I$, then there
exists $s\in\cE(U)$ with $\rho_{UV_i}(s)=s_i$ for all $i\in I$.
This $s$ is unique by~(iv).
\end{itemize}

Suppose $\cE,\cF$ are presheaves or sheaves of abelian groups on
$X$. A {\it morphism\/} $\phi:\cE\ra\cF$ consists of a morphism of
abelian groups $\phi(U):\cE(U)\ra\cF(U)$ for all open $U\subseteq
X$, such that the following diagram commutes for all open
$V\subseteq U\subseteq X$
\begin{equation*}
\xymatrix@C=90pt@R=14pt{
*+[r]{\cE(U)} \ar[r]_{\phi(U)} \ar[d]^{\rho_{UV}} & *+[l]{\cF(U)}
\ar[d]_{\rho_{UV}'} \\ *+[r]{\cE(V)} \ar[r]^{\phi(V)} & *+[l]{\cF(V),\!} }
\end{equation*}
where $\rho_{UV}$ is the restriction map for $\cE$, and $\rho_{UV}'$
the restriction map for~$\cF$.\I{sheaf!definition|)}
\label{ag4def1}
\end{dfn}

\begin{dfn} Let $\cE$ be a presheaf of abelian groups on $X$. For each
$x\in X$, the {\it stalk\/}\I{sheaf!stalk} $\cE_x$ is the direct
limit of the groups $\cE(U)$ for all $x\in U\subseteq X$, via the
restriction maps $\rho_{UV}$. It is an abelian group. A morphism
$\phi:\cE\ra\cF$ induces morphisms $\phi_x:\cE_x\ra\cF_x$ for all
$x\in X$. If $\cE,\cF$ are sheaves then $\phi$ is an isomorphism if
and only if $\phi_x$ is an isomorphism for all~$x\in X$.
\label{ag4def2}
\end{dfn}

Sheaves of abelian groups on $X$ form an {\it abelian
category\/}\I{abelian category} $\Sh(X)$.\G[Sh(X)]{$\Sh(X)$}{category
of sheaves of abelian groups on topological space $X$} Thus we have
(category-theoretic) notions of when a morphism $\phi:\cE\ra\cF$ in
$\Sh(X)$ is {\it injective\/} or {\it surjective\/} ({\it epimorphic\/}), and when a sequence $\cE\ra\cF\ra\cG$ in $\Sh(X)$ is {\it exact}. It turns out that $\phi:\cE\ra\cF$ is injective if and only if
$\phi(U):\cE(U)\ra\cF(U)$ is injective for all open $U\subseteq X$.
However $\phi:\cE\ra\cF$ surjective does not imply that
$\phi(U):\cE(U)\ra\cF(U)$ is surjective for all open $U\subseteq X$.
Instead, $\phi$ is surjective if and only if $\phi_x:\cE_x\ra\cF_x$
is surjective for all~$x\in X$.

\begin{dfn} Let $\cE$ be a presheaf of abelian groups on $X$. A
{\it sheafification\/}\I{presheaf!sheafification} of $\cE$ is a
sheaf of abelian groups $\hat\cE$ on $X$ and a morphism
$\pi:\cE\ra\hat\cE$, such that whenever $\cF$ is a sheaf of abelian
groups on $X$ and $\phi:\cE\ra\cF$ is a morphism, there is a unique
morphism $\hat\phi:\hat\cE\ra\cF$ with $\phi=\hat\phi\ci\pi$. As in
\cite[Prop.~II.1.2]{Hart}, a sheafification always exists, and is
unique up to canonical isomorphism; one can be constructed
explicitly using the stalks\I{sheaf!stalk} $\cE_x$ of~$\cE$.
\label{ag4def3}
\end{dfn}

Next we discuss {\it pushforwards\/} and {\it pullbacks\/} of
sheaves by continuous maps.

\begin{dfn} Let $f:X\ra Y$ be a continuous map of topological
spaces, and $\cE$ a sheaf of abelian groups on $X$. Define the {\it
pushforward\/}\I{sheaf!pushforward} ({\it direct
image\/})\I{sheaf!direct image} sheaf
$f_*(\cE)$\G[fEa]{$f_*(\cE)$}{pushforward (direct image) sheaf} on
$Y$ by $\bigl(f_*(\cE)\bigr)(U)=\cE\bigl(f^{-1}(U)\bigr)$ for all
open $U\subseteq V$, with restriction maps $\rho'_{UV}=
\rho_{f^{-1}(U)f^{-1}(V)}:\bigl(f_*(\cE)\bigr)(U)\ra
\bigl(f_*(\cE)\bigr)(V)$ for all open $V\subseteq U\subseteq Y$.
Then $f_*(\cE)$ is a sheaf of abelian groups on~$Y$.

If $\phi:\cE\ra\cF$ is a morphism in $\Sh(X)$ we define
$f_*(\phi):f_*(\cE)\ra f_*(\cF)$ by
$\bigl(f_*(\phi)\bigr)(u)=\phi\bigl(f^{-1}(U)\bigr)$ for all open
$U\subseteq Y$. Then $f_*(\phi)$ is a morphism in $\Sh(Y)$, and
$f_*$ is a functor $\Sh(X)\ra\Sh(Y)$. It is a left exact
functor\I{functor!left exact} between abelian categories,\I{abelian
category} but in general is not exact.\I{functor!exact} For
continuous maps $f:X\ra Y$, $g:Y\ra Z$ we have~$(g\ci f)_*=g_*\ci
f_*$.
\label{ag4def4}
\end{dfn}

\begin{dfn} Let $f:X\ra Y$ be a continuous map of topological
spaces, and $\cE$ a sheaf of abelian groups on $Y$. Define a
presheaf $\cP f^{-1}(\cE)$ on $X$ by $\bigl(\cP
f^{-1}(\cE)\bigr)(U)=\lim_{A\supseteq f(U)}\cE(A)$ for open
$A\subseteq X$, where the direct limit is taken over all open
$A\subseteq Y$ containing $f(U)$, using the restriction maps
$\rho_{AB}$ in $\cE$. For open $V\subseteq U\subseteq X$, define
$\rho_{UV}': \bigl(\cP f^{-1}(\cE)\bigr)(U)\ra \bigl(\cP
f^{-1}(\cE)\bigr)(V)$ as the direct limit of the morphisms
$\rho_{AB}$ in $\cE$ for $B\subseteq A\subseteq Y$ with
$f(U)\subseteq A$ and $f(V)\subseteq B$. Then we define the {\it
pullback\/}\I{sheaf!pullback} ({\it inverse
image\/})\I{sheaf!inverse image}
$f^{-1}(\cE)$\G[fEb]{$f^{-1}(\cE)$}{pullback (inverse image) sheaf}
to be the sheafification\I{presheaf!sheafification} of the
presheaf~$\cP f^{-1}(\cE)$.

Pullbacks $f^{-1}(\cE)$ are only unique up to canonical isomorphism,
rather than unique. By convention we choose once and for all a
pullback $f^{-1}(\cE)$ for all $X,Y,f,\cE$, using the Axiom of
Choice\I{Axiom of Choice} if necessary. If $\phi:\cE\ra\cF$ is a
morphism in $\Sh(Y)$, one can define a pullback morphism
$f^{-1}(\phi):f^{-1}(\cE)\ra f^{-1}(\cF)$. Then
$f^{-1}:\Sh(Y)\ra\Sh(X)$ is an exact functor\I{functor!exact}
between abelian categories.
\label{ag4def5}
\end{dfn}

We compare pushforwards and pullbacks:

\begin{rem}{\bf(a)} There are two kinds of pullback, with slightly
different notation. The first kind, written $f^{-1}(\cE)$ as in
Definition \ref{ag4def5}, is used for sheaves of abelian groups or
$C^\iy$-rings. The second kind, written $\uf^*(\cE)$ or $f^*(\cE)$
and discussed in \S\ref{ag53} and \S\ref{ag83}, is used for sheaves
of $\O_Y$-modules~$\cE$.
\smallskip

\noindent{\bf(b)} The definition of pushforward sheaves $f_*(\cE)$
is wholly elementary. In contrast, the definition of pullbacks
$f^{-1}(\cE)$ is complex, involving a direct limit followed by a
sheafification, and includes arbitrary choices.

Pushforwards $f_*$ are strictly functorial in the continuous map
$f:X\ra Y$, that is, for continuous $f:X\ra Y$, $g:Y\ra Z$ we have
$(g\ci f)_*=g_*\ci f_*:\Sh(X)\ra\Sh(Z)$. However, pullbacks $f^{-1}$
are only weakly functorial in $f$: if $\cE\in\Sh(Z)$ then we need
not have $(g\ci f)^{-1}(\cE)=f^{-1}(g^{-1}(\cE))$. This is because
pullbacks are only natural up to canonical isomorphism, and we make
an arbitrary choice for each pullback. So although
$f^{-1}(g^{-1}(\cE))$ is a possible pullback for $\cE$ by $g\ci f$,
it may not be the one we chose.

Thus, there is a canonical isomorphism $(g\ci f)^{-1}(\cE)\cong
f^{-1}(g^{-1}(\cE))$, which we will write as $I_{f,g}(\cE):(g\ci
f)^{-1}(\cE)\ra f^{-1}(g^{-1}(\cE))$.\G[Ifga]{$I_{f,g}(\cE):(g\ci
f)^{-1}(\cE)\ra f^{-1}(g^{-1}(\cE))$}{isomorphism of pullback
sheaves} The $I_{f,g}(\cE)$ for all $\cE\in\Sh(Z)$ comprise a
natural isomorphism of functors $I_{f,g}:(g\ci f)^{-1}\Ra f^{-1}\ci
g^{-1}$. Similarly, for $\cE\in\Sh(X)$ we may not have
$\id^{-1}_X(\cE)=\cE$, but instead there are canonical isomorphisms
$\de_X(\cE):\id^{-1}_X (\cE)\ra\cE$,\G[deXa]{$\de_X(\cE):\id^{-1}_X
(\cE)\ra\cE$}{canonical isomorphism of pullback sheaves} which make
up a natural isomorphism $\de_X:\id_X^{-1}\Ra\id_{\Sh(X)}$. Many
authors ignore the natural isomorphisms $I_{f,g},\de_X$ entirely.
\smallskip

\noindent{\bf(c)} Let $f:X\ra Y$ be a continuous map of topological
spaces. Then we have functors $f_*:\Sh(X)\ra\Sh(Y)$, and
$f^{-1}:\Sh(Y)\ra\Sh(X)$. As in \cite[Ex.~II.1.18]{Hart}, $f_*$ is
right adjoint\I{functor!adjoint}\I{adjoint functor} to $f^{-1}$.
That is, there is a natural bijection
\e
\Hom_X\bigl(f^{-1}(\cE),\cF\bigr)\cong\Hom_Y\bigl(\cE,f_*(\cF)\bigr)
\label{ag4eq1}
\e
for all $\cE\in\Sh(Y)$ and $\cF\in\Sh(X)$, with functorial
properties.\I{sheaf!on topological space|)}\I{sheaf|)}
\label{ag4rem1}
\end{rem}

We define {\it fine\/} sheaves, as in Godement \cite[\S II.3.7]{Gode} or Voisin \cite[Def.~4.35]{Vois}. They will be important in \S\ref{ag47} and~\S\ref{ag53}.

\begin{dfn} Let $X$ be a topological space (usually paracompact), and $\cE$ a sheaf of abelian groups on $X$, or more generally a sheaf of rings, or $C^\iy$-rings, or $\O_X$-modules, or any other objects which are also abelian groups. We call $\cE$ {\it fine\/}\I{fine sheaf}\I{sheaf!fine}\I{fine sheaf} if for any open cover $\{U_i:i\in I\}$ of $X$, a subordinate locally finite partition of unity $\{\ze_i:i\in I\}$\I{partition of unity} exists in the sheaf~${\cal H}om(\cE,\cE)$. 

Here $\ze_i:\cE\ra\cE$ is a morphism of sheaves of abelian groups (or rings, $C^\iy$-rings, \ldots) for each $i\in I$. For $\{\ze_i:i\in I\}$ to be {\it subordinate to\/} $\{U_i:i\in I\}$ means that $\ze_i$ is supported in $U_i$ for each $i\in I$, that is, there exists open $V_i\subseteq X$ with $\ze_i\vert_{V_i}=0$ and $U_i\cup V_i=X$. For $\{\ze_i:i\in I\}$ to be {\it locally finite\/} means that each $x\in X$ has an open neighbourhood $W$ with $\ze_i\vert_W\ne 0$ for only finitely many $i\in I$. For $\{\ze_i:i\in I\}$ to be a {\it partition of unity\/} means that $\sum_{i\in I}\ze_i=\id_\cE$, where the sum makes sense as $\{\ze_i:i\in I\}$ is locally finite.

If $\cE=\O_X$ is a sheaf of commutative rings or $C^\iy$-rings, then writing $\eta_i=\ze_i(1)$ in $\O_X(X)$, we see that $\ze_i=\eta_i\cdot{}$ is multiplication by $\eta_i$. So we can regard the partition of unity as living in $\O_X(X)$ rather than~${\cal H}om(\O_X,\O_X)$.
\label{ag4def6}
\end{dfn}

\subsection{\texorpdfstring{$C^\iy$-ringed spaces and local $C^\iy$-ringed spaces}{C∞-ringed spaces and local C∞-ringed spaces}}
\label{ag43}
\I{C-ringed space@$C^\iy$-ringed space|(}\I{sheaf!of C-rings@of
$C^\iy$-rings|(}

\begin{dfn} A {\it $C^\iy$-ringed space\/} $\uX=(X,\O_X)$ is a
topological space $X$ with a sheaf $\O_X$ of $C^\iy$-rings on $X$.
That is, for each open set $U\subseteq X$ we are given a $C^\iy$
ring $\O_X(U)$, and for each inclusion of open sets $V\subseteq
U\subseteq X$ we are given a morphism of $C^\iy$-rings
$\rho_{UV}:\O_X(U)\ra\O_X(V)$, called the {\it restriction maps},
and all this data satisfies the sheaf axioms in
Definition~\ref{ag4def1}.

Equivalently, $\O_X$ is a presheaf of $C^\iy$-rings on $X$, that is, a functor
\begin{equation*}
\O_X:\mathop{\bf Open}(X)^{\bf op}\longra\CRings,
\end{equation*}
whose underlying presheaf of abelian groups, or of sets, is a sheaf. The sheaf axioms Definition \ref{ag4def1}(iv),(v) do not use the $C^\iy$-ring structure.

A {\it morphism\/}\I{C-ringed space@$C^\iy$-ringed
space!morphism}\I{C-scheme@$C^\iy$-scheme!morphism}\G[fXYa]{$\uf:\uX\ra\uY$}{morphism of $C^\iy$-ringed spaces or $C^\iy$-schemes} $\uf=(f,f^\sh):(X,\O_X)\ra (Y,\O_Y)$ of $C^\iy$ ringed spaces is a continuous map $f:X\ra Y$ and a morphism
$f^\sh:f^{-1}(\O_Y)\ra\O_X$\G[fsha]{$f^\sh:f^{-1}(\O_Y)\ra\O_X$}{morphism
of sheaves of $C^\iy$-rings in $\uf:\uX\ra\uY$} of sheaves of
$C^\iy$-rings on $X$, for $f^{-1}(\O_Y)$ as in Definition
\ref{ag4def5}. Since $f_*$ is right
adjoint\I{functor!adjoint}\I{adjoint functor} to $f^{-1}$, as in
\eq{ag4eq1} there is a natural bijection
\e
\Hom_X\bigl(f^{-1}(\O_Y),\O_X\bigr)\cong\Hom_Y\bigl(\O_Y,f_*(\O_X)\bigr).
\label{ag4eq2}
\e
Write $f_\sh:\O_Y\ra f_*(\O_X)$\G[fshb]{$f_\sh:\O_Y\ra
f_*(\O_X)$}{morphism of sheaves of $C^\iy$-rings in $\uf:\uX\ra\uY$}
for the morphism of sheaves of $C^\iy$-rings on $Y$ corresponding to
$f^\sh$ under \eq{ag4eq2}, so that
\e
f^\sh:f^{-1}(\O_Y)\longra\O_X\quad \leftrightsquigarrow\quad
f_\sh:\O_Y\longra f_*(\O_X).
\label{ag4eq3}
\e

If $\uf:\uX\ra\uY$ and $\ug:\uY\ra\uZ$ are $C^\iy$-scheme morphisms,
the composition is
\begin{equation*}
\ug\ci\uf=\bigl(g\ci f,(g\ci f)^\sh\bigr)=\bigl(g\ci f,f^\sh\ci
f^{-1}(g^\sh)\ci I_{f,g}(\O_Z)\bigr),
\end{equation*}
where $I_{f,g}(\O_Z):(g\ci f)^{-1}(\O_Z)\ra f^{-1}(g^{-1}(\O_Z))$ is
the canonical isomorphism from Remark \ref{ag4rem1}(b). In terms of
$f_\sh:\O_Y\ra f_*(\O_X)$, composition is
\begin{equation*}
(g\ci f)_\sh=g_*(f_\sh)\ci g_\sh:\O_Z\longra (g\ci f)_*(\O_X)=g_*\ci
f_*(\O_X).
\end{equation*}

A {\it local\/ $C^\iy$-ringed space\/}\I{C-ringed
space@$C^\iy$-ringed space!local|(} $\uX=(X,\O_X)$ is a
$C^\iy$-ringed space for which the stalks\I{sheaf!stalk} $\O_{X,x}$
of $\O_X$ at $x$ are local $C^\iy$-rings\I{C-ring@$C^\iy$-ring!local} for all $x\in X$. As in Definition \ref{ag2def6}, since morphisms of local $C^\iy$-rings are automatically local morphisms, morphisms of local $C^\iy$-ringed spaces $(X,\O_X),(Y,\O_Y)$ are just morphisms of $C^\iy$-ringed spaces, without any additional locality condition. Moerdijk, van
Qu\^ e and Reyes \cite[\S 3]{MQR} call our local $C^\iy$-ringed
spaces {\it Archimedean\/ $C^\iy$-spaces}.

Write $\CRS$\G[CRS]{$\CRS$}{category of $C^\iy$-ringed spaces} for
the category of $C^\iy$-ringed spaces, and
$\LCRS$\G[LCRS]{$\LCRS$}{category of local $C^\iy$-ringed spaces}
for the full subcategory of local $C^\iy$-ringed spaces.

For brevity, we will use the notation that underlined upper case
letters $\uX,\uY,\uZ,\ldots$ represent $C^\iy$-ringed spaces
$(X,\O_X),(Y,\O_Y),(Z,\O_Z),\ldots,$ and underlined lower case
letters $\uf,\ug,\ldots$ represent morphisms of $C^\iy$-ringed
spaces $(f,f^\sh),\ab(g,g^\sh),\ab\ldots.$ When we write `$x\in\uX$'
we mean that $\uX=(X,\O_X)$ and $x\in X$. When we write `$\uU$ {\it
is open in\/} $\uX$' we mean that $\uU=(U,\O_U)$ and $\uX=(X,\O_X)$
with $U\subseteq X$ an open set and $\O_U=\O_X\vert_U$.
\label{ag4def7}
\end{dfn}

\begin{rem} As above, there are two equivalent ways to write morphisms of $C^\iy$-ringed spaces $(X,\O_X)\ra(Y,\O_Y)$, either using pullbacks as $(f,f^\sh)$ for $f^\sh:f^{-1}(\O_Y)\ra\O_X$, or using pushforwards as $(f,f_\sh)$ for $f_\sh:\O_Y\ra f_*(\O_X)$. Each definition has advantages and disadvantages. We choose to regard $f^\sh:f^{-1}(\O_Y)\ra\O_X$ as the primary object, and so define morphisms of $C^\iy$-ringed spaces as
$(f,f^\sh)$ rather than $(f,f_\sh)$, although we will use $f_\sh$ in a few places. We can always switch between the two points of view using~\eq{ag4eq3}.
\label{ag4rem3}
\end{rem}

\begin{ex} Let $X$ be a manifold,\I{manifold!C-scheme
of@$C^\iy$-scheme of}\I{manifold with corners!C-scheme
of@$C^\iy$-scheme of} which may have boundary or corners. Define a
$C^\iy$-ringed space $\uX=(X,\O_X)$ to have topological space $X$
and $\O_X(U)=C^\iy(U)$ for each open subset $U\subseteq X$, where
$C^\iy(U)$ is the $C^\iy$-ring of smooth maps $c:U\ra\R$, and if
$V\subseteq U\subseteq X$ are open we define $\rho_{UV}:C^\iy(U)\ra
C^\iy(V)$ by~$\rho_{UV}:c\mapsto c\vert_V$.

It is easy to verify that $\O_X$ is a sheaf of $C^\iy$-rings on $X$
(not just a presheaf), so $\uX=(X,\O_X)$ is a $C^\iy$-ringed space.
For each $x\in X$, the stalk\I{sheaf!stalk} $\O_{X,x}$ is the
local $C^\iy$-ring\I{C-ring@$C^\iy$-ring!local} of germs $[(c,U)]$ of
smooth functions $c:X\ra\R$ at $x\in X$, as in Example \ref{ag2ex5},
with unique maximal ideal ${\mathfrak
m}_{X,x}=\bigl\{[(c,U)]\in\O_{X,x}:c(x)=0\big\}$ and
$\O_{X,x}/{\mathfrak m}_{X,x}\cong\R$. Hence $\uX$ is a local
$C^\iy$-ringed space.

Let $X,Y$ be manifolds and $f:X\ra Y$ a weakly smooth
map.\I{manifold with corners!weakly smooth map} Define
$(X,\O_X),(Y,\O_Y)$ as above. For all open $U\subseteq Y$ define
$f_\sh(U):\O_Y(U)= C^\iy(U)\ra\O_X(f^{-1}(U))=C^\iy(f^{-1}(U))$ by
$f_\sh(U):c\mapsto c\ci f$ for all $c\in C^\iy(U)$. Then $f_\sh(U)$
is a morphism of $C^\iy$-rings, and $f_\sh:\O_Y\ra f_*(\O_X)$ is a
morphism of sheaves of $C^\iy$-rings on $Y$. Let
$f^\sh:f^{-1}(\O_Y)\ra\O_X$ correspond to $f_\sh$ under \eq{ag4eq3}.
Then $\uf=(f,f^\sh): (X,\O_X)\ra(Y,\O_Y)$ is a morphism of (local)
$C^\iy$-ringed spaces.
\label{ag4ex1}
\end{ex}

As the category $\Top$ of topological spaces has all finite limits,
and the construction of $\CRS$ involves $\Top$ in a covariant way
and the category $\CRings$ in a contravariant way, using Proposition
\ref{ag2prop1} one may prove:

\begin{prop} All finite limits\I{category!limit} exist in the
category\/~$\CRS$.
\label{ag4prop1}
\end{prop}

Dubuc \cite[Prop.~7]{Dubu3} proves:

\begin{prop} The full subcategory $\LCRS$ of local\/ $C^\iy$-ringed
spaces in $\CRS$ is closed under finite limits in~$\CRS$.\I{C-ringed
space@$C^\iy$-ringed space!local|)}\I{sheaf!of C-rings@of
$C^\iy$-rings|)}
\label{ag4prop2}
\end{prop}

\subsection{The spectrum functor}
\label{ag44}

We now define a {\it spectrum functor\/} $\Spec:\CRings^{\bf op}\ra\LCRS$. It is equivalent to those constructed by Dubuc \cite{Dubu2,Dubu3} and Moerdijk, van Qu\^e and Reyes \cite[\S 3]{MQR}, but our presentation is closer to that of Hartshorne~\cite[p.~70]{Hart}.

\begin{dfn} Let $\fC$ be a $C^\iy$-ring, and use the notation of Definition \ref{ag2def7}. Write $X_\fC$ for the set of all $\R$-points $x$ of $\fC$. Let ${\cal T}_\fC$ be the topology on $X_\fC$ generated by the basis of open sets $U_c=\bigl\{x\in X_\fC:x(c)\ne 0\bigr\}$ for all~$c\in\fC$.

For each $c\in\fC$ define $c_*:X_\fC\ra\R$ to map $c_*:x\mapsto x(c)$.
\label{ag4def8}
\end{dfn}

\begin{ex} Suppose $\fC$ is a finitely generated $C^\iy$-ring,\I{C-ring@$C^\iy$-ring!finitely generated} with exact sequence $0\ra I\hookra C^\iy(\R^n)\smash{\,{\buildrel\phi\over\longra}}\,\fC\ra 0$. Define a map $\phi_*:X_\fC\ra\R^n$ by $\phi_*:x\mapsto\bigl(x\ci\phi(x_1),\ldots,x\ci\phi(x_n)\bigr)$,
where $x_1,\ldots,x_n$ are the generators of $C^\iy(\R^n)$. Then
$\phi_*$ gives a homeomorphism
\e
\phi_*:X_\fC\,{\buildrel\cong\over\longra}\,X_\fC^\phi=\bigl\{
(x_1,\ldots,x_n)\in\R^n:\text{$f(x_1,\ldots,x_n)=0$ for all $f\in
I$}\bigr\},
\label{ag4eq4}
\e
where the right hand side is a closed subset of $\R^n$. So the
topological spaces $(X_\fC,\cT_\fC)$ for finitely generated $\fC$ are
homeomorphic to closed subsets of $\R^n$. 
\label{ag4ex2}
\end{ex}

Recall that a topological space $X$ is {\it regular\/} if whenever $S\subseteq X$ is closed and $x\in X\sm S$ then there exist open $U,V\subseteq X$ with $x\in U$, $S\subseteq V$ and~$U\cap V=\es$.

\begin{lem} In Definition\/ {\rm\ref{ag4def8},} the topology $\cT_\fC$ is also generated by the basis of open sets $c_*^{-1}(V)$ for all\/ $c\in\fC$ and open $V\subseteq\R$. That is, $\cT_\fC$ is the weakest topology on $X_\fC$ such that\/ $c_*:X_\fC\ra\R$ is continuous for all\/ $c\in\fC$. Also $(X_\fC,\cT_\fC)$ is a Hausdorff, regular topological space.
\label{ag4lem1}
\end{lem}

\begin{proof} Suppose $c\in\fC$ and $V\subseteq\R$ is open. Then there exists smooth $f:\R\ra\R$ with $V=\{x\in\R:f(x)\ne 0\}$. Set $c'=\Phi_f(c)$, using the $C^\iy$-ring operation $\Phi_f:\fC\ra\fC$. Then $c'_*=f\ci c_*$ as $c:\fC\ra\R$ is a $C^\iy$-ring morphism, so
\begin{equation*}
U_{c'}=(c'_*)^{-1}(\R\sm\{0\})=(f\ci c_*)^{-1}(\R\sm\{0\})=c_*^{-1}[f^{-1}(0)]=c_*^{-1}(V).
\end{equation*}
So $c_*^{-1}(V)$ is of the form $U_{c'}$. Conversely $U_c=c_*^{-1}(V)$ for $V=\R\sm\{0\}\subseteq\R$. So the two given bases for $\cT_\fC$ are the same, proving the first part.

Let $x,y$ be distinct points of $X_\fC$. Then there exists $c\in\fC$ with $x(c)\ne y(c)$, as $x\ne y$. Set $\ep=\ha\md{x(c)-y(c)}>0$ and $U=c_*^{-1}\bigl((x(c)-\ep,x(c)+\ep)\bigr)$, $V=c_*^{-1}\bigl((y(c)-\ep,y(c)+\ep)\bigr)$. Then $U,V\subseteq X_\fC$ are disjoint open sets with $x\in U$, $y\in V$, so $(X_\fC,\cT_\fC)$ is Hausdorff.

Suppose $S\subseteq X_\fC$ is closed, and $x\in X\sm S$. Then there exists $c\in\fC$ with $x\in U_c\subseteq X_\fC\sm S$, since $X_\fC\sm S$ is open in $X_\fC$ and the $U_c$ are a basis for $\cT_\fC$. Therefore $c_*(x)\ne 0$ and $c_*\vert_S=0$. Set $\ep=\ha\md{c_*(x)}>0$, $U=c_*^{-1}\bigl((c_*(x)-\ep,c_*(x)+\ep)\bigr)$ and $V=c_*^{-1}\bigl((-\ep,\ep)\bigr)$. Then $U,V\subseteq X_\fC$ are disjoint open sets with $x\in U$, $S\subseteq V$, so $(X_\fC,\cT_\fC)$ is regular.
\end{proof}

\begin{dfn} Let $\fC$ be a $C^\iy$-ring, and $X_\fC$ the topological space from Definition \ref{ag4def8}. For each open $U\subseteq X_\fC$, define $\O_{X_\fC}(U)$ to be the set of functions $s:U\ra \coprod_{x\in U}\fC_x$ with $s(x)\in\fC_x$ for all $x\in U$, and such that $U$ may be covered by open sets $W\subseteq U\subseteq X_\fC$ for which there exist $c\in\fC$ with $s(x)=\pi_x(c)\in\fC_x$ for all $x\in W$. Define operations $\Phi_f$ on $\O_{X_\fC}(U)$ pointwise in $x\in U$ using the operations $\Phi_f$ on $\fC_x$. This makes $\O_{X_\fC}(U)$ into a $C^\iy$-ring. If $V\subseteq U\subseteq X_\fC$ are open, the restriction map $\rho_{UV}:\O_{X_\fC}(U)\ra\O_{X_\fC}(V)$ mapping $\rho_{UV}:s\mapsto s\vert_V$ is a morphism of $C^\iy$-rings.

Clearly $\O_{X_\fC}$ is a sheaf of $C^\iy$-rings on $X_\fC$. Lemma \ref{ag4lem2} shows that the stalk\I{sheaf!stalk} $\O_{X_\fC,x}$ at $x\in X_\fC$ is $\fC_x$, which is a local $C^\iy$-ring.\I{C-ring@$C^\iy$-ring!local} Hence $(X_\fC,\O_{X_\fC})$ is a local $C^\iy$-ringed space,\I{C-ringed space@$C^\iy$-ringed space!local} which we call the {\it spectrum\/}\I{C-ring@$C^\iy$-ring!spectrum functor $\Spec$} of $\fC$, and write as~$\Spec\fC$.\G[Spec]{$\Spec:\CRings^{\bf op}\ra\LCRS$}{spectrum functor on $C^\iy$-rings}

Now let $\phi:\fC\ra\fD$ be a morphism of $C^\iy$-rings. Define
$f_\phi:X_\fD\ra X_\fC$ by $f_\phi(x)=x\ci\phi$. Then $f_\phi$ is
continuous. For $U\subseteq X_\fC$ open define
$(f_\phi)_\sh(U):\O_{X_\fC}(U)\ra\O_{X_\fD}(f_\phi^{-1}(U))$ by
$(f_\phi)_\sh(U)s:x\mapsto \phi_x(s(f_\phi(x)))$, where
$\phi_x:\fC_{f_\phi(x)}\ra\fD_x$ is the induced morphism of
local $C^\iy$-rings. Then $(f_\phi)_\sh:\O_{X_{\fC}}\ra (f_\phi)_*
(\O_{X_\fD})$ is a morphism of sheaves of $C^\iy$-rings on $X_\fC$.
Let $f_\phi^\sh:f_\phi^{-1}(\O_{X_{\fC}})\ra\O_{X_\fD}$ be the
corresponding morphism of sheaves of $C^\iy$-rings on $X_\fD$ under
\eq{ag4eq3}. Then $\uf_\phi=(f_\phi,f_\phi^\sh): (X_\fD,\O_{X_\fD})\ra
(X_\fC,\O_{X_\fC})$ is a morphism of local $C^\iy$-ringed spaces.
Define $\Spec\phi:\Spec\fD\ra\Spec\fC$ by $\Spec\phi=\uf_\phi$. Then
$\Spec$ is a functor $\CRings^{\bf op}\ra\LCRS$, the {\it spectrum functor}.
\label{ag4def9}
\end{dfn}

\begin{ex} Let $X$ be a manifold.\I{manifold!C-scheme of@$C^\iy$-scheme of}\I{manifold with corners!C-scheme of@$C^\iy$-scheme of} Then it follows from Theorem \ref{ag4thm6} below that the local $C^\iy$-ringed space\I{C-ringed space@$C^\iy$-ringed space!local} $\uX$ constructed in Example \ref{ag4ex1} is naturally isomorphic to~$\Spec C^\iy(X)$.
\label{ag4ex3}
\end{ex}

\begin{lem} In Definition\/ {\rm\ref{ag4def9},} the stalk\/ $\O_{X_\fC,x}$ of\/ $\O_{X_\fC}$ at\/ $x\in X_\fC$ is naturally isomorphic to\/~$\fC_x$.
\label{ag4lem2}
\end{lem}

\begin{proof} Elements of $\O_{X_\fC,x}$ are $\sim$-equivalence classes $[U,s]$ of pairs $(U,s)$, where $U$ is an open neighbourhood of $x$ in $X_\fC$ and $s\in\O_{X_\fC}(U)$, and $(U,s)\sim(U',s')$ if there exists open $x\in V\subseteq U\cap U'$ with $s\vert_V=s'\vert_V$. Define a $C^\iy$-ring morphism $\Pi:\O_{X_\fC,x}\ra\fC_x$ by $\Pi:[U,s]\mapsto s(x)$. 

Suppose $c_x\in\fC_x$. Then $c_x=\pi_x(c)$ for some $c\in\fC$ by Proposition \ref{ag2prop2}. The map $s:X_\fC\ra \coprod_{x'\in X_\fC}\fC_{x'}$ mapping $s:x'\mapsto\pi_{x'}(c)$ lies in $\O_{X_\fC}(X_\fC)$, and $\Pi:[X_\fC,s]\mapsto\pi_x(c)=c_x$. Hence $\Pi:\O_{X_\fC,x}\ra\fC_x$ is surjective.

Suppose $[U,s]\in \O_{X_\fC,x}$ with $\Pi([U,s])=0\in\fC_x$. As $s\in\O_{X_\fC}(U)$, there exist open $x\in V\subseteq U$ and $c\in\fC$ with $s(x')=\pi_{x'}(c)\in\fC_{x'}$ for all $x'\in V$. Then $\pi_x(c)=s(x)=\Pi([U,s])=0$, so $c$ lies in the ideal $I$ in \eq{ag2eq2} by Proposition \ref{ag2prop2}. Thus there exists $d\in\fC$ with $x(d)\ne 0$ in $\R$ and $cd=0$ in $\fC$. Set $W=\{x'\in V:x'(d)\ne 0\}$, so that $W$ is an open neighbourhood of $x$ in $U$. If $x'\in W$ then $x'(d)\ne 0$, so $\pi_{x'}(d)$ is invertible in $\fC_{x'}$. Thus
\begin{equation*}
s(x')=\pi_{x'}(c)=\pi_{x'}(c)\pi_{x'}(d)\pi_{x'}(d)^{-1}=\pi_{x'}(cd)\pi_{x'}(d)^{-1}=\pi_{x'}(0)\pi_{x'}(d)^{-1}=0.
\end{equation*}
Hence $[U,s]=[W,s\vert_W]=[W,0]=0$ in $\O_{X_\fC,x}$, so $\Pi:\O_{X_\fC,x}\ra\fC_x$ is injective. Thus $\Pi:\O_{X_\fC,x}\ra\fC_x$ is an isomorphism.
\end{proof}

\begin{dfn} The {\it global sections functor\/}\I{global sections functor
$\Ga$|(} $\Ga:\LCRS\ra\CRings^{\bf op}$\G[Gaa]{$\Ga:\LCRS\ra\CRings^{\bf op}$}{global sections functor on $C^\iy$-ringed spaces} acts on objects $(X,\O_X)$ by
$\Ga:(X,\O_X)\mapsto\O_X(X)$ and on morphisms $(f,f^\sh):(X,\O_X)\ra(Y,\O_Y)$ by $\Ga:(f,f^\sh)\mapsto f_\sh(Y)$, for $f_\sh:\O_Y\ra f_*(\O_X)$ as in~\eq{ag4eq3}. 

Then $\Ga\ci\Spec$ is a functor $\CRings^{\bf op}\ra\CRings^{\bf op}$, or equivalently a functor $\CRings\ra\CRings$. For each $C^\iy$-ring $\fC$ and $c\in\fC$, define $\Psi_\fC(c):X_\fC\ra\coprod_{x\in X_\fC}\fC_x$ by $\Psi_\fC(c):x\mapsto\pi_x(c)\in\fC_x$. Then $\Psi_\fC(c)\in\O_{X_\fC}(X_\fC)=\Ga\ci\Spec\fC$ by Definition \ref{ag4def9}, so $\Psi_\fC:\fC\ra \Ga\ci\Spec\fC$\G[Psia]{$\Psi_\fC:\fC\ra\Ga\ci\Spec\fC$}{canonical morphism for a $C^\iy$-ring $\fC$} is a map. Since $\pi_x:\fC\ra\fC_x$ is a $C^\iy$-ring morphism and the $C^\iy$-ring operations on $\O_{X_\fC}(X_\fC)$ are defined pointwise in the $\fC_x$, this $\Psi_\fC$ is a $C^\iy$-ring morphism. It is functorial in $\fC$, so that the $\Psi_\fC$ for all $\fC$ define a natural transformation $\Psi:\id_\CRings\Ra\Ga\ci\Spec$ of functors~$\id_\CRings,\Ga\ci\Spec:\CRings\ra\CRings$.
\label{ag4def10}
\end{dfn}

\begin{thm} The functor\/ $\Spec:\CRings^{\bf op}\ra\LCRS$ is \begin{bfseries}right adjoint\end{bfseries}\I{functor!adjoint|(}\I{adjoint functor|(} to $\Ga:\LCRS\ra\CRings^{\bf op}$. That is, for all\/ $\fC\in\CRings$ and\/ $\uX\in\LCRS$ there are inverse bijections
\e
\xymatrix@C=150pt{ *+[r]{\Hom_\CRings(\fC,\Ga(\uX))} \ar@<.5ex>[r]^{L_{\fC,\uX}}
& *+[l]{\Hom_\LCRS(\uX,\Spec\fC),} \ar@<.5ex>[l]^{R_{\fC,\uX}} }
\label{ag4eq5}
\e
which are functorial in the sense that if\/ $\la:\fC\ra\fD$ is a morphism in $\CRings$ and\/ $\ue:\uX\ra\uY$ a morphism in $\LCRS$ then the following commutes:
\e
\begin{gathered}
\xymatrix@C=150pt@R=15pt{
*+[r]{\Hom_\CRings(\fD,\Ga(\uY))} \ar@<.5ex>[r]^{L_{\fD,\uY}} \ar[d]^{\phi\mapsto \Ga(\ue)\ci \phi\ci\la}
& *+[l]{\Hom_\LCRS(\uY,\Spec\fD)} \ar@<.5ex>[l]^{R_{\fD,\uY}} \ar[d]_{\uf\mapsto\Spec\la\ci\uf\ci\ue} \\
*+[r]{\Hom_\CRings(\fC,\Ga(\uX))} \ar@<.5ex>[r]^{L_{\fC,\uX}}
& *+[l]{\Hom_\LCRS(\uX,\Spec\fC).\!} \ar@<.5ex>[l]^{R_{\fC,\uX}} }
\end{gathered}
\label{ag4eq6}
\e
When $\uX=\Spec\fC$ we have $\Psi_\fC=R_{\fC,\uX}(\uid_\uX),$ so that\/ $\Psi_\fC$ is the unit of the adjunction between $\Ga$ and\/~$\Spec$.
\label{ag4thm1}
\end{thm}

\begin{proof} Let $\fC\in\CRings$ and $\uX\in\LCRS$. Write $\uY=(Y,\O_Y)=\Spec\fC$. Define $R_{\fC,\uX}$ in \eq{ag4eq5} by, for each morphism $\uf:\uX\ra\uY$ in $\LCRS$, taking $R_{\fC,\uX}(\uf):\fC\ra\Ga(\uX)$ to be the composition
\e
\xymatrix@C=40pt{ \fC \ar[r]^(0.3){\Psi_\fC} & \Ga\ci\Spec\fC=\Ga(\uY) \ar[r]^(0.6){\Ga(\uf)} & \Ga(\uX). }
\label{ag4eq7}
\e
For the last part, if $\uX=\Spec\fC$ then $\Psi_\fC=R_{\fC,\uX}(\uid_\uX)$ as~$\Ga(\uid_\uX)=\id_{\Ga(\uX)}$.

Let $\phi:\fC\ra\Ga(\uX)$ be a morphism in $\CRings$. We will construct a morphism $\ug=(g,g^\sh):\uX\ra\uY$ in $\LCRS$, and set $L_{\fC,\uX}(\phi)=\ug$. For any $x\in X$ we have an $\R$-algebra morphism $x_*:\Ga(\uX)\ra\R$ by composing the stalk map $\si_x:\Ga(\uX)\ra\O_{X,x}$ with the unique morphism $\pi:\O_{X,x}\ra\R$, as $\O_{X,x}$ is a local $C^\iy$-ring. Then $x_*\ci\phi:\fC\ra\R$ is an $\R$-algebra morphism, and hence a point of $Y$. Define $g:X\ra Y$ by $g(x)=x_*\ci\phi$. If $c\in\fC$ then $U_c=\{y\in Y:y(c)\ne 0\}$ is open in $Y$, and $g^{-1}(U_c)=\{x\in X:x_*(\phi(c))\ne 0\}$ is open in $X$, as $x\mapsto x_*(d)$ is a continuous map $X\ra\R$ for any $d\in\Ga(\uX)$. Since the $U_c$ for $c\in\fC$ are a basis for the topology of $Y$ by Definition \ref{ag4def8}, $g:X\ra Y$ is continuous.

Let $x\in X$ with $g(x)=y\in Y$. Consider the diagram of $C^\iy$-rings
\e
\begin{gathered}
\xymatrix@C=100pt@R=15pt{ *+[r]{\fC} \ar[d]^{\pi_y} \ar[r]_\phi & *+[l]{\Ga(\uX)} \ar[d]_{\si_x} \\
*+[r]{\fC_y\cong\O_{Y,y}} \ar@{.>}[r]^{\phi_x} & *+[l]{\O_{X,x}.}
}
\end{gathered}
\label{ag4eq8}
\e
Here $\fC_y\cong\O_{Y,y}$ by Lemma \ref{ag4lem2}. If $c\in\fC$ with $y(c)\ne 0$ then $\si_x\ci\phi(c)\in\O_{X,x}$ with $\pi[\si_x\ci\phi(c)]\ne 0$, so $\si_x\ci\phi(c)$ is invertible in $\O_{X,x}$ as $\O_{X,x}$ is a local $C^\iy$-ring. Thus by the universal property of $\pi_y:\fC\ra\fC_y$ there is a unique morphism $\phi_x:\O_{Y,y}\ra \O_{X,x}$ making \eq{ag4eq8} commute.

For each open $V\subseteq Y$ with $U=g^{-1}(V)\subseteq X$, define $g_\sh(V):\O_Y(V)\ra g_*(\O_X)(V)=\O_X(U)$ by $g_\sh(V)s:x\mapsto \phi_x(s(g(x)))$ for $s\in \O_Y(V)$ and $x\in U\subseteq X$, so that $g(x)\in V$, and $s(g(x))\in\O_{Y,g(x)}$, and $\phi_x(s(g(x)))\in\O_{X,x}$. Here as $\O_X$ is a sheaf we may identify elements of $\O_X(U)$ with maps $t:U\ra\coprod_{x\in U}\O_{X,x}$ with $t(x)\in\O_{X,x}$ for $x\in U$, such that $t$ satisfies certain local conditions in $U$. If $s\in \O_Y(V)$ and $x\in U\subseteq X$ with $g(x)=y\in V\subseteq Y$, then by Definition \ref{ag4def9} there is an open neighbourhood $W_y$ of $y$ in $V$ and $c\in\fC$ with $s(y')=\pi_{y'}(c)\in\fC_{y'}\cong\O_{Y,y'}$ for all $y'\in W_y$. Therefore $g_\sh(V)s$ maps $x'\mapsto\si_{x'}(\phi(c))$ for all $x'$ in the open neighbourhood $g^{-1}(W_y)$ of $x$ in $U$, by \eq{ag4eq8}. Since the open subsets $g^{-1}(W_y)$ cover $U$, $g_\sh(V)s$ is a section of $\O_X\vert_U$, and $g_\sh(V)$ is well defined.

As the $\phi_x$ are $C^\iy$-ring morphisms, this defines a morphism $g_\sh:\O_Y\ra g_*(\O_X)$ of sheaves of $C^\iy$-rings on $Y$. Let $g^\sh:g^{-1}(\O_Y)\ra\O_X$ be the corresponding morphism of sheaves on $X$ under \eq{ag4eq3}. The stalk $g_x^\sh:\O_{Y,y}\ra\O_{X,x}$ of $g^\sh$ at $x\in X$ with $g(x)=y\in Y$ is $g_x^\sh=\phi_x$. Then $\ug=(g,g^\sh)$ is a morphism in $\LCRS$. Set $L_{\fC,\uX}(\phi)=\ug$. This defines $L_{\fC,\uX}$ in~\eq{ag4eq5}.

For $\phi,\ug$ as above, $c\in\fC$, and $x\in X$ with $g(x)=y=x_*\ci\phi\in Y$, we have
\begin{align*}
\si_x\bigl[\bigl(R_{\fC,\uX}\ci L_{\fC,\uX}(\phi)\bigr)(c)\bigr]&=\si_x\bigl[\Ga(\ug)\ci\Psi_\fC(c)\bigr]=g^\sh_x\ci \si_y[\Psi_\fC(c)]\\
&=\phi_x\ci\si_y[\Psi_\fC(c)]=\phi_x\ci\pi_y(c)=\si_x\ci\phi(c),
\end{align*}
using $L_{\fC,\uX}(\phi)=\ug$ and the definition \eq{ag4eq7} of $R_{\fC,\uX}(\ug)$ in the first step, $\si_x\ci\Ga(\ug)=g_x^\sh\ci\si_y:\Ga(\uY)\ra\O_{X,x}$ in the second, $g_x^\sh=\phi_x$ in the third, $\si_y\ci\Psi_\fC=\pi_y$ as maps $\fC\ra\O_{Y,y}\cong\fC_y$ in the fourth, and \eq{ag4eq8} in the fifth. As $\prod_{x\in X}\si_x:\Ga(\uX)\ra\prod_{x\in X}\O_{X,x}$ is injective, this implies that $\bigl(R_{\fC,\uX}\ci L_{\fC,\uX}(\phi)\bigr)(c)=\phi(c)$ for all $c\in\fC$, so $R_{\fC,\uX}\ci L_{\fC,\uX}(\phi)=\phi$, and~$R_{\fC,\uX}\ci L_{\fC,\uX}=\id$.

Suppose $\uf:\uX\ra\uY$ is a morphism in $\LCRS$, and set $\phi=R_{\fC,\uX}(\uf)$ and $\ug=L_{\fC,\uX}(\phi)$. Let $x\in X$ with $f(x)=y\in Y$. Then we have a commutative diagram in $\CRings$ 
\e
\begin{gathered}
\xymatrix@C=70pt@R=17pt{ \fC \ar[ddr]^(0.35)y \ar@<.8ex>@/^.5pc/[rr]^\phi \ar[d]^{\pi_y} \ar[r]_(0.5){\Psi_\fC} & \Ga\ci\Spec\fC=\Ga(\uY) \ar[d]^{\si_y} \ar[r]_(0.5){\Ga(\uf)} & \Ga(\uX) \ar[d]_{\si_x} \ar[ddl]_(0.35){x_*} \\
\fC_y \ar[dr]_\pi \ar[r]^(0.65)\cong & \O_{Y,y} \ar[d]_(0.4)\pi \ar[r]^(0.25){f^\sh_x} & \O_{X,x} \ar[dl]^\pi \\
& \R, }
\end{gathered}
\label{ag4eq9}
\e
where the isomorphism $\fC_y\cong\O_{Y,y}$ comes from Lemma \ref{ag4lem2}. Since $g(x)=x_*\ci\phi:\fC\ra\R$, this proves that $g(x)=y=f(x)$, so $f=g$. Also by definition the stalk $g^\sh_x:\O_{Y,y}\ra\O_{X,x}$ is $\phi_x$ in \eq{ag4eq8}, so comparing \eq{ag4eq8} and \eq{ag4eq9} and using $\pi_y:\fC\ra\fC_y$ surjective by Proposition \ref{ag2prop2} shows that $f^\sh_x=g^\sh_x$. As this holds for all $x\in X$ we have $f^\sh=g^\sh$, so $\uf=(f,f^\sh)=(g,g^\sh)=\ug$. Thus $L_{\fC,\uX}\ci R_{\fC,\uX}(\uf)=\uf$ for all $\uf:\uX\ra\uY$, so $L_{\fC,\uX}\ci R_{\fC,\uX}=\id$. Therefore $L_{\fC,\uX},R_{\fC,\uX}$ in \eq{ag4eq5} are inverse bijections.

It is easy to see that the rectangle in \eq{ag4eq6} involving $R_{\fD,\uY},R_{\fC,\uX}$ commutes using \eq{ag4eq7} and functoriality of the $\Psi_\fC$ and $\Ga$. Then the rectangle involving $L_{\fD,\uY},L_{\fC,\uX}$ commutes as $L_{\fD,\uY}=R_{\fD,\uY}^{-1}$ and $L_{\fC,\uX}=R_{\fC,\uX}^{-1}$. So \eq{ag4eq6} commutes. This completes the proof.
\end{proof}

\begin{rem}{\bf(a)} The fact in Theorem \ref{ag4thm1} that $\Spec:\CRings^{\bf op}\ra\LCRS$ is right adjoint to $\Ga:\LCRS\ra\CRings^{\bf op}$ determines $\Spec$ uniquely up to natural isomorphism, by properties of adjoint functors.

Dubuc \cite{Dubu3} and Moerdijk, van Qu\^e and Reyes \cite[\S 3]{MQR} both prove the existence of a right adjoint to $\Ga:\LCRS\ra\CRings^{\bf op}$, which is therefore naturally isomorphic to our functor $\Spec$ in Definition \ref{ag4def9}. But they show $\Spec$ exists by category theory, without constructing it explicitly as we do.

Moerdijk et al.\ \cite[\S 3]{MQR} call our functor $\Spec$ the {\it Archimedean spectrum}. They also give a nonequivalent definition \cite[\S 1]{MQR} of the spectrum $\Spec\fC$, in which the points are not $\R$-points, but `$C^\iy$-radical prime ideals'.
\smallskip

\noindent{\bf(b)} Since $\Spec$ is a right adjoint functor, it preserves limits, as in \cite[p.~687]{Dubu3}. Equivalently, $\Spec$ takes colimits in $\CRings$ to limits in $\LCRS$. So, for example, a pushout $\fC=\fD\amalg_\fF\fE$ of morphisms $\phi:\fF\ra\fD$, $\psi:\fF\ra\fE$ in $\CRings$ is mapped to a fibre product $\Spec\fC\cong \Spec\fD\t_{\Spec\fF}\Spec\fE$ of morphisms $\Spec\phi:\Spec\fD\ra\Spec\fF$, $\Spec\psi:\Spec\fE\ra\Spec\fF$ in $\LCRS$.\I{global sections functor $\Ga$|)}\I{functor!adjoint|)}\I{adjoint functor|)}\label{ag4rem4}
\end{rem}

Here are some properties of finitely generated and fair $C^\iy$-rings, due to Dubuc \cite[Th.~13]{Dubu3}. The reflection functor $\smash{R_{\rm fg}^{\rm fa}}$ is as in Definition~\ref{ag2def9}.

\begin{thm}{\bf(a)} If\/ $\fC$ is a finitely generated\/ $C^\iy$-ring, there is a natural isomorphism $\Ga\ci\Spec\fC\cong\smash{R_{\rm fg}^{\rm fa}}(\fC),$ which identifies $\Psi_\fC:\fC\ra\Ga (\Spec\fC)$ with the natural surjective projection\/ $\smash{\fC\ra R_{\rm fg}^{\rm fa}(\fC)}$.
 
These isomorphisms for all\/ $\fC$ form a natural isomorphism $\smash{R_{\rm fg}^{\rm fa}}\cong\Ga\ci\Spec$ of functors $\smash{R_{\rm fg}^{\rm fa}},\Ga\ci\Spec:\CRingsfg\ra\CRingsfa$.

Hence, if\/ $\fC$ is fair then $\Psi_\fC:\fC\ra\Ga(\Spec\fC)\cong\smash{R_{\rm fg}^{\rm fa}}(\fC)$ is an isomorphism.
\smallskip

\noindent{\bf(b)} If\/ $\fC$ is finitely generated then $\Spec\Psi_\fC:\Spec\fC\ra\Spec\Ga(\Spec\fC)\cong \Spec\smash{R_{\rm fg}^{\rm fa}}(\fC)$ is an isomorphism in $\LCRS$. 
\smallskip

\noindent{\bf(c)} The functor\/ $\Spec\vert_{\cdots}:(\CRingsfa)^{\bf op}\ra\LCRS$ is full\I{functor!full} and faithful, and takes finite limits in  $(\CRingsfa)^{\bf op}$ to finite limits in $\LCRS$.
\label{ag4thm2}
\end{thm}

To see that $\Spec$ is full and faithful on $(\CRingsfa)^{\bf op}$ in (c), let $\fC,\fD$ be fair $C^\iy$-rings. Then putting $\uX=\Spec\fD$ in \eq{ag4eq5} and using $\fD\cong\Ga\ci\Spec\fD$ by (a) shows that the following is a bijection.
\begin{equation*}
\Spec:\Hom_\CRings(\fC,\fD) \longra\Hom_\LCRS(\Spec\fD,\Spec\fC).
\end{equation*}
Note that $\Spec$ is neither full nor faithful on $(\CRingsfg)^{\bf op}$ or $\CRings^{\bf op}$. This is a contrast to conventional algebraic geometry, where $\Ga(\Spec R)\cong R$ for arbitrary rings $R$, as in \cite[Prop.~II.2.2]{Hart}, so that $\Spec$ is full and faithful. In \S\ref{ag46} we will generalize Theorem \ref{ag4thm2} to non-finitely-generated $C^\iy$-rings.

\subsection{\texorpdfstring{Affine $C^\iy$-schemes and $C^\iy$-schemes}{Affine C∞-schemes and C∞-schemes}}
\label{ag45}
\I{C-scheme@$C^\iy$-scheme|(}

As for the usual definitions of affine schemes and schemes, we define:

\begin{dfn} A local $C^\iy$-ringed space\I{C-ringed
space@$C^\iy$-ringed space!local} $\uX$ is called an {\it affine\/
$C^\iy$-scheme} if it is isomorphic in $\LCRS$ to $\Spec\fC$ for
some $C^\iy$-ring $\fC$. We call $\uX$ a {\it finitely
presented},\I{C-scheme@$C^\iy$-scheme!finitely presented affine} or
{\it fair},\I{C-scheme@$C^\iy$-scheme!fair affine} affine
$C^\iy$-scheme if $X\cong\Spec\fC$ for $\fC$ that kind of
$C^\iy$-ring. Write
$\ACSch,\ACSchfp,\ACSchfa$\G[ACSch]{$\ACSch$}{category of affine
$C^\iy$-schemes}\G[ACSchfa]{$\ACSchfa$}{category of fair affine
$C^\iy$-schemes}\G[ACSchfp]{$\ACSchfp$}{category of finitely
presented affine $C^\iy$-schemes} for the full subcategories of
affine $C^\iy$-schemes and of finitely presented, and fair, affine
$C^\iy$-schemes in $\LCRS$ respectively.

We do not define {\it finitely generated\/} affine $C^\iy$-schemes,
because Theorem \ref{ag4thm2}(b) implies that they coincide with fair affine $C^\iy$-schemes.

Let $\uX=(X,\O_X)$ be a local $C^\iy$-ringed space. We
call $\uX$ a $C^\iy$-{\it
scheme\/}\G[WXYZa]{$\uW,\uX,\uY,\uZ,\ldots$}{$C^\iy$-schemes} if $X$
can be covered by open sets $U\subseteq X$ such that
$(U,\O_X\vert_U)$ is an affine $C^\iy$-scheme. We call a
$C^\iy$-scheme $\uX$ {\it locally
fair},\I{C-scheme@$C^\iy$-scheme!definition}\I{C-scheme@$C^\iy$-scheme!locally
fair} or {\it locally finitely
presented},\I{C-scheme@$C^\iy$-scheme!locally finitely presented} if
$X$ can be covered by open $U\subseteq X$ with $(U,\O_X\vert_U)$ a
fair, or finitely presented, affine $C^\iy$-scheme, respectively.

We call a $C^\iy$-scheme $\uX$ {\it
Hausdorff},\I{C-scheme@$C^\iy$-scheme!Hausdorff} {\it second
countable},\I{C-scheme@$C^\iy$-scheme!second countable} {\it Lindel\"of},\I{C-scheme@$C^\iy$-scheme!Lindel\"of} {\it compact},\I{C-scheme@$C^\iy$-scheme!compact} {\it locally compact},\I{C-scheme@$C^\iy$-scheme!locally compact} {\it paracompact},\I{C-scheme@$C^\iy$-scheme!paracompact} {\it metrizable},\I{C-scheme@$C^\iy$-scheme!metrizable} {\it regular}, or \I{C-scheme@$C^\iy$-scheme!regular} {\it separable},\I{C-scheme@$C^\iy$-scheme!separable} if the topological space $X$ is.\I{topological space!Hausdorff}\I{topological space!locally compact}\I{topological space!second countable}\I{topological space!Lindel\"of}\I{topological space!paracompact}\I{topological space!metrizable} Affine $C^\iy$-schemes are Hausdorff and regular by Lemma~\ref{ag4lem1}.

Write $\CSchlf,\CSchlfp,\CSch$\G[CSch]{$\CSch$}{category of
$C^\iy$-schemes}\G[CSchlf]{$\CSchlf$}{category of locally fair
$C^\iy$-schemes}\G[CSchlfp]{$\CSchlfp$}{category of locally finitely
presented $C^\iy$-schemes} for the full subcategories in $\LCRS$ of locally fair $C^\iy$-schemes, locally finitely presented $C^\iy$-schemes, and all $C^\iy$-schemes, respectively. 
\label{ag4def11}
\end{dfn}

\begin{rem} Ordinary schemes are a much larger class than ordinary affine schemes,
and central examples such as $\CP^n$ are not affine schemes. However, affine $C^\iy$-schemes are already general enough for many purposes. For example, all second countable, metrizable $C^\iy$-schemes are affine, as in \S\ref{ag48}, including manifolds and manifolds with corners. Affine $C^\iy$-schemes are Hausdorff and regular, so any non-Hausdorff or non-regular $C^\iy$-scheme is not affine.
\label{ag4rem5}
\end{rem}

For the next theorem, part (a) follows from  Propositions \ref{ag2prop1}, \ref{ag2prop5} and \ref{ag2prop6}, Remark \ref{ag4rem4}(b), and Theorem \ref{ag4thm2}(c). Part (b) holds as finite limits in $\CSchlfp,\CSchlf,\CSch$ are locally modelled on finite limits in $\ACSchfp,\ACSchfa$ and $\ACSch$.

\begin{thm}{\bf(a)} The full subcategories\/ $\ACSchfp,\ACSchfa,\ab\ACSch$
are closed under all finite limits\I{category!limit} in\/ $\LCRS$.
Hence, fibre products\I{fibre product}\I{category!fibre
product}\I{C-scheme@$C^\iy$-scheme!fibre products} and all finite
limits exist in each of these subcategories.
\smallskip

\noindent{\bf(b)} The full subcategories\/ $\CSchlfp,\CSchlf$ and\/
$\CSch$ are closed under all finite limits in\/ $\LCRS$. Hence,
fibre products\I{fibre product}\I{category!fibre
product}\I{C-scheme@$C^\iy$-scheme!fibre products} and all finite
limits\I{category!limit} exist in each of these subcategories.
\label{ag4thm3}
\end{thm}

\begin{dfn} Define functors
\begin{align*}
F_\Man^\CSch:\Man&\longra\ACSchfp\subset \ACSch,\\
F_\Manb^\CSch:\Manb&\longra\ACSchfa\subset\ACSch,\\
F_\Manc^\CSch:\Manc&\longra\ACSchfa\subset\ACSch,
\end{align*}
by $F_{\Man^*}^\CSch=\Spec\ci F_{\Man^*}^\CRings$, in the notation
of Definitions \ref{ag3def} and \ref{ag4def9}.

By Example \ref{ag4ex3}, if $X$ is a manifold with corners then
$F_\Manc^\CSch(X)$ is naturally isomorphic to the local
$C^\iy$-ringed space $\uX$ in Example~\ref{ag4ex1}.

If $X,Y,\ldots$ are manifolds, or $f,g,\ldots$ are (weakly)
smooth\I{manifold with corners!weakly smooth map} maps, we may use
$\uX,\uY,\ldots,\uf,\ug,\ldots$ to denote the images of
$X,Y,\ldots,f,g,\ldots$ under $F_\Manc^\CSch$. So for instance we will
write $\ul{\R}^n$ and $\ul{[0,\iy)\!}\,$ for $F_\Man^\CSch(\R^n)$
and~$F_\Manb^\CSch\bigl([0,\iy)\bigr)$.

\label{ag4def12}
\end{dfn}

Our categories of spaces so far are related as
follows:
\begin{equation*}
\xymatrix@C=20pt@R=19pt{\Man \ar[d]^{F_\Man^\CSch} \ar[r]_\subset & \Manb
\ar[d]_{F_\Manb^\CSch} \ar[r]_\subset & \Manc \ar[dl]^(0.2){F_\Manc^\CSch}
\\ \ACSchfp \ar[r]_\subset \ar[d]^\subset & \ACSchfa \ar[r]_\subset
\ar[d]^\subset & \ACSch \ar[d]^\subset \ar[dr]^\subset  \\
\CSchlfp \ar[r]^\subset & \CSchlf \ar[r]^\subset & \CSch
\ar[r]^\subset & \LCRS \ar[r]^\subset & \CRS. }
\end{equation*}

By Corollary \ref{ag3cor} and Theorems \ref{ag3thm} and \ref{ag4thm2}(c), we find as in~\cite[Th.~16]{Dubu3}:

\begin{cor} $F_\Man^\CSch:\Man\hookra\ACSchfp\subset\ACSch$ is a full
and faithful functor, and $F_\Manb^\CSch:\Manb\ra\ACSchfa\subset
\ACSch,$ $F_\Manc^\CSch:\Manc\ra\ACSchfa\subset\ACSch$ are both
faithful functors, but are not full. Also these functors take
transverse fibre products\I{manifold!transverse fibre
product}\I{manifold with corners!transverse fibre product}\I{fibre
product}\I{category!fibre product}\I{C-scheme@$C^\iy$-scheme!fibre
products} in $\Man,\Manc$ to fibre products
in\/~$\ACSchfp,\ACSchfa$.
\label{ag4cor1}
\end{cor}

We study open subspaces of $C^\iy$-schemes. The definition of $\Spec\fC$ implies:

\begin{lem} Let\/ $\fC$ be a $C^\iy$-ring, and\/ $c\in\fC$. Write $\Spec\fC=(X,\O_X)$ and\/ $U_c=\{x\in X:x(c)\ne 0\}$. Then $U_c\subseteq X$ is open with\/ $(U_c,\O_X\vert_{U_c})\cong\Spec\fC[c^{-1}]$. 

\label{ag4lem3}
\end{lem}

\begin{cor} Let\/ $\uX=(X,\O_X)$ be a $C^\iy$-scheme and\/ $V\subseteq X$ be open. Then $\uV=(V,\O_X\vert_V)$ is also a $C^\iy$-scheme.
\label{ag4cor2}
\end{cor}

\begin{proof} Let $x\in V$. Then there exists an open $x\in Y\subseteq X$ with $Y\cong\Spec\fC$ for some $C^\iy$-ring $\fC$, as $\uX$ as a $C^\iy$-scheme. Identify $\uY$ with $\Spec\fC$. As $V\cap Y$ is open in $Y=X_\fC$, and the topology on $X_\fC$ is generated by subsets $U_c=\{\ti x\in X_\fC:\ti x(c)\ne 0\}$ for $c\in\fC$, there exists $c\in\fC$ such that $x\in U_c\subseteq V\cap Y$. Then $(U_c,\O_X\vert_{U_c})\cong\Spec\fC[c^{-1}]$ by Lemma \ref{ag4lem3}. So every $x\in\uV$ has an affine open neighbourhood, and $\uV$ is a $C^\iy$-scheme.
\end{proof}

\begin{lem} Let\/ $\fC$ be a finitely generated\/\I{C-ring@$C^\iy$-ring!finitely generated} $C^\iy$-ring and\/ $(X,\O_X)=\Spec\fC$. Suppose $V\subseteq X$ is open. Then there exists $c\in\fC$ with\/ $V=\{x\in X:x(c)\ne 0\}$. We call $c$ a \begin{bfseries}characteristic function\end{bfseries} for~$V$.
\label{ag4lem4}
\end{lem}

\begin{proof} As $\fC$ is a finitely generated $C^\iy$-ring it fits into an exact sequence $0\ra I\hookra C^\iy(\R^n)\,{\buildrel\phi\over\longra}\,\fC\ra 0$. Example \ref{ag4ex2} gives a  homeomorphism $\phi_*:X\ra X_\fC^\phi$ with a closed subset $X_\fC^\phi$ in
$\R^n$ given in \eq{ag4eq4}. Then $\phi_*(V)$ is open in $X_\fC^\phi$, so there exists an open $U\subseteq\R^n$ with $U\cap X_\fC^\phi=\phi_*(V)$. By \cite[Lem.~I.1.4]{MoRe2} there exists $f\in C^\iy(\R^n)$ with $U=\bigl\{x\in\R^n:f(x)\ne 0\bigr\}$. Then $c=\phi(f)\in\fC$ is a characteristic function for~$V$.
\end{proof}

\begin{ex} Let $I$ be an infinite set, and write $C^\iy(\R^I)$ for the free $C^\iy$-ring with generators $x_i$ for $i\in I$. Then $\uX=\Spec C^\iy(\R^I)$ has topological space $X=\R^I$ with points $(x_i)_{i\in I}$ for $x_i\in\R$. Elements of $C^\iy(\R^I)$ are functions $c:\R^I\ra\R$ depending only on $x_j$ for $j$ in a {\it finite\/} subset $J\subseteq I$, and which are smooth functions of these $x_j$, $j\in J$.

Let $V=\R^I\sm\{0\}$. Then $V$ is open in $X$. But no characteristic function $c$ exists for $V$ in $C^\iy(\R^I)$, since $c$ would depend only on $x_j$ for $j$ in a finite subset $J\subseteq I$, but $V$ depends on $x_i$ for all $i\in I$. Thus, infinitely generated $C^\iy$-rings need not admit characteristic functions, in contrast to Lemma~\ref{ag4lem4}.
\label{ag4ex4}
\end{ex}

If $\fC$ is a finitely generated (or finitely presented) $C^\iy$-ring and $c\in\fC$ then $\fC[c^{-1}]$ is also finitely generated (or finitely presented), since $\fC[c^{-1}]\cong \fC[x]/(c\cdot x-1)$ is the result of adding one extra generator and one extra relation to $\fC$. Thus from Lemmas \ref{ag4lem3} and \ref{ag4lem4} we deduce:

\begin{cor}{\bf(a)} Let\/ $(X,\O_X)$ be a fair (or finitely presented) affine $C^\iy$-scheme, and\/ $U\subseteq X$ be an open subset. Then\/ $(U,\O_X\vert_U)$ is also a fair (or finitely presented) affine $C^\iy$-scheme.
\smallskip

\noindent{\bf(b)} Let\/ $(X,\O_X)$ be a locally fair (or locally finitely presented)\/ $C^\iy$-scheme, and\/ $U\subseteq X$ be an open subset. Then\/ $(U,\O_X\vert_U)$ is also a locally fair (or locally finitely presented)\/ $C^\iy$-scheme.
\label{ag4cor3}
\end{cor}

Our next result describes the sheaf of $C^\iy$-rings $\O_X$ in $\Spec\fC$ for $\fC$ a finitely generated $C^\iy$-ring.\I{C-ring@$C^\iy$-ring!finitely generated} It is a version of \cite[Prop.~I.2.2(b)]{Hart} in algebraic geometry, and reduces to Moerdijk and Reyes \cite[Prop.~I.1.6]{MoRe2} when~$\fC=C^\iy(\R^n)$.

\begin{prop} Let\/ $\fC$ be a finitely generated\/
$C^\iy$-ring,\I{C-ring@$C^\iy$-ring!finitely generated} write
$(X,\O_X)=\Spec\fC,$ and let\/ $U\subseteq X$ be open. By
Lemma\/ {\rm\ref{ag4lem4}} we may choose a characteristic
function\/ $c\in\fC$ for $U$. Then there is a canonical isomorphism $\O_X(U)\cong
R_{\rm fg}^{\rm fa}(\fC[c^{-1}]),$\I{reflection
functor}\I{functor!reflection} in the notation of Definitions\/
{\rm\ref{ag2def7}} and\/ {\rm\ref{ag2def9}}. If\/ $\fC$ is finitely
presented\I{C-ring@$C^\iy$-ring!finitely presented}
then\/~$\O_X(U)\cong\fC[c^{-1}]$.
\label{ag4prop3}
\end{prop}

\begin{proof} We have morphisms of $C^\iy$-rings
$c_*:C^\iy(\R)\ra\fC$ and $i^*:C^\iy(\R)\ra C^\iy(\R\sm\{0\})$,
and $C^\iy(\R),C^\iy(\R\sm\{0\})$ are finitely presented
$C^\iy$-rings by Proposition \ref{ag3prop1}(a). So as $\Spec$
preserves limits in $(\CRingsfg)^{\bf op}$ we have
\begin{equation*}
\Spec\bigl(\fC\amalg_{c_*,C^\iy(\R),
i^*}C^\iy(\R\sm\{0\})\bigr)\cong \Spec\fC\t_{\uf,
\ul{\R},\ui}\ul{\R\sm\{0\}\!}\cong (U,\O_X\vert_U).
\end{equation*}
But $\fC\amalg_{C^\iy(\R)}C^\iy(\R\sm\{0\})\cong\fC[c^{-1}]$ for
formal reasons. Thus Theorem \ref{ag4thm2}(a) gives $\O_X(U)\cong
\Ga\bigl((U,\O_X\vert_U)\bigr)\cong R_{\rm fg}^{\rm
fa}(\fC[c^{-1}])$. If $\fC$ is finitely presented then $\fC[c^{-1}]$
is too, as in Corollary \ref{ag4cor3}, so $\fC[c^{-1}]$ is fair and $R_{\rm
fg}^{\rm fa}\bigl(\fC[c^{-1}]\bigr)=\fC[c^{-1}]$, and therefore~$\O_X(U)\cong\fC[c^{-1}]$.\I{C-ringed space@$C^\iy$-ringed
space|)}
\end{proof}

\subsection{\texorpdfstring{Complete $C^\iy$-rings}{Complete C∞-rings}}
\label{ag46}

The material of this section appears to be new.

\begin{prop} Let\/ $\fC$ be a $C^\iy$-ring, and\/ $\Psi_\fC$ be as in Definition\/ {\rm\ref{ag4def10}}. Then $\Spec\Psi_\fC:\Spec\ci\Ga\ci\Spec\fC\ra\Spec\fC$ is an isomorphism in $\LCRS$.
\label{ag4prop4}
\end{prop}

\begin{proof} Write $\fD=\Ga\ci\Spec\fC$, $\uX=\Spec\fC$, $\uY=\Spec\fD$, and $\uf=\Spec\Psi_\fC:\uY\ra\uX$. Let $x\in X$, and define $y=\pi\ci\Pi_x:\fD\ra\R$ to be the composition of the projection $\Pi_x:\fD\ra\fC_x$, noting that $\fD\subseteq\prod_{\ti x\in X}\fC_{\ti x}$ by Definition \ref{ag4def10}, and the unique morphism $\pi:\fC_x\ra\R$, as $\fC_x$ is a local $C^\iy$-ring. Then $f(y)=\pi\ci\pi_x=x:\fC\ra\R$ for $\pi_x:\fC\ra\fC_x$, so $f:Y\ra X$ is surjective.

Suppose now that $y\in Y$ with $f(y)=x$, so that $y:\fD\ra\R$ is an $\R$-algebra morphism. We will prove that $y=\pi\ci\Pi_x$ as above. Let $d\in\fD$. By definition of $\fD=\O_{X_\fC}(X_\fC)$ there exist an open neighbourhood $W$ of $x$ in $X$ and $c_1\in\fC$ such that $d(\ti x)=\pi_{\ti x}(c_1)$ in $\fC_{\ti x}$ for all $\ti x\in W$. By definition of the topology $\cT_\fC$, there exists $c_2\in\fC$ such that $U_{c_2}=\{\ti x\in X:\ti x(c_2)\ne 0\}$ is an open neighbourhood of $x$ in $W\subseteq X$. Hence $x(c_2)\ne 0$ and $\ti x(c_2)=0$ for all~$\ti x\in X\sm W$. 

Choose smooth functions $g,h:\R\ra\R$ with $g(x(c_2))=1$ and $g=0$ in an open neighbourhood $(-\ep,\ep)$ of 0 in $\R$, and $h(0)\ne 0$ and $h=0$ outside $(-\ep,\ep)$, so that $g\cdot h=0$. Set $c_3=\Phi_g(c_2)$ and $c_4=\Phi_h(c_2)$, with $\Phi_g,\Phi_h:\fC\ra\fC$ the $C^\iy$-ring operations. Then $x(c_3)=1$, and $\pi_{\ti x}(c_3)=0$ in $\fC_{\ti x}$ for all $\ti x\in X\sm W$,~as
\begin{equation*}
\pi_{\ti x}(c_3)\cdot\pi_{\ti x}(c_4)=\pi_{\ti x}\bigl(\Phi_g(c_2)\cdot \Phi_h(c_2)\bigr)=\pi_{\ti x}\ci\Phi_{gh}(c_2)=\pi_{\ti x}\ci\Phi_0(c_2)=0,
\end{equation*}
but $\pi_{\ti x}(c_4)$ is invertible in $\fC_{\ti x}$ as $\ti x(c_4)=h(\ti x(c_2))=h(0)\ne 0$. Thus we have $d\cdot\Psi_\fC(c_3)=\Psi_\fC(c_1)\cdot\Psi_\fC(c_3)=\Psi_\fC(c_1\cdot c_3)$ in $\fD$, as $d(\ti x)=\Psi_\fC(c_1)\ti x$ for all $\ti x\in W$, and $\Psi_\fC(c_3)\ti x=0$ for all $\ti x\in X\sm W$. Therefore
\begin{align*}
y(d)&=y(d)\cdot 1=y(d)\cdot x(c_3)=y(d)\cdot y(\Psi_\fC(c_3))=y\bigl(d\cdot\Psi_\fC(c_3)\bigr)\\
&=y\bigl(\Psi_\fC(c_1\cdot c_3)\bigr)\!=\!x(c_1\cdot c_3)\!=\!x(c_1)\cdot x(c_3)\!=\!\bigl(\pi\ci\Pi_x(d)\bigr)\cdot 1\!=\!\pi\ci\Pi_x(d).
\end{align*}
As this holds for all $d\in\fD$, we see that $y\in Y$ with $f(y)=x$ implies that $y=\pi\ci\Pi_x$. Hence $f:Y\ra X$ is injective, and so bijective.

From above $f:Y\ra X$ is continuous. To show $f^{-1}:X\ra Y$ is continuous, note that the topology on $Y$ is generated by the basis of open sets $V_d=\{y\in Y:y(d)\ne 0\}$ for all $d\in\fD$. So it is enough to show that $f(V_d)=\{x\in X:\pi\ci\Pi_x(d)=0\}$ is open in $X$ for all $d$. For fixed $d$, by definition we may cover $X$ by open $W\subseteq X$ for which there exist $c\in\fC$ with $d(x)=\pi_x(c)\in\fC_x$ for all $x\in W$. But then $W\cap f(V_d)=W\cap U_c$, where $U_c=\{x\in X:x(c)\ne 0\}$ is open in $X$. So we can cover $X$ by open $W\subseteq X$ with $W\cap f(V_d)$ open, and $f(V_d)$ is open. Therefore $f^{-1}$ is continuous, and $f:Y\ra X$ is a homeomorphism.

Let $y\in Y$ with $f(y)=x$. Taking stalks of $f^\sh:f^{-1}(\O_X)\ra\O_Y$ at $y$ gives a morphism $f^\sh_y:\O_{X,x}\ra\O_{Y,y}$, where $\O_{X,x}\cong\fC_x$ and $\O_{Y,y}\cong\fD_y$ by Lemma \ref{ag4lem2}, and we have a commutative diagram
\e
\begin{gathered}
\xymatrix@C=130pt@R=17pt{ *+[r]{\fC} \ar[d]^{\pi_x} \ar[r]_(0.25){\Psi_\fC} & *+[l]{\fD} \ar[d]_{\pi_y} \ar[dl]_(0.6){\Pi_x} \\
*+[r]{\fC_x\cong\O_{X,x}} \ar[r]^(0.61){\Psi_{\fC,x}\cong f^\sh_y} & *+[l]{\O_{Y,y}\cong\fD_y.}
}
\end{gathered}
\label{ag4eq10}
\e
Here the outer rectangle and top left triangle obviously commute. To see that the bottom right triangle commutes, we use that any $d\in\fD=\O_{X_\fC}(X_\fC)$ has $d(\ti x)=\Psi_\fC(c)\ti x$ for some $c\in\fC$ and all $\ti x$ in an open neighbourhood $W$ of $x$ in $X$. As in the first part of the proof, we can find $c_3\in\fC$ with $x(c_3)=1$ and $\pi_{\ti x}(c_3)=0$ in $\fC_{\ti x}$ for all $\ti x\in X\sm W$. Then evaluating at $\ti x\in W$ and $\ti x\in X\sm W$ we see that $\Psi_\fC(c)\cdot\Psi_\fC(c_3)=d\cdot\Psi_\fC(c_3)$, which forces $\pi_y(d)=\pi_y(\Psi_\fC(c))$, since $\pi_y\ci\Psi_\fC(c_3)$ is invertible in $\fD_y$ as $\pi\ci\pi_y\ci\Psi_\fC(c_3)=x(c_3)=1>0$. Thus
\begin{equation*}
\pi_y(d)=\pi_y\ci\Psi_\fC(c)=f_y^\sh\ci\pi_x(c)=f_y^\sh\ci\Pi_x\ci\Psi_\fC(c)=f_y^\sh\ci\Pi_x(d).
\end{equation*}

Since $\pi_y:\fD\ra\fD_y$ is surjective by Proposition \ref{ag2prop2}, the bottom right triangle in \eq{ag4eq10} implies that $f^\sh_y:\O_{X,x}\ra\O_{Y,y}$ is surjective. Suppose $c_x\in\O_{X,x}$ with $f^\sh_y(c_x)=0$ in $\O_{Y,y}$. As $\pi_x$ is surjective by Proposition \ref{ag2prop2} we may write $c_x=\pi_x(c)$ for $c\in\fC$. Then $\pi_y\ci\Psi_\fC(c)=f_y^\sh\ci\pi_x(c)=f_y^\sh(c_x)=0$, so $\Psi_\fC(c)\in\Ker\pi_y$. Write $I\subset\fC$ and $J\subset\fD$ for the ideals in \eq{ag2eq2} for $x,y$. Then $J=\Ker\pi_y$, so $\Psi_\fC(c)\in J$, and thus there exists $d\in\fD$ with $y(d)=\pi\ci\Pi_x(d)\ne 0$ in $\R$ and $\Psi_\fC(c)\cdot d=0$ in $\fD$. Applying $\Pi_x$ gives
\begin{equation*}
c_x\cdot \Pi_x(d)=\pi_x(c)\cdot\Pi_x(d)=\Pi_x(\Psi_\fC(c))\cdot\Pi_x(d)=\Pi_x(\Psi_\fC(c)\cdot d)=\Pi_x(0)=0.
\end{equation*}
But $\Pi_x(d)$ is invertible in $\fC_x$ as $\pi\ci\Pi_x(d)\ne 0$ in $\R$, so $c_x=0$. Thus $f^\sh_y:\O_{X,x}\ra\O_{Y,y}$ is injective, and so an isomorphism.

We have shown that $f:Y\ra X$ is a homeomorphism, and $f^\sh_y:\O_{X,f(y)}\ra\O_{Y,y}$ is an isomorphism on stalks at all $y\in Y$. Hence $\Spec\Psi_\fC=(f,f^\sh)$ is an isomorphism in $\LCRS$, as we have to prove.
\end{proof}

\begin{dfn} We call a $C^\iy$-ring $\fC$ {\it complete\/}\I{C-ring@$C^\iy$-ring!complete|(} if the morphism $\Psi_\fC:\fC\ra\Ga\ci\Spec\fC$ in Definition \ref{ag4def10} is an isomorphism. Write $\CRingsco$ for the full subcategory of complete $C^\iy$-rings $\fC$ in~$\CRings$.

If $\fC$ is any $C^\iy$-ring, applying $\Ga$ to $\Spec\Psi_\fC$ in Proposition \ref{ag4prop4} shows that
\begin{equation*}
\Ga\ci\Spec\Psi_\fC=\Psi_{\Ga\ci\Spec\fC}:\Ga\ci\Spec\fC\longra\Ga\ci\Spec(\Ga\ci\Spec\fC)
\end{equation*}
is an isomorphism in $\CRings$, where we check that $\Ga\ci\Spec\Psi_\fC=\Psi_{\Ga\ci\Spec\fC}$ from Definitions \ref{ag4def9} and \ref{ag4def10}. Hence $\Ga\ci\Spec\fC$ is a complete $C^\iy$-ring. Define a functor $R_{\rm all}^{\rm co}:\CRings\ra\CRingsco$ by $R_{\rm all}^{\rm co}=\Ga\ci\Spec$.\G[Rcaa]{$R_{\rm all}^{\rm co}:\CRings\ra\CRingsco$}{reflection functor}
\label{ag4def13}
\end{dfn}

The next result extends Definition \ref{ag2def9} and Theorem \ref{ag4thm2} from $\CRingsfa\ab\subset\CRingsfg$ to $\CRingsco\subset\CRings$.

\begin{thm}{\bf(a)} Let\/ $\uX$ be an affine $C^\iy$-scheme. Then $\uX\cong\Spec\O_X(X),$ where $\O_X(X)$ is a complete $C^\iy$-ring.
\smallskip

\noindent{\bf(b)} $\Spec\vert_{(\CRingsco)^{\bf op}}:(\CRingsco)^{\bf op}\ra\LCRS$ is full and faithful, and an equivalence of categories $\Spec\vert_{\cdots}:(\CRingsco)^{\bf op}\ra\ACSch$.
\smallskip

\noindent{\bf(c)} $R_{\rm all}^{\rm co}:\CRings\ra\CRingsco$ is left
adjoint\I{functor!adjoint}\I{adjoint functor} to the
inclusion functor\/ $\inc:\CRingsco\hookra\CRings$. That is, $R_{\rm all}^{\rm co}$ is a \begin{bfseries}reflection functor\end{bfseries}.\I{functor!reflection}\I{reflection functor}
\smallskip

\noindent{\bf(d)} All small colimits exist in $\CRingsco,$ although they may not coincide with the corresponding small colimits in $\CRings$.
\smallskip

\noindent{\bf(e)} $\Spec\vert_{(\CRingsco)^{\bf op}}=\Spec\ci\inc:(\CRingsco)^{\bf op}\ra\LCRS$ is right adjoint to $R_{\rm all}^{\rm co}\ci\Ga:\LCRS\ra(\CRingsco)^{\bf op}$. Thus $\Spec\vert_{\cdots}$ takes limits in $(\CRingsco)^{\bf op}$ (equivalently, colimits in $\CRingsco$) to limits in~$\LCRS$.

\label{ag4thm4}
\end{thm}

\begin{proof} For (a), if $\uX$ is an affine $C^\iy$-scheme then $\uX\cong\Spec\fC$ for some $C^\iy$-ring $\fC$, so $\O_X(X)\cong\Ga\ci\Spec\fC$, and thus $\uX\cong\Spec\O_X(X)$ by Proposition \ref{ag4prop4}. Also, applying $\Ga$ to $\Spec\Psi_\fC$ in Proposition \ref{ag4prop4} shows that
\begin{equation*}
\Ga\ci\Spec\Psi_\fC=\Psi_{\Ga\ci\Spec\fC}:\Ga\ci\Spec\fC\longra\Ga\ci\Spec(\Ga\ci\Spec\fC)
\end{equation*}
is an isomorphism in $\CRings$, where $\Ga\ci\Spec\Psi_\fC=\Psi_{\Ga\ci\Spec\fC}$ follows from the definitions. Hence $\Ga\ci\Spec\fC\cong\O_X(X)$ is complete, proving~(a).

For (b), if $\fC,\fD$ are complete $C^\iy$-rings then putting $\uX=\Spec\fD$ in Theorem \ref{ag4thm1} and using $\Ga\ci\Spec\fD\cong\fD$, equation \eq{ag4eq5} shows that
\begin{equation*}
\Spec=L_{\fC,\uX}:\Hom_\CRings(\fC,\fD)\longra\Hom_\LCRS(\Spec\fD,\Spec\fC)
\end{equation*}
is a bijection, where the definition of $L_{\fC,\uX}$ agrees with the definition of $\Spec$ on morphisms in this case. Thus $\Spec$ is full and faithful on complete $C^\iy$-rings. Therefore  $\Spec\vert_{\cdots}:(\CRingsco)^{\bf op}\ra\LCRS$ is an equivalence of categories from $(\CRingsco)^{\bf op}$ to its essential image in $\LCRS$, which is~$\ACSch$.

For (c), let $\fC,\fD$ be $C^\iy$-rings with $\fD$ complete. Then we have bijections
\ea
&\Hom_\CRingsco\bigl(R_{\rm all}^{\rm co}(\fC),\fD\bigr)\cong \Hom_\CRings\bigl(\Ga\ci\Spec\fC,\Ga\ci\Spec\fD\bigr)
\nonumber\\
&\cong \Hom_\LCRS\bigl(\Spec\fD,\Spec\ci\Ga\ci\Spec\fC\bigr)\cong\Hom_\LCRS\bigl(\Spec\fD,\Spec\fC\bigr)
\nonumber\\
&\cong \Hom_\CRings\bigl(\fC,\Ga\ci\Spec\fD\bigr)\cong \Hom_\CRings\bigl(\fC,\fD\bigr)
\nonumber\\
&=\Hom_\CRings\bigl(\fC,\inc(\fD)\bigr),
\label{ag4eq11}
\ea
using $\fD\cong\Ga\ci\Spec\fD$ as $\fD$ is complete in the first and fifth steps, Theorem \ref{ag4thm1} in the second and fourth, and Proposition \ref{ag4prop4} in the third. The bijections \eq{ag4eq11} are functorial in $\fC,\fD$ as each step is. Hence $R_{\rm all}^{\rm co}$ is left adjoint to~$\inc$.

For (d), note that $R_{\rm all}^{\rm co}:\CRings\ra\CRingsco$ takes colimits to colimits, as it is a left adjoint functor by (a). So given a functor $F:\cJ\ra\CRingsco$ for $\cJ$ a small category, we may take the colimit $\fC=\mathop{\rm colim}_\cJ F$ in $\CRings$, which exists by Proposition \ref{ag2prop1}, and then $\fD=R_{\rm all}^{\rm co}(\fC)$ is the colimit of $R_{\rm all}^{\rm co}\ci F$ in $\CRingsco$. But $R_{\rm all}^{\rm co}\ci F\cong F$ as $R_{\rm all}^{\rm co}\vert_{\CRingsco}\cong\id$. Hence $\fD=\mathop{\rm colim}_\cJ F$ in $\CRingsco$, and all small colimits exist in $\CRingsco$. In Example \ref{ag2ex10}, the colimits in $\CRingsco$ and $\CRings$ are different.

The first part of (e) holds by composing (c) and Theorem \ref{ag4thm1}, and the second part follows as right adjoint functors preserve limits. This completes the proof of Theorem~\ref{ag4thm4}.
\end{proof}

\begin{rem} Let $\fC$ be a $C^\iy$-ring, so that $\Psi_\fC:\fC\ra R_{\rm all}^{\rm co}(\fC)$ is a morphism of $C^\iy$-rings. If $\fC$ is finitely generated then Theorem \ref{ag4thm2}(a) gives an isomorphism $R_{\rm all}^{\rm co}(\fC)\cong R_{\rm fg}^{\rm fa}(\fC)$ identifying $\Psi_\fC$ with the surjective projection $\pi:\fC\ra R_{\rm fg}^{\rm fa}(\fC)$, for $R_{\rm fg}^{\rm fa}$ as in Definition \ref{ag2def9}. Thus $\Psi_\fC:\fC\ra R_{\rm all}^{\rm co}(\fC)$ is surjective in this case, and $R_{\rm all}^{\rm co},R_{\rm fg}^{\rm fa}$ agree on finitely generated $C^\iy$-rings up to natural isomorphism.

For $\fC$ infinitely generated, $\Psi_\fC:\fC\ra R_{\rm all}^{\rm co}(\fC)$ need not be surjective, and $R_{\rm all}^{\rm co}(\fC)$ can be much larger than $\fC$. For example, if $I$ is an infinite set and $\fC=C^\iy(\R^I)$ is as in Example \ref{ag4ex4}, then elements of $\fC$ are functions $c:\R^I\ra\R$ which depend smoothly only on $x_j$ for $j$ in a finite subset $J\subseteq I$, but elements of $R_{\rm all}^{\rm co}(\fC)$ are functions $c:\R^I\ra\R$ which {\it locally in\/} $\R^I$ depend smoothly only on $x_j$ for $j$ in a finite subset $J\subseteq I$, but globally may depend on $x_i$ for infinitely many $i\in I$. So $\Psi_\fC:\fC\ra R_{\rm all}^{\rm co}(\fC)$ is injective but not surjective.\I{C-ring@$C^\iy$-ring!complete|)}
\label{ag4rem6}
\end{rem}

\subsection{Partitions of unity}
\label{ag47}

We now study the existence of smooth partitions on unity on $C^\iy$-schemes and local $C^\iy$-ringed spaces. We will need the next definition.

\begin{dfn} Let $\uX=(X,\O_X)$ be a local $C^\iy$-ringed space. Then each $c\in\O_X(X)$ defines a continuous map $c_*:X\ra\R$ mapping $x\mapsto\pi\ci\pi_x(c)$, for $\pi_x:\O_X(X)\ra \O_{X,x}$ and $\pi:\O_{X,x}\ra\R$ the natural $C^\iy$-ring morphisms. Thus $U_c=\{x\in X:c_*(x)\ne 0\}$ is open in $X$. We say that {\it the topology on\/ $X$ is smoothly generated\/} if $\{U_c:c\in\O_X(X)\}$ is a basis for the topology on $X$.

This implies $X$ is a regular (and completely regular) topological space.
\label{ag4def14}
\end{dfn}

\begin{ex}{\bf(a)} Let $X$ be a completely regular topological space, and define a sheaf of $C^\iy$-rings $\O_X$ on $X$ by taking $\O_X(U)=C^0(U)$ to be the $C^\iy$-ring of continuous functions $c:U\ra\R$ for all open $U\subseteq X$. Then $\uX=(X,\O_X)$ is a local $C^\iy$-ringed space, and the topology on $X$ is smoothly generated.
\smallskip

\noindent{\bf(b)} Let $\uX$ be an affine $C^\iy$-scheme. Then $\uX\cong\Spec\O_X(X)$ by Theorem \ref{ag4thm4}(a). So the definition of the topology on $X$ in Definition \ref{ag4def8} implies that the topology on $X$ is smoothly generated.
\smallskip

\noindent{\bf(c)} Suppose $\uX$ is a regular $C^\iy$-scheme, and let $T\subseteq X$ be open and $x\in T$. Then $x$ has an affine open neighbourhood $\uY$ in $\uX$. Since $X$ is regular, there exist disjoint open neighbourhoods $V$ of $x$ and $W$ of $X\sm Y$ in~$X$. 

Then $x\in T\cap V\subseteq Y$, and the topology on $Y$ is smoothly generated by {\bf(b)}, so there exists $a\in\O_Y(Y)$ with $x\in U^Y_a\subseteq T\cap V$. Now $a_*(x)\ne 0$ and $a_*(y)=0$ for all $y\in Y\sm U^Y_a$, but this does not imply that $a$ is supported in $U^Y_a$, as we could have $\pi_y(a)\ne 0$ in $\O_{Y,y}$ even though $\pi\ci\pi_y(a)=0$ in $\R$. Choose smooth $f:\R\ra\R$ with $f(a_*(x))\ne 0$ and $f(t)=0$ for $t$ in an open neighbourhood of 0 in $\R$. Set $b=\Phi_f(a)$, for $\Phi_f:\O_Y(Y)\ra\O_Y(Y)$ the $C^\iy$-ring operation. 

Then $b_*(x)\ne 0$, and $U^Y_b\subseteq U^Y_a\subseteq T$, and $b$ is supported in $U^Y_a\subseteq V\subseteq Y$. Since $W$ is open in $X$ with $X\sm Y\subseteq W\subseteq Y\sm V$, there exists a unique $c\in\O_X(X)$ with $c\vert_Y=b$ and $c\vert_W=0$. We have $x\in U^X_c=U^Y_b\subseteq T$. Thus, for each open $T\subseteq X$ and $x\in T$ we can find $c\in\O_X(X)$ with $x\in U^X_c\subseteq T$. So the topology on $X$ is smoothly generated.
\smallskip

\noindent{\bf(d)} Let $X$ be an infinite-dimensional Banach space or Banach manifold, and make $X$ into a local $C^\iy$-ringed space $\uX=(X,\O_X)$ as in Example \ref{ag4ex1}. The question of when the topology of $X$ is smoothly generated (framed in terms of the existence of `smooth bump functions' on $X$) is very well understood, as in Bonic and Frampton \cite{BoNo} and Deville, Godefroy and Zizler \cite[\S V]{DGZ}. For example, if $X$ is a Hilbert manifold, or modelled on $L^q(Y)$ or $\ell^q$ for even $q\ge 2$, then the topology on $X$ is smoothly generated, but if $X$ is modelled on $L^q(Y)$ or $\ell^q$ for $q\in[1,\iy]$ not even, the topology on $X$ is not smoothly generated.
\label{ag4ex5}
\end{ex}

For the next theorem, \S\ref{ag41} defined Lindel\"of spaces, and explained their relation to other topological assumptions. Second countable implies Lindel\"of, and Lindel\"of and regular imply paracompact (note that $X$ is regular as its topology is smoothly generated). It is easy to see that $\O_X$ fine implies that the topology on $X$ is smoothly generated.

The proof of Theorem \ref{ag4thm5} is based on the proof of the existence of smooth partitions on unity on suitable separable Banach manifolds in Bonic and Frampton \cite[Th.~1]{BoNo} (see also Lang \cite[\S II.3]{Lang} and Deville et al.~\cite[\S VIII.3]{DGZ}). 

Theorem \ref{ag4thm5} applies to a very large class of $C^\iy$-schemes, showing that partitions of unity exist on most interesting examples of $C^\iy$-schemes.

\begin{thm} Let\/ $\uX=(X,\O_X)$ be a Lindel\"of local\/ $C^\iy$-ringed space, and suppose the topology on $X$ is smoothly generated. Then $\O_X$ is \begin{bfseries}fine\end{bfseries},\I{sheaf!fine|(} as in Definition\/ {\rm\ref{ag4def6}}. That is, for every open cover $\{V_i:i\in I\}$ of\/ $X$ there exists a subordinate locally finite partition of unity $\{\eta_i:i\in I\}$ in\/~$\O_X(X)$.

\label{ag4thm5}
\end{thm}

\begin{proof} For $c\in\O_X(X)$ and $x\in X$ we have $\pi_x(c)\in\O_{X,x}$ and $c_*(x)=\pi\ci\pi_x(c)\in\R$, where $\pi_x:\O_X(X)\ra\O_{X,x}$ and $\pi:\O_{X,x}\ra\R$ are the natural $C^\iy$-morphisms. Then $c_*:X\ra\R$ is continuous. Write $U_c=\{x\in X:c_*(x)\ne 0\}$, so that $U_c$ is open in $X$. The {\it support\/} of $c$ is~$\supp c=\{x\in X:\pi_x(c)\ne 0\}$. 

Then $\supp c$ is closed in $X$ with $U_c\subseteq\supp c$, but $\supp c$ may be larger than the closure of $U_c$. Note that an infinite sum $\sum_{j\in J}c_j$ in $\O_X(X)$ is defined, as a section of the sheaf $\O_X$, if $\{\supp c_j:j\in J\}$ is locally finite (that is, each $x\in X$ has an open neighbourhood $W_x$ intersecting $\supp c_j$ for only finitely many $j\in J$), but may not make sense if only $\{U_{c_j}:j\in J\}$ is locally finite. Because of this, we are careful to keep track of both $U_{c_j}$ and $\supp c_j$ in the following proof.

Let $\{V_i:i\in I\}$ be an open cover of $X$. Suppose $i\in I$ and $x\in V_i$. As the topology on $X$ is smoothly generated there exists $c\in\O_X(X)$ with $x\in U_c\subseteq V_i$. So $c_*(x)\ne 0$ and $c_*\vert_{X\sm V_i}=0$. We do not know that $\supp c\subseteq V_i$, but we can correct this as follows. Choose smooth $f:\R\ra\R$ such that $f(c_*(x))\ne 0$ and $f=0$ in a neighbourhood of 0 in $\R$. Set $c'=\Phi_f(c)$, where $\Phi_f:\O_X(X)\ra\O_X(X)$ is the $C^\iy$-ring operation. Then~$x\in U_{c'}\subseteq\supp c'\subseteq U_c\subseteq V_i\subseteq X$. 

Thus, we can choose a family $\{c_j:j\in J\}$ such that $c_j\in\O_X(X)$, and $U_{c_j}\subseteq\supp c_j\subseteq V_{i_j}\subseteq X$ for each $j\in J$ and some $i_j\in I$, and $\{U_{c_j}:j\in J\}$ is an open cover of $X$. Since $X$ is Lindel\"of we can take $J$ to be countable, and choose~$J=\N$.

Replacing $c_j$ by $c_j^2$ we have $(c_j)_*\ge 0$ on $X$. For each $j\in\N$, choose smooth $f_j:\R^{j+1}\ra\R$ such that $f_j(t_0,t_1,\ldots,t_j)>0$ if $t_i<1/j$ for $i=0,1,\ldots,j-1$ and $t_j>0$, and $f_j(t_0,t_1,\ldots,t_j)=0$ otherwise. Define $d_j=\Phi_{f_j}(c_0,c_1,\ldots,c_j)$, with $\Phi_{f_j}:\O_X(X)^{j+1}\ra\O_X(X)$ the $C^\iy$-ring operation. Then
\e
\begin{split}
&U_{d_j}=\bigl\{x\in X:(d_j)_*(x)\ne 0\bigr\}\\
&=\bigl\{x\in X:(c_i)_*(x)<1/j,\; i=0,\ldots,j-1,\;  (c_j)_*(x)\ne 0\bigr\}\subseteq V_{i_j},\\
&\supp d_j\subseteq \bigl\{x\in X:(c_i)_*(x)\le 1/j,\; i=1,\ldots,j-1\bigr\}\cap\supp c_j\subseteq V_{i_j}.
\end{split}
\label{ag4eq12}
\e

Fix $x\in X$. Then $x\in U_{c_j}$ for some $j\in\N$ as $\{U_{c_j}:j\in J\}$ covers $X$. Let $j\in\N$ be least with $x\in U_{c_j}$. Then $(c_j)_*(x)>0$ and $(c_i)_*(x)=0$ for $i=0,1,\ldots,j-1$. Thus $x\in U_{d_j}$, so $\{U_{d_j}:j\in\N\}$ is an open cover of $X$. Define $T_x=\{y\in X:(c_j)_*(y)>\ha (c_j)_*(x)\}$. Then $T_x$ is an open neighbourhood of $x$ in $X$, and $T_x\cap U_{d_k}=\es=T_x\cap\supp d_k$ provided $k>\max\bigl(j,2(c_j)_*(x)^{-1}\bigr)$ by \eq{ag4eq12}. Thus, both $\{U_{d_j}:j\in\N\}$ and $\{\supp d_j:j\in\N\}$ are locally finite.

For each $i\in I$, define $e_i=\sum_{j\in\N:i_j=i}d_j$ in $\O_X(X)$. This is well defined as $\{\supp d_j:j\in\N\}$ is locally finite. We have $U_{e_i}\subseteq\supp e_i\subseteq V_i$, since $U_{d_j}\subseteq \supp d_j\subseteq V_i$ for each $j\in\N$ with $i_j=i$. Both $\{U_{e_i}:i\in I\}$ and $\{\supp e_i:i\in I\}$ are locally finite, as $\{U_{d_j}:j\in\N\}$ and $\{\supp d_j:j\in\N\}$ are. Thus $e=\sum_{i\in I}e_i$ is well defined in $\O_X(X)$. If $x\in X$ then
\begin{equation*}
e_*(x)=\ts\sum_{i\in I}(e_i)_*(x)=\sum_{i\in I}\sum_{j\in\N:i_j=i}(d_j)_*(x)=\sum_{j\in\N}(d_j)_*(x)>0,
\end{equation*}
where each sum has only finitely many nonzero terms, and $\sum_{j\in\N}(d_j)_*(x)>0$ as $\{U_{d_j}:j\in\N\}$ covers $X$ with $(d_j)_*>0$ on $U_{d_j}$ and $(d_j)_*=0$ on $X\sm U_{d_j}$. Since $e_*$ is positive on $X$, $e$ is invertible in $\O_X(X)$. Set $\eta_i=e^{-1}\cdot e_i$ for $i\in I$. Then $\supp\eta_i\subseteq V_i$, as $\supp e_i\subseteq V_i$, and $\{\eta_i:i\in I\}$ is locally finite, as $\{\supp e_i:i\in I\}$ is, and $\sum_{i\in I}\eta_i=\sum_{i\in I}e^{-1}\cdot e_i=e^{-1}\cdot e=1$. Hence $\{\eta_i:i\in I\}$ is a locally finite partition of unity subordinate to $\{V_i:i\in I\}$, so $\O_X$ is fine.\I{sheaf!fine|)}
\end{proof}

\subsection{\texorpdfstring{A criterion for affine $C^\iy$-schemes}{A criterion for affine C∞-schemes}}
\label{ag48}

Here are sufficient conditions for a local $C^\iy$-ringed space $\uX$ to be an affine $C^\iy$-scheme. Note that affine $C^\iy$-schemes are Hausdorff with smoothly generated topology by Lemma \ref{ag4lem1} and Example \ref{ag4ex5}(b), so Lindel\"of is the only  condition in the theorem which is not also necessary.

\begin{thm} Let\/ $\uX=(X,\O_X)$ be a Hausdorff, Lindel\"of, local\/ $C^\iy$-ringed space, with smoothly generated topology. Then\/ $\uX$ is an affine\/ $C^\iy$-scheme.
\label{ag4thm6}
\end{thm}

\begin{proof} Let $\uX$ be as in the theorem. Note that Theorem \ref{ag4thm5} shows that $\O_X$ is fine. Write $\fC=\O_X(X)=\Ga(\uX)$, and $\uY=\Spec\fC$. Define a morphism $\uf:\uX\ra\uY$ by $\uf=L_{\fC,\uX}(\id_\fC)$, using the notation of Theorem \ref{ag4thm1}. We will show $\uf$ is an isomorphism, so that $\uX\cong\Spec\fC$ is an affine $C^\iy$-scheme.

Points $x\in X$ induce $C^\iy$-ring morphisms $\pi\ci\pi_x:\fC=\O_X(X)\ra\R$, where $\pi_x:\O_X(X)\ra\O_{X,x}$ and $\pi:\O_{X,x}\ra\R$ are the natural projections. Points $y\in Y$ are $C^\iy$-ring morphisms $y:\fC\ra\R$, and $f:X\ra Y$ is $f(x)=\pi\ci\pi_x$.

Suppose $x,x'\in X$ with $x\ne x'$, and set $f(x)=y$ and $f(x')=y'$. Since $X$ is Hausdorff there exists open $U\subseteq X$ with $x\in U$ and $x'\notin U$. As the topology on $X$ is smoothly generated there exists $c\in\O_X(X)$ with $c_*(x)\ne 0$ and $c_*\vert_{X\sm U}=0$, so that $c_*(x')=0$. Then $y(c)=c_*(x)\ne 0$ and $y'(c)=c_*(x')=0$, so $y\ne y'$. Hence $f:X\ra Y$ is injective.

Suppose for a contradiction that $y\in Y$, but $f(x)\ne y$ for all $x\in X$. Then for each $x\in X$, there exists $a\in\fC$ with $y(a)\ne\pi\ci\pi_x(a)$. Choose smooth $g:\R\ra\R$ with $g(y(a))=0$ and $g=1$ in an open neighbourhood of $\pi\ci\pi_x(a)$ in $\R$. Set $b=\Phi_g(a)$, where $\Phi_g:\fC\ra\fC$ is the $C^\iy$-ring operation. Then $y(b)=0$ and $\pi\ci\pi_{\ti x}(b)=1$ for $\ti x$ in an open neighbourhood $V$ of $x$ in~$X$.  

Thus we may choose a family of pairs $\{(V_j,b_j):j\in J\}$ such that for each $j\in J$ we have $V_j\subseteq X$ open and $b_j\in\fC$ with $y(b_j)=0$ and $\pi\ci\pi_x(b_j)=1$ for $x\in V_j$, and $\{V_j:j\in J\}$ is an open cover of $X$. As $X$ is Lindel\"of we can suppose $J$ is countable, and so take $J=\N$. By Theorem \ref{ag4thm5} there exists a locally finite partition of unity $\{\eta_j:j\in\N\}$ in $\fC$ subordinate to~$\{V_j:j\in\N\}$. 

Set $c=\sum_{j\in\N}j\cdot\eta_j\cdot b_j$ in $\fC=\O_X(X)$, which makes sense in global sections of $\O_X$ as $\{\eta_j:j\in\N\}$ is locally finite. Choose $n\in\N$ with $n>y(c)$, and define $d=c-y(c)\cdot 1_X+\sum_{j=0}^{n-1}(n-j)\cdot\eta_j\cdot b_j$ in $\fC$, where $1_X\in\fC$ is the identity. Then
\begin{equation*}
y(d)=y(c)-y(c)\cdot y(1_X)+\ts\sum_{j=0}^{n-1}(n-j)\cdot y(\eta_j)\cdot y(b_j)=0,
\end{equation*}
as $y(1_X)=1$ and $y(b_j)=0$. And if $x\in X$ then
\begin{align*}
\pi\ci\pi_x(d)&=\ts\pi\ci\pi_x\bigl[\sum_{j\in\N}j\cdot\eta_j\cdot b_j-y(c)\cdot \sum_{j\in\N}\eta_j+\sum_{j=0}^{n-1}(n-j)\cdot\eta_j\cdot b_j\bigr]\\
&=\ts\sum_{j\in\N}\bigl(\max(j,n)-y(c)\bigr)\pi\ci\pi_x(\eta_j)>0,
\end{align*}
where each sum has only finitely many nonzero terms, and we use $\sum_{j\in\N}\eta_j=1_X$,
$\pi\ci\pi_x(b_j)=1$, and $\max(j,n)-y(c)>0$, $\pi\ci\pi_x(\eta_j)\ge 0$ for $j\in\N$. 

Since $\pi\ci\pi_x(d)>0$ for all $x\in X$, we see that $d$ is invertible in $\fC=\O_X(X)$, but this contradicts $y(d)=0$. Hence each $y\in Y$ has $y=f(x)$ for some $x\in X$, and $f$ is surjective, so $f:X\ra Y$ is a bijection. By definition of $\uY=\Spec\fC$, the topology on $Y$ is generated by the open sets $U_c=\{y\in Y:y(c)\ne 0\}$ for all $c\in\fC$. As the topology on $X$ is smoothly generated, it is generated by the open sets $f^{-1}(U_c)=\{x\in X:c_*(x)\ne 0\}$ for $c\in\fC$. Therefore $f:X\ra Y$ is a bijection identifying bases for the topologies of $X,Y$, so $f$ is a homeomorphism.

Let $x\in X$ with $f(x)=y\in Y$. Taking stalks of $f^\sh:f^{-1}(\O_Y)\ra\O_X$ at $x$ gives a morphism $f^\sh_x:\O_{Y,y}\ra\O_{X,x}$. By the definition of $\uf=L_{\fC,\uX}(\id_\fC)$ in the proof of Theorem \ref{ag4thm1}, $f^\sh_x$ agrees with $\phi_x$ in \eq{ag4eq8}, and is the unique morphism making the following commute, where $\fC_y\cong\O_{Y,y}$ by Lemma~\ref{ag4lem2}:
\e
\begin{gathered}
\xymatrix@C=70pt@R=15pt{ *+[r]{\fC} \ar[d]^{\pi_y}\ar@/^.5pc/[ddr]^(0.42)y \ar@{=}[rr] && *+[l]{\O_X(X)} \ar[d]_{\pi_x} \\
*+[r]{\fC_y\cong\O_{Y,y}} \ar[dr] _\pi \ar[rr]^{f^\sh_x} && *+[l]{\O_{X,x}} \ar[dl]^\pi \\ & \R.\! }
\end{gathered}
\label{ag4eq13}
\e 

Suppose $a_y\in\O_{Y,y}$ with $f_x^\sh(a_y)=0$. Then $a_y=\pi_y(a)$ for some $a\in\fC=\O_X(X)$, as $\pi_y$ is surjective by Proposition \ref{ag2prop2}, and then $\pi_x(a)=0$ in $\O_{X,x}$, as \eq{ag4eq13} commutes. Hence there exists an open neighbourhood $U$ of $x$ in $X$ with $a\vert_U=0$ in $\O_X(U)$. As the topology on $X$ is smoothly generated, there exists $b\in\O_X(X)$ with $b_*(x)\ne 0$ and $b_*\vert_{X\sm U}=0$. Choose smooth $g:\R\ra\R$ with $g(b_*(x))\ne 0$ and $g=0$ near 0 in $\R$, and set $c=\Phi_g(b)$, where $\Phi_g:\O_X(X)\ra\O_X(X)$ is the $C^\iy$-ring operation. Then $y(c)=c_*(x)\ne 0$, and $c$ is supported in $U$. As $a\vert_U=0$ we see that $a\cdot c=0$ in $\O_X(X)$. Thus $a$ lies in the ideal $I$ in \eq{ag2eq2} which is the kernel of $\pi_y:\fC\ra\fC_y$, by Proposition \ref{ag2prop2}, and so $a_y=\pi_y(a)=0$. Therefore $f^\sh_x:\O_{Y,y}\ra\O_{X,x}$ is injective.

Suppose $a_x\in\O_{X,x}$. Then by definition of $\O_{X,x}$ there exists open $x\in U\subseteq X$ and $a\in\O_X(U)$ with $\pi_x(a)=a_x$. As the topology on $X$ is smoothly generated there exists $b\in\O_X(X)$ with $b_*(x)\ne 0$ and $b_*\vert_{X\sm U}=0$. Choose smooth $g:\R\ra\R$ with $g=1$ near $b_*(x)$ in $\R$ and $g=0$ near 0 in $\R$. Set $c=\Phi_g(b)$, where $\Phi_g:\O_X(X)\ra\O_X(X)$ is the $C^\iy$-ring operation. Then $c$ is supported in $U$, and there exists an open neighbourhood $V$ of $x$ in $U$ with $c\vert_V=1$. Since $c$ is supported in $U$, the section $c\vert_U\cdot a\in\O_X(U)$ can be extended by zero over $X\sm U$ to give a unique $d\in\O_X(X)$ supported in $U$ with $d\vert_U=c\vert_U\cdot a$. 

Then $d\vert_V=c\vert_V\cdot a\vert_V=1\cdot a\vert_V=a\vert_V$. Hence $f_x^\sh\ci\pi_y(d)=\pi_x(d)=a_x$, so $f^\sh_x:\O_{Y,y}\ra\O_{X,x}$ is surjective, and an isomorphism. This proves that $f^\sh:f^{-1}(\O_Y)\ra\O_X$ is an isomorphism on stalks at every $x\in X$, so $f^\sh$ is an isomorphism. As $f$ is a homeomorphism, $\uf=(f,f^\sh):\uX\ra\Spec\fC$ is an isomorphism. This completes the proof of Theorem~\ref{ag4thm6}.
\end{proof}

\begin{cor} Let\/ $\uX=(X,\O_X)$ be a local\/ $C^\iy$-ringed space. Then the following are equivalent:
\begin{itemize}
\setlength{\itemsep}{0pt}
\setlength{\parsep}{0pt}
\item[{\bf(i)}] $X$ is Hausdorff\I{topological space!Hausdorff} and second countable,\I{topological space!second countable} with smoothly generated topology. 
\item[{\bf(ii)}] $X$ is separable\I{topological space!separable} and metrizable,\I{topological space!metrizable} with smoothly generated topology.
\item[{\bf(iii)}] $\uX$ is a Hausdorff, second countable, regular\I{topological space!regular} $C^\iy$-scheme.
\item[{\bf(iv)}] $\uX$ is a separable, metrizable $C^\iy$-scheme.
\item[{\bf(v)}] $\uX$ is a second countable, affine $C^\iy$-scheme.
\end{itemize}
When these hold, $X$ is regular, normal,\I{topological space!normal} and paracompact,\I{topological space!paracompact} and\/ $\O_X$ is fine.\I{sheaf!fine}
\label{ag4cor4}
\end{cor}

\begin{proof} Section \ref{ag41} implies that (i),(ii) are equivalent (as $X$ smoothly generated topology implies $X$ regular), and (iii),(iv) are equivalent. Also (v) implies (iii) by Lemma \ref{ag4lem1}, and (iii) implies (i) by Example \ref{ag4ex5}(b), and (i) implies (v) by Theorem \ref{ag4thm6} (as second countable implies Lindel\"of). Hence (i)--(v) are equivalent. The last part follows from \S\ref{ag41} and Theorem~\ref{ag4thm5}.
\end{proof}

In comparison to Theorem \ref{ag4thm6}, we have strengthened the Lindel\"of assumption to second countable. The category of $C^\iy$-schemes in Corollary \ref{ag4cor4} is very large, and convenient to work in. They are closed under products, fibre products, and arbitrary subspaces (Lindel\"of spaces are none of these). They have partitions of unity, and as they are affine we can argue globally using $C^\iy$-rings.

\begin{ex} Let $\uX=(X,\O_X)$ be a second countable, affine $C^\iy$-scheme, and let $Y\subseteq X$ be {\it any\/} subset, not necessarily open or closed. Then $\uY=(Y,\O_X\vert_Y)$ is also a second countable, affine $C^\iy$-scheme by Corollary \ref{ag4cor4}, as being Hausdorff, second countable, and of smoothly generated topology, are all preserved under passing to subspaces, so $\uY$ satisfies Corollary \ref{ag4cor4}(i) as $\uX$ does.

\label{ag4ex6}
\end{ex}

\begin{ex} Let $X$ be a separable Banach manifold\I{Banach manifold} modelled locally on separable Banach spaces $B$ which admit `smooth bump functions' (that is, there exists a nonzero smooth function $f:B\ra\R$ with bounded support in $B$). See Deville et al.\ \cite[\S V]{DGZ} for results on when a Banach space $B$ has a smooth bump function, for example, every Hilbert space does. 

Make $X$ into a local $C^\iy$-ringed space $\uX=(X,\O_X)$ as in Example \ref{ag4ex1}. Then the topology on $X$ is smoothly generated as in Example \ref{ag4ex5}(d), so $\uX$ is an affine $C^\iy$-scheme by Corollary~\ref{ag4cor4}(ii),(v).
\label{ag4ex7}
\end{ex}

\subsection{\texorpdfstring{Quotients of $C^\iy$-schemes by finite groups}{Quotients of C∞-schemes by finite groups}}
\label{ag49}

Finally we discuss quotients of $C^\iy$-schemes by finite groups.

\begin{dfn} Let $\uX=(X,\O_X)$ be a local $C^\iy$-ringed space, $G$ a finite group, and $\ur:G\ra\Aut(\uX)$ an action of $G$ on $\uX$. We will define a local $C^\iy$-ringed space~$\uY=\uX/G$.

Set $Y=X/r(G)$ to be the quotient topological space. Open sets $V\subseteq Y$ are of the form $U/G$ for $U\subseteq X$ open and $G$-invariant. Then $\ga\mapsto r^\sh(\ga)(U)$ gives an action of $G$ on the $C^\iy$-ring $\O_X(U)$, so as in Proposition \ref{ag2prop4} we have a $C^\iy$-ring $\O_X(U)^G$, the $G$-invariant subspace in $\O_X(U)$. Define~$\O_Y(V)=\O_X(U)^G$.

If $V_2\subseteq V_1\subseteq Y$ are open then $V_1=U_1/G$, $V_2=U_2/G$ for $U_2\subseteq U_1\subseteq X$ open and $G$-invariant. The restriction morphism $\rho_{U_1U_2}:\O_X(U_1)\ra\O_X(U_2)$ in $\O_X$ is $G$-equivariant, and so restricts to $\rho_{U_1U_2}\vert_{\O_X(U_1)^G}:\O_X(U_1)^G\ra\O_X(U_2)^G$. Set $\rho_{V_1V_2}=\rho_{U_1U_2}\vert_{\O_X(U_1)^G}:\O_Y(V_1)\ra \O_Y(V_2)$. It is now easy to check that $\O_Y$ is a sheaf of $C^\iy$-rings on $Y$, so $\uY=(Y,\O_Y)$ is a $C^\iy$-ringed space.

If $x\in X$ and $y=xG\in Y$, the stalk $\O_{Y,y}$ of $\O_Y$ at $y$ is $(\O_{X,x})^H$, where $\O_{X,x}$ is a local $C^\iy$-ring, and $H=\bigl\{\ga\in G:\ga(x)=x\bigr\}$ is the stabilizer group of $x$ in $G$, which acts on $\O_{X,x}$ in the obvious way. As $\O_{X,x}$ is local there is an $\R$-algebra morphism $\pi:\O_{X,x}\ra\R$, such that $c\in\O_{X,x}$ is invertible if and only if $\pi(c)\ne 0$. Thus $\pi\vert_{(\O_{X,x})^H}:(\O_{X,x})^H\ra\R$ is an $\R$-algebra morphism, and $c\in(\O_{X,x})^H$ is invertible in $\O_{X,x}$ if and only if $\pi(c)\ne 0$. But if $c\in(\O_{X,x})^H$ is invertible in $\O_{X,x}$ then $c^{-1}$ is $H$-invariant, so $c$ is invertible in $(\O_{X,x})^H$. Therefore $\O_{Y,y}\cong(\O_{X,x})^H$ is a local $C^\iy$-ring, and $\uY$ is a local $C^\iy$-ringed space. Write~$\uX/G=\uY$.

Define $\pi:X\ra X/G$ to be the natural projection. Define a morphism $\pi_\sh:\O_Y\ra\pi_*(\O_X)$ of sheaves of $C^\iy$-rings on $Y=X/G$ by
\begin{equation*}
\pi_\sh(V)=\inc:\O_Y(V)=\O_X(U)^G\longra\O_X(U)=\pi_*(\O_X)(V)
\end{equation*}
for all open $V=U/G\subseteq Y=X/G$, where $\inc:\O_X(U)^G\hookra\O_X(U)$ is the inclusion. Let $\pi^\sh:\pi^{-1}(\O_Y)\ra\O_X$ be the morphism of sheaves of $C^\iy$-rings on $X$ corresponding to $\pi_\sh$ under \eq{ag4eq3}. Then $\upi=(\pi,\pi^\sh):\uX\ra\uX/G$ is a morphism of local $C^\iy$-ringed spaces.

It is easy to see that $\uX/G,\upi$ have the universal property that if $\uf:\uX\ra\uZ$ is a morphism in $\LCRS$ with $\uf\ci\ur(\ga)=\uf$ for all $\ga\in G$ then $\uf=\ug\ci\upi$ for a unique morphism $\ug:\uX/G\ra\uZ$ in~$\LCRS$.
\label{ag4def15}
\end{dfn}

\begin{prop} Let\/ $\uX=(X,\O_X)$ be an affine $C^\iy$-scheme, $G$ a finite group, and\/ $\ur:G\ra\Aut(\uX)$ an action of\/ $G$ on $\uX$. Suppose $X$ is Lindel\"of.

Then $\uX=\Spec\fC$ for $\fC=\O_X(X)$ a complete $C^\iy$-ring, and\/ $\ur=\Spec s$ for $s:G\ra\Aut(\fC)$ a unique action of\/ $G$ on $\fC$. Form the $G$-invariant\/ $C^\iy$-ring\/ $\fC^G\subseteq\fC$ as in Proposition\/ {\rm\ref{ag2prop4}}. Then $\fC^G$ is complete, and there is a canonical isomorphism $\uX/G\cong\Spec\fC^G$ in\/~$\LCRS$.
\label{ag4prop5}
\end{prop}

\begin{proof} Theorem \ref{ag4thm4}(a) shows that $\uX\cong\Spec\fC$, where $\fC=\O_X(X)$ is a complete $C^\iy$-ring. As $\Spec$ is full and faithful on complete $C^\iy$-rings by Theorem \ref{ag4thm4}(b), $\Spec:\Aut(\fC)\ra\Aut(\uX)$ is an isomorphism, so there is a unique action $s:G\ra\Aut(\fC)$ with $\ur=\Spec s$.

Let $\uY=\uX/G$ be as in Definition \ref{ag4def15}. Then $Y=X/G$ is Hausdorff, as $X$ is Hausdorff and $G$ is finite. Suppose $\{V_i:i\in I\}$ is an open cover of $Y$. Then $V_i=U_i/G$ for $\{U_i:i\in I\}$ an open cover of $X$. As $X$ is Lindel\"of there exists a subcover $\{U_i:i\in S\}$ for countable $S\subseteq I$, and then $\{V_i:i\in S\}$ is a countable subcover of $\{V_i:i\in I\}$. Hence $Y$ is Lindel\"of.

Suppose $V\subseteq Y$ is open and $y\in V$. Then $V=U/G$ and $y=xG$ for $G$-invariant open $U\subseteq X$ with $x\in U$. As the topology on $X$ is smoothly generated, there exists $c\in\fC$ with $c_*(x)\ne 0$ and $c_*(x')=0$ for all $x'\in X\sm U$. Define $d=\sum_{\ga\in G}\ga^*(c^2)$ in $\fC$. Then $d$ is $G$-invariant with $d_*(x)>0$ and $d_*(x')=0$ for all $x'\in X\sm U$. Hence $d\in\O_Y(Y)=\O_X(X)^G=\fC^G$, with $d_*(y)>0$ and $d_*(y')=0$ for all $y'\in Y\sm V$. Thus the topology of $Y$ is smoothly generated.

Theorem \ref{ag4thm6} now implies that $\uY=\uX/G$ is an affine $C^\iy$-scheme, and Theorem \ref{ag4thm4}(a) gives a canonical isomorphism $\uX/G\cong\Spec\O_Y(Y)=\Spec\fC^G$, where $\fC^G$ is complete.
\end{proof}

\begin{prop} Suppose $\uX$ is a Hausdorff, second countable\/ $C^\iy$-scheme, $G$ a finite group, and\/ $\ur:G\ra\Aut(\uX)$ an action of\/ $G$ on $\uX$. Then the quotient\/ $\uX/G$ is also a Hausdorff, second countable\/ $C^\iy$-scheme. If\/ $\uX$ is locally fair, or locally finitely presented, then so is $\uX/G$.
\label{ag4prop6}
\end{prop}

\begin{proof} Let $x\in X$, and write $H=\bigl\{\ga\in G:\ga(x)=x\bigr\}$. Then the $G$-orbit $xG$ is $\md{G}/\md{H}$ points. Since $X$ is Hausdorff and $G$ is finite, we can find an open neighbourhood $R$ of $x$ in $X$ such that $R$ is $H$-invariant and $R\cap\ga\cdot R=\es$ for all $\ga\in G\sm H$. As $\uX$ is a $C^\iy$-scheme, there is an open neighbourhood $S$ of $x$ in $R$ with $(S,\O_X\vert_S)$ an affine $C^\iy$-scheme. Then $T=\bigcap_{\ga\in H}\ga\cdot S$ is an $H$-invariant open neighbourhood of $x$ in $S$. Choose an open neighbourhood $U$ of $x$ in $T$ with $(U,\O_X\vert_U)$ an affine $C^\iy$-scheme.

Define $V=\bigcap_{\ga\in H}\ga\cdot U$. Then $V$ is an $H$-invariant open neighbourhood of $x$ in $U\subseteq T\subseteq S\subseteq R\subseteq X$. It is the intersection of the $\md{H}$ affine $C^\iy$-subschemes $(\ga\cdot U,\O_X\vert_{\ga\cdot U})$ for $\ga\in H$ inside the affine $C^\iy$-scheme $(S,\O_X\vert_S)$. Finite intersections of affine $C^\iy$-subschemes in an affine $C^\iy$-scheme are affine, as such intersections are fibre products and $\Spec:\CRings^{\bf op}\ra\LCRS$ preserves limits by Remark \ref{ag4rem4}(b). Thus $(V,\O_X\vert_V)$ is an affine $C^\iy$-scheme.

Set $W=\bigcup_{\ga H\in G/H}\ga\cdot V$. Then $W$ is a $G$-invariant open neighbourhood of $x$ in $X$, and $(W,\O_X\vert_W)$ is the disjoint union of $\md{G}/\md{H}$ affine $C^\iy$-schemes isomorphic to $(V,\O_X\vert_V)$, so it is affine. We have shown that every $x\in X$ has a $G$-invariant open neighbourhood $W\subseteq X$ with $\uW=(W,\O_X\vert_W)$ affine. Then $\uW/G$ is an open neighbourhood of $xG$ in $\uX/G$. As $\uX$ is second countable, $\uW$ is second countable and so Lindel\"of. Thus $\uW/G$ is an affine $C^\iy$-scheme by Proposition \ref{ag4prop5}. As we can cover $\uX/G$ by such open $\uW/G$, it is a $C^\iy$-scheme.

If $\uX$ is locally fair, or locally finitely presented, we can do the argument above with $\uS,\uU,\uV,\uW,\uW/G$ fair, or finitely presented, using Proposition \ref{ag2prop4} for $\uW/G$, so $\uX/G$ is also locally fair, or locally finitely presented.\I{C-scheme@$C^\iy$-scheme|)}\end{proof}

\section{\texorpdfstring{Modules over $C^\iy$-rings and $C^\iy$-schemes}{Modules over C∞-rings and C∞-schemes}}
\label{ag5}
\I{C-ring@$C^\iy$-ring!module|see{module over $C^\iy$-ring}}
\I{module over C-ring@module over $C^\iy$-ring|(}

Next we discuss modules over $C^\iy$-rings, and sheaves of modules on $C^\iy$-schemes. The author knows of no previous work on these, so all this section may be new, although much of it is a straightforward generalization of well known facts.

\subsection{\texorpdfstring{Modules over $C^\iy$-rings}{Modules over C∞-rings}}
\label{ag51}

\begin{dfn} Let $\fC$ be a $C^\iy$-ring. A {\it module\/ $M$
over\/} $\fC$, or $\fC$-{\it module}, is a module over $\fC$
regarded as a commutative $\R$-algebra\I{C-ring@$C^\iy$-ring!as
commutative $\R$-algebra} as in Definition \ref{ag2def3}, and morphisms of $\fC$-modules are morphisms of $\R$-algebra modules. We will write $\mu_M:\fC\t M\ra M$ for the multiplication map, and also write $\mu_M(c,m)=c\cdot m$ for $c\in\fC$ and $m\in M$. Then $\fC$-modules form an abelian category,\I{abelian category} which we write as
$\fCmod$.\G[Cmod]{$\fCmod$}{abelian category of modules over a
$C^\iy$-ring $\fC$} 

The action of $\fC$ on itself by multiplication makes $\fC$ into a $\fC$-module, and more generally $\fC\ot_\R V$ is a $\fC$-module for any $\R$-vector space $V$. A $\fC$-module $M$ is {\it finitely generated\/} if it fits into an exact sequence $\fC\ot\R^n\ra M\ra 0$ in $\fCmod$, and  {\it finitely presented\/} if it fits into an exact sequence $\fC\ot\R^m\ra \fC\ot\R^n\ra M\ra 0$.

Because $C^\iy$-rings such as $C^\iy(\R^n)$ are not noetherian, finitely generated $\fC$-modules generally need not be finitely presented.

Now let $\phi:\fC\ra\fD$ be a morphism of $C^\iy$-rings. If $M$ is a $\fC$-module then $\phi_*(M)=M\ot_\fC\fD$ is a $\fD$-module, and this induces a functor $\phi_*:\fCmod\ra\fDmod$. Also, any $\fD$-module $N$ may be regarded as a $\fC$-module $\phi^*(N)=N$ with $\fC$-action $\mu_{\phi^*(N)}(c,n)=\mu_N(\phi(c),n)$, and this defines a functor $\phi^*:\fDmod\ra\fCmod$. Note that $\phi_*:\fCmod\ra\fDmod$ takes finitely generated (or finitely presented) $\fC$-modules to finitely generated (or finitely presented) $\fD$-modules, but $\phi^*:\fDmod\ra\fCmod$ generally does not.
\label{ag5def1}
\end{dfn}

Vector bundles $E$ over manifolds $X$ give examples of modules
over~$C^\iy(X)$.

\begin{ex} Let $X$ be a manifold\I{manifold!vector bundles
on} and $E\ra X$ be a vector bundle, and write $\Ga^\iy(E)$
for the vector space of smooth sections $e$ of $E$. This is a module over the $C^\iy$-ring $C^\iy(X)$, multiplying functions on $X$ by sections of $E$.

Let $E,F\ra X$ be vector bundles over $X$ and $\la:E\ra F$ a
morphism of vector bundles. Then $\la_*:\Ga^\iy(E)\ra \Ga^\iy(F)$
defined by $\la_*:e\mapsto \la\ci e$ is a morphism of
$C^\iy(X)$-modules.

Now let $X,Y$ be manifolds and $f:X\ra Y$ a (weakly) smooth
map.\I{manifold with corners!weakly smooth map} Then
$f^*:C^\iy(Y)\ra C^\iy(X)$ is a morphism of $C^\iy$-rings. If $E\ra
Y$ is a vector bundle over $Y$, then $f^*(E)$ is a vector bundle
over $X$. Under the functor $(f^*)_*:C^\iy(Y)$-mod$\,\ra
C^\iy(X)$-mod of Definition \ref{ag5def1}, we see that
$(f^*)_*\bigl(\Ga^\iy(E) \bigr)=\Ga^\iy(E)\ot_{C^\iy(Y)}C^\iy(X)$ is
isomorphic as a $C^\iy(X)$-module to~$\Ga^\iy\bigl(f^*(E)\bigr)$.
\label{ag5ex1}
\end{ex}

If $E\ra X$ is any vector bundle over a manifold $X$ then by choosing
sections $e_1,\ldots,e_n\in \Ga^\iy(E)$ for $n\gg 0$ such that
$e_1\vert_x,\ldots,e_n\vert_x$ span $E\vert_x$ for all $x\in X$ we
obtain a surjective morphism of vector bundles $\psi:X\t\R^n\ra E$,
whose kernel is another vector bundle $F$. By choosing another
surjective morphism $\phi:X\t\R^m\ra F$ we obtain an exact sequence
of vector bundles
\begin{equation*}
\smash{\xymatrix@C=30pt{ X\t\R^m \ar[r]^\phi & X\t\R^n \ar[r]^(0.6)\psi & E \ar[r] & 0, }}
\end{equation*}
which induces an exact sequence of $C^\iy(X)$-modules
\begin{equation*}
\smash{\xymatrix@C=30pt{ C^\iy(X)\ot_\R\R^m \ar[r]^{\phi_*} & C^\iy(X)\ot_\R\R^n \ar[r]^(0.6){\psi_*} & \Ga^\iy(E) \ar[r] & 0. }}
\end{equation*}
Thus $\Ga^\iy(E)$ is a finitely presented $C^\iy(X)$-module.

\subsection{\texorpdfstring{Cotangent modules of $C^\iy$-rings}{Cotangent modules of C∞-rings}}
\label{ag52}
\I{module over C-ring@module over $C^\iy$-ring!cotangent module
$\Om_\fC$|(}\I{C-ring@$C^\iy$-ring!cotangent module $\Om_\fC$|(}

Given a $C^\iy$-ring $\fC$, we will define the {\it cotangent
module\/} $\Om_\fC$ of $\fC$. Although our definition of
$\fC$-module only used the commutative $\R$-algebra underlying the
$C^\iy$-ring $\fC$, our definition of the particular $\fC$-module
$\Om_\fC$ does use the $C^\iy$-ring structure in a nontrivial way. It is a $C^\iy$-ring version of the {\it module of relative differential forms\/} or {\it K\"ahler differentials\/} in Hartshorne \cite[p.~172]{Hart}, and is an example of a construction for Fermat theories by Dubuc and Kock~\cite{DuKo}.

\begin{dfn} Suppose $\fC$ is a $C^\iy$-ring, and $M$ a
$\fC$-module. A $C^\iy$-{\it derivation}\I{module over C-ring@module
over $C^\iy$-ring!C-derivation@$C^\iy$-derivation}%
\I{C-ring@$C^\iy$-ring!C-derivation@$C^\iy$-derivation} is an
$\R$-linear map $\d:\fC\ra M$ such that whenever $f:\R^n\ra\R$ is a
smooth map and $c_1,\ldots,c_n\in\fC$, we have
\e
\d\Phi_f(c_1,\ldots,c_n)=\ts\sum\limits_{i=1}^n\Phi_{\frac{\pd
f}{\pd x_i}}(c_1,\ldots,c_n)\cdot \d c_i.
\label{ag5eq1}
\e
Note that $\d$ is {\it not\/} a morphism of $\fC$-modules. We call
such a pair $M,\d$ a {\it cotangent module\/} for $\fC$ if it
has the universal property that for any $C^\iy$-derivation $\d':\fC\ra M'$, there exists a unique morphism of $\fC$-modules $\la:M\ra M'$ with~$\d'=\la\ci\d$.

There is a natural construction for a cotangent module: we take
$M$ to be the quotient of the free $\fC$-module with basis of
symbols $\d c$ for $c\in\fC$ by the $\fC$-submodule spanned by all
expressions of the form $\d\Phi_f(c_1,\ldots,c_n)-\sum_{i=1}^n
\Phi_{\frac{\pd f}{\pd x_i}}(c_1,\ldots,c_n)\cdot\d c_i$
for $f:\R^n\ra\R$ smooth and $c_1,\ldots,c_n\in\fC$. Thus cotangent
modules exist, and are unique up to unique isomorphism. When we
speak of `the' cotangent module, we mean that constructed above. We
write $\d_\fC:\fC\ra\Om_\fC$ for the cotangent module of~$\fC$.

Let $\fC,\fD$ be $C^\iy$-rings with cotangent modules
$\Om_\fC,\d_\fC$, $\Om_\fD,\d_\fD$, and $\phi:\fC\ra\fD$ be a morphism of $C^\iy$-rings. Then we may regard $\Om_\fD=\phi^*(\Om_\fD)$ as a $\fC$-module, and $\d_\fD\ci\phi:\fC\ra\Om_\fD$ as a $C^\iy$-derivation. Thus by the universal property of $\Om_\fC$, there exists a unique morphism of $\fC$-modules $\Om_\phi:\Om_\fC\ra\Om_\fD$ with $\d_\fD\ci\phi=\Om_\phi\ci\d_\fC$. This then induces a morphism of $\fD$-modules $(\Om_\phi)_*:
\Om_\fC\ot_\fC\fD\ra\Om_\fD$. If $\phi:\fC\ra\fD$, $\psi:\fD\ra\fE$
are morphisms of $C^\iy$-rings then~$\Om_{\psi\ci\phi}=
\Om_\psi\ci\Om_\phi:\Om_\fC\ra\Om_\fE$.
\label{ag5def2}
\end{dfn}

\begin{ex} Let $X$ be a manifold.\I{manifold!cotangent bundle} Then
the cotangent bundle $T^*X$ is a vector bundle over $X$, so as in
Example \ref{ag5ex1} it yields a $C^\iy(X)$-module $\Ga^\iy(T^*X)$.
The exterior derivative $\d:C^\iy(X)\ra \Ga^\iy(T^*X)$, $\d:c\mapsto
\d c$ is then a $C^\iy$-derivation,\I{module over C-ring@module
over $C^\iy$-ring!C-derivation@$C^\iy$-derivation}%
\I{C-ring@$C^\iy$-ring!C-derivation@$C^\iy$-derivation} since
equation \eq{ag5eq1} follows from
\begin{equation*}
\d\bigl(f(c_1,\ldots,c_n)\bigr)=\ts\sum_{i=1}^n\frac{\pd f}{\pd
x_i}(c_1,\ldots,c_n)\,\d c_n
\end{equation*}
for $f:\R^n\ra\R$ smooth and $c_1,\ldots,c_n\in C^\iy(X)$, which holds by the chain rule. It is easy to show that $\Ga^\iy(T^*X),\d$ have the universal property in Definition \ref{ag5def2}, and so form a cotangent module for~$C^\iy(X)$.

Now let $X,Y$ be manifolds, and $f:X\ra Y$ a smooth map. Then $f^*(T^*Y),\ab T^*X$ are vector bundles over $X$, and the derivative of $f$ gives a vector bundle morphism $\d f:f^*(T^*Y)\ra T^*X$. This induces a morphism of $C^\iy(X)$-modules $(\d f)_*:\Ga^\iy(f^*(T^*Y))\ra
\Ga^\iy(T^*X)$. This $(\d f)_*$ is identified with $(\Om_{f^*})_*$
under the natural isomorphism $\Ga^\iy(f^*(T^*Y))\cong
\Ga^\iy(T^*Y)\ot_{C^\iy(Y)}C^\iy(X)$, where we identify
$C^\iy(Y),C^\iy(X),f^*$ with $\fC,\fD,\phi$ in Definition~\ref{ag5def2}.
\label{ag5ex2}
\end{ex}

The importance of Definition \ref{ag5def2} is that it abstracts the
notion of cotangent bundle of a manifold in a way that makes sense
for any $C^\iy$-ring.

\begin{rem} There is a second way to define a cotangent-type module
for a $C^\iy$-ring $\fC$, namely the module ${\rm Kd}_\fC$ of {\it
K\"ahler differentials\/}\I{module over C-ring@module over
$C^\iy$-ring!module of K\"ahler
differentials}\I{C-ring@$C^\iy$-ring!module of K\"ahler
differentials} of the underlying $\R$-algebra of $\fC$. This is
defined as for $\Om_\fC$, but requiring \eq{ag5eq1} to hold only
when $f:\R^n\ra\R$ is a polynomial. Since we impose many fewer
relations, ${\rm Kd}_\fC$ is generally much larger than $\Om_\fC$,
so that ${\rm Kd}_{C^\iy(\R^n)}$ is not a finitely generated
$C^\iy(\R^n)$-module\I{module over C-ring@module over
$C^\iy$-ring!finitely generated} for $n>0$, for instance.
\label{ag5rem1}
\end{rem}

\begin{prop} If\/ $\fC$ is a finitely generated\/
$C^\iy$-ring\I{C-ring@$C^\iy$-ring!finitely generated} then
$\Om_\fC$ is a finitely generated\/ $\fC$-module. If\/ $\fC$ is finitely
presented,\I{C-ring@$C^\iy$-ring!finitely presented} then\/
$\Om_\fC$ is finitely presented.\I{module over C-ring@module over
$C^\iy$-ring!finitely presented}
\label{ag5prop1}
\end{prop}

\begin{proof} If $\fC$ is finitely generated we have an exact
sequence
\e
\smash{\xymatrix@C=30pt{ 0 \ar[r] & I\,\, \ar[r] & C^\iy(\R^n) \ar[r]^(0.6)\phi & \fC \ar[r] & 0. }}
\label{ag5eq2}
\e
Write $x_1,\ldots,x_n$ for the generators
of $C^\iy(\R^n)$. Then any $c\in\fC$ may be written as $\phi(f)$ for
some $f\in C^\iy(\R^n)$, and \eq{ag5eq1} implies that
\begin{equation*}
\d c=\d\Phi_f\bigl(\phi(x_1),\ldots,\phi(x_n)\bigr)
=\ts\sum_{i=1}^n\Phi_{\frac{\pd f}{\pd
x_i}}(\phi(x_1),\ldots,\phi(x_n))\cdot\d\ci\phi(x_i).
\end{equation*}
Hence the generators $\d c$ of $\Om_\fC$ for $c\in\fC$ are
$\fC$-linear combinations of $\d\ci\phi(x_i)$, $i=1,\ldots,n$, so
$\Om_\fC$ is spanned by the $\d\ci\phi(x_i)$, and is finitely
generated.

Suppose $\fC$ is finitely presented. Then we have an exact sequence \eq{ag5eq2} with ideal $I=(f_1,\ldots,f_m)$. We will define an exact sequence of $\fC$-modules
\e
\smash{\xymatrix@C=30pt{ \fC\ot_\R\R^m \ar[r]^\al & \fC\ot_\R\R^n \ar[r]^(0.6)\be & \Om_\fC \ar[r] & 0. }}
\label{ag5eq3}
\e
Write $(a_1,\ldots,a_m)$, $(b_1,\ldots,b_n)$ for bases of
$\R^m,\R^n$. As $\fC\ot_\R\R^m,\fC\ot_\R\R^n$ are free
$\fC$-modules, the $\fC$-module morphisms $\al,\be$ are specified
uniquely by giving $\al(a_i)$ for $i=1,\ldots,m$ and $\be(b_j)$ for
$j=1,\ldots,n$, which we define to be
\begin{equation*}
\al:a_i\longmapsto \ts\sum_{j=1}^n\Phi_{\frac{\pd
f_i}{\pd x_j}}\bigl(\phi(x_1),\ldots,\phi(x_n)\bigr)\cdot b_j
\quad\text{and}\quad \be:b_j\longmapsto \d_\fC\bigl(\phi(x_j)\bigr).
\end{equation*}
Then for $i=1,\ldots,m$ we have
\begin{align*}
\be\ci\al(a_i)&=\ts\sum_{j=1}^n\Phi_{\frac{\pd f_i}{\pd
x_j}}\bigl(\phi(x_1),\ldots,\phi(x_n)\bigr)\cdot
\d_\fC\bigl(\phi(x_j)\bigr)\\
&=\d_\fC\bigl(\Phi_{f_i}\bigl(\phi(x_1),\ldots,\phi(x_n)\bigr)\bigr)\\
&=\d_\fC\ci\phi\bigl(\Phi_{f_i}(x_1,\ldots,x_n)\bigr)
=\d_\fC\ci\phi\bigl(f_i(x_1,\ldots,x_n))=\d_\fC(0)=0,
\end{align*}
using \eq{ag5eq1} in the second step, $\phi$ a
morphism of $C^\iy$-rings in the third, the definition of
$C^\iy(\R^n)$ as a $C^\iy$-ring in the fourth, and
$f_i(x_1,\ldots,x_n)\in I=\Ker\phi$ in the fifth. Hence
$\be\ci\al=0$, and \eq{ag5eq3} is a complex.

Thus $\be$ induces $\be_*:(\fC\ot_\R\R^n)/\al(\fC\ot_\R\R^m)\ra
\Om_\fC$. We will show $\be_*$ is an isomorphism, so that
\eq{ag5eq3} is exact. Define $\d:\fC\ra(\fC\ot_\R\R^n)/
\al(\fC\ot_\R\R^m)$ by
\e
\ts\d\bigl(\phi(h)\bigr)=
\ts\sum_{j=1}^n\Phi_{\frac{\pd h}{\pd
x_j}}\bigl(\phi(x_1),\ldots,\phi(x_n)\bigr)\cdot b_j
+\al(\fC\ot_\R\R^m).
\label{ag5eq4}
\e
Here every $c\in\fC$ may be written as $\phi(h)$ for some $h\in C^\iy(\R^n)$ as $\phi$ is surjective. To show \eq{ag5eq4} is well-defined we must show the right hand side is independent of the choice of $h$ with $\phi(h)=c$, that is, we must show that the right hand side is zero if $h\in I$. It is enough to check this when $h$ is a generator $f_1,\ldots,f_m$ of $I$, and this holds by definition
of $\al$. Hence $\d$ in \eq{ag5eq4} is well-defined.

It is easy to see that $\d$ is a $C^\iy$-derivation, and that
$\be_*\ci\d=\d_\fC$. So by the universal property of $\Om_\fC$,
there is a unique $\fC$-module morphism $\psi:\Om_\fC\ra(\fC\ot_\R
\R^n)/\al(\fC\ot_\R\R^m)$ with $\d=\psi\ci\d_\fC$. Thus
$\be_*\ci\psi\ci\d_\fC=\be_*\ci\d=\d_\fC=\id_{\Om_\fC}\ci\d_\fC$, so
as $\Im\d_\fC$ generates $\Om_\fC$ as an $\fC$-module we see that
$\be_*\ci\psi=\id_{\Om_\fC}$. Similarly $\psi\ci\be_*$ is the
identity, so $\psi,\be_*$ are inverse, and $\be_*$ is an
isomorphism. Therefore \eq{ag5eq3} is exact, and $\Om_\fC$ is
finitely presented.
\end{proof}

Cotangent modules behave well under
localization.\I{C-ring@$C^\iy$-ring!localization}

\begin{prop} Let\/ $\fC$ be a $C^\iy$-ring, $S\subseteq\fC,$ and\/ $\fD=\fC[s^{-1}:s\in S]$ be the localization of\/ $\fC$ at\/ $S$ with projection $\pi:\fC\ra\fD,$ as in Definition\/ {\rm\ref{ag2def7}}. Then 
$(\Om_\pi)_*:\Om_\fC\ot_\fC\fD\ra\Om_\fD$ is
an isomorphism of\/ $\fD$-modules.
\label{ag5prop2}
\end{prop}

\begin{proof} Let $\Om_\fC,\Om_\fD$ be constructed as in Definition \ref{ag5def2}. As $\fD=\fC[s^{-1}:s\in S]$ is $\fC$ together with an extra generator $s^{-1}$ and an extra relation $s\cdot s^{-1}=1$ for each $s\in S$, we see that the $\fD$-module $\Om_\fD$ may be constructed from $\Om_\fC\ot_\fC\fD$ by adding an extra generator
$\d(s^{-1})$ and an extra relation $\d(s\cdot s^{-1}-1)=0$ for each $s\in S$. But
using \eq{ag5eq1} and $s\cdot s^{-1}=1$ in $\fD$, we see that this extra relation is equivalent to $\d(s^{-1})=-(s^{-1})^2\d s$. Thus the extra relations exactly cancel the effect of adding the extra generators, so $(\Om_\pi)_*$ is an isomorphism.
\end{proof}

Here is a useful exactness property of cotangent modules.

\begin{thm} Suppose we are given a pushout\I{category!pushout}\I{pushout}\I{C-ring@$C^\iy$-ring!pushout} diagram of\/ $C^\iy$-rings:
\e
\begin{gathered}
\xymatrix@C=60pt@R=14pt{ *+[r]{\fC} \ar[r]_\be \ar[d]^\al & *+[l]{\fE} \ar[d]_\de \\
*+[r]{\fD} \ar[r]^\ga & *+[l]{\fF,\!} }
\end{gathered}
\label{ag5eq5}
\e
so that\/ $\fF=\fD\amalg_\fC\fE$. Then the following sequence of\/
$\fF$-modules is exact:
\e
\xymatrix@C=15pt{ \Om_\fC\ot_{\fC,\ga\ci\al}\fF
\ar[rrr]^(0.5){(\Om_\al)_*\op -(\Om_\be)_*} &&&
{\raisebox{5pt}{$\begin{subarray}{l}\ts \Om_\fD\ot_{\fD,\ga}\fF\,\, \op\\
\ts\,\,\,\Om_\fE\ot_{\fE,\de}\fF \end{subarray}$}}
\ar[rrr]^(0.65){(\Om_\ga)_*\op(\Om_\de)_*} &&& \Om_\fF \ar[r] & 0. }
\label{ag5eq6}
\e
Here\/ $(\Om_\al)_*:\Om_\fC\ot_{\fC,\ga\ci\al}\fF\ra
\Om_\fD\ot_{\fD,\ga}\fF$ is induced by\/ $\Om_\al:\Om_\fC\ra
\Om_\fD,$ and so on. Note the sign of\/ $-(\Om_\be)_*$
in\/~\eq{ag5eq6}.
\label{ag5thm1}
\end{thm}

\begin{proof} By $\Om_{\psi\ci\phi}=\Om_\psi\ci\Om_\phi$ in
Definition \ref{ag5def2} and commutativity of \eq{ag5eq5} we have
$\Om_\ga\ci\Om_\al=\Om_{\ga\ci\al}=\Om_{\de\ci\be}=\Om_\de\ci
\Om_\be:\Om_\fC\ra\Om_\fF$. Tensoring with $\fF$ then gives
$(\Om_\ga)_*\ci(\Om_\al)_*=(\Om_\de)_*\ci(\Om_\be)_*:
\Om_\fC\ot_\fC\fF\ra\Om_\fF$. As the composition of morphisms in
\eq{ag5eq6} is $(\Om_\ga)_*\ci(\Om_\al)_*-(\Om_\de)_*\ci
(\Om_\be)_*$, this implies \eq{ag5eq6} is a complex.

For simplicity, first suppose $\fC,\fD,\fE,\fF$ are finitely
presented.\I{C-ring@$C^\iy$-ring!finitely presented} Use the
notation of Example \ref{ag2ex9} and the proof of Proposition
\ref{ag2prop5}, with exact sequences \eq{ag2eq3} and \eq{ag2eq4},
where $I=(h_1,\ldots,h_i)\subset C^\iy(\R^l)$,
$J=(d_1,\ldots,d_j)\subset C^\iy(\R^m)$ and
$K=(e_1,\ldots,e_k)\subset C^\iy(\R^n)$. Then $L$ is given by
\eq{ag2eq5}. Applying the proof of Proposition \ref{ag5prop1} to
\eq{ag2eq3}--\eq{ag2eq4} yields exact sequences of $\fF$-modules
\ea
\xymatrix{ \fF\ot_\R\R^i \ar[r]^{\ep_1} & \fF\ot_\R\R^l
\ar[r]^{\ze_1} & \Om_\fC\ot_\fC\fF \ar[r] & 0, }&
\label{ag5eq7}\\
\xymatrix{ \fF\ot_\R\R^j \ar[r]^{\ep_2} & \fF\ot_\R\R^m
\ar[r]^{\ze_2} & \Om_\fD\ot_\fD\fF \ar[r] & 0, }&
\label{ag5eq8}\\
\xymatrix{ \fF\ot_\R\R^k \ar[r]^{\ep_3} & \fF\ot_\R\R^n
\ar[r]^{\ze_3} & \Om_\fE\ot_\fE\fF \ar[r] & 0, }&
\label{ag5eq9}\\
\xymatrix@C=6pt{ \fF\ot_\R\R^{j+k+l} \ar[rr]^(0.3){\ep_4} &&
\fF\ot_\R\R^{m+n} \!=\!\fF\ot_\R\R^m\!\op\!\fF\ot_\R\R^n
\ar[rr]^(0.8){\ze_4} && \Om_\fF \ar[r] & 0, }&
\label{ag5eq10}
\ea
where for \eq{ag5eq7}--\eq{ag5eq9} we have tensored \eq{ag5eq3} for
$\fC,\fD,\fE$ with~$\fF$.

Define $\fF$-module morphisms $\th_1:\fF\ot_\R\R^l\ra\fF\ot_\R\R^m$,
$\th_2:\fF\ot_\R\R^l\ra\fF\ot_\R\R^n$ by
$\th_1(a_1,\ldots,a_l)=(b_1,\ldots,b_m)$,
$\th_2(a_1,\ldots,a_l)=(c_1,\ldots,c_n)$ with
\begin{equation*}
b_q=\sum_{p=1}^l \Phi_{\frac{\pd f_p}{\pd
y_q}}(\xi(y_1),\ldots,\xi(y_m))\cdot a_p, \quad c_r=\sum_{p=1}^l
\Phi_{\frac{\pd g_p}{\pd y_r}}(\xi(z_1),\ldots,\xi(z_n))\cdot a_p,
\end{equation*}
for $a_p,b_q,c_r\in\fF$. Now consider the diagram
\e
\begin{gathered}
\xymatrix@C=18pt@R=18pt{
{\begin{subarray}{l}\ts \fF\ot_\R\R^j\,\,\op \\
\ts \fF\ot_\R\R^k\,\,\op \\
\ts \fF\ot_\R\R^l\end{subarray}}
\ar[rrr]_(0.46){\ep_4=\left(\begin{subarray}{l} \ep_2 \,\, 0 \,\,
\phantom{-}\th_1 \\ 0 \,\, \ep_3 \,\, -\th_2\end{subarray}\right)}
\ar[d]^(0.6){\left(0 \,\, 0 \,\, \ze_1\right)} &&&
{\begin{subarray}{l}\ts \fF\ot_\R\R^m\,\,\op \\
\ts \fF\ot_\R\R^n\end{subarray}} \ar[rrr]_(0.55){\ze_4}
\ar[d]^{\left(\begin{subarray}{l} \ze_2 \,\, 0 \\ 0 \,\,
\ze_3\end{subarray}\right)} &&& \Om_\fF \ar[r]
\ar@{=}[d]^{{}\,\,\id_{\Om_\fF}} & 0 \\
\Om_\fC\ot_\fC\fF
\ar[rrr]^(0.5){\left(\begin{subarray}{l} (\Om_\al)_*\\
-(\Om_\be)_*\end{subarray}\right)} &&&
{\begin{subarray}{l}\ts \Om_\fD\ot_\fD\fF\, \op\\
\ts\Om_\fE\ot_\fE\fF \end{subarray}}
\ar[rrr]^(0.55){\left((\Om_\ga)_*\,\, (\Om_\de)_*\right)} &&&
\Om_\fF \ar[r] & 0, }
\end{gathered}
\label{ag5eq11}
\e
using matrix notation. The top line is the exact sequence
\eq{ag5eq10}, where the sign in $-\th_2$ comes from the sign of
$g_p$ in the generators $f_p(y_1,\ldots,y_m)-g_p(z_1,\ldots,z_n)$ of
$L$ in \eq{ag2eq5}. The bottom line is the complex~\eq{ag5eq6}.

The left hand square commutes as $\ze_2\ci\ep_2=\ze_3\ci\ep_3=0$ by
exactness of \eq{ag5eq8}--\eq{ag5eq9} and
$\ze_2\ci\th_1=(\Om_\al)_*\ci\ze_1$ follows from
$\al\ci\phi(x_p)=\psi(f_p)$, and $\ze_3\ci\th_2=(\Om_\be)_*\ci\ze_1$
follows from $\be\ci\phi(x_p)=\chi(g_p)$. The right hand square
commutes as $\ze_4$ and $(\Om_\ga)_*\ci\ze_2$ act on $\fF\ot_\R\R^m$
by $(a_1,\ldots,a_m)\mapsto\sum_{q=1}^ma_q\d_\fF\ci\xi(y_q)$, and
$\ze_4$ and $(\Om_\de)_*\ci\ze_3$ act on $\fF\ot_\R\R^n$ by
$(b_1,\ldots,b_n)\mapsto\sum_{r=1}^nb_r\d_\fF\ci\xi(z_r)$. Hence
\eq{ag5eq11} is commutative. The columns are surjective since
$\ze_1,\ze_2,\ze_3$ are surjective as \eq{ag5eq7}--\eq{ag5eq9} are
exact and identities are surjective.

The bottom right morphism $\bigl((\Om_\ga)_*\,(\Om_\de)_*\bigr)$ in
\eq{ag5eq11} is surjective as $\ze_4$ is and the right hand square
commutes. Also surjectivity of the middle column implies that it
maps $\Ker\ze_4$ surjectively onto $\Ker\bigl((\Om_\ga)_*\,
(\Om_\de)_*\bigr)$. But $\Ker\ze_4=\Im\ep_4$ as the top row is
exact, so as the left hand square commutes we see that
$\bigl((\Om_\al)_*\, -(\Om_\be)_*\bigr){}^T$ surjects onto
$\Ker\bigl((\Om_\ga)_*\, (\Om_\de)_*\bigr)$, and the bottom row of
\eq{ag5eq11} is exact. This proves the theorem for $\fC,\fD,\fE,\fF$
finitely presented. For the general case we can use the
same proof, but allowing $i,j,k,l,m,n$ infinite.
\end{proof}

Here is an example of the situation of Theorem \ref{ag5thm1} for manifolds.

\begin{ex} Let $W,X,Y,Z,e,f,g,h$ be as in Theorem \ref{ag3thm}, so
that \eq{ag3eq1} is a Cartesian square\I{Cartesian square} of
manifolds\I{manifold!transverse fibre product}\I{manifold with
corners!transverse fibre product}\I{fibre product}\I{category!fibre
product} and \eq{ag3eq2} a pushout\I{C-ring@$C^\iy$-ring!pushout}
square of $C^\iy$-rings. We have the following sequence of morphisms
of vector bundles on~$W$:
\e
\xymatrix@C=9.5pt{ 0 \ar[r] & (g\ci e)^*(T^*Z)
\ar[rrrr]^(0.47){\raisebox{5pt}{$\st e^*(\d g^*)\op -f^*(\d h^*)$}}
&&&& e^*(T^*X)\!\op\!f^*(T^*Y) \ar[rr]^(0.71){\raisebox{5pt}{$\st \d
e^*\op \d f^*$}} && T^*W \ar[r] & 0.}\!\!
\label{ag5eq12}
\e
Here $\d g:TX\ra g^*(TZ)$ is a morphism of vector bundles over $X$,
and $\d g^*:g^*(T^*Z)\ra T^*X$ is the dual morphism, and $e^*(\d
g^*):(g\ci e)^*(T^*Z)\ra e^*(T^*X)$ is the pullback of this dual
morphism to~$W$.

Since $g\ci e=h\ci f$, we have $\d e^*\ci e^*(\d g^*)=\d f^*\ci
f^*(\d h^*)$, and so \eq{ag5eq12} is a complex. As $g,h$ are
transverse and \eq{ag3eq1} is Cartesian, \eq{ag5eq12} is exact. So
passing to smooth sections in \eq{ag5eq12} we get an exact sequence
of $C^\iy(W)$-modules:
\begin{equation*}
\xymatrix@C=13pt{ 0 \ar[r] & \Ga^\iy\bigl((g\ci e)^*(T^*Z)\bigr)
\ar[rrr]^(0.53){\begin{subarray}{l}(e^*(\d g^*)\op \\ -f^*(\d h^*))_*
\end{subarray}} &&& {\raisebox{7pt}{$\begin{subarray}{l}\ts
\Ga^\iy\bigl(e^*(T^*X)\\
\ts{}\,\op f^*(T^*Y)\bigr)\end{subarray}$}}
\ar[rr]^(0.52){\begin{subarray}{l}(\d e^*\op\\ {}\,\,\,\d f^*)_*
\end{subarray}} && \Ga^\iy(T^*W) \ar[r] & 0.}
\end{equation*}
The final four terms are the exact sequence \eq{ag5eq6}
for the pushout\I{category!pushout}\I{pushout}
diagram~\eq{ag3eq2}.\I{module over C-ring@module over $C^\iy$-ring|)}\I{module over C-ring@module over $C^\iy$-ring!cotangent module $\Om_\fC$|)}\I{C-ring@$C^\iy$-ring!cotangent module $\Om_\fC$|)}

\label{ag5ex3}
\end{ex}

\subsection{\texorpdfstring{Sheaves of $\O_X$-modules on a $C^\iy$-ringed space $(X,\O_X)$}{Sheaves of modules on a C∞-ringed space}}
\label{ag53}
\I{C-ringed space@$C^\iy$-ringed space!sheaves of $\O_X$-modules on|(}%
\I{C-scheme@$C^\iy$-scheme!sheaves of $\O_X$-modules on|(}

We define sheaves of $\O_X$-modules on a $C^\iy$-ringed space,
following~\cite[\S II.5]{Hart}.

\begin{dfn} Let $(X,\O_X)$ be a $C^\iy$-ringed space. A {\it sheaf
of\/ $\O_X$-modules}, or simply an $\O_X$-{\it module}, $\cE$ on $X$
assigns a module $\cE(U)$ over the $C^\iy$-ring
$\O_X(U)$ for each open set $U\subseteq X$, and a linear map
$\cE_{UV}:\cE(U)\ra \cE(V)$ for each inclusion of open sets $V\subseteq
U\subseteq X$, such that the following commutes
\e
\begin{gathered}
\xymatrix@R=15pt@C=100pt{ *+[r]{\O_X(U)\t \cE(U)} \ar[d]^{\rho_{UV}\t\cE_{UV}}
\ar[r]_(0.6){\mu_{\cE(U)}} & *+[l]{\cE(U)}
\ar[d]_{\cE_{UV}} \\
*+[r]{\O_X(V)\t \cE(V)} \ar[r]^(0.6){\mu_{\cE(V)}} & *+[l]{\cE(V),\!} }
\end{gathered}
\label{ag5eq13}
\e
and all this data $\cE(U),\cE_{UV}$ satisfies the sheaf axioms in
Definition~\ref{ag4def1}.

A {\it morphism of sheaves of\/ $\O_X$-modules\/} $\phi:\cE\ra\cF$
assigns a morphism of $\O_X(U)$-modules $\phi(U):\cE(U)\ra\cF(U)$
for each open set $U\subseteq X$, such that
$\phi(V)\ci\cE_{UV}=\cF_{UV}\ci\phi(U)$ for each inclusion of open
sets $V\subseteq U\subseteq X$. Then $\O_X$-modules form an abelian
category,\I{abelian category} which we write
as~$\OXmod$.\G[OXmoda]{$\OXmod$}{abelian category of $\O_X$-modules
on $C^\iy$-scheme $\uX$}

An $\O_X$-module $\cE$ is called a {\it vector bundle of rank\/} $n$ if we may cover $X$ by open $U\subseteq X$ with $\cE\vert_U\cong\O_X\vert_U\ot_\R\R^n$.
\label{ag5def3}
\end{dfn}

In Definition \ref{ag4def6} we defined {\it fine\/} sheaves $\cE$ on a topological space $X$. In \S\ref{ag47} we gave sufficient conditions for when a $C^\iy$-ringed space $\uX=(X,\O_X)$ has $\O_X$ fine, which hold if $\uX$ is an affine $C^\iy$-scheme with $X$ Lindel\"of. Now if $\O_X$ is fine, then any $\O_X$-module $\cE$ is also fine, since partitions of unity in $\O_X$ induce partitions of unity in $\mathcal{H}om(\cE,\cE)$. 

As in Voisin \cite[Prop.~4.36]{Vois}, a fundamental property  of fine sheaves $\cE$ is that their cohomology groups $H^i(\cE)$ are zero for all $i>0$. This means that $H^0$ is an exact functor\I{functor!exact} on fine sheaves, rather than just left exact,\I{functor!left exact} since
$H^1$ measures the failure of $H^0$ to be right exact.\I{functor!right exact} If $X$ is second countable then $(U,\O_X\vert_U)$ is a Lindel\"of affine $C^\iy$-scheme for all open $U\subseteq X$. Thus we deduce:

\begin{prop} Let\/ $(X,\O_X)$ be an affine $C^\iy$-scheme with\/ $X$ Lindel\"of, and
\begin{equation*}
\smash{\xymatrix@C=30pt{ 
\cdots \ar[r] & \cE^i \ar[r]^{\phi^i} & \cE^{i+1} \ar[r]^{\phi^{i+1}} & \cE^{i+2} \ar[r] & \cdots }}
\end{equation*}
be an exact sequence in $\OXmod$. Then
\begin{equation*}
\smash{\xymatrix@C=32pt{ \cdots \ar[r] & \cE^i(X) \ar[r]^(0.45){\phi^i(X)} & \cE^{i+1}(X) \ar[r]^{\phi^{i+1}(X)} & \cE^{i+2}(X) \ar[r] & \cdots }}
\end{equation*}
is an exact sequence of\/ $\O_X(X)$-modules. If\/ $X$ is also second countable then the following is an exact sequence of\/ $\O_X(U)$-modules for all open\/~$U\subseteq X:$
\begin{equation*}
\smash{\xymatrix@C=32pt{ \cdots \ar[r] & \cE^i(U) \ar[r]^(0.45){\phi^i(U)} & \cE^{i+1}(U) \ar[r]^{\phi^{i+1}(U)} & \cE^{i+2}(U) \ar[r] & \cdots. }}
\end{equation*}

\label{ag5prop3}
\end{prop}

\begin{rem} Recall that a $C^\iy$-ring $\fC$ has an underlying
commutative $\R$-algebra,\I{C-ring@$C^\iy$-ring!as commutative
$\R$-algebra} and a module over $\fC$ is a module over this
$\R$-algebra, by Definitions \ref{ag2def3} and \ref{ag5def1}. Thus,
by truncating the $C^\iy$-rings $\O_X(U)$ to commutative
$\R$-algebras, regarded as rings, a $C^\iy$-ringed space $(X,\O_X)$
has an underlying ringed space in the usual sense of algebraic
geometry \cite[p.~72]{Hart}, \cite[\S 0.4]{Grot}. Our definition of
$\O_X$-modules are simply $\O_X$-modules on this underlying ringed
space \cite[\S II.5]{Hart}, \cite[\S 0.4.1]{Grot}. Thus we can apply
results from algebraic geometry without change, for instance that
$\OXmod$ is an abelian category, as in~\cite[p.~202]{Hart}.
\label{ag5rem2}
\end{rem}

\begin{dfn} Let $\uf=(f,f^\sh):(X,\O_X)\ra(Y,\O_Y)$ be a morphism
of $C^\iy$-ringed spaces, and $\cE$ be an $\O_Y$-module. Define the
{\it pullback\/}\I{C-ringed space@$C^\iy$-ringed space!sheaves of
$\O_X$-modules on!pullback}
$\uf^*(\cE)$\G[fEc]{$\uf^*(\cE)$}{pullback of sheaf of
$\O_Y$-modules under $\uf:\uX\ra\uY$} by $\uf^*(\cE)=f^{-1}(\cE)
\ot_{f^{-1}(\O_Y)}\O_X$, where $f^{-1}(\cE)$ is as in Definition
\ref{ag4def5}, a sheaf of modules over the sheaf of $C^\iy$-rings
$f^{-1}(\O_Y)$ on $X$, and the tensor product uses the morphism
$f^\sh:f^{-1}(\O_Y)\ra\O_X$. If $\phi:\cE\ra\cF$ is a morphism of
$\O_Y$-modules we have a morphism of
$\O_X$-modules~$\uf^*(\phi)=f^{-1}(\phi)
\ot\id_{\O_X}:\uf^*(\cE)\ra\uf^*(\cF)$.
\label{ag5def4}
\end{dfn}

\begin{rem} Pullbacks $\uf^*(\cE)$ are a kind of fibre
product,\I{fibre product}\I{category!fibre product} and may be
characterized by a universal property in $\OXmod$. So they should be regarded as
being {\it unique up to canonical isomorphism}, rather than unique.
One can give an explicit construction for pullbacks, or use the
Axiom of Choice\I{Axiom of Choice} to choose $\uf^*(\cE)$ for all
$\uf,\cE$, and so speak of `the' pullback $\uf^*(\cE)$. However, it
may not be possible to make these choices strictly functorial
in~$\uf$.

That is, if $\uf:\uX\ra\uY$, $\ug:\uY\ra\uZ$ are morphisms and
$\cE\in\OZmod$ then $(\ug\ci\uf)^*(\cE)$, $\uf^*(\ug^*(\cE))$ are
canonically isomorphic in $\OXmod$, but may not be equal. We will
write $I_{\uf,\ug}(\cE):(\ug\ci\uf)^*(\cE)\ra
\uf^*(\ug^*(\cE))$\G[Ifgb]{$I_{\uf,\ug}(\cE):(\ug\ci\uf)^*(\cE)
\ra\uf^{-1}(\ug^{-1}(\cE))$}{isomorphism of pullbacks in $\OXmod$}
for these canonical isomorphisms, as in Remark \ref{ag4rem1}(b).
Then $I_{\uf,\ug}:(\ug\ci\uf)^*\Ra \uf^*\ci\ug^*$ is a natural
isomorphism of functors. It is common to ignore this point and
identify $(\ug\ci\uf)^*$ with $\uf^*\ci\ug^*$. Vistoli \cite{Vist} makes careful use of
natural isomorphisms $(g\ci f)^*\Ra f^*\ci g^*$ in his treatment of
descent theory.

When $\uf$ is the identity $\uid_\uX:\uX\ra\uX$ and $\cE\in\OXmod$
we do not require $\uid^*_\uX(\cE)=\cE$, but as $\cE$ is a possible
pullback for $\uid^*_\uX(\cE)$ there is a canonical isomorphism
$\de_\uX(\cE):\uid^*_\uX(\cE)\ra\cE$,\G[deXb]{$\de_\uX(\cE):\uid^*_\uX
(\cE)\ra\cE$}{canonical isomorphism of pullbacks in $\OXmod$} and
then $\de_\uX:\uid^*_\uX\Ra\id_\OXmod$ is a natural isomorphism of
functors.
\label{ag5rem3}
\end{rem}

By Grothendieck \cite[\S 0.4.3.1]{Grot} we have:

\begin{prop} Let\/ $\uX,\uY$ be $C^\iy$-ringed spaces and\/
$\uf:\uX\ra\uY$ a morphism. Then pullback\/ $\uf^*:\OYmod\ra\OXmod$
is a \begin{bfseries}right exact
functor\end{bfseries}\I{functor!right exact} between abelian
categories.\I{abelian category} That is, if\/
$\cE\smash{\,{\buildrel\phi \over\longra}\,\cF
\,{\buildrel\psi\over\longra}\,}\cG\ra 0$ is exact in $\OYmod$ then
$\uf^*(\cE)\,{\buildrel\uf^*(\phi)\over \longra}\,\uf^*(\cF)
\,{\buildrel\uf^*(\psi)\over\longra} \,\uf^*(\cG)\ra 0$ is exact
in\/~$\OXmod$.
\label{ag5prop4}
\end{prop}

In general $\uf^*$ is not exact,\I{functor!exact} or left
exact,\I{functor!left exact} unless $\uf:\uX\ra\uY$ is
flat.\I{C-ringed space@$C^\iy$-ringed space!sheaves of
$\O_X$-modules on|)}

\subsection{\texorpdfstring{Sheaves on affine $C^\iy$-schemes, $\MSpec$ and $\Ga$}{Sheaves on affine C∞-schemes, MSpec and Γ}}
\label{ag54}
\I{C-scheme@$C^\iy$-scheme!affine!sheaves of $\O_X$-modules on|(}

In \S\ref{ag44} we defined $\Spec:\CRings^{\bf op}\ra\LCRS$.\I{C-ring@$C^\iy$-ring!spectrum functor $\Spec$} In a similar way, if $\fC$ is a $C^\iy$-ring and $(X,\O_X)=\Spec\fC$ we can define $\MSpec:\fCmod\ra\OXmod$, a spectrum functor for modules.

\begin{dfn} Let $(X,\O_X)=\Spec\fC$ for some $C^\iy$-ring $\fC$ and $M$ be a $\fC$-module. We will define an $\O_X$-module $\cE=\MSpec M$. For each open $U\subseteq X$, define $\cE(U)$ to be the $\R$-vector space of functions $e:U\ra \coprod_{x\in U}(M\ot_\fC\fC_x)$ with $e(x)\in M\ot_\fC\fC_x$ for all $x\in U$, and such that $U$ may be covered by open sets $W\subseteq U\subseteq X$ for which there exist $m\in M$ with $e(x)=m\ot 1\in M\ot_\fC\fC_x$ for all $x\in W$. Here the $\fC_x$-module $M\ot_\fC\fC_x$ is defined using the $\fC$-module structure on $M$ and the projection~$\pi_x:\fC\ra\fC_x$.

Definition \ref{ag4def9} defines $\O_X(U)$ as a set of functions $U\ra\coprod_{x\in U}\fC_x$. Define an $\O_X(U)$-module structure $\mu_{\cE(U)}:\O_X(U)\t\cE(U)\ra\cE(U)$ on $\cE(U)$ by
\begin{equation*}
\mu_{\cE(U)}(s,e):x\longmapsto s(x)\cdot e(x),
\end{equation*}
for all $s\in\O_X(U)$, $e\in\cE(U)$ and $x\in U$. For open $V\subseteq U\subseteq X$, define $\cE_{UV}:\cE(U)\ra\cE(V)$ by $\cE_{UV}:e\mapsto e\vert_V$. It is now easy to check that $\cE$ is a sheaf of $\O_X$-modules on $X$. Define $\MSpec M=\cE$ in~$\OXmod$.

An equivalent way to define $\MSpec M$ is as the sheafification of the presheaf $U\mapsto M\ot_\fC\O_X(U)$. The definition above performs the sheafification explicitly.

Now let $\al:M\ra N$ be a morphism in $\fCmod$, and set $\cE=\MSpec M$ and $\cF=\MSpec N$. For each open $U\subseteq X$, define $\la(U):\cE(U)\ra\cF(U)$ by 
\begin{equation*}
\la(U)(e):x\mapsto (\al\ot\id)(e(x)) \quad\text{for $x\in U$,}
\end{equation*}
where $\al\ot\id$ maps $M\ot_\fC\fC_x\ra N\ot_\fC\fC_x$. It is easy to check that $\la(U)$ is an $\O_X(U)$-module morphism and $\la(V)\ci\cE_{UV}=\cF_{UV}\ci\la(U):\cE(U)\ra\cF(V)$ for all open $V\subseteq U\subseteq X$. Hence $\la:\cE\ra\cF$ is a morphism in $\OXmod$. Define $\MSpec\al=\la$, so that $\MSpec\al:\MSpec M\ra\MSpec N$. This defines a functor $\MSpec:\fCmod\ra\OXmod$.\G[MSpec]{$\MSpec:\fCmod\ra\OXmod$}{spectrum functor for modules over a $C^\iy$-ring $\fC$} It is an exact functor of abelian categories, since $M\mapsto M\ot_\fC\fC_x$ is an exact functor $\fCmod\ra\fC_x\text{-mod}$ for each $x\in X$, as the localization $\pi_x:\fC\ra\fC_x$ is a flat morphism of $\R$-algebras.
\label{ag5def5}
\end{dfn}

\begin{dfn} Let $\fC$ be a $C^\iy$-ring, and $(X,\O_X)=\Spec\fC$. If $\cE$ is an $\O_X$-module then $\cE(X)$ is a module over $\O_X(X)$, so using $\Psi_\fC:\fC\ra\Ga(\Spec\fC)=\O_X(X)$ we may regard $\cE(X)$ as a $\fC$-module. Define $\Ga(\cE)$ to be the $\fC$-module $\cE(X)$. If $\al:\cE\ra\cF$ is a morphism of $\O_X$-modules then $\Ga(\al):=\al(X):\cE(X)\ra\cF(X)$ is a morphism $\Ga(\al):\Ga(\cE)\ra\Ga(\cF)$ in $\fCmod$. This defines the {\it global sections functor\/}~$\Ga:\OXmod\ra\fCmod$.\G[Gab]{$\Ga:\OXmod\ra\fCmod$}{global sections functor on $\O_X$-modules}

In general $\Ga$ is a left exact functor of abelian categories, but may not be right exact. However, if $X$ is Lindel\"of (for example, if $\fC$ is finitely  or countably generated) then Proposition \ref{ag5prop3} shows that $\Ga$ is an exact functor.

Now $\Ga\ci\MSpec$ is a functor $\fCmod\ra\fCmod$. For each $\fC$-module $M$ and $m\in M$, define $\Psi_M(m):X\ra\coprod_{x\in X}M\ot_\fC\fC_x$ by $\Psi_M(m):x\mapsto m\ot 1_{\fC_x}\in M\ot_\fC\fC_x$. Then $\Psi_M(m)\in \MSpec M(X)=\Ga\ci\MSpec M$ by Definition \ref{ag5def5}, so $\Psi_M:M\ra\Ga\ci\MSpec M$\G[Psib]{$\Psi_M:M\ra\Ga\ci\MSpec M$}{canonical morphism for a $\fC$-module $M$} is a linear map, and in fact a $\fC$-module morphism.

It is functorial in $M$, so that the $\Psi_M$ for all $M$ define a natural transformation $\Psi:\id_\fCmod\Ra\Ga\ci\MSpec$ of functors~$\id_\fCmod,\Ga\ci\MSpec:\fCmod\ra\fCmod$.
\label{ag5def6}
\end{dfn}

Here are the analogues of Lemma \ref{ag4lem2} and Theorem~\ref{ag4thm1}:

\begin{lem} In Definition\/ {\rm\ref{ag5def5},} the stalk\/ $(\MSpec M)_x=\cE_x$ of\/ $\MSpec M$ at\/ $x\in X$ is naturally isomorphic to\/ $M\ot_\fC\fC_x,$ as modules over\/~$\fC_x\cong\O_{X,x}$.
\label{ag5lem1}
\end{lem}

\begin{proof} Elements of $\cE_x$ are $\sim$-equivalence classes $[U,e]$ of pairs $(U,e)$, where $U$ is an open neighbourhood of $x$ in $X$ and $e\in\cE(U)$, and $(U,e)\sim(U',e')$ if there exists open $x\in V\subseteq U\cap U'$ with $e\vert_V=e'\vert_V$. Define a $\fC_x$-module morphism $\Pi:\cE_x\ra M\ot_\fC\fC_x$ by~$\Pi:[U,e]\mapsto e(x)$. 

Proposition \ref{ag2prop2} shows that $\fC_x\cong\fC/I$ for $I$ the ideal in \eq{ag2eq2}. Hence $M\ot_\fC\fC_x\cong M/(I\cdot M)$, and thus every element of $M\ot_\fC\fC_x$ is of the form $m\ot 1_{\fC_x}$ for some $m\in M$. But $\Psi_M(m)\in\cE(X)$, so that $[X,\Psi_M(m)]\in\cE_x$, with $\Pi:[X,\Psi_M(m)]\mapsto m\ot 1_{\fC_x}$. Hence $\Pi:\cE_x\ra M\ot_\fC\fC_x$ is surjective.

Suppose $[U,e]\in\cE_x$ with $\Pi([U,e])=0\in M\ot_\fC\fC_x$. As $e\in\cE(U)$, there exist open $x\in V\subseteq U$ and $m\in M$ with $e(x')=m\ot 1_{\fC_{x'}}\in M\ot_\fC\fC_{x'}$ for all $x'\in V$. Then $m\ot 1_{\fC_x}=e(x)=\Pi([U,e])=0$ in $M\ot_\fC\fC_x$, so $m\in I\cdot M\subseteq M$, and we may write $m=\sum_{a=1}^ki_a\cdot m_a$ for $i_a\in I$ and $m_a\in M$. By \eq{ag2eq2} we may choose $d_1,\ldots,d_k\in\fC$ with $x(d_a)\ne 0$ and $i_a\cdot d_a=0$ in $\fC$ for~$a=1,\ldots,k$. 

Set $W=\{x'\in V:x'(d_a)\ne 0,$ $a=1,\ldots,k\}$, so that $W$ is an open neighbourhood of $x$ in $U$. If $x'\in W$ then $x'(d_a)\ne 0$, so $\pi_{x'}(d_a)$ is invertible in $\fC_{x'}$. But $i_a\cdot d_a=0$, so $\pi_{x'}(i_a)=0$ in $\fC_{x'}$ for $a=1,\ldots,k$. As $m=\sum_{a=1}^ki_a\cdot m_a$ it follows that $e(x')=m\ot 1_{\fC_{x'}}=0$ in $M\ot_\fC\fC_{x'}$ for all $x'\in W$. Thus $e\vert_W=0$ in $\cE(W)$, so $[U,e]=[W,e\vert_W]=0$ in $\cE_x$. Therefore $\Pi:\cE_x\ra M\ot_\fC\fC_x$ is injective, and so an isomorphism.
\end{proof}
 
\begin{thm} Let\/ $\fC$ be a $C^\iy$-ring, and\/ $(X,\O_X)=\Spec\fC$. Then $\Ga:\OXmod\ra\fCmod$ is \begin{bfseries}right adjoint\end{bfseries}\I{functor!adjoint}\I{adjoint functor} to $\MSpec:\fCmod\ra\OXmod$. That is, for all\/ $M\in\fCmod$ and\/ $\cE\in\OXmod$ there are inverse bijections
\e
\xymatrix@C=150pt{ *+[r]{\Hom_\fCmod(M,\Ga(\cE))} \ar@<.5ex>[r]^(0.45){L_{M,\cE}}
& *+[l]{\Hom_\OXmod(\MSpec M,\cE),} \ar@<.5ex>[l]^(0.55){R_{M,\cE}} }
\label{ag5eq14}
\e
which are functorial in $M,\cE$. When $\cE=\MSpec M$ we have $\Psi_M=R_{M,\cE}(\id_\cE),$ so that\/ $\Psi_M$ is the unit of the adjunction between $\Ga$ and\/~$\MSpec$.
\label{ag5thm2}
\end{thm}

\begin{proof} Let $M\in\fCmod$ and $\cE\in\OXmod$, and set $\cD=\MSpec M$. Define $R_{M,\cE}$ in \eq{ag5eq14} by, for each morphism $\al:\cD\ra\cE$ in $\OXmod$, taking $R_{M,\cE}(\al):M\ra\Ga(\cE)$ to be the composition
\begin{equation*}
\xymatrix@C=40pt{ M \ar[r]^(0.3){\Psi_M} & \Ga\ci\MSpec M=\Ga(\cD) \ar[r]^(0.65){\Ga(\al)} & \Ga(\cE). }
\end{equation*}
For the last part, if $\cE=\MSpec M$ then $\Psi_M=R_{M,\cE}(\id_\cE)$ as~$\Ga(\id_\cE)=\id_{\Ga(\cE)}$.

Let $\be:M\ra\Ga(\cE)$ be a morphism in $\fCmod$. We will construct a morphism $\la:\cD\ra\cE$ in $\OXmod$, and set $L_{M,\cE}(\be)=\la$.
Let $x\in X$. Consider the diagram
\e
\begin{gathered}
\xymatrix@C=100pt@R=15pt{ *+[r]{M\ot_\fC\fC=M} \ar[d]^{\id\ot\pi_x} \ar[r]_(0.6)\be & *+[l]{\Ga(\cE)} \ar[d]_{\si_x} \\
*+[r]{M\ot_\fC\fC_x\cong\cD_x} \ar@{.>}[r]^(0.6){\be_x} & *+[l]{\cE_x}
}
\end{gathered}
\label{ag5eq15}
\e
in $\fCmod$, where the isomorphism $M\ot_\fC\fC_x\cong\cD_x$ comes from Lemma \ref{ag5lem1}. Here $\cE_x$ is the stalk of $\cE$ at $x$, and $\si_x:\Ga(\cE)=\cE(X)\ra\cE_x$ takes stalks at $x$. The $\fC$-action on $\Ga(\cE)$ factors via $\fC\,{\buildrel\Psi_\fC\over\longra}\,\O_X(X)$, and the $\fC$-action on $\cE_x$ factors via $\fC\,{\buildrel\Psi_\fC\over\longra}\,\O_X(X)\,{\buildrel\pi\over\longra}\,\O_{X,x}$, and $\be,\si_x$ are both $\fC$-module morphisms. But $\O_{X,x}\cong\fC_x$ by Lemma \ref{ag4lem2}, so $\si_x\ci\be:M\ra\cE_x$ is a $\fC$-module morphism, where the $\fC$-action on $\cE_x$ factors via $\fC\,{\buildrel\pi_x\over\longra}\,\fC_x$. Hence there is a unique $\O_{X,x}$-module morphism $\be_x:\cD_x\ra\cE_x$ making \eq{ag5eq15} commute.

For each open $U\subseteq X$, define $\la(U):\cD(U)\ra\cE(U)$ by $\la(U)d:x\mapsto \be_x(d(x))$ for $d\in\cD(U)$ and $x\in U\subseteq X$, and $d(x)\in\cD_x$, and $\be_x(d(x))\in\cE_x$. Here as $\cE$ is a sheaf we may identify elements of $\cE(U)$ with maps $e:U\ra\coprod_{x\in U}\cE_x$ with $e(x)\in\cE_x$ for $x\in U$, such that $e$ satisfies certain local conditions in $U$. 

If $d\in\cD(U)=\MSpec M(U)$ and $x\in U$ then by Definition \ref{ag5def5} we may cover $U$ by open $W\subseteq U$ for which there exist $m\in M$ with $d(x)=m\ot 1_{\fC_x}$ in $M\ot_\fC\fC_x$ for all $x\in W$. Therefore $\la(U)d$ maps $x\mapsto\si_x(\be(m))$ for all $x\in W$ by \eq{ag5eq15}, so $\la(U)d$ is a section $\be(m)\vert_W$ of $\cE$ on $W$. Hence $\la(U)d$ is a section of $\cE\vert_U$, as such $W$ cover $U$, and $\la(U):\cD(U)\ra\cE(U)$ is well defined.

As $\be_x$ is an $\O_{X,x}$-module morphism for all $x\in U$, $\la(U):\cD(U)\ra\cE(U)$ is an $\O_X(U)$-module morphism. The definition of $\la(U)$ is clearly compatible with restriction to open $V\subseteq U\subseteq X$. Thus the $\la(U)$ for all open $U\subseteq X$ define a sheaf morphism $\la:\cD\ra\cE$ in $\OXmod$. Set $L_{M,\cE}(\be)=\la$. This defines $L_{M,\cE}$ in \eq{ag5eq14}. A very similar proof to that of Theorem \ref{ag4thm1} shows that $L_{M,\cE},R_{M,\cE}$ are inverse maps, so they are bijections, and that they are functorial in~$M,\cE$. 
\end{proof}

We show that $\Ga$ is a right inverse for $\MSpec$:

\begin{prop} Let\/ $\fC$ be a $C^\iy$-ring, and\/ $(X,\O_X)=\Spec\fC,$ and\/ $\cE$ be an\/ $\O_X$-module. Set\/ $M=\Ga(\cE)$ in $\fCmod,$ and write\/ $\Psi_\cE=L_{M,\cE}(\id_M)$. Then\/ $\Psi_\cE:\MSpec\ci\Ga(\cE)\ra\cE$ is an isomorphism in\/ $\OXmod,$ for any\/~$\cE$. 

These isomorphisms\/ $\Psi_\cE$ are functorial in\/ $\cE,$ and so define a natural isomorphism\/ $\Psi:\MSpec\ci\Ga\Ra\id_\OXmod$ of functors\/~$\OXmod\ra\OXmod$.
\label{ag5prop5}
\end{prop}

\begin{proof} Set $\cD=\MSpec M=\MSpec\ci\Ga(\cE)$, and let $x\in X$. Then by definition of $\Psi_\cE=L_{M,\cE}(\id_M):\cD\ra\cE$ in the proof of Theorem \ref{ag5thm2}, as in \eq{ag5eq15} the stalk map $\Psi_{\cE,x}:\cD_x\ra\cE_x$ is the unique morphism of modules over $\fC_x\cong\O_{X,x}$ making the following diagram of $\fC$-modules commute:
\e
\begin{gathered}
\xymatrix@C=120pt@R=15pt{ *+[r]{M\ot_\fC\fC=M} \ar[d]^{\id\ot\pi_x} \ar[r]_(0.6){\id_M} & *+[l]{M=\Ga(\cE)} \ar[d]_{\si_x} \\
*+[r]{M\ot_\fC\fC_x\cong\cD_x} \ar[r]^(0.6){\Psi_{\cE,x}} & *+[l]{\cE_x.\!} }
\end{gathered}
\label{ag5eq16}
\e

Let $[U,e]\in\cE_x$, so that $x\in U\subseteq X$ is open and $e\in\cE(U)$. By Definition \ref{ag4def8} there exists $c\in\fC$ such that $x(c)\ne 0$ and $y(c)=0$ for all $y\in X\sm U$. Choose smooth $f:\R\ra\R$ such that $f=0$ near 0 in $\R$ and $f=1$ near $x(c)$ in $\R$. Set $c'=\Phi_f(c)$, where $\Phi_f:\fC\ra\fC$ is the $C^\iy$-ring operation. Then $\eta=\Psi_\fC(c')\in\O_X(X)$, and there exist open neighbourhoods $V$ of $X\sm U$ and $W$ of $x$ in $X$ with $\eta\vert_V=0$ and $\eta\vert_W=1$. Clearly $V\cap W=\es$, so $x\in W\subseteq U$. We have $\eta\vert_U\cdot e\in\cE(U)$, with $(\eta\vert_U\cdot e)\vert_{U\cap V}=0$ and $(\eta\vert_U\cdot e)\vert_W=e\vert_W$. 

Since $\{U,V\}$ is an open cover of $X$ and $(\eta\vert_U\cdot e)\vert_{U\cap V}=0=0\vert_{U\cap V}$, by the sheaf property of $\cE$ there is a unique $e'\in\cE(X)$ with $e'\vert_U=\eta\vert_U\cdot e$ and $e'\vert_V=0$. Then $e'\vert_W=(\eta\vert_U\cdot e)\vert_W=e\vert_W$. Thus
\begin{equation*}
\si_x(e')=[X,e']=[W,e'\vert_W]=[W,e\vert_W]=[U,e]
\end{equation*}
in $\cE_x$. Hence $\si_x:\Ga(\cE)\ra\cE_x$ is surjective, so $\Psi_{\cE,x}:\cD_x\ra\cE_x$ is surjective by \eq{ag5eq16}, as $\pi_x:\fC\ra\fC_x$ is surjective by Proposition~\ref{ag2prop2}.

Suppose $d\in\fD_x$ with $\Psi_{\cE,x}(d)=0$. We may write $m\ot 1_{\fC_x}\cong d$ under the isomorphism $M\ot_\fC\fC_x\cong\cD_x$ for some $m\in M$, and then \eq{ag5eq16} gives $\si_x(m)=\Psi_{\cE,x}(d)=0$. Hence there exists open $x\in U\subseteq X$ with $m\vert_U=0$. As above we may construct $\eta\in\O_X(X)$ and open $V,W\subseteq X$ with $X\sm U\subseteq V$, $x\in W\subseteq U$, $\eta\vert_V=0$ and $\eta\vert_W=1$. Then $\eta\cdot m=0$ in $M$ as $m\vert_U=0$, $\eta\vert_V=0$ with $U\cup V=X$, and $\pi_x(\eta)=1_{\fC_x}$ in $\fC_x$ as $\eta=1$ near $x$ in $X$. Hence
\begin{equation*}
m\ot 1_{\fC_x}\!=\!1_{\fC_x}\cdot(m\ot 1_{\fC_x})\!=\!\pi_x(\eta)\cdot(m\ot 1_{\fC_x})\!=\!(\eta\cdot m)\ot 1_{\fC_x}\!=\!0\ot 1_{\fC_x}\!=\!0
\end{equation*}
in $M\ot_\fC\fC_x$. Therefore $d=0$ in $\fD_x$, and $\Psi_{\cE,x}:\cD_x\ra\cE_x$ is injective, and so an isomorphism. As this holds for all $x\in X$, $\Psi_\cE:\cD\ra\cE$ is an isomorphism, proving the first part of the proposition. The second part follows from $L_{M,\cE}$ functorial in $M,\cE$ in Theorem~\ref{ag5thm2}.
\end{proof}

As for quasicoherent sheaves in conventional algebraic geometry, we define:

\begin{dfn} Let $\uX=(X,\O_X)$ be a $C^\iy$-scheme, and $\cE$ be an $\O_X$-module. We call $\cE$ {\it quasicoherent\/}\I{sheaf!quasicoherent} if we may cover $X$ with open $U\subseteq X$ such that $(U,\O_X\vert_U)\cong\Spec\fC$ and $\cE\vert_U\cong\MSpec M$ for some $C^\iy$-ring $\fC$ and $\fC$-module~$M$.

We write $\qcoh(\uX)$\G[qcoh(X)a]{$\qcoh(\uX)$}{abelian category of quasicoherent sheaves on $C^\iy$-scheme $\uX$} for the category of quasicoherent sheaves on~$\uX$.
\label{ag5def7}
\end{dfn}

If $(X,\O_X)$ is a $C^\iy$-scheme and $\cE$ an $\O_X$-module, we can cover $X$ by open $U\subseteq X$ with $(U,\O_X\vert_U)\cong\Spec\fC$ affine, and then Proposition \ref{ag5prop5} shows that $\cE\vert_U\cong\MSpec M$ for $M=\cE(U)$. Thus we have:

\begin{cor} Let\/ $\uX=(X,\O_X)$ be a\/ $C^\iy$-scheme. Then every $\O_X$-module $\cE$ is quasicoherent, so that\/~$\qcoh(\uX)=\OXmod$.
\label{ag5cor1}
\end{cor}

\begin{rem}{\bf(a)} In conventional algebraic geometry, as in Hartshorne \cite[\S II.5]{Hart}, if $R$ is a ring and $(X,\O_X)=\Spec R$ the corresponding affine scheme, we also have functors $\MSpec:\Rmod\ra\OXmod$ and $\Ga:\OXmod\ra\Rmod$. In $C^\iy$-algebraic geometry, as in Proposition \ref{ag5prop5}, $\Ga$ is a right inverse for $\MSpec$, but may not be a left inverse. But in algebraic geometry the opposite happens, as $\Ga$ is a left inverse for $\MSpec$ \cite[Cor.~II.5.5]{Hart}, but may not be a right inverse.

The fact that $\Ga$ is a right inverse for $\MSpec$ in $C^\iy$-algebraic geometry means that all $\O_X$-modules on a $C^\iy$-scheme $(X,\O_X)$ are quasicoherent, so quasicoherence is not a very useful idea. But in algebraic geometry, as $\Ga$ is not a right inverse for $\MSpec$, this is {\it false\/}: there are many examples of schemes $(X,\O_X)$ and $\O_X$-modules $\cE$ which are not quasicoherent. For instance, we may take $X={\mathbb A}^1$ and $\cE(U)=0$ if $0\in U$, $\cE(U)=\O_X(U)$ if $0\notin U$ for all open~$U\subseteq X$.

In \S\ref{ag55} we will define a module $M$ over a $C^\iy$ ring $\fC$ to be {\it complete\/} if $M\cong\Ga\ci\MSpec M$. Then $\Ga$ is a left inverse for $\MSpec$ on the subcategory $\fCmodco\subset\fCmod$ of complete $\fC$-modules. In general $\fC$-modules need not be complete. But in conventional algebraic geometry, as $\Ga$ is a left inverse for $\MSpec$ all $R$-modules are complete, so completeness is not a useful idea.
\smallskip

\noindent{\bf(b)} In conventional algebraic geometry one defines {\it coherent sheaves\/}\I{sheaf!coherent} \cite[\S II.5]{Hart} to be quasicoherent sheaves $\cE$ locally modelled on $\MSpec M$ for $M$ a finitely generated $\fC$-module. However, coherent sheaves are only well behaved on {\it noetherian\/} schemes, and most interesting $C^\iy$-rings, such as $C^\iy(\R^n)$ for $n>0$, are not noetherian $\R$-algebras. Because of this, coherent sheaves do not seem to be a useful idea in $C^\iy$-algebraic geometry (for instance, $\coh(\uX)$ is not closed under kernels in $\qcoh(\uX)$, and is not an abelian category), and we do not discuss them.

\label{ag5rem4}
\end{rem}

We can understand the pullback functor\I{C-scheme@$C^\iy$-scheme!sheaves of $\O_X$-modules on!pullback} $\uf^*$ in Definition \ref{ag5def4} explicitly in terms of modules over the corresponding $C^\iy$-rings:

\begin{prop} Let\/ $\fC,\fD$ be\/ $C^\iy$-rings, $\phi:\fD\ra\fC$ a
morphism, $M,N$ be $\fD$-modules, and\/ $\al:M\ra N$ a morphism of\/
$\fD$-modules. Write $\uX=\Spec\fC,$ $\uY=\Spec\fD,$
$\uf=\Spec\phi:\uX\ra\uY,$ and\/ $\cE=\MSpec M,$
$\cF=\MSpec N$ in $\qcoh(\uY)$. Then there are natural
isomorphisms\/ $\uf^*(\cE)\cong\MSpec(M\ot_\fD\fC)$ and\/
$\uf^*(\cF)\cong\MSpec(N\ot_\fD\fC)$ in $\qcoh(\uX)$. These identify
$\MSpec(\al\ot\id_\fC):\MSpec(M\ot_\fD\fC)\ra\MSpec(N\ot_\fD\fC)$
with\/~$\uf^*(\MSpec\al):\uf^*(\cE)\ra\uf^*(\cF)$.
\label{ag5prop6}
\end{prop}

\begin{proof} Write $\uX=(X,\O_X)$, $\uY=(Y,\O_Y)$ and
$\uf=(f,f^\sh)$. Then $\cE$ is the sheafification of the presheaf
$V\mapsto M\ot_\fD\O_Y(V)$, and $f^{-1}(\cE)$ is the sheafification
of the presheaf $U\mapsto\lim_{V\supseteq f(U)}\cE(V)$, and
$f^{-1}(\O_Y)$ is the sheafification of the presheaf
$U\mapsto\lim_{V\supseteq f(U)}\O_Y(V)$. In
$\uf^*(\cE)=f^{-1}(\cE)\ot_{f^{-1}(\O_Y)}\O_X$, these three
sheafifications combine into one, so $\uf^*(\cE)$ is the
sheafification of the presheaf $U\mapsto\lim_{V\supseteq
f(U)}(M\ot_\fD\O_Y(V)) \ot_{\O_Y(V)}\O_X(U)$. But
\begin{equation*}
(M\ot_\fD\O_Y(V))\ot_{\O_Y(V)}\O_X(U) \cong M\ot_\fD\O_X(U)\cong
(M\ot_\fD\fC)\ot_\fC\O_X(U),
\end{equation*}
so this is canonically isomorphic to the presheaf $U\mapsto
(M\ot_\fD\fC)\ot_\fC\O_X(U)$ whose sheafification is
$\MSpec(M\ot_\fD\fC)$. This gives a natural isomorphism
$\uf^*(\cE)\cong\MSpec(M\ot_\fD\fC)$. The same holds for $N$. The
identification of $\MSpec(\al\ot\id_\fC)$ and $\uf^*(\MSpec\al)$
follows by passing from morphisms of presheaves to morphisms of the
associated sheaves.\I{C-scheme@$C^\iy$-scheme!affine!sheaves of
$\O_X$-modules on|)}
\end{proof}

\subsection{\texorpdfstring{Complete modules over $C^\iy$-rings}{Complete modules over C∞-rings}}
\label{ag55}

Here are the module analogues of Definition \ref{ag4def13} and Theorem~\ref{ag4thm4}(b),(c).\I{module over C-ring@module over $C^\iy$-ring!complete|(}

\begin{dfn} Let $\fC$ be a $C^\iy$-ring, and $M$ a $\fC$-module. We call $M$ {\it complete\/} if $\Psi_M:M\ra\Ga\ci\MSpec M$ in Definition \ref{ag5def6} is an isomorphism. 

Write $\fCmodco$ for the full subcategory of complete $\fC$-modules in $\fCmod$.

If $M$ is a $\fC$-module then applying $\Ga$ to Proposition \ref{ag5prop5} shows that
\begin{equation*}
\Ga(\Psi_{\MSpec M}):\Ga\ci\MSpec(\Ga\ci\MSpec M)\longra\Ga\ci\MSpec M
\end{equation*}
is an isomorphism. From the definitions we can show that $\Psi_{\Ga\ci\MSpec M}=\Ga(\Psi_{\MSpec M})^{-1}$. Thus $\Ga\ci\MSpec M$ is complete, for any $\fC$-module $M$. Define a functor $R_{\rm all}^{\rm co}=\Ga\ci\MSpec:\fCmod\ra \fCmodco$.\G[Rcab]{$R_{\rm all}^{\rm co}:\fCmod\ra \fCmodco$}{reflection functor}\label{ag5def8}
\end{dfn}

\begin{thm} Let\/ $\fC$ be a $C^\iy$-ring, and\/ $\uX=(X,\O_X)=\Spec\fC$. Then
\begin{itemize}
\setlength{\itemsep}{0pt}
\setlength{\parsep}{0pt}
\item[{\bf(a)}] $\MSpec\vert_{\fCmodco}:\fCmodco\ra\qcoh(\uX)$ is an equivalence of categories.
\item[{\bf(b)}] $R_{\rm all}^{\rm co}:\fCmod\ra\fCmodco$ is left
adjoint\I{functor!adjoint}\I{adjoint functor} to the
inclusion functor\/ $\inc:\fCmodco\hookra\fCmod$. That is, $R_{\rm all}^{\rm co}$ is a \begin{bfseries}reflection functor\end{bfseries}.\I{functor!reflection}\I{reflection functor}
\end{itemize}
\label{ag5thm3}
\end{thm}

\begin{proof} For (a), if $M,N$ are complete $\fC$-modules then putting $\cE=\MSpec N$ in Theorem \ref{ag5thm2} and using $\Ga\ci\MSpec N\cong N$, equation \eq{ag5eq14} shows that
\begin{equation*}
\MSpec=L_{M,\cE}:\Hom_\fCmodco(M,N)\longra\Hom_\OXmod(\MSpec M,\MSpec N)
\end{equation*}
is a bijection, where the definition of $L_{M,\cE}$ agrees with the definition of $\MSpec$ on morphisms in this case. Thus $\MSpec$ is full and faithful on complete $\fC$-modules. 

If $\cE\in\OXmod=\qcoh(\uX)$ then $\cE\cong\MSpec\ci\Ga(\cE)$ by Proposition \ref{ag5prop5}. Thus $\Ga(\cE)\cong\Ga\ci\MSpec\ci\Ga(\cE)$, so $\Ga(\cE)$ is complete by Definition \ref{ag5def8}. Hence $\cE\cong\MSpec\vert_{\fCmodco}[\Ga(\cE)]$, and the essential image of $\MSpec\vert_{\fCmodco}$ is $\qcoh(\uX)$. Therefore $\MSpec\vert_{\fCmodco}$ is an equivalence of categories.

For (b), let $M,N$ be $\fC$-modules with $N$ complete. Then we have bijections
\ea
&\Hom_\fCmodco\bigl(R_{\rm all}^{\rm co}(M),N\bigr)\cong \Hom_\fCmod\bigl(\Ga\ci\MSpec M,\Ga\ci\MSpec N\bigr)
\nonumber\\
\begin{split}
&\cong \Hom_\OXmod\bigl(\MSpec \ci\Ga\ci\MSpec M,\MSpec N\bigr)\\
&\cong\Hom_\OXmod\bigl(\MSpec M,\MSpec N\bigr)
\end{split}
\label{ag5eq17}\\
&\cong \Hom_\fCmod\bigl(M,\Ga\ci\MSpec N\bigr)\!\cong\! \Hom_\fCmod\bigl(M,N\bigr)\!=\!\Hom_\fCmod\bigl(M,\inc(N)\bigr),
\nonumber
\ea
using $N\cong\Ga\ci\MSpec N$ as $N$ is complete in the first and fifth steps, Theorem \ref{ag5thm2} in the second and fourth, and Proposition \ref{ag5prop5} in the third. The bijections \eq{ag5eq17} are functorial in $M,N$ as each step is. Hence $R_{\rm all}^{\rm co}$ is left adjoint to~$\inc$.
\end{proof}

\begin{prop} Let\/ $\fC$ be a\/ $C^\iy$-ring and\/ $(X,\O_X)=\Spec\fC,$ and suppose $X$ is Lindel\"of. Then\/ $\fCmodco$ is closed under kernels, cokernels and extensions in\/ $\fCmod,$ that is, $\fCmodco$ is an abelian subcategory\I{abelian category} of\/~$\fCmod$.
\label{ag5prop7}
\end{prop}

\begin{proof} As in \S\ref{ag54}, $\MSpec:\fCmod\ra\OXmod$ is an exact functor, and as $X$ is Lindel\"of $\Ga:\OXmod\ra\fCmod$ is also exact by Proposition \ref{ag5prop3}. Hence $R_{\rm all}^{\rm co}=\Ga\ci\MSpec:\fCmod\ra\fCmod$ is an exact functor. Let $0\ra M_1\ra M_2\ra M_3$ be exact in $\fCmod$ with $M_2,M_3$ complete. Then we have a commutative diagram 
\begin{equation*}
\xymatrix@C=30pt@R=15pt{ 0 \ar[r] & M_1 \ar[d]^{\Psi_{M_1}} \ar[r] & M_2 \ar[d]^{\Psi_{M_2}}_\cong \ar[r] & M_3 \ar[d]^{\Psi_{M_3}}_\cong \\
0 \ar[r] & R_{\rm all}^{\rm co}(M_1) \ar[r] & R_{\rm all}^{\rm co}(M_2) \ar[r] & R_{\rm all}^{\rm co}(M_3) }
\end{equation*}
in $\fCmod$, where both rows are exact as $R_{\rm all}^{\rm co}$ is an exact functor, and the second and third columns are isomorphisms. Hence the first column is also an isomorphism, and $M_1$ is complete, so $\fCmodco$ is closed under kernels in $\fCmod$. It is closed under cokernels and extensions by very similar arguments.
\end{proof}

\begin{ex} Let $\fC$ be a $C^\iy$-ring with $(X,\O_X)=\Spec\fC$. Then:
\begin{itemize}
\setlength{\itemsep}{0pt}
\setlength{\parsep}{0pt}
\item[{\bf(a)}] Considering $\fC$ as a $\fC$-module, we have $\Ga\ci\MSpec\fC=\Ga\ci\Spec\fC=\O_X(X)$, and $\Psi_\fC:\fC\ra\O_X(X)$ in Definitions \ref{ag4def10} and \ref{ag5def6} coincide. Hence $\fC$ is complete as a $\fC$-module if and only if it is complete as a $C^\iy$-ring, in the sense of \S\ref{ag46}. So, if $\fC$ is a finitely generated but not fair $C^\iy$-ring, as in Examples \ref{ag2ex7} and \ref{ag2ex8}, then $\fC$ is a non-complete $\fC$-module.
\item[{\bf(b)}] Suppose $\fC$ is complete and $X$ is Lindel\"of. Let $M$ be a finitely presented $\fC$-module, so we have an exact sequence $\fC\ot\R^m\ra\fC\ot\R^n\ra M\ra 0$ in $\fCmod$. Here $\fC\ot\R^m,\fC\ot\R^n$ are complete as $\fC$ is by {\bf(a)}, so $M$ is complete by Proposition \ref{ag5prop7} as $\fCmod$ is closed under cokernels.
\item[{\bf(c)}] Suppose $\fC$ is complete, $X$ is Lindel\"of, and $I\subseteq \fC$ is a finitely generated ideal. Choose generators $i_1,\ldots,i_n$ for $I$. Then we have an exact sequence $\fC\ot\R^n\ra\fC\ra\fC/I\ra 0$ in $\fCmod$ with $\fC\ot\R^n,\fC$ complete, so $\fC/I$ is a complete $\fC$-module by Proposition \ref{ag5prop7}. Also we have an exact sequence $0\ra I\ra\fC\ra\fC/I$ with $\fC,\fC/I$ complete, so $I$ is a complete $\fC$-module.
\item[{\bf(d)}] Let $\fC$ be complete and $V$ be an infinite-dimensional $\R$-vector space. One can show that $\fC\ot_\R V$ is a complete $\fC$-module if and only if $X$ is compact.\I{module over C-ring@module over $C^\iy$-ring!complete|)}

\end{itemize}
\label{ag5ex4}
\end{ex}

\subsection{\texorpdfstring{Cotangent sheaves of $C^\iy$-schemes}{Cotangent sheaves of C∞-schemes}}
\label{ag56}
\I{C-scheme@$C^\iy$-scheme!cotangent sheaf|(}

We now define {\it cotangent sheaves}, the sheaf version of
cotangent modules in~\S\ref{ag52}.

\begin{dfn} Let $\uX=(X,\O_X)$ be a $C^\iy$-ringed space. Define $\cP T^*\uX$ to associate to each open $U\subseteq X$ the cotangent module $\Om_{\O_X(U)}$ of Definition \ref{ag5def2}, regarded as a module over the $C^\iy$-ring $\O_X(U)$, and to each inclusion of open sets $V\subseteq U\subseteq X$ the morphism of $\O_X(U)$-modules $\Om_{\rho_{UV}}:\Om_{\O_X(U)}\ra\Om_{\O_X(V)}$ associated to the morphism of $C^\iy$-rings $\rho_{UV}:\O_X(U)\ra\O_X(V)$. Then as we want for \eq{ag5eq13} the following commutes:
\begin{equation*}
\xymatrix@R=15pt@C=110pt{ *+[r]{\O_X(U)\t \Om_{\O_X(U)}}
\ar[d]^{\rho_{UV}\t\Om_{\rho_{UV}}} \ar[r]_(0.6){\mu_{\O_X(U)}} &
*+[l]{\Om_{\O_X(U)}} \ar[d]_{\Om_{\rho_{UV}}} \\
*+[r]{\O_X(V)\t \Om_{\O_X(V)}} \ar[r]^(0.6){\mu_{\O_X(V)}} & *+[l]{\Om_{\O_X(V)}.\!} }
\end{equation*}
Using this and functoriality of cotangent modules
$\Om_{\psi\ci\phi}=\Om_\psi\ci\Om_\phi$ in Definition \ref{ag5def2},
we see that $\cP T^*\uX$ is a presheaf\I{presheaf} of $\O_X$-modules on $\uX$. Define the {\it cotangent sheaf\/\I{C-ringed space@$C^\iy$-ringed space!cotangent sheaf} $T^*\uX$ of\/} $\uX$ to be the sheaf\I{presheaf!sheafification} of $\O_X$-modules associated to~$\cP T^*\uX$.

If $U\subseteq X$ is open then we have an equality of sheaves of
$\O_X\vert_U$-modules
\begin{equation*}
T^*(U,\O_X\vert_U)=T^*\uX\vert_U.
\end{equation*}

As in Example \ref{ag5ex2}, if $f:X\ra Y$ is a smooth map of
manifolds we have a morphism $\d f:f^*(T^*Y)\ra T^*X$ of vector
bundles over $X$. Here is an analogue for $C^\iy$-ringed spaces. Let
$\uf:\uX\ra\uY$ be a morphism of $C^\iy$-ringed spaces. Then by
Definition \ref{ag5def4}, $\uf^*(T^*\uY)=f^{-1}(T^*\uY)
\ot_{f^{-1}(\O_Y)}\O_X,$ where $T^*\uY$ is the sheafification of the
presheaf $V\mapsto\Om_{\O_Y(V)}$, and $f^{-1}(T^*\uY)$ the
sheafification of the presheaf $U\mapsto\lim_{V\supseteq f(U)}
(T^*\uY)(V)$, and $f^{-1}(\O_Y)$ the sheafification of the presheaf
$U\mapsto\lim_{V\supseteq f(U)}\O_Y(V)$. These three sheafifications
combine into one, so that $\uf^*(T^*\uY)$ is the
sheafification of the presheaf $\cP(\uf^*(T^*\uY))$ acting by
\begin{equation*}
U\longmapsto\cP(\uf^*(T^*\uY))(U)=
\ts\lim_{V\supseteq f(U)}\Om_{\O_Y(V)}\ot_{\O_Y(V)}\O_X(U).
\end{equation*}

Define a morphism of presheaves $\cP\Om_\uf:\cP(\uf^*(T^*\uY))\ra\cP
T^*\uX$ on $X$ by
\begin{equation*}
(\cP\Om_\uf)(U)=\ts\lim_{V\supseteq f(U)}
(\Om_{\rho_{f^{-1}(V)\,U}\ci f_\sh(V)})_*,
\end{equation*}
where $(\Om_{\rho_{f^{-1}(V)\,U}\ci f_\sh(V)})_*:\Om_{\O_Y(V)}
\ot_{\O_Y(V)}\O_X(U)\ra\Om_{\O_X(U)}=(\cP T^*\uX)(U)$ is constructed
as in Definition \ref{ag5def2} from the $C^\iy$-ring morphisms
$f_\sh(V):\O_Y(V)\ra\O_X(f^{-1}(V))$ from $f_\sh:\O_Y\ra f_*(\O_X)$
corresponding to $f^\sh$ in $\uf$ as in \eq{ag4eq3}, and
$\rho_{f^{-1}(V)\,U}:\O_X(f^{-1}(V))\ra\O_X(U)$ in $\O_X$. Define
$\Om_\uf:\uf^*(T^*\uY)\ra T^*\uX$ to be the induced
morphism of the associated sheaves.
\label{ag5def9}
\end{dfn}

\begin{rem} There is an alternative definition of the cotangent
sheaf $T^*\uX$ following Hartshorne \cite[p.~175]{Hart}. We can form
the product $\uX\t\uX$ in $\CRS$, and there is a natural diagonal
morphism $\uDe_\uX:\uX\ra\uX\t\uX$. Write $\cI_X$ for the
sheaf of ideals in $\O_{X\t X}$ vanishing on the closed
$C^\iy$-ringed subspace $\uDe_{\uX}$. Then $T^*\uX\cong
\uDe^*_\uX (\cI_X/\cI_X^2)$. This can be proved using the
equivalence of two definitions of cotangent module in \cite[Prop.~II.8.1A]{Hart}. An affine version of this also appears in Dubuc and Kock~\cite{DuKo}.
\label{ag5rem5}
\end{rem}

\begin{prop} Let\/ $\fC$ be a\/ $C^\iy$-ring and\/ $\uX=\Spec\fC$. Then there is a canonical isomorphism\/~$T^*\uX\cong\MSpec\Om_\fC$.
\label{ag5prop8}
\end{prop}

\begin{proof} By Definitions \ref{ag5def5} and \ref{ag5def9}, $\MSpec\Om_\fC$ and $T^*\uX$ are sheafifications of presheaves $\cP\MSpec\Om_\fC,\cP T^*\uX$, where for open $U\subseteq X$ we have
\begin{equation*}
\cP\MSpec\Om_\fC(U)=\Om_\fC\ot_\fC\O_X(U)\quad\text{and}\quad \cP T^*\uX(U)=\Om_{\O_X(U)}.
\end{equation*}
We have $C^\iy$-ring morphisms $\Psi_\fC:\fC\ra\O_X(X)$ from Definition \ref{ag4def10} and restriction $\rho_{XU}:\O_X(X)\ra\O_X(U)$ from $\O_X$, and so as in Definition \ref{ag5def2} a morphism of $\O_X(U)$-modules $\cP\rho(U):=(\rho_{XU}\ci\Psi_\fC)_*:\Om_\fC\ot_\fC\O_X(U)\ra\Om_{\O_X(U)}$. This defines a morphism of presheaves $\cP\rho:\cP\MSpec\Om_\fC\ra\cP T^*\uX$, and so sheafifying induces a morphism~$\rho:\MSpec\Om_\fC\ra T^*\uX$.

The induced morphism on stalks at $x\in X$ is $\rho_x=(\pi_x)_*:\Om_\fC\ot_\fC\fC_x\ra\Om_{\fC_x}$, where $\pi_x:\fC\ra\fC_x$ is projection to the local $C^\iy$-ring $\fC_x$, noting that $\O_{X,x}\cong\fC_x$. But $\fC_x$ is the localization $\fC[c^{-1}:c\in\fC$, $c(x)\ne 0]$, so Proposition \ref{ag5prop2} implies that $(\pi_x)_*:\Om_\fC\ot_\fC\fC_x\ra\Om_{\fC_x}$ is an isomorphism. Hence $\rho:\MSpec\Om_\fC\ra T^*\uX$ is a sheaf morphism which induces isomorphisms on stalks at all $x\in X$, so $\rho$ is an isomorphism.
\end{proof}

Here are some properties of the morphisms $\Om_\uf$ in
Definition \ref{ag5def9}. Equation \eq{ag5eq20} is an analogue of
\eq{ag5eq6} and~\eq{ag5eq12}.

\begin{thm}{\bf(a)} Let\/ $\uf:\uX\ra\uY$ and\/ $\ug:\uY\ra\uZ$ be
morphisms of\/ $C^\iy$-schemes. Then
\e
\Om_{\ug\ci\uf}=\Om_\uf\ci \uf^*(\Om_\ug)\ci I_{\uf,\ug}(T^*\uZ)
\label{ag5eq18}
\e
as morphisms $(\ug\ci\uf)^*(T^*\uZ)\ra T^*\uX$ in $\qcoh(\uX)$. Here
$\Om_\ug:\ug^*(T^*\uZ)\ra T^*\uY$ is a morphism in $\qcoh(\uY),$ so applying $\uf^*$ gives $\uf^*(\Om_\ug):\uf^*(\ug^*(T^*\uZ))\ra \uf^*(T^*\uY)$ in\/ $\qcoh(\uX),$ and\/ $I_{\uf,\ug}(T^*\uZ):(\ug\ci\uf)^*(T^*\uZ)
\ra\uf^*(\ug^*(T^*\uZ))$ is as in Remark\/~{\rm\ref{ag5rem3}}.

\smallskip

\noindent{\bf(b)} Suppose we are given a Cartesian square\I{Cartesian square} in $\CSch$
\e
\begin{gathered}
\xymatrix@C=80pt@R=14pt{ *+[r]{\uW} \ar[r]_\uf \ar[d]^\ue & *+[l]{\uY} \ar[d]_\uh \\
*+[r]{\uX} \ar[r]^\ug & *+[l]{\uZ,\!} }
\end{gathered}
\label{ag5eq19}
\e
so that\/ $\uW=\uX\t_\uZ\uY$. Then the following is exact in $\qcoh(\uW)\!:$
\e
\xymatrix@C=15pt{ (\ug\ci\ue)^*(T^*\uZ)
\ar[rrrr]^(0.49){\begin{subarray}{l}\ue^*(\Om_\ug)\ci
I_{\ue,\ug}(T^*\uZ)\op\\ -\uf^*(\Om_\uh)\ci
I_{\uf,\uh}(T^*\uZ)\end{subarray}} &&&&
{\raisebox{5pt}{$\begin{subarray}{l}\ts \;\>\ue^*(T^*\uX) \\ \ts {}\op\uf^*(T^*\uY)\end{subarray}$}} \ar[rr]^(0.55){\Om_\ue\op \Om_\uf}
&& T^*\uW \ar[r] & 0.}
\label{ag5eq20}
\e

\label{ag5thm4}
\end{thm}

\begin{proof} Combining several sheafifications\I{presheaf!sheafification} into one as in the proof of Proposition \ref{ag5prop6}, we see that the sheaves
$T^*\uX,\uf^*(T^*\uY),\uf^*(\ug^*(T^*\uZ))$ and
$(\ug\ci\uf)^*(T^*\uZ)$ on $\uX$ are isomorphic to the
sheafifications of the following presheaves:
\ea
&T^*\uX &\rightsquigarrow & \quad U \longmapsto \Om_{\O_X(U)},
\label{ag5eq21}\\
& \uf^*(T^*\uY) &\rightsquigarrow & \quad
U\longmapsto\lim_{V\supseteq f(U)} \Om_{\O_Y(V)}\ot_{\O_Y(V)}\O_X(U),
\label{ag5eq22}\\
& \uf^*(\ug^*(T^*\uZ)) &\rightsquigarrow & \quad
U\longmapsto\lim_{V\supseteq f(U)} \lim_{W\supseteq g(V)}\!\!\!
\begin{aligned}[t] \bigl(\Om_{\O_Z(W)}\ot_{\O_Z(W)}\O_Y(V)\bigr)
&\\[-3pt] \ot_{\O_Y(V)}\O_X(U),\,\,\,&
\end{aligned}
\label{ag5eq23}\\
& (\ug\ci\uf)^*(T^*\uZ) &\rightsquigarrow & \quad
U\longmapsto\lim_{W\supseteq g\ci f(U)}
\Om_{\O_Z(W)}\ot_{\O_Z(W)}\O_X(U).
\label{ag5eq24}
\ea

Then $\Om_\uf,\Om_{\ug\ci\uf},\uf^*(\Om_\ug),I_{\uf,\ug}(T^*\uZ)$
are the morphisms of sheaves associated to the following morphisms
of the presheaves in \eq{ag5eq21}--\eq{ag5eq24}:
\ea
&\Om_{\uf} &\rightsquigarrow & \quad U \longmapsto \lim_{V\supseteq
f(U)}(\Om_{\rho_{f^{-1}(V)\,U}\ci f_\sh(V)})_*,
\label{ag5eq25}\\
&\Om_{\ug\ci\uf} &\rightsquigarrow & \quad U \longmapsto
\lim_{W\supseteq g\ci f(U)}(\Om_{\rho_{(g\ci f)^{-1}(W)\,U}\ci (g\ci
f)_\sh(W)})_*,
\label{ag5eq26}\\
&\uf^*(\Om_\ug) &\rightsquigarrow & \quad U \longmapsto
\lim_{V\supseteq f(U)} \lim_{W\supseteq g(V)}
(\Om_{\rho_{g^{-1}(W)\,V}\ci g_\sh(W)})_*,
\label{ag5eq27}\\
&I_{\uf,\ug}(T^*\uZ)
&\rightsquigarrow & \quad U \longmapsto \lim_{V\supseteq f(U)}
\lim_{W\supseteq g(V)} I_{UVW},
\label{ag5eq28}
\ea
by Definition \ref{ag5def9}, where $I_{UVW}:\Om_{\O_Z(W)}
\ot_{\O_Z(W)}\O_X(U)\ra\bigl(\Om_{\O_Z(W)}\ot_{\O_Z(W)}\O_Y(V)
\bigr)\ot_{\O_Y(V)}\O_X(U)$ is the natural isomorphism.

Now if $U\!\subseteq\!X$, $V\!\subseteq\! Y$, $W\!\subseteq\! Z$ are
open with $V\!\supseteq\! f(U)$, $W\!\supseteq\! g(V)$ then
\begin{equation*}
\rho_{(g\ci f)^{-1}(W)\,U}\ci(g\ci f)_\sh(W)=\bigl[
\rho_{f^{-1}(V)\,U}\ci f_\sh(V)\bigr]\ci\bigl[\rho_{g^{-1}(W)\,V}\ci
g_\sh(W)\bigr]
\end{equation*}
as morphisms $\O_Z(W)\ra\O_X(U)$, so
$\Om_{\phi\ci\psi}=\Om_\phi\ci\Om_\psi$ in Definition \ref{ag5def2}
implies
\begin{equation*}
(\Om_{\rho_{(g\ci f)^{-1}(W)\,U}\ci(g\ci f)_\sh(W)})_*=
(\Om_{\rho_{f^{-1}(V)\,U}\ci
f_\sh(V)})_*\ci (\Om_{\rho_{g^{-1}(W)\,V}\ci
g_\sh(W)})_*\ci I_{UVW}.
\end{equation*}
Taking limits $\lim_{V\supseteq f(U)}\lim_{W\supseteq g(V)}$ implies
that the morphisms of presheaves in \eq{ag5eq25}--\eq{ag5eq28}
satisfy the analogue of \eq{ag5eq18}, so passing to sheaves
proves~(a).

For (b), first observe that as \eq{ag5eq19} is commutative, by (a) we
have
\begin{gather*}
\Om_\ue\ci\ue^*(\Om_\ug)\ci
I_{\ue,\ug}(T^*\uZ)=\Om_{\ug\ci\ue}=\Om_{\uh\ci\uf}=
\Om_\uf\ci\uf^*(\Om_\uh)\ci I_{\uf,\uh}(T^*\uZ),\\
\text{so}\qquad
\Om_\ue\ci\bigl(\ue^*(\Om_\ug)\ci I_{\ue,\ug}(T^*\uZ)\bigr)-\Om_\uf
\ci\bigl(\uf^*(\Om_\uh)\ci I_{\uf,\uh}(T^*\uZ)\bigr)=0,
\end{gather*}
and \eq{ag5eq20} is a complex. To show it is exact, note that as in the first part of the proof, \eq{ag5eq20} is the sheafification of a complex of presheaves, and the presheaves are defined as direct limits. Let $S\subseteq W$ be open. Then the complex of presheaves corresponding to \eq{ag5eq20} evaluated at $S\subseteq W$ is the direct limit over all open $T\subseteq X$, $U\subseteq Y$, $V\subseteq Z$ with $e(S)\subseteq T$, $f(S)\subseteq U$, $g(T)\subseteq V$, $h(U)\subseteq V$ of equation \eq{ag5eq6} with $\O_Z(V),\O_X(T),\O_Y(U),\O_W(S)$ in place of $\fC,\fD,\fE,\fF$. 

Since \eq{ag5eq6} is exact by Theorem \ref{ag5thm1} and direct limits are exact, the complex of presheaves whose sheafification is \eq{ag5eq20} is exact when evaluated on each open $S\subseteq W$, so it is exact. As sheafification is an exact functor, this implies that equation \eq{ag5eq20} is exact. This completes the proof.\I{C-scheme@$C^\iy$-scheme!cotangent
sheaf|)}\I{C-scheme@$C^\iy$-scheme!sheaves of $\O_X$-modules on|)}
\end{proof}

\section{\texorpdfstring{$C^\iy$-stacks}{C∞-stacks}}
\label{ag6}
\I{C-stack@$C^\iy$-stack|(}

We now discuss $C^\iy$-{\it stacks}, that is, geometric stacks over
the site $(\CSch,\cJ)$ of $C^\iy$-schemes with the open cover
topology. The author knows of no previous work on these. For the rest of the book, we will assume the reader has some familiarity with stacks in algebraic geometry. Appendix \ref{agA} summarizes the main definitions and results on stacks that we will use, but it is too brief to help someone learn about stacks for the first time. Readers with little experience of stacks are advised to first consult an introductory text such as Vistoli \cite{Vist}, Gomez \cite{Gome}, Laumon and Moret-Bailly \cite{LaMo}, or the online `Stacks Project'~\cite{Jong}.

The author found Metzler \cite{Metz} and Noohi \cite{Nooh} useful in writing this section.

\subsection{\texorpdfstring{$C^\iy$-stacks}{C∞-stacks}}
\label{ag61}

We use the material of \S\ref{agA2}--\S\ref{agA5}.

\begin{dfn} Define a Grothendieck pretopology\I{Grothendieck pretopology}
$\cP\cJ$ on the category of $C^\iy$-schemes $\CSch$ to have coverings
$\{\ui_a:\uU_a\ra\uU\}_{a\in A}$ where $V_a=i_a(U_a)$ is open in $U$
with $\ui_a:\uU_a\ra(V_a,\O_U\vert_{V_a})$ and isomorphism for all
$a\in A$, and $U=\bigcup_{a\in A}V_a$. Using Corollary
\ref{ag4cor2} we see that up to isomorphisms of the $\uU_a$, the
coverings $\{\ui_a:\uU_a\ra\uU\}_{a\in A}$ of $\uU$ correspond
exactly to open covers $\{V_a:a\in A\}$ of $U$. Write $\cJ$ for the associated Grothendieck topology.\I{Grothendieck topology}
\label{ag6def1}
\end{dfn}

It is a straightforward exercise in sheaf theory to prove:

\begin{prop} The site $(\CSch,\cJ)$ has descent for objects and
morphisms, in the sense of\/ {\rm\S\ref{agA3}}. Thus it is subcanonical.\I{site!subcanonical}
\label{ag6prop1}
\end{prop}

The point here is that since coverings of $\uU$ in $\cJ$ are just
open covers of the underlying topological space $U$, rather than
something more complicated like \'etale covers in algebraic
geometry, proving descent is easy: for objects, we glue the
topological spaces $X_a$ of $\uX_a$ together in the usual way to get
a topological space $X$, then we glue the $\O_{X_a}$ together to get
a presheaf of $C^\iy$-rings $\ti{\cal O}_X$ on $X$ isomorphic to
$\O_{X_a}$ on $X_a\subseteq X$ for all $a\in A$, and finally we
sheafify $\ti{\cal O}_X$ to a sheaf of $C^\iy$-rings $\O_X$ on $X$,
which is still isomorphic to $\O_{X_a}$ on $X_a\subseteq X$.

\begin{dfn} A $C^\iy$-{\it stack\/}\I{C-stack@$C^\iy$-stack!definition}
$\cX$\G[WXYZb]{$\cW,\cX,\cY,\cZ,\ldots$}{$C^\iy$-stacks} is a
geometric stack\I{stack!geometric} on the site\I{site}
$(\CSch,\cJ)$. Write $\CSta$\G[CSta]{$\CSta$}{2-category of
$C^\iy$-stacks} for the 2-category of $C^\iy$-stacks,
$\CSta=\GSta_{(\CSch,\cJ)}$. 

As in Definition \ref{agAdef12}, we will very often use the notation that if $\uX$ is a $C^\iy$-scheme then $\bar\uX$ is the associated $C^\iy$-stack,\G[X]{$\bar\uX$}{$C^\iy$-stack associated to a $C^\iy$-scheme $\uX$} and if $\uf:\uX\ra\uY$ is a morphism of $C^\iy$-schemes then $\bar\uf:\bar\uX\ra\bar\uY$\G[fXYb]{$\bar\uf:\bar\uX\ra\bar\uY$}{$C^\iy$-stack 1-morphism from a $C^\iy$-scheme morphism $\uf:\uX\ra\uY$} is the associated 1-morphism of $C^\iy$-stacks. Write $\bCSchlfp,\ab\bCSchlf,\ab\bCSch$\G[CSchx]{$\bCSch$}{2-subcategory of $\cX$ in $\CSta$ equivalent to a $C^\iy$-scheme
$\bar\uX$}\G[CSchy]{$\bCSchlf$}{2-subcategory of $\cX$ in $\CSta$
equivalent to $\bar\uX$ for $\uX$ locally
fair}\G[CSchz]{$\bCSchlfp$}{2-subcategory of $\cX$ in $\CSta$
equivalent to $\bar\uX$ for $\uX$ locally finitely presented} for
the full 2-subcategories of $C^\iy$-stacks $\cX$ in $\CSta$ which
are equivalent to $\bar\uX$ for $\uX$ in $\CSchlfp,\ab\CSchlf$ or
$\CSch$, respectively. When we say that a $C^\iy$-stack $\cX$ {\it
is a $C^\iy$-scheme},\I{C-stack@$C^\iy$-stack!is a $C^\iy$-scheme}
we mean that~$\cX\in\bCSch$.

Since $(\CSch,\cJ)$ is a subcanonical site, the embedding
$\CSch\ra\CSta$ taking $\uX\mapsto\bar\uX$, $\uf\mapsto\buf$ is
fully faithful. We write this as a full and faithful functor
$F_\CSch^\CSta:\CSch\ra\CSta$\G[FCSch]{$F_\CSch^\CSta:
\CSch\ra\CSta$}{inclusion from $C^\iy$-schemes to $C^\iy$-stacks}
mapping $F_\CSch^\CSta:\uX\mapsto\bar\uX$ on objects and
$F_\CSch^\CSta:\uf\mapsto\buf$ on (1-)morphisms. Hence
$\bCSchlfp,\ab\bCSchlf,\ab\bCSch$ are equivalent to
$\CSchlfp,\ab\CSchlf,\ab\CSch$, considered as 2-categories with only
identity 2-morphisms. In practice one often does not distinguish
between schemes and stacks which are equivalent to schemes, that is,
one identifies $\CSchlfp,\ldots,\CSch$
and~$\bCSchlfp,\ldots,\bCSch$.
\label{ag6def2}
\end{dfn}

\begin{rem} Behrend and Xu \cite[Def.~2.15]{BeXu} use `$C^\iy$-stack' to mean something different, a stack $X$ over the site\I{site} $(\Man,\cJ_\Man)$ of manifolds with Grothendieck topology\I{Grothendieck topology} $\cJ_\Man$ associated to the Grothendieck pretopology\I{Grothendieck pretopology} $\cP\cJ_\Man$ given by open covers, such that there exists a surjective representable submersion $\pi:\bar U\ra X$ from some manifold $U$. These are also called `smooth stacks' or `differentiable stacks' in \cite{BeXu,Hein,Metz,Nooh}. The quotient $[V/G]$ of a manifold $V$ by a Lie group $G$ is an example of a differentiable stack. By Zung's linearization theorem \cite[Th.~2.3]{Zung}, a differentiable stack $\cX$ with proper diagonal is Zariski locally equivalent to such a quotient $[V/G]$ with $G$ compact. Our $C^\iy$-stacks are a far larger class of more singular objects than the differentiable stacks of~\cite{BeXu,Hein,Metz,Nooh}.
\label{ag6rem1}
\end{rem}

Theorems \ref{ag4thm3}(b) and \ref{agAthm}, Corollary \ref{agAcor2} and
Proposition \ref{ag6prop1} imply:

\begin{thm} Let\/ $\cX$ be a $C^\iy$-stack. Then $\cX$ is equivalent to the stack\/\I{groupoid object}\I{category!groupoid object in}\I{stack!associated to a groupoid}\I{C-stack@$C^\iy$-stack!associated to a groupoid} $[\uV\rra\uU]$ associated to a groupoid\/ $(\uU,\uV,\us,\ut,\uu,\ui,\um)$ in\/ $\CSch$. Conversely, any groupoid in\/ $\CSch$ defines a $C^\iy$-stack\/ $[\uV\rra\uU]$. All fibre products exist in the\/ $2$-category\/~$\CSta$.\I{C-stack@$C^\iy$-stack!fibre products}
\label{ag6thm1}
\end{thm}

{\it Quotient\/ $C^\iy$-stacks\/} $[\uX/\uG]$ are a special class of $C^\iy$-stacks.\I{C-stack@$C^\iy$-stack!quotients
$[\protect\uX/G]$|(}\I{quotient C-stack@quotient $C^\iy$-stack|(}

\begin{dfn} A $C^\iy$-{\it group\/}\I{C-group@$C^\iy$-group}%
\I{C-scheme@$C^\iy$-scheme!C-group@$C^\iy$-group} $\uG$ is a group object in $\CSch$, that is, a $C^\iy$-scheme $\uG=(G,\O_G)$ equipped with an identity element
$1\in G$ and multiplication and inverse morphisms
$\um:\uG\t\uG\ra\uG$, $\ui:\uG\ra\uG$ in $\CSch$ such that
$(\ul{*},\uG,\upi,\upi, 1,\ui,\um)$ is a groupoid in $\CSch$.
Here $\ul{*}=\Spec\R$ is a point, and $\upi:\uG\ra\ul{*}$ is
the projection, and we regard $1\in G$ as a morphism~$1:\ul{*}\ra\uG$.

Let $\uG$ be a $C^\iy$-group, and $\uX$ a $C^\iy$-scheme. A ({\it
left\/}) {\it action\/} of $\uG$ on $\uX$ is a morphism
$\umu:\uG\t\uX\ra\uX$ such that
\e
\bigl(\uX,\uG\!\t\!\uX,\upi_{\uX},\umu,1\!\t\!\uid_\uX,
(\ui\ci\upi_{\uG})\!\t\!\umu,(\um\!\ci\!((\upi_\uG\!\ci\!\upi_1)
\!\t\!(\upi_\uG\!\ci\!\upi_2)))\!\t\!(\upi_\uX\!\ci\!\upi_2)\bigr)
\label{ag6eq1}
\e
is a groupoid object in $\CSch$, where in the final morphism
$\upi_1,\upi_2$ are the projections from $(\uG\t\uX)\t_{\upi_\uX,
\uX,\umu}(\uG\t\uX)$ to the first and second factors $\uG\t\uX$.
Then define the {\it quotient\/ $C^\iy$-stack\/} $[\uX/\uG]$ to be
the stack $[\uG\t\uX\rra\uX]$ associated to the groupoid \eq{ag6eq1}.
It is a $C^\iy$-stack.

If $\uG=(G,\O_G)$ is a $C^\iy$-group then the underlying space $G$
is a topological group, and is in particular a group, and if
$\uG=(G,\O_G)$ acts on $\uX=(X,\O_X)$ then $G$ acts continuously
on~$X$.

If $G$ is a Lie group then $\uG=F_\Man^\CSch(G)$ is a $C^\iy$-group
in a natural way, by applying $F_\Man^\CSch$ to the smooth
multiplication and inverse maps $m:G\t G\ra G$ and $i:G\ra G$. If a
Lie group $G$ acts smoothly on a manifold $X$ with action $\mu:G\t
X\ra X$ then the $C^\iy$-group $\uG=F_\Man^\CSch(G)$ acts on the
$C^\iy$-scheme $\uX=F_\Man^\CSch(X)$ with action
$\umu=F_\Man^\CSch(\mu):\uG\t\uX\ra\uX$, so we can form the quotient
$C^\iy$-stack~$[\uX/\uG]$.
\label{ag6def3}
\end{dfn}

\begin{ex} Let $\uG$ be a $C^\iy$-group, and $\uX=\ul{*}$ be the point in $\CSch$, with trivial $\uG$-action. The quotient $C^\iy$-stack $[\ul{*}/\uG]$ is known as $B\uG$, the classifying stack for principal $\uG$-bundles on $C^\iy$-schemes. 

If $\uS$ is a $C^\iy$-scheme, a {\it principal\/ $\uG$-bundle\/} $(\uP,\upi,\umu)$ over $\uS$ is a $C^\iy$-scheme $\uP$, a morphism $\upi:\uP\ra\uS$, and a $\uG$-action $\umu:\uG\t\uP\ra\uP$ of $\uG$ on $\uP$, such that $\upi$ is $\uG$-invariant, and $\uS$ may be covered by open $C^\iy$-subschemes $\uU\subseteq\uS$ such that there exists an isomorphism $\upi^{-1}(\uU)\cong\uG\t\uU$ which identifies the $\uG$-action on $\upi^{-1}(\uU)\subseteq\uP$ with the product of the left $\uG$-action on $\uG$ and the trivial $\uG$-action on $\uU$, and identifies $\upi\vert_{\cdots}:\upi^{-1}(\uU)\ra\uU$ with $\upi_\uU:\uG\t\uU\ra\uU$. Often we write $\uP$ as the principal bundle, leaving $\upi,\umu$ implicit.

One well known way to write $B\uG$ explicitly as a category fibred in groupoids $p_\cX:\cX\ra\CSch$, as in \S\ref{agA2}, is to define $\cX$ to be the category with objects pairs $(\uS,\uP)$ of a $C^\iy$-scheme $\uS$ and $\uP$ a principal $\uG$-bundle over $\uS$, and morphisms $(\uf,\uu):(\uS,\uP)\ra(\uT,\uQ)$ consisting of $C^\iy$-scheme morphisms $\uf:\uS\ra\uT$ and $\uu:\uP\ra\uQ$, such that $\uu$ is $\uG$-equivariant and
\e
\begin{gathered}
\xymatrix@C=90pt@R=15pt{ *+[r]{\uP} \ar[d]^\upi \ar[r]_\uu & *+[l]{\uQ} \ar[d]_\upi \\
*+[r]{\uS} \ar[r]^\uf & *+[l]{\uT} }
\end{gathered}
\label{ag6eq2}
\e
is a Cartesian square in $\CSch$, which implies that $\uP$ is canonically isomorphic to the pullback principal $\uG$-bundle $\uf^*(\uQ)$. Composition of morphisms is $(\ug,\uv)\ci(\uf,\uu)=(\ug\ci\uf,\uv\ci\uu)$, and identity morphisms are $\id_{(\uS,\uP)}=(\uid_\uS,\uid_\uP)$. The functor $p_\cX:\cX\ra\CSch$ maps $p_\cX:(\uS,\uP)\mapsto\uS$ on objects and $p_\cX:(\uf,\uu)\mapsto\uf$ on morphisms.
\label{ag6ex1}
\end{ex}

In \S\ref{ag71} we will give a more detailed treatment of
quotient $C^\iy$-stacks $[\uX/G]$ of a $C^\iy$-scheme $\uX$ by a
finite group~$G$.\I{C-stack@$C^\iy$-stack!quotients $[\protect\uX/G]$|)}%
\I{quotient C-stack@quotient $C^\iy$-stack|)}

\subsection{\texorpdfstring{Properties of 1-morphisms of $C^\iy$-stacks}{Properties of 1-morphisms of C∞-stacks}}
\label{ag62}
\I{C-scheme@$C^\iy$-scheme!morphism|(}

We use the material of \S\ref{agA4}. We define some classes of $C^\iy$-scheme morphisms.

\begin{dfn} Let $\uf=(f,f^\sh):\uX=(X,\O_X)\ra\uY=(Y,\O_Y)$ be a
morphism in $\CSch$. Then:
\begin{itemize}
\setlength{\itemsep}{0pt}
\setlength{\parsep}{0pt}
\item We call $\uf$ an {\it open
embedding\/}\I{C-scheme@$C^\iy$-scheme!open embedding} if
$V=f(X)$ is an open subset in $Y$ and
$(f,f^\sh):(X,\O_X)\ra(V,\O_Y\vert_V)$ is an isomorphism.
\item We call $\uf$ a {\it closed
embedding\/}\I{C-scheme@$C^\iy$-scheme!closed embedding} if
$f:X\ra Y$ is a homeomorphism with a closed subset of $\uY$, and
$f^\sh:f^{-1}(\O_Y)\ra\O_X$ is a surjective morphism of sheaves
of $C^\iy$-rings. Equivalently, $\uf$ is an isomorphism with a
closed $C^\iy$-subscheme of $\uY$. Over affine open subsets
$\uU\cong\Spec\fC$ in $\uY$, $\uf$ is modelled on the natural
morphism $\Spec(\fC/I)\hookra\Spec\fC$ for some ideal $I$
in~$\fC$.
\item We call $\uf$ an {\it
embedding\/}\I{C-scheme@$C^\iy$-scheme!embedding} if we may
write $\uf=\ug\ci\uh$ where $\uh$ is an open embedding and $\ug$
is a closed embedding.
\item We call $\uf$ {\it
\'etale\/}\I{C-scheme@$C^\iy$-scheme!etale morphism@\'etale
morphism} if each $x\in X$ has an open neighbourhood $U$ in $X$
such that $V=f(U)$ is open in $Y$ and
$(f\vert_U,f^\sh\vert_U):(U,\O_X\vert_U)\ra (V,\O_Y\vert_V)$ is
an isomorphism. That is, $\uf$ is a local isomorphism.
\item We call $\uf$ {\it
proper\/}\I{C-scheme@$C^\iy$-scheme!proper morphism} if $f:X\ra
Y$ is a proper map of topological spaces, that is, if
$S\subseteq Y$ is compact then $f^{-1}(S)\subseteq X$ is
compact.
\item We say that $\uf$ {\it has finite
fibres\/}\I{C-scheme@$C^\iy$-scheme!morphism with finite fibres}
if $f:X\ra Y$ is a finite map, that is, $f^{-1}(y)$ is a finite
subset of $X$ for all~$y\in Y$.
\item We call $\uf$ {\it
separated\/}\I{C-scheme@$C^\iy$-scheme!separated morphism} if
$f:X\ra Y$ is a separated map of topological spaces, that is,
$\De_X=\bigl\{ (x,x):x\in X\bigr\}$ is a closed subset of the
topological fibre product $X\t_{f,Y,f}X=\bigl\{(x,x')\in X\t
X:f(x)=f(x')\bigr\}$.
\item We call $\uf$ {\it
closed\/}\I{C-scheme@$C^\iy$-scheme!closed morphism} if $f:X\ra
Y$ is a closed map of topological spaces, that is, $S\subseteq
X$ closed implies $f(S)\subseteq Y$ closed.
\item We call $\uf$ {\it universally
closed\/}\I{C-scheme@$C^\iy$-scheme!universally closed morphism}
if whenever $\ug:\uW\ra\uY$ is a morphism then
$\upi_\uW:\uX\t_{\uf,\uY,\ug}\uW\ra\uW$ is closed.
\item We call $\uf$ a {\it
submersion\/}\I{C-scheme@$C^\iy$-scheme!submersion} if for all
$x\in X$ with $f(x)=y$, there exists an open neighbourhood $U$
of $y$ in $Y$ and a morphism
$\ug=(g,g^\sh):(U,\O_Y\vert_U)\ra(X,\O_X)$ with $g(y)=x$
and~$\uf\ci\ug=\id_{(U,\O_Y\vert_U)}$.
\item We call $\uf$ {\it locally
fair},\I{C-scheme@$C^\iy$-scheme!locally fair morphism} or {\it
locally finitely presented},\I{C-scheme@$C^\iy$-scheme!locally
finitely presented morphism} if whenever $\uU$ is a locally
fair,\I{C-scheme@$C^\iy$-scheme!locally fair} or locally
finitely presented
$C^\iy$-scheme,\I{C-scheme@$C^\iy$-scheme!locally finitely
presented} respectively, and $\ug:\uU\ra\uY$ is a morphism then
$\uX\t_{\uf,\uY,\ug}\uU$ is locally fair, or locally finitely
presented, respectively.
\end{itemize}
\label{ag6def4}
\end{dfn}

\begin{rem} These are mostly analogues of standard concepts in
algebraic geometry, as in Hartshorne \cite{Hart} for instance. But
because the topology on $C^\iy$-schemes is finer than the Zariski
topology in algebraic geometry --- for example, affine
$C^\iy$-schemes are Hausdorff --- our definitions of \'etale and
proper are simpler than in algebraic geometry. (Open or closed)
embeddings correspond to (open or closed) immersions in algebraic
geometry, but we prefer the word `embedding', as immersion has a
different meaning in differential geometry. Closed morphisms are not
invariant under base change, which is why we define universally
closed. If $X,Y$ are manifolds and $\uX,\uY=F_\Man^\CSch(X,Y)$, then
$\uf:\uX\ra\uY$ is a submersion of $C^\iy$-schemes if and only if
$\uf=F_\Man^\CSch(f)$ for $f:X\ra Y$ a submersion of manifolds.
\label{ag6rem2}
\end{rem}

\begin{dfn} Let $\bs P$ be a property of morphisms in $\CSch$. We
say that $\bs P$ {\it is stable under open embedding\/} if whenever
$\uf:\uU\ra\uV$ is $\bs P$ and $\ui:\uV\ra\uW$ is an open embedding,
then $\ui\ci\uf:\uU\ra\uW$ is $\bs P$.
\label{ag6def5}
\end{dfn}

The next proposition is elementary. See Laumon and Bailly \cite[\S
3.10]{LaMo} and Noohi \cite[Ex.~4.6]{Nooh} for similar lists for the
\'etale and topological sites.\I{site}

\begin{prop} The following properties of morphisms in $\CSch$ are
invariant under base change and local in the target in the site
$(\CSch,\cJ),$ in the sense of\/ {\rm\S\ref{agA4}:} open embedding,\I{C-scheme@$C^\iy$-scheme!open
embedding} closed embedding,\I{C-scheme@$C^\iy$-scheme!closed
embedding} embedding,\I{C-scheme@$C^\iy$-scheme!embedding}
\'etale,\I{C-scheme@$C^\iy$-scheme!etale morphism@\'etale morphism}
proper,\I{C-scheme@$C^\iy$-scheme!proper morphism} has finite
fibres,\I{C-scheme@$C^\iy$-scheme!morphism with finite fibres}
separated,\I{C-scheme@$C^\iy$-scheme!separated morphism} universally
closed,\I{C-scheme@$C^\iy$-scheme!universally closed morphism}
submersion,\I{C-scheme@$C^\iy$-scheme!submersion} locally
fair,\I{C-scheme@$C^\iy$-scheme!locally fair} locally finitely
presented.\I{C-scheme@$C^\iy$-scheme!locally finitely presented} The
following properties are also stable under open embedding, in the sense of Definition\/ {\rm\ref{ag6def5}:} open embedding, embedding, \'etale, has finite fibres, separated,
submersion, locally fair, locally finitely
presented.\I{C-scheme@$C^\iy$-scheme!morphism|)}
\label{ag6prop2}
\end{prop}

As in \S\ref{agA4}, this implies that these properties are also
defined for representable
1-morphisms\I{C-stack@$C^\iy$-stack!representable 1-morphism} in
$\CSta$. In particular, if $\cX$ is a $C^\iy$-stack then
$\De_\cX:\cX\ra\cX\t\cX$ is representable, and if $\Pi:\bar
\uU\ra\cX$ is an atlas\I{atlas}\I{C-stack@$C^\iy$-stack!atlas} then
$\Pi$ is representable, so we can
require that $\De_\cX$ or $\Pi$ has some of these properties.%
\I{C-stack@$C^\iy$-stack!open embedding}%
\I{C-stack@$C^\iy$-stack!closed embedding}%
\I{C-stack@$C^\iy$-stack!embedding}%
\I{C-stack@$C^\iy$-stack!etale 1-morphism@\'etale 1-morphism}%
\I{C-stack@$C^\iy$-stack!proper 1-morphism}%
\I{C-stack@$C^\iy$-stack!1-morphism with finite fibres}%
\I{C-stack@$C^\iy$-stack!separated 1-morphism}%
\I{C-stack@$C^\iy$-stack!universally closed 1-morphism}%
\I{C-stack@$C^\iy$-stack!submersion}%

\begin{dfn} Let $\cX$ be a $C^\iy$-stack. Following
\cite[Def.~7.6]{LaMo}, we say that $\cX$ is {\it
separated\/}\I{C-stack@$C^\iy$-stack!separated} if the diagonal
1-morphism $\De_\cX:\cX\ra\cX\t\cX$ is universally closed. If
$\cX=\bar\uX$ for some $C^\iy$-scheme $\uX=(X,\O_X)$ then $\cX$ is
separated if and only if $\De_X:X\ra X\t X$ is closed, that is, if
and only if $X$ is Hausdorff.
\label{ag6def6}
\end{dfn}

\begin{prop} Let\/ $\cW=\cX\t_{f,\cZ,g}\cY$ be a fibre
product\I{C-stack@$C^\iy$-stack!fibre products} of\/ $C^\iy$-stacks
with\/ $\cX,\cY$ separated. Then $\cW$ is separated.
\label{ag6prop3}
\end{prop}

\begin{proof} We have a 2-commutative diagram with both squares
2-Cartesian:\I{2-category!2-Cartesian square}
\e
\begin{gathered}
\xymatrix@!0@C=65pt@R=25pt{ & \cW \ar[rr]_{\De_\cW}
\ar[dl] \ar[dr]^(0.55){\pi_1} && \cW\t\cW \ar[dr] \\
\cZ \ar[dr]_{\jmath_\cZ} && \cX\t_{f\ci\De_\cZ,\cZ\t\cZ,g\ci\cZ}\cY
\ar[dl] \ar[dr]
\ar[ur]^(0.45){\pi_2} && \cX\t\cX\t\cY\t\cY. \\
& I_\cZ && \cX\t\cY \ar[ur]_{{}\,\,\,\De_\cX\t\De_\cY}}
\end{gathered}
\label{ag6eq3}
\e
Let $[\uV\rra\uU]$ be a groupoid
presentation\I{C-stack@$C^\iy$-stack!associated to a groupoid} of
$\cZ$, and consider the fourth 2-Cartesian diagram of \eq{agAeq12},
with surjective rows. The left hand morphism $\bar\uu\t\bar\uid_\uU$
has a left inverse $\pi_\uU$, and so is automatically universally
closed. Hence $\jmath_\cZ$ is universally closed by Propositions
\ref{agAprop2}(c) and \ref{ag6prop2}, so $\pi_1$ in \eq{ag6eq3} is
universally closed by Propositions \ref{agAprop2}(a) and
\ref{ag6prop2}. Also $\De_\cX,\De_\cY$ are universally closed as
$\cX,\cY$ are separated, so $\De_\cX\t\De_\cY$ in \eq{ag6eq3} is
universally closed, and $\pi_2$ is universally closed. Thus
$\De_\cW\cong\pi_2\ci\pi_1$ is universally closed, and $\cW$ is
separated.
\end{proof}

\subsection{\texorpdfstring{Open $C^\iy$-substacks and open covers}{Open C∞-substacks and open covers}}
\label{ag63}
\I{C-stack@$C^\iy$-stack!C-substack@$C^\iy$-substack}%
\I{C-stack@$C^\iy$-stack!C-substack@$C^\iy$-substack!open|(}%
\I{C-stack@$C^\iy$-stack!C-substack@$C^\iy$-substack!closed}%
\I{C-stack@$C^\iy$-stack!C-substack@$C^\iy$-substack!locally closed}%
\I{C-substack@$C^\iy$-substack}%
\I{C-substack@$C^\iy$-substack!open|(}%
\I{C-substack@$C^\iy$-substack!closed}%
\I{C-substack@$C^\iy$-substack!locally closed}%
\I{C-stack@$C^\iy$-stack!open cover|(}

\begin{dfn} Let $\cX$ be a $C^\iy$-stack. A $C^\iy$-{\it
substack\/} $\cY$ in $\cX$ is a substack of $\cX$, in the sense of
Definition \ref{agAdef7}, which is also a $C^\iy$-stack. It has a
natural inclusion 1-morphism $i_\cY:\cY\hookra\cX$. We call $\cY$ an
{\it open\/ $C^\iy$-substack\/} of $\cX$ if $i_\cY$ is a
representable\I{C-stack@$C^\iy$-stack!representable 1-morphism} open
embedding, a {\it closed\/ $C^\iy$-substack\/} of $\cX$ if $i_\cY$
is a representable closed embedding, and a {\it locally closed\/
$C^\iy$-substack\/} of $\cX$ if $i_\cY$ is a representable
embedding.

An {\it open cover\/} $\{\cU_a:a\in A\}$ of $\cX$ is a family of
open $C^\iy$-substacks $\cU_a$ in $\cX$ with $\coprod_{a\in
A}i_{\cU_a}:\coprod_{a\in A}\cU_a\ra\cX$ surjective. We write
$\cU\subseteq\cX$ when $\cU$ is an open $C^\iy$-substack of $\cX$,
and $\bigcup_{a\in A}\cU=\cX$ to mean that $\coprod_{a\in
A}i_{\cU_a}$ is surjective.
\label{ag6def7}
\end{dfn}

Some properties of $\De_\cX,\io_\cX,\jmath_\cX$ and
atlases\I{atlas}\I{C-stack@$C^\iy$-stack!atlas} for $\cX$ can be
tested on the elements of an open cover. The proof is elementary.

\begin{prop} Let\/ $\cX$ be a\/ $C^\iy$-stack, and\/
$\{\cU_a:a\in A\}$ an open cover of\/ $\cX$. Suppose $\bs P$ and\/
$\bs Q$ are properties of morphisms in $\CSch$ which are invariant
under base change and local in the target in $(\CSch,\cJ),$ and
that\/ $\bs P$ is stable under open embedding. Then:
\begin{itemize}
\setlength{\itemsep}{0pt}
\setlength{\parsep}{0pt}
\item[{\bf(a)}] Let\/ $\Pi_a:\bar\uU_a\ra\cU_a$ be an atlas
for\/ $\cU_a$ for\/ $a\in A$. Set\/ $\uU=\coprod_{a\in A}\uU_a$
and\/ $\Pi=\coprod_{a\in A}i_{\cU_a}\ci\Pi_a:\bar\uU\ra\cX$.
Then $\Pi$ is an atlas for $\cX,$ and\/ $\Pi$ is $\bs P$ if and
only if\/ $\Pi_a$ is $\bs P$ for all\/ $a\in A$.
\item[{\bf(b)}] $\De_\cX:\cX\!\ra\!\cX\!\t\!\cX$
is $\bs P$ if and only if\/
$\De_{\cU_a}:\cU_a\!\ra\!\cU_a\!\t\!\cU_a$ is $\bs P$ for
all\/~$a\!\in\!A$.
\item[{\bf(c)}] $\io_\cX:I_\cX\ra\cX$ is $\bs Q$
if and only if\/ $\io_{\cU_a}:I_{\cU_a}\ra\cU_a$ is $\bs Q$ for
all\/~$a\in A$.
\item[{\bf(d)}] $\jmath_\cX:\cX\ra I_\cX$ is $\bs Q$
if and only if\/ $\jmath_{\cU_a}:\cU_a\ra I_{\cU_a}$ is $\bs Q$
for all\/~$a\in A$.\I{C-stack@$C^\iy$-stack!open cover|)}
\end{itemize}
\label{ag6prop4}
\end{prop}

If $\cX=\bar\uU$ for some $C^\iy$-scheme $\uU=(U,\O_U)$, then the
open $C^\iy$-substacks of $\cX$ are precisely those subsheaves of the form 
$\,\overline{\!(V,\O_U\vert_V)\!}\,$ for all open $V\subseteq U$, that
is, they are the images in $\CSta$ of the open $C^\iy$-subschemes of
$U$. We can also describe the open substacks of stacks $[\uV\rra\uU]$ associated to groupoids:\I{C-stack@$C^\iy$-stack!associated to a groupoid}

\begin{prop} Let\/ $(\uU,\uV,\us,\ut,\uu,\ui,\um)$ be a groupoid
in\/ $\CSch$ and\/ $\cX=[\uV\rra\uU]$ the associated\/
$C^\iy$-stack, and write $\uU=(U,\O_U),$ and so on. Then open\/
$C^\iy$-substacks $\cX'$ of\/ $\cX$ are naturally in $1$-$1$
correspondence with open subsets\/ $U'\subseteq U$ with\/
$s^{-1}(U')=t^{-1}(U'),$ where $\cX'=[\uV'\rra\uU']$ for
$\uU'=(U',\O_U\vert_{U'})$ and\/ $\uV'=(s^{-1}(U'),\O_V
\vert_{s^{-1}(U')})$. If\/ $(\uU,\uV,\us,\ut,\uu,\ui,\um)$ is as in
{\rm\eq{ag6eq1},} so that\/ $\cX$ is a quotient\/ $C^\iy$-stack\/
$[\uU/\uG],$ then open\/ $C^\iy$-substacks $\cX'$ of\/ $\cX$
correspond to\/ $G$-invariant open subsets\/~$U'\subseteq U$.
\label{ag6prop5}
\end{prop}

\begin{proof} From Theorem \ref{agAthm}, as $\cX=[\uV\rra\uU]$ we
have a natural surjective, representable
1-morphism\I{C-stack@$C^\iy$-stack!representable 1-morphism}
$\Pi:\bar\uU\ra\cX$. If $\cX'$ is an open $C^\iy$-substack of $\cX$
then $\bar\uU\t_{\Pi,\cX,i_{\cX'}}\cX'$ is an open $C^\iy$-substack
of $\bar\uU$, and so is of the form
$\,\overline{\!(U',\O_U\vert_{U'})\!}\,$ for some open $U'\subseteq
U$. We have natural equivalences
\begin{align*}
&\,\overline{\!(s^{-1}(U'),\O_V\vert_{s^{-1}(U')})\!}\simeq\!
\bar U'\!\t_{i_{\bar U'},\bar U,\bar s}\!\bar V\!\simeq\!
{\cX'}\!\t_\cX\!(\bar U\!\t_{\id_{\bar U},\bar U,\bar s}\!
\bar V)\!\simeq\!{\cX'}\!\t_{i_\cX',\cX,\pi_\cX}\!\bar V\\
&\qquad\simeq {\cX'}\t_\cX(\bar U\t_{\id_{\bar U},\bar U,\bar t}
\bar V)\simeq \bar U'\t_{i_{\bar U'},\bar U,\bar t}\bar V\simeq\,
\overline{\!(t^{-1}(U'),\O_V\vert_{t^{-1}(U')})\!}\,,
\end{align*}
by associativity properties of fibre products in 2-categories, which implies that $s^{-1}(U')=t^{-1}(U')$. Conversely, if $s^{-1}(U')=t^{-1}(U')$ then defining $\uU',\uV'$ as in the proposition, we get a $C^\iy$-stack $\cX'=[\uV'\rra\uU']$ which is naturally an open $C^\iy$-substack of $\cX$. When $\cX=[\uU/\uG]$, we see that $s^{-1}(U')= t^{-1}(U')$ if
and only if $U'$ is $G$-invariant.%
\I{C-stack@$C^\iy$-stack!C-substack@$C^\iy$-substack!open|)}%
\I{C-substack@$C^\iy$-substack!open|)}%
\end{proof}

\subsection{\texorpdfstring{The underlying topological space of a $C^\iy$-stack}{The underlying topological space of a C∞-stack}}
\label{ag64}
\I{C-stack@$C^\iy$-stack!underlying topological space $\cX_\top$|(}

Following Noohi \cite[\S 4.3, \S 11]{Nooh} in the case of
topological stacks, we associate a topological space $\cX_\top$ to a
$C^\iy$-stack $\cX$. In \S\ref{ag74}, if $\cX$ is a Deligne--Mumford
$C^\iy$-stack, we will also give $\cX_\top$ the structure of a
$C^\iy$-scheme.

\begin{dfn} Let $\cX$ be a $C^\iy$-stack. Write $\ul{*}$ for the
point $\Spec\R$ in $\CSch$, and $\bar{\ul{*}}$ for the associated
point in $\CSta$. Define $\cX_\top$\G[Xtop]{$\cX_\top$}{underlying
topological space of a $C^\iy$-stack $\cX$} to be the set of
2-isomorphism classes $[x]$ of 1-morphisms $x:\bar{\ul{*}}\ra\cX$.

Suppose $\cU\subseteq\cX$ is an open $C^\iy$-substack. Since $\cU$
is a subcategory of $\cX$, any 1-morphism $u: \bar{\ul{*}}\ra\cU$,
regarded as a functor from the category $\bar{\ul{*}}$ to the
category $\cU$, is also a 1-morphism $u: \bar{\ul{*}}\ra\cX$. Also,
as $\cU$ is a strictly full subcategory of $\cX$, if
$x:\bar{\ul{*}}\ra\cX$ is a 1-morphism and $\eta:u\Ra x$ a
2-morphism of 1-morphisms $\bar{\ul{*}}\ra\cX$, then $x$ is also a
1-morphism $u: \bar{\ul{*}}\ra\cU$, and $\eta$ is also a 2-morphism
of 1-morphisms $\bar{\ul{*}}\ra\cU$. This implies that $\cU_\top$ is
a subset of $\cX_\top$.

Define $\cT_{\cX_\top}=\bigl\{\cU_\top:\cU\subseteq\cX$ is an
open $C^\iy$-substack in $\cX\bigr\}$, a set of subsets of
$\cX_\top$. We claim that $\cT_{\cX_\top}$ is a topology on
$\cX_\top$. To see this, note that taking $\cU$ to be $\cX$ or the
empty $C^\iy$-substack gives $\cX_\top,\es\in\cT_{\cX_\top}$.
If $\cU,\cV\subseteq\cX$ are open $C^\iy$-substacks of $\cX$ then
the intersection of subcategories $\cW=\cU\cap\cV$ is an open
$C^\iy$-substack of $\cX$ equivalent to the fibre
product\I{C-stack@$C^\iy$-stack!fibre products}
$\cU\t_{i_\cU,\cX,i_\cV}\cV$, with $\cW_\top=\cU_\top\cap\cV_\top$,
so $\cT_{\cX_\top}$ is closed under finite intersections.

If $\{\cU_a:a\in A\}$ is a family of open $C^\iy$-substacks in
$\cX$, define $\cV$ to be the unique smallest strictly full
subcategory of $\cX$ which contains $\cU_a$ for each $a\in A$ and is
closed under the stack axiom \eq{agAeq9} in Definition
\ref{agAdef6}. Then $\cV$ is an open $C^\iy$-substack of $\cX$,
which we write as $\cV=\bigcup_{a\in A}\cU_a$, and
$\cV_\top=\bigcup_{a\in A}\cU_{a\,\top}$. So $\cT_{\cX_\top}$
is closed under arbitrary unions.

Thus $(\cX_\top,\cT_{\cX_\top})$ is a topological space, which
we call the {\it underlying topological space\/} of $\cX$, and
usually write as $\cX_\top$. It has the following properties. If
$f:\cX\ra\cY$ is a 1-morphism of $C^\iy$-stacks then there is a
natural continuous map $f_\top:\cX_\top\ra\cY_\top$ defined by
$f_\top([x])=[f\ci x]$. If $f,g:\cX\ra\cY$ are 1-morphisms and
$\eta:f\Ra g$ is a 2-isomorphism then $f_\top=g_\top$. Mapping
$\cX\mapsto\cX_\top$, $f\mapsto f_\top$ and 2-morphisms to
identities defines a 2-functor $F_\CSta^\Top:\CSta\ra\Top$, where
the category of topological spaces $\Top$ is regarded as a
2-category with only identity 2-morphisms.

If $\uX=(X,\O_X)$ is a $C^\iy$-scheme, so that $\bar\uX$ is a
$C^\iy$-stack, then $\bar\uX_\top$ is naturally homeomorphic to $X$,
and we will identify $\bar\uX_\top$ with $X$. If
$\uf=(f,f^\sh):\uX=(X,\O_X)\ra\uY=(Y,\O_Y)$ is a morphism of
$C^\iy$-schemes, so that $\buf:\bar\uX\ra\bar\uY$ is a 1-morphism of
$C^\iy$-stacks, then $\buf_\top:\bar\uX_\top\ra\bar\uY_\top$
is~$f:X\ra Y$.

For a $C^\iy$-stack $\cX$, we can characterize $\cX_\top$ by the
following universal property. We are given a topological space
$\cX_\top$ and for every 1-morphism $f:\bar\uU\ra\cX$ for a
$C^\iy$-scheme $\uU=(U,\O_U)$ we are given a continuous map
$f_\top:U\ra\cX_\top$, such that if $f$ is 2-isomorphic to
$h\ci\bar\ug$ for some morphism $\ug=(g,g^\sh):\uU\ra\uV$ and
1-morphism $h:\uV\ra\cX$ then $f_\top=h_\top\ci g$. If $\cX_\top'$,
$f_\top'$ are alternative choices of data with these properties then
there is a unique continuous map $j:\cX_\top\ra \cX_\top'$ with
$f_\top'=j\ci f_\top$ for all~$f$.
\label{ag6def8}
\end{dfn}

We can also make $\cX_\top$ into a $C^\iy$-ringed
space~$\ucX_\top$:\I{C-stack@$C^\iy$-stack!underlying $C^\iy$-ringed
space $\protect\ucX_\top$|(}

\begin{dfn} Let $\cX$ be a $C^\iy$-stack. Define a sheaf of $C^\iy$-rings $\O_{\cX_\top}$ on $\cX_\top$ as follows: each open set in $\cX_\top$ is $\cU_\top$ for some unique open $C^\iy$-substack $\cU\subseteq\cX$. Define $\O_{\cX_\top}(\cU_\top)$ to be the set of 2-isomorphism classes
$[c]$ of 1-morphisms $c:\cU\ra\bar{\ul{\R}}$. If $f:\R^n\ra\R$ is
smooth and $[c_1],\ldots,[c_n]\in\O_{\cX_\top}(\cU_\top)$, define $\Phi_f\bigl([c_1],\ldots,[c_n]\bigr)=\bigl[\buf\ci(c_1\t\cdots\t c_n)\bigr]$, using the composition $\cU\,{\buildrel c_1\t\cdots\t c_n\over\longra}\,\bar{\ul{\R}}\t\cdots\t\bar{\ul{\R}}\,{\smash{\buildrel\buf\over\longra}}\,\bar{\ul{\R}}$. Then $\O_{\cX_\top}(\cU_\top)$ is a $C^\iy$-ring.

If $\cV_\top\subseteq\cU_\top\subseteq\cX_\top$ are open, so that
$\cV\subseteq\cU\subseteq\cX$, define a $C^\iy$-ring morphism
$\rho_{\cU\cV}:\O_{\cX_\top}(\cU_\top)\ra\O_{\cX_\top}(\cV_\top)$
by $\rho_{\cU\cV}:[c]\mapsto[c\vert_\cV]$. It is now easy to check
that $\O_{\cX_\top}$ is a presheaf of $C^\iy$-rings on $\cX_\top$, but it is less obvious that it is a sheaf. To see this, note that by general properties of stacks, $\cU\mapsto\bs\Hom(\cU,\bar{\ul{\R}})$ is a 2-sheaf (stack) of groupoids on the topological space $\cX_\top$, where $\bs\Hom(\cU,\bar{\ul{\R}})$ is the groupoid of 1- and 2-morphisms $\cU\ra\bar{\ul{\R}}$, and $\O_{\cX_\top}(\cU_\top)$ is its set of isomorphism classes. 

Starting with a 2-sheaf and taking sets of isomorphism classes generally yields only a presheaf of sets, not a sheaf. But as $\bar{\ul{\R}}$ is a $C^\iy$-scheme the groupoids $\bs\Hom(\cU,\bar{\ul{\R}})$ are discrete (have no nontrivial automorphisms), so taking isomorphism classes loses no information, and the 2-sheaf property implies that $\O_{\cX_\top}$ is a sheaf of sets, and so of $C^\iy$-rings. Thus $\ucX_\top=(\cX_\top,\O_{\cX_\top})$\G[Xtopb]{$\ucX_\top$}{underlying $C^\iy$-ringed space or $C^\iy$-scheme of a $C^\iy$-stack $\cX$} is a $C^\iy$-ringed space, the {\it underlying $C^\iy$-ringed space\/} of~$\cX$.

For general $\cX$ this $\ucX_\top$ need not be a $C^\iy$-scheme. If
it is, we call $\ucX_\top$ the {\it coarse moduli\/
$C^\iy$-scheme\/}\I{C-stack@$C^\iy$-stack!coarse moduli
$C^\iy$-scheme $\protect\ucX_\top$} of $\cX$. Coarse moduli
$C^\iy$-schemes have the following universal property: there is a
1-morphism $\pi:\cX\ra\bucX_\top$ called the {\it structural
morphism}, such that if $f:\cX\ra\bar\uY$ is a 1-morphism for any
$C^\iy$-scheme $\uY$ then $f$ is 2-isomorphic to $\bar\ug\ci\pi$ for
some unique $C^\iy$-scheme
morphism~$\ug:\ucX_\top\ra\uY$.\I{C-stack@$C^\iy$-stack!underlying
$C^\iy$-ringed space $\protect\ucX_\top$|)}

\label{ag6def9}
\end{dfn}

We can think of a $C^\iy$-stack $\cX$ as being a topological space
$\cX_\top$ equipped with some complicated extra geometrical
structure, just as manifolds and orbifolds\I{orbifold} are usually
thought of as topological spaces equipped with extra structure
coming from an atlas of charts. As in Noohi \cite[Ex.~4.13]{Nooh},
it is easy to describe $\cX_\top$ using a
groupoid\I{C-stack@$C^\iy$-stack!associated to a groupoid}
presentation $[\uV\rra\uU]$ of~$\cX$:

\begin{prop} Let\/ $\cX$ be equivalent to the $C^\iy$-stack\/
$[\uV\rra\uU]$ associated to a groupoid\/ $(\uU,\uV,\us,\ut,\uu,\ui,\um)$
in\/ $\CSch,$ where\/ $\uU=(U,\O_U),\us=(s,s^\sh),$ and so on.
Define $\sim$ on\/ $U$ by\/ $p\sim p'$ if there exists $q\in V$
with\/ $s(q)=p$ and\/ $t(q)=p'$. Then $\sim$ is an equivalence
relation on $U,$ so we can form the quotient\/ $U/\!\sim,$ with the
quotient topology. There is a natural homeomorphism\/~$\cX_\top\cong
U/\!\sim$.

For a quotient\/ $C^\iy$-stack\/ $\cX\simeq[\uU/\uG]$ we
have\/~$\cX_\top\cong U/G$.
\label{ag6prop6}
\end{prop}

Using this we can deduce properties of $\cX_\top$ from properties of
$\cX$ expressed in terms of $\uV\rra\uU$. For instance, if $\cX$ is
separated\I{C-stack@$C^\iy$-stack!separated} then $s\times t:V\ra
U\t U$ is (universally) closed, and we can take $U$ Hausdorff. But
the quotient of a Hausdorff topological space by a closed
equivalence relation is Hausdorff, yielding:

\begin{lem} Let\/ $\cX$ be a separated\/ $C^\iy$-stack. Then the
underlying topological space $\cX_\top$ is Hausdorff.\I{topological
space!Hausdorff}
\label{ag6lem}
\end{lem}

Next we discuss {\it isotropy groups\/} of $C^\iy$-stacks.

\begin{dfn} Let $\cX$ be a $C^\iy$-stack, and $[x]\in\cX_\top$.
Pick a representative $x$ for $[x]$, so that $x:\bar{\ul{*}}\ra\cX$
is a 1-morphism. Then there exists a $C^\iy$-scheme $\uG=(G,\O_G)$,
unique up to isomorphism, with
$\bar\uG=\bar{\ul{*}}\t_{x,\cX,x}\bar{\ul{*}}$. Applying the
construction of the groupoid\I{groupoid object}\I{category!groupoid
object in} in Definition \ref{agAdef18} with $\Pi:U\ra\cX$ replaced
by $x:\bar{\ul{*}}\ra\cX$, we give $\uG$ the structure of a
$C^\iy$-{\it group}. The underlying group $G$ is canonically isomorphic to the group of 2-morphisms~$\eta:x\Ra x$.

With $[x]$ fixed, this $C^\iy$-group\I{C-group@$C^\iy$-group}%
\I{C-scheme@$C^\iy$-scheme!C-group@$C^\iy$-group} $\uG$ is
independent of choices up to noncanonical isomorphism; roughly,
$\uG$ is canonical up to conjugation in $\uG$. We define the {\it
isotropy group\/}\I{C-stack@$C^\iy$-stack!isotropy group
$\Iso_\cX([x])$} (or {\it orbifold group}, or {\it stabilizer
group\/}) $\Iso_\cX([x])$ or $\Iso([x])$ of $[x]$ to be this
$C^\iy$-group $\uG$, regarded as a $C^\iy$-group up to noncanonical
isomorphism.\I{C-stack@$C^\iy$-stack!orbifold group|see{isotropy group}}\I{C-stack@$C^\iy$-stack!stabilizer group|see{isotropy group}}

If $\cX=[\uV\rra\uU]$ is associated to a groupoid $(\uU,\uV,\us,\ut,\uu,\ui,\um)$ then $x:\bar{\ul{*}}\ra\cX$ factors through $\bar w:\bar{\ul{*}}\ra\bar\uU$ up to 2-isomorphism for some point $w\in\uU$, and then $\uG$ is isomorphic to the $C^\iy$-subscheme $\uG'=\us^{-1}(w)\cap\ut^{-1}(w)$ in $\uV$, with identity $\uu\vert_w:\ul{*}\ra\uG'$, inverse $\ui\vert_{\uG'}:\uG'\ra\uG'$, and multiplication~$\um\vert_{\uG'\t\uG'}:\uG'\t\uG'\ra\uG'$.

If $f:\cX\ra\cY$ is a 1-morphism of $C^\iy$-stacks and
$[x]\in\cX_\top$ with $f_\top([x])=[y]\in\cY_\top$, for $y=f\ci x$,
then at the level of sets we define $f_*:\Iso_\cX([x])\ra\Iso_\cY([y])$ by
$f_*(\eta)=\id_f*\eta$. This is a group morphism, by compatibility of horizontal and vertical composition in 2-categories. We can extend $f_*$ naturally to a morphism $\uf{}_*:\Iso_\cX([x])\ra\Iso_\cY([y])$ of $C^\iy$-groups, such that 
\begin{equation*}
\bar\uf{}_*:\,\,\,\,\ov{\!\!\!\!\Iso_\cX([x])\!\!\!\!}\,\,\,\,=
\bar{\ul{*}}\t_{x,\cX,x}\bar{\ul{*}}\longra\bar{\ul{*}}\t_{f\ci x,\cY,f\ci x}\bar{\ul{*}}=\,\,\,\,\ov{\!\!\!\!\Iso_\cY([y])\!\!\!\!}\,\,\,\,
\end{equation*}
is induced from $f:\cX\ra\cY$ by the universal property of fibre products. Then $f_*,\uf_*$ are independent of the choice of $x\in[x]$ up to conjugation
in~$\Iso_\cY([y])$.\I{C-stack@$C^\iy$-stack!underlying topological
space $\cX_\top$|)}
\label{ag6def10}
\end{dfn}

\subsection{\texorpdfstring{Gluing $C^\iy$-stacks by equivalences}{Gluing C∞-stacks by equivalences}}
\label{ag65}
\I{C-stack@$C^\iy$-stack!gluing by equivalences|(}

Here are two propositions on gluing $C^\iy$-stacks by equivalences.
They are exercises in stack theory, with no special $C^\iy$ issues,
and also hold for other classes of stacks. See Rydh
\cite[Th.~C]{Rydh} for stronger results for algebraic stacks.

\begin{prop} Suppose $\cX,\cY$ are\/ $C^\iy$-stacks,
$\cU\subseteq\cX,$ $\cV\subseteq\cY$ are open $C^\iy$-substacks,
and\/ $f:\cU\ra\cV$ is an equivalence in $\CSta$. Then there exist
a\/ $C^\iy$-stack\/ $\cZ,$ open $C^\iy$-substacks $\hcX,\hat\cY$ in
$\cZ$ with\/ $\cZ=\hcX\cup\hat\cY,$ equivalences $g:\cX\ra\hcX$
and\/ $h:\cY\ra\hat\cY$ such that\/ $g\vert_\cU$ and\/ $h\vert_\cV$
are both equivalences with\/ $\hat \cX\cap\hat\cY,$ and a
$2$-morphism $\eta:g\vert_\cU\Ra h\ci f:\cU\ra\hcX\cap\hat\cY$ in
$\CSta$. Furthermore, $\cZ$ is independent of choices up to
equivalence.
\label{ag6prop7}
\end{prop}

\begin{prop} Suppose $\cX,\cY$ are $C^\iy$-stacks,
$\cU,\cV\subseteq\cX$ are open $C^\iy$-substacks with\/
$\cX=\cU\cup\cV,$ $f:\cU\ra\cY$ and\/ $g:\cV\ra\cY$ are
$1$-morphisms, and\/ $\eta:f\vert_{\cU\cap\cV}\Ra
g\vert_{\cU\cap\cV}$ is a $2$-morphism in $\CSta$. Then there exists
a $1$-morphism $h:\cX\ra\cY$ and\/ $2$-morphisms $\ze:h\vert_\cU\Ra
f,$ $\th:h\vert_\cV\Ra g$ such that\/ $\th\vert_{\cU\cap\cV}=
\eta\od\ze\vert_{\cU\cap\cV}:h\vert_{\cU\cap\cV}\Ra
g\vert_{\cU\cap\cV}$.\I{2-category!2-morphism!vertical composition}
This $h$ is unique up to $2$-isomorphism.

In general, $h$ is \begin{bfseries}not\end{bfseries} independent up
to $2$-isomorphism of the choice of\/~$\eta$.
\label{ag6prop8}
\end{prop}

Here is an example in which $h$ is not independent of $\eta$ up to
2-isomorphism in the last part of Proposition~\ref{ag6prop8}.

\begin{ex} Let $\cX$ be the $C^\iy$-stack associated to the circle
$X=\bigl\{(x,y)\in\R^2:x^2+y^2=1\bigr\}$, and $\cU,\cV\subseteq\cX$
the substacks associated to the open sets $U=\bigl\{(x,y)\in
X:x>-\ha\bigr\}$ and $V=\bigl\{(x,y)\in X:x<\ha\bigr\}$. Let $\cY$
be the quotient $C^\iy$-stack\I{C-stack@$C^\iy$-stack!quotients
$[\protect\uX/G]$}\I{quotient C-stack@quotient $C^\iy$-stack}
$[\ul{*}/\Z_2]$. Then 1-morphisms $f:\cX\ra\cY$ correspond to
principal $\Z_2$-bundles $P_f\ra X$, and for 1-morphisms
$f,g:\cX\ra\cY$ with principal $\Z_2$-bundles $P_f,P_g\ra X$, a
2-morphism $\eta:f\Ra g$ corresponds to an isomorphism of principal
$\Z_2$-bundles $P_f\cong P_g$. The same holds for 1-morphisms
$\cU,\cV,\cU\cup\cV\ra\cY$ and their 2-morphisms.

Let $f:\cU\ra\cY$ and $g:\cV\ra\cY$ be the 1-morphisms corresponding
to the trivial $\Z_2$-bundles $P_f=\Z_2\t U\ra U$, $P_g=\Z_2\t V\ra
V$. Then 2-morphisms $\eta:f\vert_{\cU\cap\cV}\Ra
g\vert_{\cU\cap\cV}$ correspond to automorphisms of the trivial
$\Z_2$-bundle $\Z_2\t(U\cap V)\ra U\cap V$, that is, to continuous
maps $U\cap V\ra\Z_2$. Note that $U\cap V$ has two connected
components $\bigl\{(x,y)\in X:-\ha<x<\ha$, $y>0\bigr\}$ and
$\bigl\{(x,y)\in X:-\ha<x<\ha$, $y<0\bigr\}$.

Define 2-morphisms $\eta_1,\eta_2:f\vert_{\cU\cap\cV}\Ra
g\vert_{\cU\cap\cV}$ such that $\eta_1$ corresponds to the map
$1:(U\cap V)\ra\Z_2=\{\pm 1\}$, and $\eta_1$ corresponds to the map
${\mathop{\rm sign}}(y):(U\cap V)\ra\Z_2=\{\pm 1\}$. Then
Proposition \ref{ag6prop8} gives 1-morphisms $h_1,h_2:\cX\ra\cY$
from $\eta_1,\eta_2$. The associated principal $\Z_2$-bundles
$P_{h_1},P_{h_2}$ over $X$ come from gluing $P_f,P_g$ over $U,V$
using the transition functions $1,{\mathop{\rm sign}}(y)$. Therefore
$P_{h_1}$ is the trivial $\Z_2$-bundle over $X={\cal S}^1$, and
$P_{h_2}$ the nontrivial $\Z_2$-bundle. Hence $P_{h_1},P_{h_2}$ are
not isomorphic as principal $\Z_2$-bundles, and $h_1,h_2$ are not
2-isomorphic. Hence in this example, $h$ is not independent up to
2-isomorphism of the choice
of~$\eta$.\I{C-stack@$C^\iy$-stack!gluing by equivalences|)}
\label{ag6ex2}
\end{ex}

\section{\texorpdfstring{Deligne--Mumford $C^\iy$-stacks}{Deligne--Mumford C∞-stacks}}
\label{ag7}
\I{C-stack@$C^\iy$-stack!Deligne--Mumford|see{Deligne-- \\ Mumford
$C^\iy$-stack}}%
\I{Deligne--Mumford $C^\iy$-stack|(}

We now introduce {\it Deligne--Mumford\/ $C^\iy$-stacks}, which are
$C^\iy$-stacks locally modelled on quotients $[\uU/G]$ for $\uU$ an
affine $C^\iy$-scheme and $G$ a finite group. As we explain in
\S\ref{ag76}, orbifolds\I{orbifold} may be defined as a
2-subcategory of Deligne--Mumford $C^\iy$-stacks.

\subsection{\texorpdfstring{Quotient $C^\iy$-stacks, 1-morphisms, and 2-morphisms}{Quotient C∞-stacks, 1-morphisms, and 2-morphisms}}
\label{ag71}
\I{C-stack@$C^\iy$-stack!quotients $[\protect\uX/G]$|(}\I{quotient
C-stack@quotient $C^\iy$-stack|(}

When a $C^\iy$-group\I{C-group@$C^\iy$-group}%
\I{C-scheme@$C^\iy$-scheme!C-group@$C^\iy$-group} $\uG$ acts on a
$C^\iy$-scheme $\uX$, Definition \ref{ag6def3} gives the quotient
$C^\iy$-stack $[\uX/\uG]$. It is a stack associated to a groupoid\I{stack!associated to a groupoid}\I{C-stack@$C^\iy$-stack!associated to a groupoid} $[\uG\t\uX\rra\uX]$ from Definition \ref{agAdef19}, which is the stackification\I{prestack!stackification} of a certain prestack. By
Proposition \ref{agAprop1}, stackifications always exist, and are
unique up to equivalence. Thus, Definition \ref{ag6def3} actually
only specifies $[\uX/\uG]$ up to equivalence in~$\CSta$.

When a finite group $G$ acts on a $C^\iy$-scheme $\uX$, we will now
define an explicit $C^\iy$-stack $[\uX/G]$, which is in the
equivalence class of $[\uX/\uG]$ in Definition \ref{ag6def3} for
$\uG=F_\Man^\CSch(G)$. These quotient $C^\iy$-stacks $[\uX/G]$ (for
$\uX$ Hausdorff\I{C-scheme@$C^\iy$-scheme!Hausdorff}) will be our
local models for defining Deligne--Mumford $C^\iy$-stacks
in~\S\ref{ag72}.

We will also define {\it quotient\/ $1$-morphisms\/}
$[\uf,\rho]:[\uX/G]\ra[\uY/H]$ of quotient $C^\iy$-stacks
$[\uX/G],[\uY/H]$ when $\rho:G\ra H$ is a group morphism and
$\uf:\uX\ra\uY$ a $\rho$-equivariant $C^\iy$-morphism, and {\it
quotient\/ $2$-morphisms\/} $[\de]:[\uf,\rho]\Ra [\ug,\si]$ for quotient 1-morphisms
$[\uf,\rho],[\ug,\si]:[\uX/G]\ra[\uY/H]$, when $\de\in H$ with
$\si(\ga)=\de\,\rho(\ga)\,\de^{-1}$ for all $\ga\in G$, and
$\ug=\de\cdot\uf$. We will see in \S\ref{ag74} that all 1- and
2-morphisms of Deligne--Mumford $C^\iy$-stacks are locally modelled
on quotient 1- and 2-morphisms.

\begin{dfn} Let $\uX$ be a $C^\iy$-scheme, $G$ a finite group, and
$\ur:G\ra\Aut(\uX)$ an action of $G$ on $\uX$ by isomorphisms. We
will define the {\it quotient $C^\iy$-stack\/}
\I{C-stack@$C^\iy$-stack!quotients
$[\protect\uX/G]$!definition}\I{quotient C-stack@quotient
$C^\iy$-stack!definition} $\cX=[\uX/G]$,\G[XG]{$[\uX/G]$}{quotient
$C^\iy$-stack} generalizing the description of $[\ul{*}/\uG]$ in Example \ref{ag6ex2}. It is a well known construction, as in Behrend et al.\ \cite[Ex.~2.6]{BEFF} and Noohi~\cite[Ex.~12.4]{Nooh}.

Define a category $\cX$ to have objects triples $(\uS,\uP,\ulp)$ where $\uS$ is a $C^\iy$-scheme, and $\uP$ is a principal $G$-bundle over $\uS$ in the sense of Example \ref{ag6ex1} (or $(\uP,\upi,\umu)$ rather than $\uP$, but we leave $\upi:\uP\ra\uS$ and the $G$-action $\umu$ implicit), and $\ulp:\uP\ra\uX$ is a $G$-equivariant morphism. Define morphisms $(\um,\uu):(\uS,\uP,\ulp)\ra(\uT,\uQ,\uq)$ in $\cX$ to be $C^\iy$-scheme morphisms $\um:\uS\ra\uT$ and $\uu:\uP\ra\uQ$, such that $\uu$ is $\uG$-equivariant, and \eq{ag6eq2} is a Cartesian square in $\CSch$, and $\ulp=\uq\ci\uu:\uP\ra\uX$. Composition of morphisms is $(\un,\uv)\ci(\um,\uu)=(\un\ci\um,\uv\ci\uu)$, and identity morphisms are $\id_{(\uS,\uP,\ulp)}=(\uid_\uS,\uid_\uP)$. The functor $p_\cX:\cX\ra\CSch$ maps $p_\cX:(\uS,\uP,\ulp)\mapsto\uS$ on objects and $p_\cX:(\um,\uu)\mapsto\um$ on morphisms.

Then $\cX$ is a $C^\iy$-stack, which we write as $[\uX/G]$. It is equivalent in $\CSta$ to the quotient $C^\iy$-stack $[\uX/\uG]$ in Definition \ref{ag6def3} for $\uG=F_\Man^\CSch(G)$. To see this, note that by Definition \ref{agAdef19} $[\uX/\uG]$ is the stackification of a prestack $p_{\cX'}:\cX'\ra\CSch$, where $\cX'$ may be written as the category whose objects are pairs $(\uS,\ulp')$ of a $C^\iy$-scheme $\uS$ and a morphism $\ulp':\uS\ra\uX$, and whose morphisms $(\um,\uu'):(\uS,\ulp')\ra(\uT,\uq')$ consist of morphisms $\um:\uS\ra\uT$  and $\uu':\uS\ra\uG$ with $\ulp'=\uq'\ci\um$, with composition $(\un,\uv')\ci(\um,\uu')=(\un\ci\um,\uu'\cdot (\uv'\ci\um))$, and $p_{\cX'}$ maps $p_{\cX'}:(\uS,\ulp')\mapsto\uS$ and~$p_{\cX'}:(\um,\uu')\mapsto\uu'$.

We may identify $\cX'$ with the full subcategory of $\cX$ with objects $(\uS,G\t\uS,\ulp)$ in which $\uP$ is the trivial principal $G$-bundle $G\t\uS\ra\uS$, where $(\uS,\ulp')$ in $\cX'$ is identified with $(\uS,G\t\uS,\ulp)$ in $\cX$  for $\ulp'=\ulp\vert_{\uS\t\{1\}}:\uS\cong\uS\t\{1\}\ra\uX$, and $(\um,\uu'):(\uS,\ulp')\ra(\uT,\uq')$ in $\cX'$ is identified with $(\um,\uu):(\uS,G\t\uS,\ulp)\ra(\uT,G\t\uT,\uq)$ in $\cX$, where $\uu:G\t\uS\ra G\t\uT$ maps $(\ga,s)\mapsto (\ga\cdot\uu'(s),\um(s))$. Stackifying $\cX'$ enlarges from trivial principal $G$-bundles to all principal $G$-bundles.

Define a functor $\pi_{[\uX/G]}:\ul{\bar X\!}\,\ra[\uX/G]$ by
$\pi_{[\uX/G]}:(\uS,\ulp')\mapsto(\uS,G\t\uS,\ulp)$ on objects, where $\ulp:G\t\uS\ra\uX$ is the unique $G$-equivariant morphism with $\ulp'=\ulp\vert_{\uS\t\{1\}}:\uS\cong\uS\t\{1\}\ra\uX$, and $\pi_{[\uX/G]}:\um\mapsto(\um,\id_G\t\um)$ on morphisms. Then $\pi_{[\uX/G]}:\ul{\bar X\!}\,\ra[\uX/G]$ is a representable 1-morphism,\I{C-stack@$C^\iy$-stack!representable 1-morphism} and makes $\ul{\bar X\!}\,$ into a principal $G$-bundle over~$[\uX/G]$.
\label{ag7def1}
\end{dfn}

\begin{dfn} Let $\uX,\uY$ be $C^\iy$-schemes acted on by finite
groups $G,H$ with actions $\ur:G\ra\Aut(\uX)$, $\us:H\ra\Aut(\uY)$,
so that we have quotient $C^\iy$-stacks $\cX=[\uX/G]$ and
$\cY=[\uY/H]$ as in Definition \ref{ag7def1}. Suppose we have
morphisms $\uf:\uX\ra\uY$ of $C^\iy$-schemes and $\rho:G\ra H$ of
groups, with $\uf\ci \ur(\ga)=\us(\rho(\ga))\ci\uf$ for all $\ga\in
G$. We will define a {\it quotient\/
$1$-morphism\/}\I{C-stack@$C^\iy$-stack!quotients
$[\protect\uX/G]$!quotient 1-morphism|(}\I{quotient C-stack@quotient
$C^\iy$-stack!1-morphism|(}
$[\uf,\rho]:\cX\ra\cY$.\G[frho]{$[\uf,\rho]:[\uX/G]\ra[\uY/H]$}{quotient
1-morphism of quotient $C^\iy$-stacks} 

Define a functor $[\uf,\rho]:\cX\ra\cY$ by $[\uf,\rho]:(\uS,\uP,\ulp)\mapsto(\uS,\ti\uP,\ti\ulp)$ on objects $(\uS,\uP,\ulp)$ in $\cX$, where $\ti\uP=(H\t\uP)/_\rho G$ is the principal $H$-bundle on $\uS$ constructed from $\uP$ and $\rho:G\ra H$, and $\ti\ulp:\ti\uP\ra\uY$ is the $H$-equivariant $C^\iy$-scheme morphism induced from the $\rho$-equivariant morphism $\uf\ci\ulp:\uP\ra\uY$, which acts on points by $\ti\ulp:(h,p)G\mapsto h\cdot \uf\ci\ulp(p)$. Define $[\uf,\rho]:(\um,\uu)\mapsto(\um,\ti\uu)$ on morphisms $(\um,\uu):(\uS,\uP,\ulp)\ra(\uT,\uQ,\uq)$ in $\cX$, where $\ti\uu:(H\t\uP)/_\rho G\ra(H\t\uQ)/_\rho G$ is induced by $\id_H\t\uu:H\t\uP\ra H\t\uQ$. Then
$[\uf,\rho]:\cX\ra\cY$ is a functor, with $p_\cX=p_\cY\ci[\uf,\rho]$, so $[\uf,\rho]$ is a 1-morphism of $C^\iy$-stacks, which we write as~$[\uf,\rho]:[\uX/G]\ra[\uY/H]$.

It is easy to check that $[\uf,\rho]\ci\pi_{[\uX/G]}\cong
\pi_{[\uY/H]}\ci\ul{\bar f\!}\,$, and if $[\uf,\rho]:[\uX/G]
\ra[\uY/H]$, $[\ug,\si]:[\uY/H]\ra[\uZ/I]$ are quotient 1-morphisms
then there is a canonical 2-isomorphism $[\ug,\si]\ci[\uf,\rho]\cong[\ug\ci\uf,
\si\ci\rho]$ coming from the canonical $C^\iy$-scheme isomorphisms~$\bigl(I\t ((H\t\uP)/_\rho G)\bigr)/_\si H\cong (I\t\uP)/_{\si\ci\rho}G$.
\label{ag7def2}
\end{dfn}

\begin{dfn} Let $[\uf,\rho]:[\uX/G]\ra[\uY/H]$ and
$[\ug,\si]:[\uX/G]\ra[\uY/H]$ be quotient 1-morphisms, so that
$\uf,\ug:\uX\ra\uY$ and $\rho,\si:G\ra H$ are morphisms. Suppose
$\de\in H$ satisfies $\si(\ga)=\de\,\rho(\ga)\,\de^{-1}$ for all
$\ga\in G$, and $\ug=\us(\de)\ci\uf$. We will define a 2-morphism
$[\de]:[\uf,\rho]\Ra [\ug,\si]$,\G[de]{$[\de]:[\uf,\rho]\Ra
[\ug,\si]$}{quotient 2-morphism of quotient 1-morphisms} which we
call a {\it quotient\/
$2$-morphism}.\I{C-stack@$C^\iy$-stack!quotients
$[\protect\uX/G]$!quotient 2-morphism|(}\I{quotient C-stack@quotient
$C^\iy$-stack!2-morphism|(}

Here $[\de]$ must be a natural isomorphism of functors
$[\uf,\rho]\Ra[\ug,\si]$. Let $(\uS,\uP,\ulp)$ be an
object in $[\uX/G]$. Define an isomorphism in $[\uY/H]$:
\begin{align*}
&[\de]\bigl((\uS,\uP,\ulp)\bigr)=\bigl(\uid_\uS,(r_{\de^{-1}}\t\id_\uP)_*\bigr):
[\uf,\rho]\bigl((\uS,\uP,\ulp)\bigr)=\\
&\qquad(\uS,(H\t\uP)/_\rho G,\ti\ulp)\longra
[\ug,\si]\bigl((\uS,\uP,\ulp)\bigr)=
(\uS,(H\t\uP)/_\si G,\ti\ulp),
\end{align*}
where $r_{\de^{-1}}:H\ra H$ maps $\ep\mapsto\ep\de^{-1}$, and $r_{\de^{-1}}\t\id_\uP:H\t\uP\ra H\t\uP$ is equivariant under the actions of $G$ on $H\t\uP$ induced by $\rho$ on the domain and $\si$ on the target, so that it descends to an isomorphism $(r_{\de^{-1}}\t\id_\uP)_*:(H\t\uP)/_\rho G\ra (H\t\uP)/_\si G$. It is now easy to check that $[\de]\bigl((\uS,\uP,\ulp)\bigr)$ is an isomorphism in $[\uY/H]$, and $[\de]$ is a
natural isomorphism of functors, and a 2-morphism $[\de]:[\uf,\rho]\Ra [\ug,\si]$ in $\CSta$. Quotient 2-morphisms have functorial properties under horizontal and vertical composition. For instance, if $[\uf,\rho],[\ug,\si],[\uh,\tau]:[\uX/G]\ra[\uY/H]$ are quotient 1-morphisms and $[\de]:[\uf,\rho]\Ra [\ug,\si]$,
$[\ep]:[\ug,\si]\Ra [\uh,\tau]$ are quotient 2-morphisms
then~$[\ep]\od[\de]=[\ep\de]:[\uf,\rho]\Ra[\uh,\tau]$.%
\I{2-category!2-morphism!vertical composition}
\label{ag7def3}
\end{dfn}

\begin{rem} Studying $C^\iy$-stacks $[\uX/G]$ and their 1- and 2-morphisms is a good way to develop geometric intuition about Deligne--Mumford $C^\iy$-stacks (including
orbifolds)\I{orbifold} and their 1- and 2-morphisms. If $[\uX/G],[\uY/H]$ are quotient $C^\iy$-stacks, then general 1-morphisms $f:[\uX/G]\ra[\uY/H]$ in $\CSta$ need not
be quotient 1-morphisms $[\uf,\rho]$, or even 2-isomorphic to
$[\uf,\rho]$. But Theorem \ref{ag7thm6}(b) says that $f\cong[\uf,\rho]$ locally in $[\uX/G]$. If $[\uf,\rho],[\ug,\si]:[\uX/G]\ra [\uY/H]$ are quotient 1-morphisms, and $[\uX/G]$ is connected, then Proposition \ref{ag7prop5} says that all 2-morphisms $\eta:[\uf,\rho]\Ra[\ug,\si]$ are quotient 2-morphisms $[\de]:[\uf,\rho]\Ra [\ug,\si]$.%
\I{C-stack@$C^\iy$-stack!quotients $[\protect\uX/G]$!quotient
1-morphism|)}\I{quotient
C-stack@quotient $C^\iy$-stack!1-morphism|)}%
\I{C-stack@$C^\iy$-stack!quotients $[\protect\uX/G]$!quotient
2-morphism|)}\I{quotient C-stack@quotient
$C^\iy$-stack!2-morphism|)}

\label{ag7rem1}
\end{rem}

\subsection{\texorpdfstring{Deligne--Mumford $C^\iy$-stacks}{Deligne--Mumford C∞-stacks}}
\label{ag72}

{\it Deligne--Mumford stacks\/} in algebraic geometry were
introduced in \cite{DeMu} to study moduli spaces of algebraic
curves. As in \cite[Th.~6.2]{LaMo}, Deligne--Mumford stacks are
locally modelled (in the \'etale topology, at least, but with
isomorphisms of isotropy groups) on quotient $C^\iy$-stacks $[X/G]$
for $X$ an affine scheme and $G$ a finite group. This motivates:

\begin{dfn} A {\it Deligne--Mumford\/
$C^\iy$-stack\/}\I{Deligne--Mumford $C^\iy$-stack!definition} is a
$C^\iy$-stack $\cX$ which admits a (Zariski) open cover
$\{\cU_a:a\in A\}$, as in Definition \ref{ag6def7}, with each
$\cU_a$ equivalent to a quotient $C^\iy$-stack $[\uU_a/G_a]$ in
Definition \ref{ag7def1} for $\uU_a$ an affine $C^\iy$-scheme and
$G_a$ a finite group. We call $\cX$ {\it locally
fair},\I{Deligne--Mumford $C^\iy$-stack!locally fair} or {\it
locally finitely presented},\I{Deligne--Mumford
$C^\iy$-stack!locally finitely presented} or {\it
locally Lindel\"of},\I{Deligne--Mumford
$C^\iy$-stack!locally Lindel\"of} if it admits such an open
cover with each $\uU_a$ a fair,\I{C-scheme@$C^\iy$-scheme!locally
fair} or finitely presented,\I{C-scheme@$C^\iy$-scheme!locally
finitely presented} or Lindel\"of, affine $C^\iy$-scheme, respectively. We call $\cX$ {\it second countable}\I{Deligne--Mumford $C^\iy$-stack!second countable} if the underlying topological
space\I{C-stack@$C^\iy$-stack!underlying topological space
$\cX_\top$} $\cX_\top$ is second countable. 

Write $\DMCStalf,\DMCStalfp$\G[DMCStalf]{$\DMCStalf$}{2-category of
locally fair Deligne--Mumford
$C^\iy$-stacks}\G[DMCStalfp]{$\DMCStalfp$}{2-category of locally
finitely presented Deligne--Mumford $C^\iy$-stacks} and
$\DMCSta$\G[DMCSta]{$\DMCSta$}{2-category of Deligne--Mumford
$C^\iy$-stacks} for the full 2-sub\-cat\-ego\-ries of locally fair,
locally finitely presented, and all, Deligne--Mumford $C^\iy$-stacks
in $\CSta$, respectively.

The functor $F_\CSch^\CSta:\CSch\ra\CSta$ in Definition
\ref{ag6def2} maps into $\DMCSta\subset\CSta$, so the 2-categories
$\bCSchlf,\ab\bCSchlfp,\ab\bCSch$ are 2-subcategories of
$\DMCStalf,\DMCStalfp,\DMCSta$, respectively. If a $C^\iy$-stack
$\cX$ is a $C^\iy$-scheme,\I{C-stack@$C^\iy$-stack!is a
$C^\iy$-scheme} then it is a Deligne--Mumford $C^\iy$-stack.
\label{ag7def4}
\end{dfn}

\begin{prop} $\DMCStalf,\DMCStalfp,\DMCSta$ are closed under taking
open\/ $C^\iy$-substacks in\/~$\CSta$.%
\I{C-stack@$C^\iy$-stack!C-substack@$C^\iy$-substack!open}%
\I{C-substack@$C^\iy$-substack!open}%
\label{ag7prop1}
\end{prop}

\begin{proof} Let $\cX$ lie in one of these 2-categories, and
$\cX'$ be an open $C^\iy$-substack of $\cX$. Then $\cX$ admits an
open cover $\{\cU_a:a\in A\}$ with $\cU_a\simeq[\uU_a/G_a]$ with
$\uU_a$ affine and $G_a$ finite, and $\{\cU_a':a\in A\}$ is an open
cover of $\cX'$, where $\cU_a'=\cU_a\t_\cX\cX'$ is an open
$C^\iy$-substack of $\cU_a$. Thus $\cU_a'\simeq[\uU_a'/G_a]$ by
Proposition \ref{ag6prop5}, where $\uU_a'$ is a $G_a$-invariant open
$C^\iy$-subscheme of $\uU_a$. If the $\uU_a$ are fair, or finitely
presented then the $\uU_a'$ are too by Corollary \ref{ag4cor3}(a).
Thus $\DMCStalf,\DMCStalfp$ are closed under open subsets.

For $\DMCSta$, as open subsets of affine $C^\iy$-schemes need not be
affine, the $\uU_a'$ need not be affine. We will show that we can
cover $\uU_a'$ by $G_a$-invariant open affine $C^\iy$-subschemes
$\uU_{au}'$. Write $\uU_a'=(U_a',\O_{\uU_a'})$ and
$G_a=(G_a,\O_{G_a})$. Then the finite group $G_a$ acts continuously
on $U_a'$. Let $u\in U_a'$, and $H_u=\{\ga\in G_a:\ga u=u\}$ be the
stabilizer of $u$ in $G_a$. Then the orbit $\{\ga u:\ga\in G\}\cong
G_a/H_u$ of $u$ is a finite set, so as $U_a'$ is Hausdorff we can
choose affine open neighbourhoods $V_{\ga u}$ of $\ga u$ for each
point in the orbit such that $V_{\ga u}\cap V_{\ga'u}=\es$ if $\ga
u\ne\ga'u$. Define $W_u=\bigcap_{\ga\in G}\ga^{-1}V_{\ga u}$. Then
$W_u$ is an $H_u$-invariant open neighbourhood of $u$ in $U_a'$, and
if $\ga\in G_a\sm H_u$ then~$\ga W_u\cap W_u=\es$.

As in Corollary \ref{ag4cor2} we can choose an affine open neighbourhood
$W_u'$ of $u$ in $W_u$. Define $W_u''=\bigcap_{\ga\in H_u}W_u'$, an
$H_u$-invariant open neighbourhood of $u$ in $W_u$. This a finite
intersection of affine open $C^\iy$-subschemes $\uW_u'$ in the
affine $C^\iy$-scheme $\uV_u$, and so is affine, since intersection
is a kind of fibre product, and $\ACSch$ is closed under fibre
products by Theorem \ref{ag4thm3}(a). Define $U_{au}'=\bigcup_{\ga\in
G_a}W_u''$. Then $U_{au}'$ is a $G_a$-invariant open neighbourhood
of $u$ in $U_a'$. Since $W_u''$ is $H_u$-invariant and $\ga
W_u''\cap W_u''=\es$ if $\ga\in G_a\sm H_u$, we see that $U_{au}'$
is isomorphic to the disjoint union of $\md{G_a}/\md{H_u}$ copies of
$W_u''$. Hence $\uU_{au}'=(U_{au}',\O_{\uU_a'}\vert_{U_{au}'})$ is a
finite disjoint union of affine $C^\iy$-schemes, and is an affine
$C^\iy$-scheme. Therefore we may cover $\uU_a'$ by $G_a$-invariant
open affine $C^\iy$-subschemes $\uU_{au}'$. Using these we obtain an
open cover $\bigl\{\cU'_{au}:a\in A$, $u\in U_a\bigr\}$ of $\cX'$
with $\cU'_{au}\simeq[\uU_{au}'/G_a]$, so $\cX'$ is
Deligne--Mumford.
\end{proof}

The proof of Proposition \ref{ag7prop1} only uses $\uU_a=(U_a,
\O_{U_a})$ a $C^\iy$-scheme and $U_a$ Hausdorff, it does not need
$\uU_a$ to be affine. So the same proof yields:

\begin{prop} Any\/ $C^\iy$-stack of the form\/ $[\uX/G]$ in\/
{\rm\S\ref{ag71}} with\/ $\uX$ a Hausdorff\/ $C^\iy$-scheme and\/
$G$ finite is a separated\I{C-stack@$C^\iy$-stack!separated}
Deligne--Mumford\/ $C^\iy$-stack.
\label{ag7prop2}
\end{prop}

However, if $\uX$ is not Hausdorff then $[\uX/G]$ need not be
Deligne--Mumford:

\begin{ex} Let $\uX$ be the non-Hausdorff $C^\iy$-scheme
$(\ul{\R}\amalg\ul{\R})/\sim$, where $\sim$ is the equivalence
relation which identifies the two copies of $\ul{\R}$ on
$\,\ul{\!(0,\iy)\!}\,$. Let $G=\Z_2$ act on $\uX$ by exchanging the
two copies of $\ul{\R}$. Let $\cX$ be the quotient $C^\iy$-stack
$[\uX/G]$. We can think of $\cX$ as a like copy of $\R$, where the
stabilizer group of $x\in\R$ is $\{1\}$ if $x\in(-\iy,0]$ and $\Z_2$
if $x\in(0,\iy)$. Using the obvious
atlas\I{atlas}\I{C-stack@$C^\iy$-stack!atlas}
$\Pi:\bar{\ul{\R}}\ra\cX$, the third diagram of \eq{agAeq12} yields a
2-Cartesian square\I{2-category!2-Cartesian square}
\begin{equation*}
\xymatrix@C=50pt@R=14pt{ *+[r]{\bar{\ul{\R}}\amalg\,\,\overline{\!
\ul{\!(0,\iy)\!}\!}\,\,} \ar[rr] \ar[d]
\drrtwocell_{}\omit^{}\omit{^{}} && *+[l]{I_\cX} \ar[d]_{\io_\cX} \\
*+[r]{\bar{\ul{\R}}} \ar[rr]^(0.7){\Pi} && *+[l]{\cX.\!{}} }
\end{equation*}
As the left hand column is not proper, $\io_\cX$ is not proper, so
$\cX=[\uX/G]$ is not Deligne--Mumford by Corollary \ref{ag7cor1}
below.
\label{ag7ex1}
\end{ex}

We show that the 2-subcategory of quotient $C^\iy$-stacks $[\uX/G]$
in $\CSta$ is closed under fibre
products:\I{C-stack@$C^\iy$-stack!fibre products}

\begin{prop} Suppose $g:[\uX/F]\ra[\uZ/H],$ $h:[\uY/G]\ra[\uZ/H]$
are $1$-morphisms of quotient\/ $C^\iy$-stacks, where $\uX,\uY,\uZ$
are $C^\iy$-schemes and\/ $F,G,H$ are finite groups. Then we have a
$2$-Cartesian square
\e
\begin{gathered}
\xymatrix@C=100pt@R=14pt{ *+[r]{[\uW/(F\t G)]} \ar[r]_(0.6){f}
\ar[d]^{e} \drtwocell_{}\omit^{}\omit{^{}} & *+[l]{[\uY/G]} \ar[d]_h \\
*+[r]{[\uX/F]} \ar[r]^(0.4)g & *+[l]{[\uZ/H],\!\!} }
\end{gathered}
\label{ag7eq1}
\e
where $\Pi_\uX:\bar\uX\ra[\uX/F],$ $\Pi_\uY:\bar\uY
\ra[\uY/G]$ are the natural atlases\I{atlas}\I{C-stack@$C^\iy$-stack!atlas} and\/ $\bar\uW=\bar\uX\t_{g\ci\Pi_\uX,[\uZ/H],h\ci\Pi_\uY}\bar \uY$. If\/ $\uX,\uY,\uZ$ are Hausdorff, or locally fair, or locally finitely
presented,\I{C-scheme@$C^\iy$-scheme!locally finitely presented}
then $\uW$ is too.
\label{ag7prop3}
\end{prop}

\begin{proof} Write $\cW=[\uX/F]\t_{[\uZ/H]}[\uY/G]$.
Then from the atlases $\Pi_\uX,\Pi_\uY$, Example \ref{agAex}
constructs an atlas\I{atlas}\I{C-stack@$C^\iy$-stack!atlas}
$\Pi_\uW:\bar\uW\ra\cW$ for $\cW$. Since
$[\uX/F]\simeq[F\t\uX\rra\uX]$ and $[\uY/G]\simeq[G\t\uY\rra\uY]$ it
follows from \eq{agAeq14} that $\cW$ is equivalent to the stack associated to the groupoid $[(F\t G)\t\uW\rra\uW]$ for a natural action of $F\t G$ on $\uW$. This proves~\eq{ag7eq1}.

If $\uX,\uY,\uZ$ are Hausdorff then $[\uZ/H]$ is Deligne--Mumford by
Proposition \ref{ag7prop2}, so $\De_{[\uZ/H]}$ is separated by
Corollary \ref{ag7cor1} below, and thus $\uW$ is Hausdorff as
$\uX,\uY$ are and $\bar\uW\cong (\bar\uX\t\bar\uY)\t_{[\uZ/H]\t
[\uZ/H], \De_{[\uZ/H]}}[\uZ/H]$. Form the diagram
\begin{equation*}
\xymatrix@!0@R=19pt@C=70pt{ & \bar\uW' \ar[dl]_{\pi_\uW} \ar@<-1ex>[ddr]
\ar[rr] && \bar\uY' \ar[dl]^(0.3){\pi_\uY} \ar[ddr] \\ \bar\uW
\ar[ddr] \ar[rr] && \bar\uY \ar@<1ex>[ddr]^(0.4){h\ci\Pi_\uY} \\ &&
\bar\uX'
\ar[rr] \ar[dl]_(0.6){\pi_\uX} && \bar\uZ \ar[dl]^(0.4){\Pi_\uZ} \\
& \uX \ar[rr]^(0.7){g\ci\Pi_\uX} && [\uZ/H]}
\end{equation*}
with squares 2-Cartesian, where $\uW',\uX',\uY'$ are
$C^\iy$-schemes. Then $\pi_\uW,\pi_\uX,\pi_\uY$ are \'etale and
surjective, as $\Pi_\uZ$ is. If $\uX,\uY,\uZ$ are locally fair, then
$\uX',\uY'$ are locally fair as $\uX,\uY$ are and $\pi_\uX,\pi_\uY$
are \'etale, so $\uW'\cong\uX'\t_\uZ\uY'$ is locally fair by Theorem
\ref{ag4thm3}(b), and thus $\uW$ is locally fair as
$\pi_\uW:\uW'\ra\uW$ is \'etale and surjective. The proof for
locally finitely presented is the same.
\end{proof}

Using this we prove:

\begin{thm} $\DMCSta,\DMCStalf$ and\/ $\DMCStalfp$ are closed under fibre
products\I{C-stack@$C^\iy$-stack!fibre products} in\/~$\CSta$.
\label{ag7thm1}
\end{thm}

\begin{proof} Let $\cW=\cX\t_\cZ\cY$ be a fibre product in $\CSta$
of Deligne--Mumford $C^\iy$-stacks $\cX,\cY,\cZ$. We must show $\cW$
is Deligne--Mumford. Now $\cZ$ admits an open cover $\{\cZ_c:c\in
C\}$ with $\cZ_c\simeq[\uZ_c/H_c]$ for $\uZ_c$ an affine
$C^\iy$-scheme and $H_c$ finite. For $c\in C$ define
$\cX_c=\cX\t_\cZ\cZ_c$ and $\cY_c=\cY\t_\cZ\cZ_c$, which are open
$C^\iy$-substacks%
\I{C-stack@$C^\iy$-stack!C-substack@$C^\iy$-substack!open}%
\I{C-substack@$C^\iy$-substack!open} of $\cX,\cY$, and so are
Deligne--Mumford by Proposition \ref{ag7prop1}. Then
$\{\cX_c\t_{\cZ_c}\cY_c:c\in C\}$ is an open cover of $\cW$, so it
is enough to prove $\cX_c\t_{\cZ_c}\cY_c$ is Deligne--Mumford. That
is, we may replace $\cZ$ by $\cZ_c\simeq[\uZ_c/H_c]$.

Similarly, by choosing open covers of $\cX_c,\cY_c$ by substacks
equivalent to $[\uX/F],[\uY/G]$, we reduce the problem to showing
$[\uX/F]\t_{[\uZ/H]}[\uY/G]$ is Deligne--Mumford, for $\uX,\uY,\uZ$
affine $C^\iy$-schemes and $F,G,H$ finite groups. This follows from
Propositions \ref{ag7prop2} and \ref{ag7prop3}, noting that
$\uX,\uY,\uZ$ are Hausdorff as they are affine, so $\uW$ is
Hausdorff in Proposition \ref{ag7prop3}. This shows $\DMCSta$ is
closed under fibre products. For $\DMCStalf,\DMCStalfp$ we use the
same argument with $\uZ_c,\uZ,\uX,\uY,\uW$ locally fair, or locally
finitely presented.
\end{proof}

Under weak conditions Deligne--Mumford $C^\iy$-stacks have coarse moduli $C^\iy$-schemes,\I{C-stack@$C^\iy$-stack!coarse moduli $C^\iy$-scheme $\protect\ucX_\top$}\I{Deligne--Mumford $C^\iy$-stack!coarse moduli $C^\iy$-scheme $\protect\ucX_\top$} in the sense of~\S\ref{ag64}.

\begin{thm} Let\/ $\cX$ be a locally Lindel\"of Deligne--Mumford\/ $C^\iy$-stack.\I{Deligne--Mumford $C^\iy$-stack!locally Lindel\"of} Then the $C^\iy$-ringed space\I{C-stack@$C^\iy$-stack!underlying $C^\iy$-ringed space $\protect\ucX_\top$} $\ucX_\top$ in Definition\/ {\rm\ref{ag6def9}}
is a\/ $C^\iy$-scheme. If\/ $\cX$ is locally fair,\I{Deligne--Mumford
$C^\iy$-stack!locally fair} or locally finitely presented,\I{Deligne--Mumford $C^\iy$-stack!locally finitely presented} then $\ucX_\top$ is too.
\label{ag7thm2}
\end{thm}

\begin{proof} By definition $\cX$ can be covered by open $C^\iy$-substacks\I{C-stack@$C^\iy$-stack!C-substack@$C^\iy$-substack!open}\I{C-substack@$C^\iy$-substack!open} $\cU$ equivalent to $[\uY/G]$
for $\uY$ a Lindel\"of affine $C^\iy$-scheme. Then the $C^\iy$-ringed space $\ucU_\top$ is isomorphic to $\uY/G$ in Definition \ref{ag4def15}, so Proposition \ref{ag4prop5} shows that $\ucU_\top$ is an affine $C^\iy$-scheme. Hence $\ucX_\top$ can be covered by open affine $\ucU_\top\subseteq\ucX_\top$, so $\ucX_\top$ is a $C^\iy$-scheme. If $\cX$ is locally fair (or locally finitely presented) the same argument works with $\uY$ fair (or finitely presented), and then $\ucU_\top\cong\uY/G$ is fair (or finitely presented), and $\ucX_\top$ is locally fair (or locally finitely presented).
\end{proof}

\begin{rem}\I{partition of unity|(}\I{Deligne--Mumford
$C^\iy$-stack!partition of unity on|(} In \S\ref{ag47} we discussed
{\it partitions of unity\/} on $C^\iy$-schemes. We can use Theorem
\ref{ag7thm2} to extend these ideas to Deligne--Mumford
$C^\iy$-stacks.

Let $\cX$ be a Deligne--Mumford $C^\iy$-stack, and suppose $\cX_\top$ is regular and Lindel\"of. Then $\ucX_\top$ is a $C^\iy$-scheme by Theorem \ref{ag7thm2}, and the topology on $\ucX_\top$ is smoothly generated by Example \ref{ag4ex5}(c) as $\cX_\top$ is regular. Hence Theorem \ref{ag4thm5} shows that $\O_{\cX_\top}$ is fine. Suppose $\{\cU_a:a\in A\}$ is a (Zariski) open cover of $\cX$. Then $\bigl\{\ucU_{a,\top}:a\in A\bigr\}$ is an open cover of $\ucX_\top$, so there exists a partition of unity $\{\eta_a:a\in A\}$ on $\ucX_\top$ subordinate to $\bigl\{\ucU_{a,\top}:a\in A\bigr\}$. Therefore $\{\pi^*(\eta_a):a\in A\}$ is (in a suitable sense) a partition of unity on $\cX$ subordinate to $\{\cU_a:a\in A\}$, where $\pi:\cX\ra\bucX_\top$ is the structural morphism.\I{partition of unity|)}\I{Deligne--Mumford $C^\iy$-stack!partition of unity on|)}
\label{ag7rem2}
\end{rem}

\subsection{\texorpdfstring{Characterizing Deligne--Mumford $C^\iy$-stacks}{Characterizing Deligne--Mumford C∞-stacks}}
\label{ag73}

We now explore ways to characterize when a $C^\iy$-stack $\cX$ is
Deligne--Mumford.

\begin{prop} Let\/ $\cX$ be a quotient $C^\iy$-stack\/
$[\uU/G]$ for\/ $\uU$ affine and\/ $G$ finite. Then the natural
$1$-morphism $\Pi:\bar\uU\ra\cX$ is an \'etale
atlas,\I{atlas!etale@\'etale}\I{C-stack@$C^\iy$-stack!atlas!etale@\'etale}
and\/ $\De_\cX:\cX\ra\cX\t\cX,$ $\io_\cX:I_\cX\ra\cX$ are
universally closed,\I{C-stack@$C^\iy$-stack!universally closed
1-morphism} proper,\I{C-stack@$C^\iy$-stack!proper 1-morphism} and
separated,\I{C-stack@$C^\iy$-stack!separated 1-morphism} with finite
fibres,\I{C-stack@$C^\iy$-stack!1-morphism with finite fibres} and\/
$\jmath_\cX:\cX\ra I_\cX$ is an open\I{C-stack@$C^\iy$-stack!open
embedding} and closed embedding.\I{C-stack@$C^\iy$-stack!closed
embedding}
\label{ag7prop4}
\end{prop}

\begin{proof} As in \eq{agAeq12} we have
2-Cartesian\I{2-category!2-Cartesian square} diagrams with
surjective rows:
\begin{equation*}
\xymatrix@!0@C=35pt@R=10pt{
*+[r]{\bar G\t\bar\uU} \ar[rrrr]_(0.7){\bar\umu} \ar[ddd]^{\bar\upi_\uU}
&&&& *+[l]{\bar\uU} \ar[ddd]_\Pi &
*+[r]{\bar G\t\bar\uU} \ar[rrrr]_(0.7){\Pi\ci\bar\upi_\uU}
\ar[ddd]^{\bar\upi_\uU\t\bar\umu} &&&& *+[l]{\cX}
\ar[ddd]_{\De_\cX}
\\
& \drrtwocell_{}\omit^{}\omit{^{}} &&&&& \drrtwocell_{}\omit^{}\omit{^{}}
\\
&&&&&&&&&
\\
*+[r]{\bar\uU} \ar[rrrr]^(0.7)\Pi &&&& *+[l]{\cX,\!\!{}}
& *+[r]{\bar\uU\t\bar\uU}
\ar[rrrr]^(0.7){\Pi\t\Pi} &&&& *+[l]{\cX\t\cX,\!\!{}}
\\
\\
*+[r]{(\bar G\t\bar\uU)\t_{\bar\uU\t\bar\uU}\bar\uU}
\ar[rrrr]_(0.7){\bar\umu} \ar[ddd]^{\bar\upi_\uU}
&&&& *+[l]{I_\cX}
\ar[ddd]_{\io_\cX} & *+[r]{\bar U}
\ar[ddd]^{(1\t\uid_\uU)\t\uid_\uU} \ar[rrrr]_(0.7)\Pi
&&&& *+[l]{\cX} \ar[ddd]_{\jmath_\cX}
\\
& \drrtwocell_{}\omit^{}\omit{^{}} &&&&& \drrtwocell_{}\omit^{}\omit{^{}}
\\
&&&&&&&&&
\\
*+[r]{\bar U} \ar[rrrr]^(0.7)\Pi &&&& *+[l]{\cX,\!\!{}} &
*+[r]{(\bar G\t\bar\uU)\t_{\bar\uU\t
\bar\uU}\bar\uU} \ar[rrrr]^(0.7){\bar\umu} &&&& *+[l]{I_\cX.\!\!{}} }
\end{equation*}
The left column $\bar\pi_\uU$ in the first diagram is \'etale. The
left columns in the second and third diagrams are both universally
closed, proper, and separated, with finite fibres, since $G$ is
finite with the discrete topology, and $U$ is Hausdorff as $\uU$ is
affine. This left column in the fourth is an open and closed
embedding. The result now follows from Propositions \ref{ag6prop2} and~\ref{agAprop2}(c).
\end{proof}

Propositions \ref{ag6prop2}, \ref{ag6prop4} and \ref{ag7prop4} now
imply:

\begin{cor} Let\/ $\cX$ be a Deligne--Mumford\/ $C^\iy$-stack.
Then $\cX$ has an \'etale
atlas\I{atlas!etale@\'etale}\I{C-stack@$C^\iy$-stack!atlas!etale@\'etale}
$\Pi:\bar U\!\ra\!\cX,$ the diagonal\/
$\De_\cX:\cX\!\ra\!\cX\!\t\!\cX$ is
separated\I{C-stack@$C^\iy$-stack!separated 1-morphism} with finite
fibres,\I{C-stack@$C^\iy$-stack!1-morphism with finite fibres} and
the inertia morphism\/ $\io_\cX:I_\cX\!\ra\!\cX$ is universally
closed,\I{C-stack@$C^\iy$-stack!universally closed 1-morphism}
proper,\I{C-stack@$C^\iy$-stack!proper 1-morphism} and separated,
with finite fibres, and\/ $\jmath_\cX:\cX\ra I_\cX$ is an
open\I{C-stack@$C^\iy$-stack!open embedding} and
closed\I{C-stack@$C^\iy$-stack!closed embedding} embedding. If\/
$\cX$ is separated\I{C-stack@$C^\iy$-stack!separated} then $\De_\cX$
is also universally closed and proper.
\label{ag7cor1}
\end{cor}

The last part holds as then $\De_\cX$ is universally closed with
finite fibres, which implies $\De_\cX$ is proper. Note that for
$\cX$ not separated we cannot conclude from Proposition
\ref{ag7prop4} that $\De_\cX$ is universally closed or proper, since
these properties are not stable under open embedding. Some of the
conclusions of Corollary \ref{ag7cor1} are sufficient for $\cX$ to
be separated and Deligne--Mumford.

\begin{thm} Let\/ $\cX$ be a $C^\iy$-stack, and suppose\/ $\cX$ has
an \'etale atlas\I{atlas!etale@\'etale} $\Pi:\bar\uU\ra\cX,$ and the
diagonal\/ $\De_\cX:\cX\ra\cX\t\cX$ is universally
closed\I{C-stack@$C^\iy$-stack!universally closed 1-morphism} and
separated.\I{C-stack@$C^\iy$-stack!separated 1-morphism} Then\/
$\cX$ is a separated Deligne--Mumford\/ $C^\iy$-stack.
\label{ag7thm3}
\end{thm}

\begin{proof} Let $(\uU,\uV,\us,\ut,\uu,\ui,\um)$ be the
groupoid\I{groupoid object|(}\I{category!groupoid object in|(} in
$\CSch$ constructed from $\Pi:\bar\uU\ra\cX$ as in \S\ref{agA5}, so
that $\cX\simeq[\uV\rra\uU]$. Then \eq{agAeq12} gives 2-Cartesian
diagrams with surjective rows. From the first and Propositions
\ref{agAprop2}(a) and \ref{ag6prop2} we see that $\us,\ut$ are
\'etale, since $\Pi$ is. From the second $\us\t\ut:\uV\ra\uU\t\uU$
is universally closed and separated, as $\De_\cX$ is. Let $p\in U$.
Define
\begin{equation*}
H=\bigl\{q\in V:s(q)=t(q)=p\bigr\}\subseteq s^{-1}(\{p\}).
\end{equation*}
It has the discrete topology, as $\us,\ut$ are \'etale.

Suppose for a contradiction that $H$ is infinite. Define a
$C^\iy$-ring
\begin{equation*}
\fC=\bigl\{c:H\amalg\{\iy\}\ra\R:\text{$c(q)=c(\iy)$ for
all but finitely many $q\in H$}\bigr\},
\end{equation*}
with $C^\iy$ operations defined pointwise in $H\amalg\{\iy\}$. Then
$\Spec\fC$ has underlying topological space the one point
compactification $H\amalg\{\iy\}$ of the discrete topological space
$H$, since $\fC=C^0(H\amalg\{\iy\})$ is the set of continuous maps $H\amalg\{\iy\}\ra\R$, with the natural $C^\iy$-ring structure. Define $\ug:\Spec\fC\ra\uU\t\uU$ to project $\Spec\fC$ to the point $(p,p)$. Then the morphism
\e
\upi_{\Spec\fC}:\uV\t_{\us\t\ut,\uU\t\uU,\ug}\Spec\fC\longra\Spec\fC
\label{ag7eq2}
\e
is the projection $H\t(H\amalg\{\iy\})\ra H\amalg\{\iy\}$. The
diagonal in $H$ is closed in $H\t(H\amalg\{\iy\})$, but its image is
$H$, which is not closed in $H\amalg\{\iy\}$. Hence \eq{ag7eq2} is
not a closed morphism, contradicting $\us\t\ut$ universally closed.
So $H$ is finite.

As $(\uU,\uV,\us,\ut,\uu,\ui,\um)$ is a groupoid, $H$ is a {\it
finite group}, with identity $u(p)$, inverse map $i\vert_H$, and
multiplication $m_H=m\vert_{H\t H}$. Since $\us,\ut$ are \'etale, we
can choose small open neighbourhoods $Z_q$ of $q$ in $V$ for all
$q\in H$ such that $\us\vert_{Z_q},\ut\vert_{Z_q}$ are isomorphisms
with open subsets of $\uU$. As $\us\t\ut$ is separated,
$\bigl\{(v,v):v\in V\bigr\}$ is closed in $\bigl\{(v,v')\in V\t
V:s(v)=s(v')$, $t(v)=t(v')\bigr\}$, which has the subspace topology
from $V\t V$. If $q\ne q'\in H$ then $(q,q')$ lies in
$\bigl\{(v,v')\in V\t V:s(v)=s(v')$, $t(v)=t(v')\bigr\}$ but not in
$\bigl\{(v,v):v\in V\bigr\}$, so $(q,q')$ has an open neighbourhood
in $V\t V$ which does not intersect $\bigl\{(v,v):v\in V\bigr\}$.
Making $Z_q,Z_{q'}$ smaller if necessary, we can take this open
neighbourhood to be $Z_q\t Z_{q'}$, and then $Z_q\cap Z_{q'}=\es$.
Thus, we can choose these open neighbourhoods $Z_q$ for $q\in H$ to
be {\it disjoint}.

Define $Y=\bigcap_{q\in H}s(Z_q)$ and $\uY=(Y,\O_U\vert_Y)$. Then
$Y$ is a small open neighbourhood of $p$ in $U$. Making $Y$ smaller
if necessary we can suppose it is contained in an affine open
neighbourhood of $p$ in $U$, and so is Hausdorff. Replace $Z_q$ by
$Z_q\cap s^{-1}(Y)$ for all $q\in H$. Then
$s\vert_{Z_q}:(Z_q,\O_V\vert_{Z_q})\ra \uY$ is an isomorphism for
$q\in H$. Set $Z=\bigcup_{q\in H}Z_q$, noting the union is disjoint,
and $\uZ=(Z,\O_V\vert_Z)$. Then we have an isomorphism
$\uphi=(\phi,\phi^\sh):H\t\uY\ra\uZ$, such that
$\us\ci\uphi=\uid_\uY$ and $\phi(q\t Y)=Z_q$ for~$q\in H$.

Now $\uZ$ is open in $\uV$, so $\uZ\t_{\us,\uU,\ut}\uZ$ is open in
$\uV\t_{\us,\uU,\ut}\uV$, and we can restrict the morphism $\um:
\uV\t_{\us,\uU,\ut}\uV\ra\uV$ to $\um\vert_{\uZ\t_\uU\uZ}:\uZ
\t_{\us,\uU,\ut}\uZ\ra\uV$. But
\begin{align*}
\uZ\t_{\us,\uU,\ut}\uZ&\cong(H\t\uY)\t_{\ui_\uY\ci\upi_\uY,\uU,\ut}
\uZ\\ &\cong H\t (Z\cap t^{-1}(Y),\O_V\vert_{Z\cap t^{-1}(Y)})
\subseteq H\t\uZ \cong H\t H\t\uY,
\end{align*}
using $\uphi$ an isomorphism and $\us\ci\uphi=\uid_\uY$. Write
$\ul{\Phi}:\uZ\t_{\us,\uU,\ut}\uZ\hookra H\t H\t\uY$ for the induced
open embedding. Define a second morphism $\um':\uZ
\t_{\us,\uU,\ut}\uZ\ra\uV$ by $\um'=\uphi\ci(\um_H\t\id_{\uY})
\ci\ul{\Phi}$, where $\um_H:H\t H\ra H$ is the group multiplication
$m_H:H\t H\ra H$, regarded as a morphism of $C^\iy$-schemes.

Following the definitions we find that $\us\ci(\um\vert_{\uZ
\t_\uU\uZ})=\us\ci\um':\uZ\t_{\us,\uU,\ut}\uZ\ra\uY\subset\uU$. Also
$H\subset Z$, and the definition of $m_H$ from $\um$ implies that
$m\vert_{Z\t_UZ}$ and $m'$ coincide on the finite set $H\t_UH$ in
$Z\t_UZ$. Since $\us$ is \'etale, this implies that
$\um\vert_{\uZ\t_\uU\uZ}$ and $\um'$ must coincide near the finite
set $H\t_UH$ in $Z\t_UZ$. Therefore by making the open neighbourhood
$Y$ of $p$ in $U$ smaller, and hence making $W_q,W,Z$ smaller too,
we can assume that~$\um\vert_{\uZ\t_\uU\uZ}=\um'$.

Let us summarize what we have done so far. We have constructed a
finite group $H$, a Hausdorff open neighbourhood $\uY$ of $p$ in
$\uU$, an open and closed subset $\uZ$ of $\us^{-1}(\uY)$ in $\uZ$
which contains $s^{-1}(p)\cap t^{-1}(p)$, and an isomorphism
$\uphi:H\t\uY\ra\uZ$ with $\us\ci\uphi=\upi_\uY$ which identifies
the groupoid multiplication $\um\vert_{\uZ\t_\uU\uZ}$ with the
restriction to $\uZ\t_\uU\uZ$ of the morphism $\um_H\t\uid_\uY:H\t
H\t\uY\ra\uY$ from multiplication in the finite group~$H$.

Consider the morphism $\ut\ci\uphi:H\t\uY\ra\uU\supset\uY$. Roughly
speaking, $\ut\ci\uphi$ is an $H$-action on $\uY$. More accurately,
there should an $H$-action on some open subset of $\uU$ containing
$\uY$, but $\uY$ may not be $H$-invariant, so that $\ut\ci\uphi$
need not map $H\t\uY$ to $\uY$. Replace $Y$ by $Y'=\bigcap_{q\in
H}t(Z_q)$, which is an open subset of $Y$ since when $q$ is the
identity $u(p)$ in $H$ we have $t(Z_{u(p)})=s(Z_{u(p)})=Y$, and
$p\in Y'$ as $p=t(q)\in t(Z_q)$ for $q\in H$. Replace $Z_q$ by
$Z_q'=Z_q\cap s^{-1}(Y)$ and $Z$ by $Z'=\bigcup_{q\in H}Z_q'$. Then
using $\um\vert_{\uZ\t_\uU\uZ}=\um'$ we can show that
$s(Z_q')=t(Z_q')=Y'$ for all $q\in H$, so $Y'$ is an $H$-invariant
open set, and $\ut\ci\uphi$ maps $H\t\uY'\ra\uY'$. Restricting the
groupoid\I{groupoid object|)}\I{category!groupoid object in|)}
axioms shows that $\ut\ci\uphi$ gives an action of $H$ on~$\uY'$.

Now consider the morphism
\begin{equation*}
\us\t\ut\vert_{s^{-1}(Y')\cap t^{-1}(Y')}:\bigl(s^{-1}(Y')\cap
t^{-1}(Y'),\O_V\vert_{s^{-1}(Y')\cap
t^{-1}(Y')}\bigr)\longra\uY'\t\uY'.
\end{equation*}
This is closed, as $\us\t\ut$ is universally closed. Since $Z'$ is
open and closed in $s^{-1}(Y')\cap t^{-1}(Y')$, its complement is
closed, so its image $\bigl\{(s(v),t(v))\in Y'\t Y':v\in V\sm
Z'\bigr\}$ is closed in $Y'$. But $(p,p)$ does not lie in this
image, since $s^{-1}(p)\cap t^{-1}(p)\subseteq Z'$. Thus, by making
the $H$-invariant open neighbourhood $Y'$ of $p$ in $U$ smaller if
necessary, we can suppose that $s^{-1}(Y')\cap t^{-1}(Y')=Z'$.

The quotient $C^\iy$-stack $[\uY'/H]$ is Deligne--Mumford by
Proposition \ref{ag7prop2}, since $Y'$ is Hausdorff. Thus there
exists an open embedding $\cY_p\hookra[\uY'/H]$ with
$\cY_p\simeq[\uU_p/G_p]$ for $\uU_p$ affine and $G_p$ finite, which
includes $p$ in its image. The inclusion morphisms $\uY'\hookra\uU$,
$\uZ'\hookra\uV$ induce a 1-morphism $[\uZ'\rra\uY']\hookra
[\uV\rra\uU]$, which is an open embedding as $\uY'$ is open in
$\uU$, $\uZ'$ is open in $\uV$ and $s^{-1}(Y')\cap t^{-1}(Y')=Z'$ in
$V$. Let $i_{\cY_p}:\cY_p\ra\cX$ be the composition $\cY_p
\hookra[\uY'/H] \simeq[\uZ'\rra\uY']\hookra [\uV\rra\uU]\simeq\cX$.
Then $i_{\cY_p}$ is an open embedding, as it is a composition of
open embeddings and equivalences. This works for all $p\in U$, and
$\{\cY_p:p\in U\}$ is an open cover of $\cX$ with
$\cY_p\simeq[\uU_p/G_p]$ for $\uU_p$ affine and $G_p$ finite. Hence
$\cX$ is Deligne--Mumford. It is
separated\I{C-stack@$C^\iy$-stack!separated} as $\De_\cX$ is
universally closed, by assumption.
\end{proof}

Suppose $\uf:\uX\ra\uY$ is a separated morphism of $C^\iy$-schemes
with finite fibres. Then $\uf$ universally closed implies $\uf$
proper. Conversely, if $X,Y$ are compactly generated topological
spaces then $\uf$ proper implies $\uf$ universally closed. If
$\uX,\uY$ are locally fair\I{C-scheme@$C^\iy$-scheme!locally fair}
then $X,Y$ are compactly generated, as they are locally homeomorphic
to closed subsets of $\R^n$. Thus, in Theorem \ref{ag7thm3}, if
$\uU,\uV$ are locally fair then we can replace $\De_\cX$ universally
closed by $\De_\cX$ proper, yielding:

\begin{thm} Let\/ $\cX$ be a $C^\iy$-stack, and suppose\/ $\cX$ has
an \'etale
atlas\I{atlas!etale@\'etale}\I{C-stack@$C^\iy$-stack!atlas!etale@\'etale}
$\Pi:\bar\uU\ra\cX$ with\/ $\uU$ locally fair, and the diagonal\/
$\De_\cX:\cX\ra\cX\t\cX$ is proper\I{C-stack@$C^\iy$-stack!proper
1-morphism} and separated.\I{C-stack@$C^\iy$-stack!separated
1-morphism} Then\/ $\cX$ is a separated, locally
fair\I{Deligne--Mumford $C^\iy$-stack!locally fair}
Deligne--Mumford\/ $C^\iy$-stack.
\label{ag7thm4}
\end{thm}

The same holds with locally finitely presented\I{Deligne--Mumford
$C^\iy$-stack!locally finitely presented} in place of locally fair.
If $\cX\simeq[\uV\rra\uU]$ with $\uU$ a Hausdorff $C^\iy$-scheme
then $\uV$ is Hausdorff if and only if $\De_\cX$ is separated. We
can always choose $\uU$ Hausdorff, by replacing $\uU$ by the
disjoint union of an open cover of $\uU$ by affine open subsets.
Thus we can replace the condition that $\De_\cX$ is separated by
$\uU,\uV$ Hausdorff. Combining this and the results above proves:

\begin{thm}{\bf(a)} A\/ $C^\iy$-stack\/ $\cX$ is
separated\I{C-stack@$C^\iy$-stack!separated} and Deligne--Mumford if
and only if it is equivalent to the stack associated to a groupoid\/\I{C-stack@$C^\iy$-stack!associated to a groupoid} $[\uV\rra\uU]$ where $\uU,\uV$ are Hausdorff\/ $C^\iy$-schemes, $\us:\uV\ra\uU$ is \'etale, and\/ $\us\t\ut:\uV\ra\uU\t\uU$ is universally closed.
\smallskip

\noindent{\bf(b)} A\/ $C^\iy$-stack\/ $\cX$ is separated,
Deligne--Mumford and locally fair (or locally finitely presented) if
and only if it is equivalent to some\/ $[\uV\rra\uU]$ with\/
$\uU,\uV$ Hausdorff, locally fair (or locally finitely presented)\/
$C^\iy$-schemes, $\us:\uV\ra\uU$ \'etale, and\/
$\us\t\ut:\uV\ra\uU\t\uU$ proper.
\label{ag7thm5}
\end{thm}

\subsection[\texorpdfstring{Quotient $C^\iy$-stacks, 1- and 2-morphisms as
local models for \\ objects, 1- and 2-morphisms in
$\DMCSta$}{Quotient C∞-stacks, 1- and 2-morphisms as
local models for objects, 1- and 2-morphisms in DMCSta}]{Quotient $C^\iy$-stacks, 1- and 2-morphisms as local \\
models for objects, 1- and 2-morphisms in $\DMCSta$}
\label{ag74}
\I{C-stack@$C^\iy$-stack!quotients $[\protect\uX/G]$!quotient 1-morphism|(}%
\I{quotient C-stack@quotient $C^\iy$-stack!1-morphism|(}%
\I{C-stack@$C^\iy$-stack!quotients $[\protect\uX/G]$!quotient
2-morphism|(}\I{quotient
C-stack@quotient $C^\iy$-stack!2-morphism|(}%

In our next theorem, we prove that Deligne--Mumford $C^\iy$-stacks
and their 1- and 2-morphisms are (Zariski) locally modelled on
quotient $C^\iy$-stacks $[\uX/G]$, quotient 1-morphisms
$[\uf,\rho]:[\uX/G]\ra[\uY/H]$, and quotient 2-morphisms
$[\de]:[\uf,\rho]\Ra [\ug,\si]$ from~\S\ref{ag71}.

\begin{thm}{\bf(a)} Let\/ $\cX$ be a Deligne--Mumford\/
$C^\iy$-stack and\/ $[x]\in\cX_\top,$ and write\/
$G=\Iso_\cX([x])$.\I{C-stack@$C^\iy$-stack!isotropy group
$\Iso_\cX([x])$} Then there exists a quotient\/ $C^\iy$-stack\/
$[\uU/G]$ for\/ $\uU$ an affine $C^\iy$-scheme, and a $1$-morphism
$i:[\uU/G]\ra\cX$ which is an equivalence with an open
$C^\iy$-substack\/%
\I{C-stack@$C^\iy$-stack!C-substack@$C^\iy$-substack!open}%
\I{C-substack@$C^\iy$-substack!open} $\cU$ in $\cX,$ such that\/
$i_\top:[u]\mapsto[x]\in \cU_\top\subseteq\cX_\top$ for some fixed
point\/ $u$ of\/ $G$ in~$\uU$.

\smallskip

\noindent{\bf(b)} Let\/ $f:\cX\ra\cY$ be a $1$-morphism of
Deligne--Mumford\/ $C^\iy$-stacks, and\/ $[x]\in\cX_\top$ with\/
$f_\top:[x]\mapsto[y]\in\cY_\top,$ and write\/ $G=\Iso_\cX([x])$
and\/ $H=\Iso_\cY([y])$. Part\/ {\bf(a)} gives $1$-morphisms
$i:[\uU/G]\ra\cX,$ $j:[\uV/H]\ra\cY$ which are equivalences with
open $\cU\subseteq\cX,$ $\cV\subseteq\cY,$ such that\/
$i_\top:[u]\mapsto[x]\in \cU_\top\subseteq\cX_\top,$
$j_\top:[v]\mapsto[y]\in \cV_\top\subseteq\cY_\top$ for $u,v$ fixed
points of\/ $G,H$ in~$\uU,\uV$.

Then there exists a $G$-invariant open neighbourhood\/ $\uU'$ of\/
$u$ in $\uU$ and a quotient\/ $1$-morphism
$[\uf,\rho]:[\uU'/G]\ra[\uV/H]$ such that\/ $\uf(u)=v,$ and\/
$\rho:G\ra H$ is $f_*:\Iso_\cX([x])\ra\Iso_\cY([y]),$ fitting into a
$2$-commutative diagram:
\e
\begin{gathered}
\xymatrix@C=130pt@R=13pt{ *+[r]{[\uU'/G]} \ar[r]_(0.3){[\uf,\rho]}
\ar[d]^{i\vert_{[\uU'/G]}}
\drtwocell_{}\omit^{}\omit{^\ze} & *+[l]{[\uV/H]} \ar[d]_j \\
*+[r]{\cX} \ar[r]^(0.63){f} & *+[l]{\cY.\!} }
\end{gathered}
\label{ag7eq3}
\e

\noindent{\bf(c)} Let\/ $f,g:\cX\ra\cY$ be $1$-morphisms of
Deligne--Mumford\/ $C^\iy$-stacks and\/ $\eta:f\Ra g$ a
$2$-morphism, let\/ $[x]\in\cX_\top$ with\/
$f_\top:[x]\mapsto[y]\in\cY_\top,$ and write\/ $G=\Iso_\cX([x])$
and\/ $H=\Iso_\cY([y])$. Part\/ {\bf(a)} gives $i:[\uU/G]\ra\cX,$
$j:[\uV/H]\ra\cY$ which are equivalences with open
$\cU\subseteq\cX,$ $\cV\subseteq\cY$ and map $i_\top:[u]\mapsto[x],$
$j_\top:[v]\mapsto[y]$ for $u,v$ fixed points of\/~$G,H$.

By making $\uU'$ smaller, we can take the same $\uU'$ in {\bf(b)}
for both\/ $f,g$. Thus part\/ {\bf(b)} gives a $G$-invariant open
$\uU'\subseteq\uU,$ quotient morphisms
$[\uf,\rho]:[\uU'/G]\ra[\uV/H]$ and\/ $[\ug,\si]:[\uU'/G]\ra[\uV/H]$
with\/ $\uf(u)=\ug(u)=v$ and\/ $\rho=f_*:\Iso_\cX([x])\ra
\Iso_\cY([y]),$ $\si=g_*:\Iso_\cX([x])\ra\Iso_\cY([y]),$ and\/
$2$-morphisms $\ze:f\ci i\vert_{[\uU'/G]}\Ra j\ci[\uf,\rho],$
$\th:g\ci i\vert_{[\uU'/G]}\Ra j\ci[\ug,\si]$.

Then there exists a $G$-invariant open neighbourhood\/ $\uU''$ of\/
$u$ in $\uU'$ and\/ $\de\in H$ such that\/
$\si(\ga)=\de\,\rho(\ga)\ci\de^{-1}$ for all\/ $\ga\in G$ and\/
$\ug\vert_{\smash{\uU''}}=\us(\de)\ci\uf\vert_{\smash{\uU''}},$ so
that\/ $[\de]:[\uf\vert_{\smash{\uU''}},\rho]
\Ra[\ug\vert_{\smash{\uU''}},\si]$ is a quotient\/ $2$-morphism, and
the following diagram of\/ $2$-morphisms in $\CSta$ commutes:
\e
\begin{gathered}
\xymatrix@C=130pt@R=15pt{ *+[r]{f\ci i\vert_{[\uU''/G]}}
\ar@{=>}[r]_{\eta*\id_{i\vert_{[\uU''/G]}}}
\ar@{=>}[d]^{{}\,\,\ze\vert_{[\uU''/G]}} & *+[l]{g\ci
i\vert_{[\uU''/G]}}
\ar@{=>}[d]_{\th\vert_{[\uU''/G]}\,\,{}} \\
*+[r]{j\ci [\uf\vert_{\uU''},\rho]} \ar@{=>}[r]^{\id_j*[\de]} &
*+[l]{j\ci [\ug\vert_{\uU''},\si].\!{}} }
\end{gathered}
\label{ag7eq4}
\e

\label{ag7thm6}
\end{thm}

\begin{proof} In this proof we will use the theory of 2-categories
from \S\ref{agA1}, including
vertical\I{2-category!2-morphism!vertical composition} and
horizontal composition\I{2-category!2-morphism!horizontal
composition} of 2-morphisms `$*$', `$\od$', and the definition of
fibre products\I{C-stack@$C^\iy$-stack!fibre products|(} and
2-Cartesian squares\I{2-category!2-Cartesian square} in
Definition~\ref{agAdef3}.

For (a), as $\cX$ is Deligne--Mumford it is covered by open
$C^\iy$-substacks $\cV$ equivalent to $[\uV/H]$ for $\uV$ affine and
$H$ finite, so we can choose such $\cV$ with $[x]\in\cV_\top$. Then
$\cV$ has an \'etale atlas $\Pi:\bar\uV\ra\cV$ and $\De_\cV$ is
universally closed and separated by Proposition \ref{ag7prop4}, so
we can apply the proof of Theorem \ref{ag7thm3} to $\cV$ for a point
$p\in V$ with $\Pi_*(p)=[x]$. This constructs an open
$C^\iy$-substack $\cU$ in $\cV$ equivalent to $[\uU/G]$, where $\uU$
is affine and $G=\Iso_\cX([x])$, as we want.

For (b), write $\pi_{[\uU/G]}:\ul{\bar U\!}\,\ra[\uU/G]$ and
$\pi_{[\uV/H]}:\ul{\bar V\kern -.25em}\kern .2em\ra[\uV/H]$ for the
projection 1-morphisms in $\CSta$. They are proper and
representable. Let $\ur:G\ra\Aut(\uU)$ and $\us:H\ra\Aut(\uV)$ be
the $G$- and $H$-actions on $\uU,\uV$. Then $\ul{\bar r\kern
-0.15em}\kern 0.15em(\ga):\ul{\bar U\!}\,\ra \ul{\bar U\!}\,$ for
$\ga\in G$ and $\ul{\bar s\kern -0.15em}\kern 0.15em(\de):\ul{\bar
V\kern -.25em}\kern .2em\ra \ul{\bar V\kern -.25em}\kern .2em$ for
$\de\in H$ are the corresponding $C^\iy$-stack 1-morphisms, and
there are natural 2-morphisms $\la_\ga: \pi_{[\uU/G]}\ci\ul{\bar
r\kern -0.15em}\kern 0.15em(\ga) \Ra \pi_{[\uU/G]}$ and~$\mu_\de:
\pi_{[\uV/H]}\ci\ul{\bar s\kern -0.15em}\kern 0.15em(\de) \Ra
\pi_{[\uV/H]}$.

Consider the $C^\iy$-stack fibre product $\ul{\bar U\kern
-.25em}\kern .2em\t_{f\ci
i\ci\pi_{[\uU/G]},\cY,j\ci\pi_{[\uV/H]}}\ul{\bar V\kern -.25em}\kern
.2em$. As $\pi_{[\uV/H]}$ is representable and $j$ is an equivalence
with an open $C^\iy$-substack, $j\ci\pi_{[\uV/H]}$ is representable,
and $\ul{\bar U\kern -.25em}\kern .2em$ is a $C^\iy$-stack, so this
fibre product is a $C^\iy$-scheme. So changing the fibre product up
to equivalence, we can take $\ul{\bar U\kern -.25em}\kern
.2em\t_\cY\ul{\bar V\kern -.25em}\kern .2em=\ul{\bar W\!\!}\,\,$ for
some $C^\iy$-scheme $\uW$ unique up to isomorphism. The fibre
product projections are 1-morphisms $\ul{\bar W\!\!}\,\,\ra\ul{\bar
U\kern -.25em}\kern .2em$ and $\ul{\bar W\!\!}\,\,\ra\ul{\bar V\kern
-.25em}\kern .2em,$ so they are 2-isomorphic to $\ul{\bar
a},\ul{\bar b\kern -0.1em}\kern 0.1em$ for unique morphisms
$\ua:\uW\ra\uU$, $\ub:\uW\ra\uV$. Hence we have a 2-Cartesian square
in $\CSta$, for some 2-morphism~$\om$:
\e
\begin{gathered}
\xymatrix@C=140pt@R=14pt{ *+[r]{\ul{\bar W\!\!}\,\,}
\ar[r]_(0.3){\ul{\bar b\kern -0.1em}\kern 0.1em} \ar[d]^{\ul{\bar
a}} \drtwocell_{}\omit^{}\omit{^\om} &
*+[l]{\ul{\bar V\kern -.25em}\kern .2em}
\ar[d]_{j\ci\pi_{[\uV/H]}} \\
*+[r]{\ul{\bar U\kern -.25em}\kern .2em} \ar[r]^(0.25){f\ci
i\ci\pi_{[\uU/G]}} & *+[l]{\cY.\!} }
\end{gathered}
\label{ag7eq5}
\e

We will show that the data $\ur(\ga),\la_\ga$ for $\ga\in G$ induces
an action of $G$ on $\uW$. Let $\ga\in G$, and apply the universal
property of the 2-Cartesian square \eq{ag7eq5} in Definition
\ref{agAdef3} to the 1-morphisms $\ul{\bar r\kern -0.15em}\kern
0.15em(\ga)\ci \ul{\bar a}:\ul{\bar W\!\!}\,\,\ra\ul{\bar U\kern
-.25em}\kern .2em$, $\ul{\bar b\kern -0.1em}\kern 0.1em:\ul{\bar
W\!\!}\,\,\ra\ul{\bar V\kern -.25em}\kern .2em$ and 2-morphism
$\om\od(\id_{f\ci i}*\la_\ga*\id_{\smash{\ul{\bar a}}}):(f\ci
i\ci\pi_{[\uU/G]})\ci(\ul{\bar r\kern -0.15em}\kern
0.15em(\ga)\ci\ul{\bar a})\Ra(j\ci\pi_{[\uV/H]})\ci \ul{\bar b\kern
-0.1em}\kern 0.1em$. This gives a 1-morphism $c_\ga:\ul{\bar
W\!\!}\,\,\ra \ul{\bar W\!\!}\,\,$, unique up to 2-isomorphism, and
2-morphisms $\ze_\ga: \ul{\bar a}\ci c_\ga\Ra \ul{\bar r\kern
-0.15em}\kern 0.15em(\ga)\ci\ul{\bar a}$, $\th_\ga:\ul{\bar b\kern
-0.1em}\kern 0.1em\ci c_\ga\Ra\ul{\bar b\kern -0.1em}\kern 0.1em$
such that \eq{agAeq6} commutes.

Now $c_\ga$ is 2-isomorphic to $\bar\uc_\ga$ for some unique
$\uc_\ga:\uW\ra\uW$, so we may replace $c_\ga$ by $\bar\uc_\ga$.
Then $\ze_\ga: \ul{\bar a}\ci \bar\uc_\ga\Ra\ul{\bar r\kern
-0.15em}\kern 0.15em(\ga)\ci\ul{\bar a}$, so we must have
$\ua\ci\uc_\ga=\ur(\ga)\ci\ua$ and $\ze_\ga=\id_{\smash{\ul{\bar
r\kern -0.15em}\kern 0.15em(\ga)\ci\ul{\bar a}}}$. Similarly
$\ub\ci\uc_\ga=\ub$ and $\th_\ga=\id_{\smash{\ul{\bar b\kern
-0.1em}\kern 0.1em}}$. Therefore \eq{agAeq6} reduces to
$\om*\id_{\smash{\bar\uc_\ga}}= \om\od(\id_{f\ci
i}*\la_\ga*\id_{\smash{\ul{\bar a}}})$. Using
$\ur(\ga)\ur(\ga')=\ur(\ga\ga')$ and a natural compatibility between
$\la_\ga,\la_\ga',\la_{\ga\ga'}$ we find that
$\uc_\ga\ci\uc_{\ga'}=\uc_{\ga\ga'}$ for $\ga,\ga'\in G$, and as
$\ur(1)=\uid_\uU$ and $\la_1=\id_{\smash{\pi_{[\uU/G]}}}$ we have
$\uc_1=\uid_\uW$. Hence $\ga\mapsto\uc_\ga$ is an action of $G$ on
$\uW$, and $\ua\ci\uc_\ga=\ur(\ga)\ci\ua$ means that $\ua:\uW\ra\uU$
is $G$-equivariant.

In the same way, we obtain unique isomorphisms $\ud_\de:\uW\ra\uW$
for $\de\in H$ with $\ua\ci\ud_\de=\ua$,
$\ub\ci\ud_\de=\us(\de)\ci\ub$ and $\om*\id_{\smash{\bar\ud_\de}}=
(\id_{j}*\mu_\de*\id_{\smash{\ul{\bar b\kern -0.1em}\kern
0.1em}})\od\om$, and $\de\mapsto\ud_\de$ is an action of $H$ on
$\uW$, and $\ub:\uW\ra\uV$ is $H$-equivariant. Using associativity
of $\od$ in $(\id_{j}*\mu_\de*\id_{\smash{\ul{\bar b\kern
-0.1em}\kern 0.1em}})\od\om\od(\id_{f\ci
i}*\la_\ga*\id_{\smash{\ul{\bar a}}})$, we see that $\uc_\ga$ and
$\ud_\de$ commute. Hence $(\ga,\de)\mapsto \uc_\ga\ci\ud_\de$ is an
action of $G\t H$ on~$\uW$.

Since $\pi_{[\uV/H]}:\ul{\bar V\kern -.25em}\kern .2em\ra[\uV/H]$ is
a principal $H$-bundle, and $j:[\uV/H]\ra\cY$ is an equivalence with
$\cV\subseteq\cY$, and \eq{ag7eq5} is 2-Cartesian, it follows that
$\ua:\uW\ra\uU$ is a principal $H$-bundle over the open
$C^\iy$-subscheme $\ul{\ti U\!}\,$ of $\uU$ mapped to $\cV$ by $f\ci
i\ci\pi_{[\uU/G]}$, where the $H$-action for the principal
$H$-bundle is $\de\mapsto\ud_\de$. As $u\in\ul{\ti U\!}\,$, this
implies that we can choose a $G$-invariant open neighbourhood $\uU'$
of $u$ in $\ul{\ti U\!}\,\subseteq\uU$ with an isomorphism
$\uW'=\ua^{-1}(\uU')\cong \uU'\t H$, that identifies
$\ud_\de\vert_{\uW'}:\uW'\ra\uW'$ with the product of
$\uid_{\smash{\uU'}}$ on $\uU'$ and $\ep\mapsto\de\ep$ on $H$.

Then $\ga\mapsto\uc_\ga\vert_{\smash{\uW'}}$ is an action of $G$ on
$\uW'\cong \uU'\t H$, and the projection $\uU'\t H\ra\uU'$ is
$G$-equivariant. Since $u\in\uU'$ is a fixed point of $G$, this
implies that $\uc_\ga$ fixes the finite subset $\{(u,\de):\de\in
H\}$ in $\uU'\t H$. Define $\rho:G\ra H$ by
$\uc_\ga(u,1)=(u,\rho(\ga)^{-1})$ for $\ga\in G$. Since $\ud_\de$
acts by $(u,\ep)\mapsto(u,\de\ep)$ and $\uc_\ga,\ud_\de$ commute, it
follows that $\uc_\ga(u,\de)=(u,\de\rho(\ga)^{-1})$ for $\ga\in G$,
$\de\in H$. Hence
\begin{equation*}
(u,\rho(\ga\ga')^{-1})\!=\!\uc_{\ga\ga'}(u,1)\!=\!\uc_\ga\ci
\uc_{\ga'}(u,1)\!=\!\uc_\ga(u,\rho(\ga')^{-1})\!=\!
(u,\rho(\ga')^{-1}\rho(\ga)^{-1}),
\end{equation*}
so $\rho(\ga\ga')^{-1}=\rho(\ga')^{-1}\rho(\ga)^{-1}$, and
$\rho(\ga\ga')=\rho(\ga)\rho(\ga')$ for $\ga,\ga'\in G$. Thus
$\rho:G\ra H$ is a group morphism.

Using $\uW'\cong\uU'\t H$, $\ua\ci\uc_\ga=\ur(\ga)\ci\ua$, and
$\uc_\ga(u,\de)=(u,\de\rho(\ga)^{-1})$, we see that close to
$\{u\}\t H$, $\uc_\ga\vert_{\uW'}:\uU'\t H\ra \uU'\t H$ acts as
$\ur(\ga)$ on $\uU'$ and $\de\mapsto\de\rho(\ga)^{-1}$ on $H$.
Making $\uU'$ smaller if necessary, we can suppose this happens on
all of $\uU'$. Write $\uk:\uU'\hookra\uW$ for the inclusion of
$\uU'$ as an open $C^\iy$-subscheme in $\uW$ via the identifications
$\uU'\cong\uU'\t\{1\}\subseteq \uU'\t H\cong\uW'\subseteq\uW$, and
define $\uf=\ub\ci\uk:\uU'\ra\uV$.

Let $\ga\in G$. Since $\uc_\ga\vert_{\uW'}$ acts as $\ur(\ga)$ on
$\uU'$ and $\de\mapsto\de\rho(\ga)^{-1}$ on $H$, and
$\ud_{\rho(\ga)}$ acts as $\de\mapsto\rho(\ga)\de$ on $H$, we see
that $\ud_{\rho(\ga)}\ci\uc_\ga$ acts as $\ur(\ga)\t\id_1$ on
$\uU'\t\{1\}$. Hence $\uk\ci\ur(\ga)\vert_{\uU'}=\ud_{\rho(\ga)}\ci
\uc_\ga\ci\uk$. Composing with $\ub$ gives
\begin{align*}
\uf\ci\ur(\ga)\vert_{\uU'}&=\ub\ci\uk\ci\ur(\ga)\vert_{\uU'}=
\ub\ci\ud_{\rho(\ga)}\ci\uc_\ga\ci\uk\\
&=\us(\rho(\ga))\ci\ub
\ci\uc_\ga\ci\uk=\us(\rho(\ga))\ci\ub\ci\uk=\us(\rho(\ga))\ci\uf,
\end{align*}
using $\ub\ci\ud_\de=\us(\de)\ci\ub$ and $\ub\ci\uc_\ga=\ub$. We
have now constructed a $C^\iy$-scheme morphism $\uf:\uU'\ra\uV$ and
a group morphism $\rho:G\ra H$ with $\uf\ci\ur(\ga)\vert_{\uU'}=
\us(\rho(\ga))\ci\uf$ for all $\ga\in G$. Thus Definition
\ref{ag7def2} defines~$[\uf,\rho]:[\uU'/G]\ra[\uV/H]$.

Consider the diagram of 2-morphisms:
\e
\begin{gathered}
\xymatrix@C=55pt@R=10pt{*+[r]{f\ci i\vert_{[\uU'/G]}\ci
\pi_{[\uU'/G]}} \ar@{=>}[r]_(0.63)\nu \ar@{=}[d] &
j\ci[\uf,\rho]\ci\pi_{[\uU'/G]} \ar@{=}[r]
& *+[l]{j\ci \pi_{[\uV/H]}\ci\ul{\bar f\!}\,{}} \ar@{=}[d] \\
*+[r]{f\ci i\ci \pi_{[\uU/G]}\ci \bar\ua\ci\ul{\bar k\kern -0.1em}
\kern 0.1em} \ar@{=>}[rr]^{\om*\id_{\ul{\bar k\kern -0.1em}\kern
0.1em}} && *+[l]{j\ci \pi_{[\uV/H]}\ci \ul{\bar b\kern -0.1em}\kern
0.1em\ci\ul{\bar k\kern -0.1em}\kern 0.1em.}}
\end{gathered}
\label{ag7eq6}
\e
Here $\om$ is as in \eq{ag7eq5}, and we have used $\uf=\ub\ci\uk$,
so that $\ul{\bar f\!}\,=\ul{\bar b\kern -0.1em}\kern
0.1em\ci\ul{\bar k\kern -0.1em}\kern 0.1em$, and
$\pi_{[\uU'/G]}=\pi_{[\uU/G]}\ci \bar\ua\ci\ul{\bar k\kern
-0.1em}\kern 0.1em$ since $\ua\ci\uk$ is the inclusion
$\uU'\hookra\uU$, and $[\uf,\rho]\ci\pi_{[\uU'/G]}=
\pi_{[\uV/H]}\ci\ul{\bar f\!}\,$. Thus there is a unique 2-morphism
$\nu=\om*\id_{\ul{\bar k\kern -0.1em}\kern 0.1em}$ making
\eq{ag7eq6} commute.

Using $\om*\id_{\smash{\bar\uc_\ga}}= \om\od(\id_{f\ci
i}*\la_\ga*\id_{\smash{\ul{\bar a}}})$ for $\ga\in G$ we can show
that $\nu$ is $G$-invariant in a suitable sense, and so pushes down
from $\ul{\bar U\kern -.25em}\kern .25em'$ to $[\uU'/G]$. That is,
there exists a unique 2-morphism $\ze:f\ci i\vert_{[\uU'/G]}\Ra
j\ci[\uf,\rho]$ with $\nu=\ze*\id_{\smash{\pi_{[\uU'/G]}}}$. So
\eq{ag7eq3} 2-commutes, completing part~(b).

For (c), let $\uW,\ua,\ub,\om,\uc_\ga,\ud_\de,\uW',\uk,\uf,\rho$ be
the data constructed in (b) above for $f:\cX\ra\cY$, and let
$\smash{\ul{\hat W\!\!}\,\,,\hat\ua,\ul{\hat b\kern -0.1em}\kern
0.1em,\hat\om, \hat\uc_\ga,\hat\ud_\de,\ul{\hat W\!\!}\,\,',\ul{\hat
k\kern -0.1em}\kern 0.1em,\ug,\si}$ be the corresponding data
constructed in (b) for $g:\cX\ra\cY$. Then combining $\eta:f\Ra g$
with the analogue of \eq{ag7eq5} for $g$, we have a 2-morphism
\begin{equation*}
\smash{(\eta*\id_{i\ci\pi_{[\uU/G]}\ci\ul{\bar{\hat a}}})\od\hat\om:
(f\ci i\ci\pi_{[\uU/G]})\ci\ul{\bar{\hat a}}\Longra
(j\ci\pi_{[\uV/H]})\ci \ul{\bar{\hat b}\kern -0.1em}\kern 0.1em.}
\end{equation*}
Arguing as in the construction of $\uc_\ga$ above, by the
2-Cartesian property of \eq{ag7eq5}, there exists a 1-morphism
$e:\ul{\bar{\hat W}\!\!}\,\,\ra \ul{\bar W\!\!}\,\,$, unique up to
2-isomorphism, and 2-morphisms $\hat\ze: \ul{\bar a}\ci e\Ra
\ul{\bar{\hat a}}$, $\hat\th:\ul{\bar b\kern -0.1em}\kern 0.1em\ci
e\Ra \ul{\bar{\hat b}\kern -0.1em}\kern 0.1em$ satisfying
\eq{agAeq6}. Then $e\cong\bar\ue$ for a unique $\ue:\ul{\hat
W\!\!}\,\,\ra\uW$. Replacing $e$ by $\bar\ue$, we have
$\ua\ci\ue=\hat\ua$, $\ub\ci\ue=\ul{\hat b\kern -0.1em}\kern 0.1em$,
$\hat\ze=\id_{\smash{\hat\ua}}$ and $\hat\th=\id_{\smash{\ul{\hat
b\kern -0.1em}\kern 0.1em}}$, and \eq{agAeq6} reduces
to~$\om*\id_{\smash{\bar\ue}}=(\eta*\id_{i\ci\pi_{[\uU/G]}\ci
\ul{\bar{\hat a}}})\od\hat\om$.

By repeating this for $\eta^{-1}:g\Ra f$, we can easily show that
$\ue:\ul{\hat W\!\!}\,\,\ra\uW$ is an isomorphism, and identifies
$\ua,\ub,\om,\uc_\ga,\ud_\de,\uW'$ with $\smash{\ul{\hat
W\!\!}\,\,,\hat\ua,\ul{\hat b\kern -0.1em}\kern 0.1em,\hat\om,
\hat\uc_\ga,\hat\ud_\de,\ul{\hat W\!\!}\,\,'}$, respectively.
However, the isomorphisms $\uW'\cong\uU'\t H$ and $\ul{\hat
W\!\!}\,\,'\cong\uU'\t H$ involved arbitrary choices of local
trivializations of the principal $H$-bundles $\ua:\uW\ra\uU$ and
$\hat\ua:\ul{\hat W\!\!}\,\,\ra\uU$, so $\ue$ need not identify
these isomorphisms.

Abuse notation by identifying $\uW'=\uU'\t H$ and $\ul{\hat
W\!\!}\,\,'=\uU'\t H$. Since $\ua\ci\ue(u,1)=\hat\ua(u,1)=u$ we see
that $\ue'(u,1)=(u,\de)$ for some unique $\de\in H$. As $\ue$
identifies $\ud_\ep$ and $\hat\ud_\ep$ for $\ep\in H$ we have
\e
\ue(u,\ep)=\ue\ci\hat\ud_\ep(u,1)=\ud_\ep\ci\ue(u,1)=
\ud_\ep(u,\de)=(u,\ep\de).
\label{ag7eq7}
\e
Similarly, as $\ue$ identifies $\uc_\ep$ and $\hat\uc_\ga$ for
$\ga\in G$, and $\uc_\ga,\hat\uc_\ga$ act on $\{u\}\t H$ by right
multiplication by $\rho(\ga)^{-1},\si(\ga)^{-1}$ in $H$, we have
\e
\ue(u,\si(\ga)^{-1})=\ue\ci\hat\uc_\ga(u,1)=\uc_\ga\ci\ue(u,1)=
\uc_\ga(u,\de)=(u,\de\rho(\ga)^{-1}).
\label{ag7eq8}
\e
Comparing \eq{ag7eq8} and \eq{ag7eq7} with $\ep=\si(\ga)^{-1}$, we
see that $\si(\ga)^{-1}\de=\de\rho(\ga)^{-1}$, so
$\si(\ga)=\de\,\rho(\ga)\,\de^{-1}$ for all~$\ga\in\Ga$.

Since $\ua\ci\ue=\hat\ua$, regarding $\ue\vert_{\smash{\uW'}}$ as a
morphism $\uU'\t H\ra\uU'\t H$, we have $\upi_{\smash{\uU'}}\ci
\ue\vert_{\smash{\uW'}}=\upi_{\smash{\uU'}}$. So by \eq{ag7eq7},
$\ue\vert_{\smash{\uW'}}$ is near $\{u\}\t H$ the product of
$\id_{\uU'}$ on $\uU'$ and $\ep\mapsto\ep\de$ on $H$. Choose a
$G$-invariant open neighbourhood $\uU''$ of $u$ in $\uU'$ such that
$\ue\vert_{\uU''\t H}$ is the product of $\id_{\uU''}$ and
$\ep\mapsto\ep\de$. Then
\begin{equation*}
\ug\vert_{\uU''}=\ul{\hat b\kern -0.1em}\kern 0.1em\ci\ul{\hat
k\kern -0.1em}\kern 0.1em\vert_{\uU''}=\ub\ci\ue\ci\ul{\hat
k\kern -0.1em}\kern 0.1em\vert_{\uU''}=\ub\ci\ud_\de\ci\uk\vert_{\uU''}
=\us(\de)\ci\ub\ci\uk\vert_{\uU''}=\us(\de)\ci\uf\vert_{\uU''}.
\end{equation*}
Hence $\si(\ga)=\de\,\rho(\ga)\,\de^{-1}$ for all $\ga\in G$ and
$\ug\vert_{\uU''}=\us(\de)\ci\uf\vert_{\uU''}$. Thus by Definition
\ref{ag7def3} we have a quotient 2-morphism $[\de]:
[\uf\vert_{\smash{\uU''}},\rho]\Ra[\ug\vert_{\smash{\uU''}},\si]$.
An argument similar to the last part of (b) then shows that
\eq{ag7eq4} commutes.\I{C-stack@$C^\iy$-stack!fibre products|)}
\end{proof}

Using the method of Theorem \ref{ag7thm6}(c), we can also prove:

\begin{prop} Let\/ $[\uf,\rho],[\ug,\si]:[\uX/G]\ra[\uY/H]$ be
quotient\/ $1$-morphisms of quotient\/ $C^\iy$-stacks in the sense
of\/ {\rm\S\ref{ag71},} and suppose\/ $[\uX/G]$ is connected, that
is, $X/G$ is connected as a topological space. Then every
$2$-morphism $\eta:[\uf,\rho]\Ra[\ug,\si]$ in $\CSta$ is a quotient
$2$-morphism $[\de]:[\uf,\rho]\Ra[\ug,\si]$ from Definition\/
{\rm\ref{ag7def3},} for some unique\/~$\de\in H$.
\label{ag7prop5}
\end{prop}

\begin{proof} Let $\eta:[\uf,\rho]\Ra[\ug,\si]$ be a 2-morphism.
The proof of Theorem \ref{ag7thm6}(c) shows that for each
$[x]\in[\uX/G]_\top\cong X/G$, there exists a unique $\de_{[x]}\in
H$ and an open neighbourhood $[\uU_{[x]}/G]$ of $[x]$ in $[\uX/G]$,
where $\uU_{[x]}\subseteq\uX$ is $G$-invariant and open, such that
$\eta\vert_{[\uU_{[x]}/G]}=[\de_{[x]}]\vert_{[\uU_{[x]}/G]}:
[\uf,\rho]\vert_{[\uU_{[x]}/G]}\Ra [\ug,\si]\vert_{[\uU_{[x]}/G]}$.
The map $X/G\ra H$ taking $[x]\mapsto\de_{[x]}$ is locally constant,
as it is constant on each such open $[\uU_{[x]}/G]$, so it is
globally constant as $X/G$ is connected, and $\de_{[x]}=\de\in H$
for all $[x]\in X/G$. Thus, $[\uX/G]$ may be covered by open
$[\uU_{[x]}/G]\subseteq [\uX/G]$ with
$\eta\vert_{[\uU_{[x]}/G]}=[\de] \vert_{[\uU_{[x]}/G]}$. As
2-morphisms in $\CSta$ form a sheaf, this proves that~$\eta=[\de]$.
\end{proof}

If $\cX=\bar\uX$ for some $C^\iy$-scheme $\uX$ then
$\Iso_\cX([x])\cong\{1\}$ for all $[x]\in\cX_\top$. Conversely, a
Deligne--Mumford $C^\iy$-stack with trivial isotropy groups is a
$C^\iy$-scheme.\I{C-stack@$C^\iy$-stack!is a $C^\iy$-scheme} Note
that in conventional algebraic geometry, a Deligne--Mumford stack
with trivial stabilizers is an {\it algebraic space},\I{algebraic
space} but need not be a scheme.

\begin{thm} Suppose $\cX$ is a Deligne--Mumford\/ $C^\iy$-stack
with\/ $\Iso_\cX([x])\cong\{1\}$\I{C-stack@$C^\iy$-stack!isotropy
group $\Iso_\cX([x])$} for all\/ $[x]\in\cX_\top$. Then\/ $\cX$ is
equivalent to\/ $\bar\uX$ for some\/ $C^\iy$-scheme\/~$\uX$.
\label{ag7thm7}
\end{thm}

\begin{proof} As $\Iso_\cX([x])\cong\{1\}$ for all $[x]\in\cX_\top$, by
Theorem \ref{ag7thm6}(a) there is an open cover $\{\cX_a:a\in A\}$
of $\cX$ with $\cX_a\simeq[\uX_a/\{1\}]\simeq\bar\uX_a$ for affine
$C^\iy$-schemes $\uX_a$, $a\in A$. Write $i_a:\bar\uX_a\ra\cX$ for
the corresponding open embedding. As $\De_\cX$ is representable, for
$a,b\in A$ the fibre product $\bar\uX_a\t_{i_a,\cX,i_b}\bar X_b$ is
represented by a $C^\iy$-scheme $\uX_{ab}=\uX_{ba}$ with open
embeddings $\ui_{ab}:\uX_{ab}\ra\uX_a$, $\ui_{ba}:\uX_{ba}\ra\uX_b$
identifying $\uX_{ab}$ with open $C^\iy$-subschemes
of~$\uX_a,\uX_b$.

The idea now is that the $C^\iy$-stack $\cX$ is made by gluing the
$C^\iy$-schemes $\uX_a$ for $a\in A$ together on the overlaps
$\uX_{ab}$, that is, we identify $\uX_a\supset
\ui_{ab}(\uX_{ab})\cong\uX_{ab}=\uX_{ba}\cong
\ui_{ab}(\uX_{ba})\subset\uX_b$. This is similar to the notion of
descent for objects in \S\ref{agA3}, and it is easy to check that
the natural 1-isomorphisms
\begin{equation*}
\bar\uX_{ab}\t_\cX\bar\uX_c\cong\bar\uX_{bc}\t_\cX
\bar\uX_a\cong\bar\uX_{ca}\t_\cX\bar\uX_b\cong\bar\uX_a
\t_\cX\bar\uX_b\t_\cX\bar\uX_c
\end{equation*}
imply the obvious compatibility conditions of the gluing morphisms
$\ui_{ab}$ on triple overlaps, and that $\uX_{aa}\cong\uX_a$. So by
a minor modification of the proof in Proposition \ref{ag6prop1} that
$(\CSch,\cJ)$ has descent for objects, we construct a $C^\iy$-scheme
$\uX$ with open embeddings $\uj_a:\uX_a\hookra\uX$ such that
$\{\uX_a:a\in A\}$ is an open cover of $\uX$, and
$\uX_a\t_{\uj_a,\uX,\uj_b}\uX_b$ is identified with $\uX_{ab}$ for
$a,b\in A$. Then by descent for morphisms in $(\CSch,\cJ)$, there
exists a 1-morphism $i:\bar\uX\ra\cX$ with $i_a$ 2-isomorphic to
$i\ci\bar\uj_a$ for all $a\in A$. This $i$ is an equivalence, so
$\cX\simeq\bar\uX$, as we have to prove.
\end{proof}

In fact in Theorem \ref{ag7thm7} we can take $\uX=\ucX_\top$, for
$\ucX_\top$\I{C-stack@$C^\iy$-stack!underlying $C^\iy$-ringed space
$\protect\ucX_\top$} as in Definition \ref{ag6def9}. Recall from
Definition \ref{agAdef13} that a 1-morphism of $C^\iy$-stacks
$f:\cX\ra\cY$ is {\it representable\/} if whenever $\uU$ is a
$C^\iy$-scheme and $g:\bar\uU\ra\cY$ a 1-morphism then the fibre
product $\cW=\cX\t_{f,\cY,g}\bar\uU$ in $\CSta$ is equivalent to a
$C^\iy$-scheme~$\bar\uV$.

\begin{cor} Let\/ $f:\cX\ra\cY$ be a $1$-morphism of
Deligne--Mumford\/ $C^\iy$-stacks. Then $f$ is representable if and
only if\/ $f_*:\Iso_\cX([x])\ra\Iso_\cY([y])$ in Definition\/
{\rm\ref{ag6def10}} is injective for all\/ $[x]\in\cX_\top$ with\/
$f_\top([x])=[y]\in\cY_\top$.
\label{ag7cor2}
\end{cor}

\begin{proof} Suppose $f$ is representable, and let $[x]\in
\cX_\top$ with $f_\top([x])=[y]\in\cY_\top$. Then $y:\bar{\ul *}
\ra\cY$, and $\cX\t_{f,\cY,y}\bar{\ul *}\simeq[{\ul *}/H]$, where
$H=\Ker\bigl(f_*:\Iso_\cX([x])\ra \Iso_\cY([y])\bigr)$. As $f$ is
representable, $[{\ul *}/H]$ is equivalent to a $C^\iy$-scheme, so
$H=\{1\}$, and $f_*$ is injective. This proves the `only if' part.

Now suppose $f_*$ is injective for all $[x]\in\cX_\top$. Let $\uU$
be a $C^\iy$-scheme and $g:\bar\uU\ra\cY$ a 1-morphism, and define
$\cW=\cX\t_{f,\cY,g}\bar\uU$, with projections $d:\cW\ra\cX$ and
$e:\cW\ra\bar\uU$. Then $\cW$ is a Deligne--Mumford $C^\iy$-stack by
Theorem \ref{ag7thm1}, as $\cX,\cY,\bar\uU$ are. Let $[w]\in
\cW_\top$, and set $[x]=d_\top([w])$ in $\cX_\top$,
$[u]=e_\top([w])$ in $\bar\uU_\top$, and $[y]=f_\top([x])=
g_\top([u])$ in $\cY_\top$. Then by properties of fibre products of
$C^\iy$-stacks we have a Cartesian square of groups
\begin{equation*}
\xymatrix@C=90pt@R=13pt{ *+[r]{\Iso_\cW([w])} \ar[r]_{e_*} \ar[d]^{d_*} &
*+[l]{\Iso_{\bar\uU}([u])} \ar[d]_{g_*} \\
*+[r]{\Iso_\cX([x])} \ar[r]^{f_*} & *+[l]{\Iso_\cY([y]).\!} }
\end{equation*}
But $\Iso_{\bar\uU}([u])=\{1\}$ as $\bar\uU$ is a $C^\iy$-scheme,
and $f_*$ is injective by assumption, so $\Iso_\cW([w])=\{1\}$, for
all $[w]\in \cW_\top$. Thus Theorem \ref{ag7thm7} shows $\cW$ is a
$C^\iy$-scheme, and $f$ is representable, proving the `if' part.
\end{proof}

We show that $\cX$ being Deligne--Mumford is essential in
Theorem~\ref{ag7thm7}:

\begin{ex} Let the group $\Z^2$ act on $\R$ by $(a,b):x\mapsto
x+a+b\sqrt{2}$ for $a,b\in\Z$ and $x\in\R$. As $\sqrt{2}$ is
irrational, this is a free action. It defines a groupoid\I{groupoid
object}\I{category!groupoid object in} $\Z^2\t\R\rra\R$ in $\Man$
which is \'etale, but not proper. Applying $F_\Man^\CSch$ gives a
groupoid $\ul{\Z}^2\t\ul{\R}\rra \ul{\R}$ in $\CSch$, and an
associated $C^\iy$-stack
$\cX=[\ul{\R}/\ul{\Z}^2]=[\ul{\Z}^2\t\ul{\R}\rra \ul{\R}]$. The
underlying topological space $\cX_\top$ is~$\R/\Z^2$.

Since each orbit of $\Z^2$ in $\R$ is dense in $\R$, $\cX_\top$ has
the indiscrete topology, that is, the only open sets are $\es$ and
$\cX_\top$. Thus $\cX_\top$ is not homeomorphic to $X$ for any
$C^\iy$-scheme $\uX=(X,\O_X)$, as each point of $X$ has an affine
and hence Hausdorff open neighbourhood. Therefore $\cX$ is not
equivalent to $\bar\uX$ for any $C^\iy$-scheme $\uX$. So $\cX$ is
not Deligne--Mumford by Theorem \ref{ag7thm7}. Hence, $C^\iy$-stacks
with finite isotropy groups need not be
Deligne--Mumford.\I{C-stack@$C^\iy$-stack!quotients
$[\protect\uX/G]$|)}\I{quotient
C-stack@quotient $C^\iy$-stack|)}%
\I{C-stack@$C^\iy$-stack!quotients $[\protect\uX/G]$!quotient 1-morphism|)}%
\I{quotient C-stack@quotient $C^\iy$-stack!1-morphism|)}%
\I{C-stack@$C^\iy$-stack!quotients $[\protect\uX/G]$!quotient
2-morphism|)}\I{quotient C-stack@quotient
$C^\iy$-stack!2-morphism|)}
\label{ag7ex2}
\end{ex}

\subsection{\texorpdfstring{Effective Deligne--Mumford $C^\iy$-stacks}{Effective Deligne--Mumford C∞-stacks}}
\label{ag75}
\I{Deligne--Mumford $C^\iy$-stack!effective|(}

\begin{dfn} A Deligne--Mumford $C^\iy$-stack $\cX$ is called {\it
effective\/} if whenever $[x]\in\cX_\top$ and $\cX$ near $[x]$ is
locally modelled near $[x]$ on a quotient $C^\iy$-stack $[\uU/G]$
near $[u]$, where $G=\Iso_\cX([x])$ and $u\in\uU$ is fixed by $G$,
as in Theorem \ref{ag7thm6}(a), then $G$ acts effectively on $\uU$
near $u$. That is, for each $1\ne\ga\in G$, we have
$\ur(\ga)\not\equiv\uid_\uU$ near $u$ in $\uU$, where
$\ur:G\ra\Aut(\uU)$ is the $G$-action.

Here the $C^\iy$-scheme $\uU$ in Theorem \ref{ag7thm6}(a) is
determined by $\cX,[x]$ up to $G$-equivariant isomorphism locally
near $u$. Hence to test whether $\cX$ is effective, it is enough to
consider one choice of $[\uU/G]$ for each~$[x]\in\cX_\top$.

A quotient $C^\iy$-stack $[\uX/G]$ is effective if and only if the
action $\ur:G\ra\Aut(\uX)$ of $G$ on $\uX$ is {\it locally
effective},\I{locally effective group action} that is, if for each
$1\ne\ga\in G$ we have $\ur(\ga)\vert_\uU\not\equiv\uid_\uU$ for
every nonempty open $C^\iy$-subscheme $\uU\subseteq\uX$. If a
Deligne--Mumford $C^\iy$-stack $\cX$ is a
$C^\iy$-scheme,\I{C-stack@$C^\iy$-stack!is a $C^\iy$-scheme} it is
automatically effective. Quotients $[\ul{*}/G]$ for $G\ne\{1\}$ are
not effective.
\label{ag7def5}
\end{dfn}

Here is a uniqueness property of 2-morphisms of effective
Deligne--Mumford $C^\iy$-stacks. Embeddings and submersions of
$C^\iy$-stacks are defined in~\S\ref{ag62}.

\begin{prop} Let\/ $f,g:\cX\ra\cY$ be $1$-morphisms of
Deligne--Mumford\/ $C^\iy$-stacks. Suppose any one of the following
conditions hold:
\begin{itemize}
\setlength{\itemsep}{0pt}
\setlength{\parsep}{0pt}
\item[{\bf(i)}] $\cX$ is effective and\/ $f$ is an embedding of\/
$C^\iy$-stacks (this implies $f_*:\Iso_\cX([x])\ra
\Iso_\cY(f_\top([x]))$\I{C-stack@$C^\iy$-stack!isotropy group
$\Iso_\cX([x])$} is an isomorphism for each $[x]\in\cX_\top$);
\item[{\bf(ii)}] $\cY$ is effective and\/ $f$ is a submersion;
or
\item[{\bf(iii)}] $\cY$ is a
$C^\iy$-scheme.\I{C-stack@$C^\iy$-stack!is a $C^\iy$-scheme}
\end{itemize}
Then there exists at most one $2$-morphism\/ $\eta:f\Ra g$. That is, the groupoid of such\/ $1$-morphisms is equivalent to a set. 
\label{ag7prop6}
\end{prop}

\begin{proof} Suppose $\eta,\ti\eta:f\Ra g$ are 2-morphisms. Let
$[x]\in\cX_\top$ with $f_\top([x])=[y]\in\cY_\top$. Apply Theorem
\ref{ag7thm6}(c) to $\eta,\ti\eta$. This first applies (a) to
$\cX,\cY$ at $[x],[y]$, giving $i:[\uU/G]\,{\buildrel\sim
\over\longra} \,\cU\subseteq\cX$ identifying $u\in\uU$ with $[x]$
and $j:[\uV/H]\,{\buildrel\sim\over\longra}\,\cV\subseteq\cY$
identifying $v\in\uU$ with $[y]$, and then applies (b) to $f,g$
giving $u\in\uU'\subseteq\uU$ and 1-morphisms
$[\uf,\rho],[\ug,\si]:[\uU/G]\ra[\uV/H]$. Then (c) for $\eta$ and
$\hat\eta$ gives $G$-invariant open $u\in\uU'',\utU{}''
\subseteq\uU'$ and elements $\de,\ti\de\in H$ with 2-morphisms
$[\de]:[\uf\vert_{\smash{\uU''}},\rho]\Ra[\ug
\vert_{\smash{\uU''}},\si]$, $[\ti\de]:[\uf\vert_{\smash{\utU{}''}},
\rho]\Ra[\ug\vert_{\smash{\utU{}''}},\si]$ such that \eq{ag7eq4} and
its analogue for $\ti\eta,\ti\de,\utU{}''$ commutes. Making
$\uU'',\utU{}''$ smaller, we can
take~$\uU''=\utU{}''$.\I{C-stack@$C^\iy$-stack!quotients
$[\protect\uX/G]$!quotient 1-morphism}%
\I{quotient C-stack@quotient $C^\iy$-stack!1-morphism}%
\I{C-stack@$C^\iy$-stack!quotients $[\protect\uX/G]$!quotient
2-morphism}\I{quotient C-stack@quotient $C^\iy$-stack!2-morphism}

The 2-morphisms $[\de],[\ti\de]:[\uf\vert_{\smash{\uU''}},
\rho]\Ra[\ug \vert_{\smash{\uU''}},\si]$ imply that
\e
\smash{\us(\de)\ci\uf\vert_{\smash{\uU''}}=\ug\vert_{\smash{\uU''}}
=\us(\hat\de)\ci\uf\vert_{\smash{\uU''}}.}
\label{ag7eq9}
\e
We will show that \eq{ag7eq9} and each of conditions (i)--(iii)
force $\de=\hat\de$. In case (i), as $f$ is an embedding, $\rho:G\ra
H$ is an isomorphism, and $\uf:\uU\ra\uV$ is an embedding of
$C^\iy$-schemes. Hence $\de=\rho(\ga)$, $\hat\de=\rho(\hat\ga)$ for
$\ga,\hat\ga\in G$, and
\begin{equation*}
\uf\ci\ur(\ga)\vert_{\smash{\uU''}}=\us(\de)\ci\uf\vert_{\smash{\uU''}}
=\us(\hat\de)\ci\uf\vert_{\smash{\uU''}}=\uf\ci\ur(\hat\ga)
\vert_{\smash{\uU''}}
\end{equation*}
by \eq{ag7eq9}. As $\uf$ is an embedding this implies that
$\ur(\ga)\vert_{\smash{\uU''}}=\ur(\hat\ga) \vert_{\smash{\uU''}}$,
so $\ga=\hat\ga$ as $G$ acts effectively on $\uU$ near $u$ since
$\cX$ is effective, and thus $\de=\hat\de$.

In case (ii), as $f$ is a submersion, $\uf:\uU\ra\uV$ is surjective
near $\uf(u)=v\in V$. Hence \eq{ag7eq9} implies that
$\us(\de)\vert_{\smash{\uV''}}= \us(\hat\de)\vert_{\smash{\uV''}}$
for some open neighbourhood $\uV''$ of $v$ in $\uV$. But $H$ acts
effectively on $\uV$ near $v$ as $\cY$ is effective, so
$\de=\hat\de$. In case (iii) $H=\Iso_\cY([y])=\{1\}$ as $\cY$ is a
$C^\iy$-scheme, so $\de=\hat\de=1$. Therefore $\de=\hat\de$ in each
case. Equation \eq{ag7eq4} for $\eta,\hat\eta$ now implies that
$\eta*\id_{\smash{i\vert_{[\uU''/G]}}}=\hat\eta*\id_{i\vert_{[\uU''/G]}}$.

Let $\cU''\subseteq\cU\subseteq\cX$ be the open $C^\iy$-substack
identified with $[\uU''/G]$. Then $i\vert_{\smash{[\uU''/G]}}:
[\uU''/G]\ra\cU''$ is an equivalence, so $\eta*\id_{\smash{
i\vert_{[\uU''/G]}}}=\hat\eta*\id_{\smash{i\vert_{[\uU''/G]}}}$
implies that $\eta\vert_{\smash{\cU''}}=\hat\eta
\vert_{\smash{\cU''}}$. Thus, each $[x]\in\cX_\top$ has an open
neighbourhood $\cU''$ in $\cX$ with $\eta\vert_{\smash{\cU''}}=
\hat\eta\vert_{\smash{\cU''}}$. As 2-morphisms form a sheaf on
restriction to Zariski open $C^\iy$-substacks, this implies that
$\eta=\hat\eta$, so $\eta:f\Ra g$ is unique.
\end{proof}

Similar arguments show that if $f,g:\cX\ra\cY$ are arbitrary
1-morphisms of Deligne--Mumford $C^\iy$-stacks with $\cX$ connected,
then there are at most finitely many 2-morphisms~$\eta:f\Ra
g$.\I{Deligne--Mumford $C^\iy$-stack!effective|)}

\subsection{\texorpdfstring{Orbifolds as Deligne--Mumford $C^\iy$-stacks}{Orbifolds as Deligne--Mumford C∞-stacks}}
\label{ag76}
\I{orbifold|(}\I{Deligne--Mumford $C^\iy$-stack!orbifold|(}

Orbifolds are geometric spaces locally modelled on $\R^n/G$ for $G$
a finite group acting linearly on $\R^n$, just as manifolds are
geometric spaces locally modelled on $\R^n$. Much has been written
about orbifolds, and there are several definitions, as either categories or 2-categories. See Lerman \cite{Lerm} for a good overview.

There are three main definitions of ordinary categories of orbifolds:
\begin{itemize}
\setlength{\itemsep}{0pt}
\setlength{\parsep}{0pt}
\item[(a)] Satake \cite{Sata} and Thurston \cite{Thur} defined an orbifold $\cX$ to be a Hausdorff topological space $X$ with an atlas $\bigl\{(V_i,\Ga_i,\psi_i):i\in I\bigr\}$ of orbifold charts $(V_i,\Ga_i,\psi_i)$, where $V_i$ is a manifold, $\Ga_i$ a finite group acting on $V_i$, and $\psi_i:V_i/\Ga_i\ra X$ a homeomorphism with an open set in $X$. Smooth maps $f:\cX\ra\cY$ between orbifolds are continuous maps $f:X\ra Y$ which lift locally to smooth maps on the charts, giving a category~$\Orb_{\rm ST}$.
\item[(b)] Chen and Ruan \cite[\S 4]{ChRu2} defined orbifolds $\cX$ in a similar way to \cite{Sata,Thur}, but using germs of orbifold charts $(V_p,\Ga_p,\psi_p)$ for $p\in X$. Their morphisms $f:\cX\ra\cY$ are called {\it good maps}, giving a category $\Orb_{\rm CR}$. 
\item[(c)] Moerdijk and Pronk \cite{Moer,MoPr} defined a category of orbifolds $\Orb_{\rm MP}$ as {\it proper \'etale Lie groupoids\/} in $\Man$. Their smooth maps $f:\cX\ra\cY$, called {\it strong maps\/} \cite[\S 5]{MoPr}, are equivalence classes of diagrams $\smash{\cX\,{\buildrel\phi\over \longleftarrow}\,\cX'\,{\buildrel\psi\over\longra}\,\cY}$, where $\cX'$ is a third orbifold, and $\phi,\psi$ are morphisms of groupoids with $\phi$ an equivalence.
\end{itemize}
A book on orbifolds in the sense of \cite{ChRu2,Moer,MoPr} is Adem, Leida and Ruan~\cite{ALR}.

There are four main definitions of 2-categories of orbifolds: 
\begin{itemize}
\setlength{\itemsep}{0pt}
\setlength{\parsep}{0pt}
\item[(i)] Pronk \cite{Pron} defines a strict 2-category $\bf LieGpd$ of Lie groupoids in $\Man$ as in (c), with the obvious 1-morphisms of groupoids, and localizes by a class of weak equivalences $\cW$ to get a weak 2-category $\Orb_{\rm Pr}={\bf LieGpd}[\cW^{-1}]$.
\item[(ii)] Lerman \cite[\S 3.3]{Lerm} defines a weak 2-category $\Orb_{\rm Le}$ of Lie groupoids in $\Man$ as in (c), with a non-obvious notion of 1-morphism called `Hilsum--Skandalis morphisms' involving `bibundles', and does not need to localize.

Henriques and Metzler \cite{HeMe} also use Hilsum--Skandalis morphisms.
\item[(iii)] Behrend and Xu \cite[\S 2]{BeXu}, Lerman \cite[\S 4]{Lerm} and Metzler \cite[\S 3.5]{Metz} define a strict 2-category of orbifolds $\Orb_{\rm ManSta}$ as a class of Deligne--Mumford stacks on the site $(\Man,{\cal J}_\Man)$ of manifolds with Grothendieck topology ${\cal J}_\Man$ coming from open covers.
\item[(iv)] The author \cite[\S 4.5]{Joyc5} defines a weak 2-category of orbifolds $\Orb_{\rm Kur}$ as special examples of Kuranishi spaces.
\end{itemize}

As in Behrend and Xu \cite[\S 2.6]{BeXu}, Lerman \cite{Lerm}, Pronk \cite{Pron}, and the author \cite[Rem.~4.51(a)]{Joyc5}, approaches (i)--(iv) give equivalent weak 2-categories $\Orb_{\rm Pr},\ab\Orb_{\rm Le},\ab\Orb_{\rm ManSta},\ab\Orb_{\rm Kur}$. Properties of localization imply that $\Orb_{\rm MP}\simeq \Ho(\Orb_{\rm Pr})$. Thus, all of (c) and (i)--(iv) are equivalent at the level of weak 2-categories or homotopy categories. 

Here is yet another definition of a strict 2-category of orbifolds $\Orb_{C^\iy{\rm Sta}}$, which is similar to (iii), but defining orbifolds
as a class of $C^\iy$-stacks, that is, as stacks on the site
$(\CSch,\cJ)$ rather than on~$(\Man,\cJ_\Man)$.

\begin{dfn} A $C^\iy$-stack $\cX$ is called an {\it orbifold\/}
if it is equivalent to the $C^\iy$-stack\I{groupoid
object}\I{category!groupoid object in}\I{stack!associated to a
groupoid}\I{C-stack@$C^\iy$-stack!associated to a groupoid}
$[\uV\rra\uU]$ associated to a groupoid $(\uU,\uV,\us,\ut,\uu,\ui,\um)$ in
$\CSch$ which is the image under $F_\Man^\CSch$ of a groupoid
$(U,V,s,t,u,i,m)$ in $\Man$, where $s:V\ra U$ is an \'etale smooth
map, and $s\times t:V\ra U\t U$ is a proper smooth map. That is,
$\cX$ is the $C^\iy$-stack associated to a {\it proper \'etale Lie
groupoid\/} in $\Man$. Write $\Orb_{C^\iy{\rm Sta}}$ for the full 2-subcategory of
orbifolds in~$\CSta$.

As in \S\ref{ag44}, $\uU,\uV$ are finitely presented
affine\I{C-scheme@$C^\iy$-scheme!finitely presented affine}
$C^\iy$-schemes, and thus $\cX$ is a
separated,\I{C-stack@$C^\iy$-stack!separated} locally finitely
presented\I{Deligne--Mumford $C^\iy$-stack!locally finitely
presented} Deligne--Mumford $C^\iy$-stack by Theorem
\ref{ag7thm5}(b). Hence~$\Orb_{C^\iy{\rm Sta}}\subset\DMCStalfp$.
\label{ag7def6}
\end{dfn}

The next theorem follows from the proofs in \cite{BeXu,Lerm,Joyc5,Pron} that (i)--(iv) above are equivalent 2-categories (in particular, that orbifolds in (iii) as stacks on $(\Man,{\cal J}_\Man)$ associated to proper \'etale Lie groupoids are equivalent to (i),(ii),(iv)), and the fact that the inclusion $\Man\hookra\CSch$ is full and faithful, with open covers $\cJ$ in $\CSch$ restricting to open covers $\cJ_\Man$ in~$\Man$. 

\begin{thm} This $2$-category of orbifolds $\Orb_{C^\iy{\rm Sta}}$ is equivalent to the $2$-categories of orbifolds $\Orb_{\rm Pr},\ab\Orb_{\rm Le},\ab\Orb_{\rm ManSta},\ab\Orb_{\rm Kur}$ in {\rm\cite{BeXu,Lerm,Metz,Pron,Joyc5}} described in {\rm(i)--(iv)} above. Also the homotopy category\I{homotopy category} $\Ho(\Orb_{C^\iy{\rm Sta}})$\G[HoOrb]{$\Ho(\Orb)$}{homotopy category of the
2-category of orbifolds $\Orb$} is equivalent to the category of orbifolds $\Orb_{\rm MP}$ in {\rm\cite{Moer,MoPr}} described in {\rm(c)} above.\label{ag7thm8}
\end{thm}

By Corollary \ref{ag4cor1} $F_\Man^\CSch$ takes transverse fibre
products\I{manifold!transverse fibre product} in $\Man$ to fibre
products in $\CSch$. As fibre products of orbifolds are locally
modelled on fibre products of manifolds, and fibre products of
Deligne--Mumford $C^\iy$-stacks are locally modelled on fibre
products of $C^\iy$-schemes, we deduce:

\begin{cor} Transverse fibre products\I{orbifold!transverse fibre
product} in\/ $\Orb_{C^\iy{\rm Sta}}$ agree with the corresponding fibre
products\I{C-stack@$C^\iy$-stack!fibre products} in\/~$\CSta$.\I{orbifold|)}\I{Deligne--Mumford $C^\iy$-stack|)}\I{Deligne--Mumford $C^\iy$-stack!orbifold|)}\I{C-stack@$C^\iy$-stack|)}\label{ag7cor3}
\end{cor}

\section{\texorpdfstring{Sheaves on Deligne--Mumford $C^\iy$-stacks}{Sheaves on Deligne--Mumford C∞-stacks}}
\label{ag8}
\I{Deligne--Mumford $C^\iy$-stack!quasicoherent sheaves on|(}

Next we discuss quasicoherent sheaves on Deligne--Mumford $C^\iy$-stacks $\cX$, generalizing \S\ref{ag5} for $C^\iy$-schemes. Some references on
sheaves on orbifolds or stacks are Behrend and Xu \cite[\S
3.1]{BeXu}, Deligne and Mumford \cite[Def.~4.10]{DeMu}, Heinloth
\cite[\S 4]{Hein}, Laumon and Moret-Bailly \cite[\S 13]{LaMo}, and
Moerdijk and Pronk \cite[\S 2]{MoPr}. Our definitions are closest to \cite{Hein,MoPr}. Almost everything in this section is an exercise in stack theory, not special to $C^\iy$-stacks, and would also work for sheaves (with \'etale descent) on other kinds of Deligne--Mumford stacks.

\subsection{Quasicoherent sheaves}
\label{ag81}

We build our notions of sheaves on Deligne--Mumford $C^\iy$-stacks $\cX$
from those of sheaves on $C^\iy$-schemes $\uU$ in \S\ref{ag5}, by lifting
to \'etale covers $\bar\uU\ra\cX$. Since all $\O_U$-modules on a $C^\iy$-scheme $\uU$ are quasicoherent by Corollary \ref{ag5cor1}, we do not distinguish between $\O_\cX$-modules and quasicoherent sheaves on a Deligne--Mumford $C^\iy$-stack $\cX$, and we will just call them quasicoherent sheaves.

\begin{dfn} Let $\cX$ be a Deligne--Mumford $C^\iy$-stack. Define
a category $\cC_\cX$ to have objects pairs $(\uU,u)$ where $\uU$ is
a $C^\iy$-scheme and $u:\bar\uU\ra\cX$ is an \'etale morphism, and
morphisms $(\uf,\eta):(\uU,u)\ra(\uV,v)$ where $\uf:\uU\ra\uV$ is an
\'etale morphism of $C^\iy$-schemes, and $\eta:u\Ra v\ci\buf$ is a
2-isomorphism. (Here $\uf$ \'etale is implied by $u,v$ \'etale and
$u\cong v\ci\buf$.) If $(\uf,\eta):(\uU,u)\ra(\uV,v)$ and
$(\ug,\ze):(\uV,v)\ra(\uW,w)$ are morphisms in $\cC_\cX$ then we
define the composition $(\ug,\ze)\ci(\uf,\eta)$ to be
$(\ug\ci\uf,\th):(\uU,u)\ra(\uW,w)$, where $\th$ is the composition
of 2-morphisms across the diagram:
\begin{equation*}
\xymatrix@C=35pt@R=11pt{ \bar\uU \ar[dr]^(0.6){\buf}
\ar@/^/@<1ex>[drrr]_(0.4)u \ar[dd]_{\overline{\ug\ci\uf}}
\dduppertwocell_{}\omit^{}\omit{<-2.5>^{\id}} & \\
& \bar\uV \ar[rr]^(0.5){v} \ar[dl]^(0.45){\bar\ug} &
\ultwocell_{}\omit^{}\omit{^\eta} & \cX.  \\
\bar\uW \ar@/_/@<-1ex>[urrr]^(0.35)w && {}
\ultwocell_{}\omit^{}\omit{^\ze} }
\end{equation*}

Define a {\it quasicoherent sheaf\/ $\cE$ on\/} $\cX$ to assign a quasicoherent sheaf $\cE(\uU,u)$ on $\uU$ for all objects $(\uU,u)$ in $\cC_\cX$, and an isomorphism $\cE_{(\uf,\eta)}:\uf^*(\cE(\uV,v))\ra\cE(\uU,u)$ in $\qcoh(\uU)$ for all morphisms $(\uf,\eta):(\uU,u)\ra(\uV,v)$ in $\cC_\cX$, such that for all $(\uf,\eta),(\ug,\ze), (\ug\ci\uf,\th)$ as above the following diagram of isomorphisms in $\qcoh(\uU)$ commutes:
\e
\begin{gathered}
\xymatrix@C=7pt@R=11pt{ (\ug\ci\uf)^*\bigl(\cE(\uW,w)\bigr)
\ar[rrrrr]_(0.4){\cE_{(\ug\ci\uf,\th)}}
\ar[dr]_{I_{\uf,\ug}(\cE(\uW,w))
\,\,\,\,\,\,\,\,\,\,\,\,{}} &&&&& \cE(\uU,u), \\
& \uf^*\bigl(\ug^*(\cE(\uW,w)\bigr)
\ar[rrr]^(0.55){\uf^*(\cE_{(\ug,\ze)})} &&&
\uf^*\bigl(\cE(\uV,v)\bigr) \ar[ur]_{{}\,\,\,\cE_{(\uf,\eta)}} }
\end{gathered}
\label{ag8eq1}
\e
for $I_{\uf,\ug}(\cE)$ as in Remark~\ref{ag5rem3}.

A {\it morphism of quasicoherent sheaves\/}\I{Deligne--Mumford $C^\iy$-stack!quasicoherent sheaves on!morphism} $\phi:\cE\ra\cF$ assigns a morphism $\phi(\uU,u):\cE(\uU,u)\ra\cF(\uU,u)$ in $\qcoh(\uU)$ for each object $(\uU,u)$ in $\cC_\cX$, such that for all morphisms $(\uf,\eta):(\uU,u)\ra(\uV,v)$ in $\cC_\cX$ the following commutes:
\begin{equation*}
\xymatrix@C=110pt@R=15pt{ *+[r]{\uf^*\bigl(\cE(\uV,v)\bigr)}
\ar[d]^{\uf^*(\phi(\uV,v))} \ar[r]_{\cE_{(\uf,\eta)}} & *+[l]{\cE(\uU,u)}
\ar[d]_{\phi(\uU,u)} \\ *+[r]{\uf^*\bigl(\cF(\uV,v)\bigr)}
\ar[r]^{\cF_{(\uf,\eta)}} & *+[l]{\cF(\uU,u).\!} }
\end{equation*}
We call $\cE$ a {\it vector bundle\I{Deligne--Mumford $C^\iy$-stack!vector bundles on} of rank\/} $n$ if $\cE(\uU,u)$ is a vector bundle of rank $n$ for all $(\uU,u)\in\cC_\cX$. Write $\qcoh(\cX)$\G[qcoh(X)b]{$\qcoh(\cX)$}{abelian category of quasicoherent sheaves on Deligne--Mumford $C^\iy$-stack $\cX$} for the category of quasicoherent sheaves on~$\cX$. 

\label{ag8def1}
\end{dfn}

\begin{rem}{\bf(a)} Here is a second way to define quasicoherent sheaves, closer to \cite[\S 3.1]{BeXu}, \cite[Def.~4.10]{DeMu}. Define a Grothendieck pretopology\I{Grothendieck pretopology} $\cP\cJ_\cX$ on $\cC_\cX$ to have coverings $\bigl\{(\ui_a,\eta_a):(\uU_a,u_a)\ra(\uU,u)
\bigr\}{}_{a\in A}$ where $\ui_a:\uU_a\ra\uU$ is an open embedding
for all $a\in A$ and $U=\bigcup_{a\in A}i_a(U_a)$. Let $\cJ_\cX$ be the associated Grothendieck topology.\I{Grothendieck topology} Then $(\cC_\cX,\cJ_\cX)$ is a site.\I{site}

We can now use the standard notion of {\it sheaves on a site}, as in
Artin \cite{Arti} or Metzler~\cite[\S 2.1]{Metz}. For all $(\uU,u)$
in $\cC_\cX$, define a $C^\iy$-ring $\O_\cX(\uU,u)=\O_U(U)$, where
$\uU=(U,\O_U)$. For all morphisms $(\uf,\eta):(\uV,v)\ra(\uU,u)$,
define a morphism of $C^\iy$-rings
$\rho_{(\uU,u)(\uV,v)}:\O_\cX(\uU,u)\ra\O_\cX(\uV,v)$ by
$\rho_{(\uU,u)(\uV,v)}=f_\sh(U):\O_U(U)\ra\O_V(V)$. Then $\O_\cX$ is
a {\it sheaf of\/ $C^\iy$-rings on the site\/} $(\cC_\cX,\cJ_\cX)$.

Define a {\it quasicoherent sheaf\/} $\cE'$ to be a sheaf of
$\O_\cX$-modules on $(\cC_\cX,\cJ_\cX)$. That is, $\cE'$ assigns
an $\O_\cX(\uU,u)$-module $\cE'(\uU,u)$ for all $(\uU,u)$ in
$\cC_\cX$, and a linear map $\cE'_{(\uf,\eta)}:\cE(\uU,u)\ra\cE(\uV,v)$ for all $(\uf,\eta):(\uV,v)\ra(\uU,u)$ in $\cC_\cX$, such that the analogue of \eq{ag5eq13} commutes, and the axioms for sheaves on a site hold.

If $\cE$ is as in Definition \ref{ag8def1} then defining
$\cE'(\uU,u)=\Ga\bigl(\cE(\uU,u)\bigr)$ gives a quasicoherent sheaf in
the sense of this second definition. Conversely, any quasicoherent sheaf in this second sense extends to one in the first sense uniquely up
to canonical isomorphism. Thus the two definitions yield equivalent
categories.

\smallskip

\noindent{\bf(b)} As quasicoherent sheaves are a kind of sheaves of sets
on a site,\I{site} not sheaves of categories on a site as stacks
are, $\qcoh(\cX)$ is a category not a 2-category.
\smallskip

\noindent{\bf(c)} In Definition \ref{ag8def1} we require the
1-morphisms $u,v,w$ and morphisms $\uf,\ug$ to be {\it \'etale}.
This is important in several places below: for instance, if
$\uf:\uU\ra\uV$ is \'etale then $\uf^*:\qcoh(\uV)\ra\qcoh(\uU)$ is
exact,\I{functor!exact} not just right exact,\I{functor!right exact}
which is needed in Proposition \ref{ag8prop1} to show $\qcoh(\cX)$ is abelian, and also $\Om_\uf:\uf^*(T^*\uV)\ra T^*\uU$ is an isomorphism, which is needed to define the cotangent sheaf~$T^*\cX$. We restricted to Deligne--Mumford $C^\iy$-stacks $\cX$ in order to be able to use
\'etale (1-)morphisms in this way. For $C^\iy$-stacks $\cX$ which do
not admit an \'etale atlas,\I{atlas!etale@\'etale}\I{C-stack@$C^\iy$-stack!atlas!etale@\'etale} the approach above is inadequate and would need to be modified.
\smallskip

\noindent{\bf(d)} Our notion of vector bundles $\cE$ over $\cX$
correspond to {\it orbifold vector bundles\/} when $\cX$ is an
orbifold.\I{orbifold!vector bundle} That is, the isotropy groups
$\Iso_\cX([x])$\I{C-stack@$C^\iy$-stack!isotropy group
$\Iso_\cX([x])$} of $\cX$ are allowed to act nontrivially on the
vector space fibres $\cE\vert_x$ of~$\cE$.
\smallskip

\noindent{\bf(e)} We can also use the method of Definition \ref{ag8def1}  (or the approach of {\bf(a)}) to define other kinds of sheaves on a Deligne--Mumford $C^\iy$-stack $\cX$, such as sheaves of sets, abelian groups, $C^\iy$-rings, \ldots, in the obvious way: we just take the $\cE(\uU,u)$ to be a sheaf of sets, \ldots\ on $\uU$ instead of a quasicoherent sheaf.
\label{ag8rem1}
\end{rem}

\begin{prop} Let\/ $\cX$ be a Deligne--Mumford\/ $C^\iy$-stack.
Then\/ $\qcoh(\cX)$ is an abelian category.\I{abelian category}\label{ag8prop1}
\end{prop}

\begin{proof} We define a complex in $\qcoh(\cX)$
\begin{equation*}
\smash{\xymatrix@C=30pt{ 0 \ar[r] & \cE \ar[r]^\phi & \cF \ar[r]^\psi & \cG \ar[r] & 0 }}
\end{equation*} 
to be {\it exact\/} if and only if
\begin{equation*}
\smash{\xymatrix@C=35pt{ 0 \ar[r] & \cE(\uU,u) \ar[r]^{\phi(\uU,u)} & \cF(\uU,u)
\ar[r]^{\psi(\uU,u)} & \cG(\uU,u) \ar[r] & 0 }}
\end{equation*}
is exact in $\qcoh(\uU)$ for all $(\uU,u)$ in $\cC_\cX$. Since each $\qcoh(\uU)$ in Definition \ref{ag8def1} is abelian, and the functors $\uf^*$ in Definition \ref{ag8def1} are exact\I{functor!exact} (not just right exact)\I{functor!right exact} as $\uf$ is \'etale, it is easy to show this makes $\qcoh(\cX)$ into an abelian category.
\end{proof}

\begin{ex} Let $\uX$ be a $C^\iy$-scheme. Then $\cX=\bar\uX$ is
a Deligne--Mumford $C^\iy$-stack. We will define an equivalence of categories $\cI_\uX:\qcoh(\uX)\ra\qcoh(\cX)$.\G[IX]{$\cI_\uX:\qcoh(\uX)\ra\qcoh(\cX)$}{inclusion functor from sheaves on a $C^\iy$-scheme $\uX$ to sheaves on the associated Deligne--Mumford $C^\iy$-stack $\cX=\bar\uX$} 

Let $\cE$ be an object in $\qcoh(\uX)$. If $(\uU,u)$ is an object in
$\cC_\cX$ then $u:\bar\uU\ra\cX=\bar\uX$ is a 1-morphism, so as
$\CSch,\bCSch$ are equivalent (2-)categories $u$ is 2-isomorphic to
$\bar\uu:\bar\uU\ra\bar\uX$ for some unique morphism
$\uu:\uU\ra\uX$. Define $\cE'(\uU,u)=\uu^*(\cE)$. If
$(\uf,\eta):(\uU,u)\ra(\uV,v)$ is a morphism in $\cC_\cX$ and
$\uu,\uv$ are associated to $u,v$ as above, so that $\uu=\uv\ci\uf$,
then define
\begin{equation*}
\smash{\cE'_{(\uf,\eta)}= I_{\uf,\uv}(\cE)^{-1}:\uf^*(\cE'(\uV,v))=
\uf^*\bigl(\uv^*(\cE)\bigr)\ra (\uv\ci\uf)^*(\cE)=\cE'(\uU,u).}
\end{equation*}
Then \eq{ag8eq1} commutes for all $(\uf,\eta),(\ug,\ze)$, so $\cE'$
is a quasicoherent sheaf on~$\cX$.

If $\phi:\cE\ra\cF$ is a morphism in $\qcoh(\uX)$ define a morphism $\phi':\cE'\ra\cF'$ in $\qcoh(\cX)$ by $\phi'(\uU,u)=\uu^*(\phi)$ for $\uu$ associated to $u$ as above. Then defining $\cI_\uX:\cE\mapsto\cE'$, $\cI_\uX:\phi\mapsto\phi'$ gives a functor $\qcoh(\uX)\ra\qcoh(\cX)$. There is a natural inverse construction: if $\ti{\cal E}$ is an object in $\qcoh(\cX)$ then $\ti{\cal E}(\uX,\bar\uid_\uX)$ is an object in $\qcoh(\uX)$, and $\ti{\cal E}$ is canonically isomorphic to $\cI_\uX\bigl(\ti{\cal E}(\uX,\bar\uid{}_\uX)\bigr)$. Using this we can show $\cI_\uX$ is an equivalence of categories.
\label{ag8ex1}
\end{ex}

\subsection{Writing sheaves in terms of a groupoid presentation}
\label{ag82}
\I{groupoid object|(}\I{category!groupoid object in|(}\I{stack!associated to a
groupoid|(}\I{C-stack@$C^\iy$-stack!associated to a groupoid|(}

Let $\cX$ be a Deligne--Mumford $C^\iy$-stack. Then $\cX$ admits an
\'etale
atlas\I{atlas!etale@\'etale}\I{C-stack@$C^\iy$-stack!atlas!etale@\'etale}
$\Pi:\bar\uU\ra\cX$, and as in \S\ref{agA5} from $\Pi$ we can
construct a groupoid $(\uU,\uV,\us,\ut,\uu,\ui,\um)$ in $\CSch$,
with $\us,\ut:\uV\ra\uU$ \'etale, such that $\cX$ is equivalent to
the associated $C^\iy$-stack $[\uV\rra\uU]$, and we have a 2-Cartesian
diagram\I{2-category!2-Cartesian square}
\begin{equation*}
\xymatrix@C=100pt@R=11pt{ *+[r]{\bar\uV} \ar[r]_(0.3){\bar\ut} \ar[d]^{\bar\us}
\drtwocell_{}\omit^{}\omit{^\eta} & *+[l]{\bar{\uU}} \ar[d]_\Pi \\
*+[r]{\bar{\uU}} \ar[r]^(0.7)\Pi & *+[l]{\cX.\!} }
\end{equation*}
We can now consider the objects $(\uU,\Pi)$ and $(\uV,\Pi\ci\us)$ in
$\cC_\cX$, and the two morphisms
$(\us,\id_{\Pi\ci\us}):(\uV,\Pi\ci\us)\ra(\uU,\Pi)$
and~$(\ut,\eta):(\uV,\Pi\ci\us)\ra(\uU,\Pi)$.

Now let $\cE$ be an object in $\qcoh(\cX)$. Then we have quasicoherent sheaves $E=\cE(\uU,\Pi)$ on $\uU$ and $E'=\cE(\uV,\Pi\ci\us)$ on $\uV$, and isomorphisms $\cE_{(\us,\id_{\Pi\ci\us})}:\us^*(E)\ra E'$ and
$\cE_{(\ut,\eta)}:\ut^*(E)\ra E'$ in $\qcoh(\uV)$. Hence
$\Phi=\cE_{(\ut,\eta)}^{-1}\ci \cE_{(\us,\id_{\Pi\ci\us})}$ is an
isomorphism of $\Phi:\us^*(E)\ra\ut^*(E)$ in~$\qcoh(\uV)$.

We also have a 2-commutative diagram with all squares 2-Cartesian:
\begin{equation*}
\xymatrix@!0@C=65pt@R=21pt{ & \bar\uW \ar[dl]_(0.6){\bar\upi_1}
\ar[ddr]^(0.4){\bar\upi_2} \ar[rr]_(0.55){\bar\um}
&& \bar\uV \ar[dl]^(0.4){\bar\ut} \ar[ddr]^{\bar\us} \\
\bar\uV \ar[ddr]_{\bar\us} \ar[rr]_(0.4){\bar\ut} && \bar\uU
\ar[ddr]^(0.35){\Pi} \\ && \bar\uV \ar[rr]^(0.6){\bar\us}
\ar[dl]_(0.4){\bar\ut} && \bar\uU \ar[dl]^(0.4){\Pi} \\
& \bar\uU \ar[rr]^(0.6){\Pi} && \cX,}
\end{equation*}
omitting 2-morphisms, where $\uW=\uV\t_{\bar\us,\bar\uU,\ut}
\bar\uV$, and $\upi_1,\upi_2:\uW\ra\uV$ are projections to the first
and second factors in the fibre
product.\I{C-stack@$C^\iy$-stack!fibre products} So we have an
object $(\uW,\Pi\ci\bar\us\ci\bar\upi_1)$ in $\cC_\cX$, and we can
define $E''=\cE(\uW,\Pi\ci\bar\us\ci\bar\upi_1)$. Then we have a
commutative diagram of isomorphisms in $\qcoh(\uW)$:
\e
\begin{gathered}
\xymatrix@!0@C=62pt@R=37pt{ & E'' && \um^*(E')
\ar[ll]^(0.45){\cE_{(\um,\th_3)}}
\\
\upi_1^*(E') \ar[ur]^(0.4){\cE_{(\upi_1,\th_1)}} &&
{\begin{subarray}{l}\ts(\ut\ci\upi_1)^*(E)=\\
{}\quad\ts(\ut\ci\um)^*(E)\end{subarray}}
\ar[ll]^(0.6){\begin{subarray}{l} \upi_1^*(\cE_{(\ut,\eta)})\\
\ci I_{\upi_1,\ut}(E)\end{subarray}}
\ar[ur]_(0.7){\begin{subarray}{l} \um^*(\cE_{(\ut,\eta)})\\
\ci I_{\um,\ut}(E)\end{subarray}}
\\
&& \upi_2^*(E') \ar@<1ex>[uul]_(0.75){\cE_{(\upi_2,\th_2)}} &&
{\begin{subarray}{l}\ts(\us\ci\upi_2)^*(E)=\\
\ts(\us\ci\um)^*(E)\end{subarray}} \ar[uul]_{\begin{subarray}{l}
\um^*(\cE_{(\us,\id_{\Pi\ci\us})})\\
\ci I_{\um,\us}(E)\end{subarray}} \ar[ll]^(0.6){\begin{subarray}{l}
\upi_2^*(\cE_{(\us,\id_{\Pi\ci\us})})\\ \ci
I_{\upi_2,\us}(E)\end{subarray}} \ar@{-->}@(d,l)[dlll]_(0.53){\be}
\ar@{-->}[ull]^{\al}
\\
& {\begin{subarray}{l}\ts(\us\ci\upi_1)^*(E)= (\ut\ci\upi_2)^*(E)
\end{subarray}}\ar[uul]^(0.65){\begin{subarray}{l}
\upi_1^*(\cE_{(\us,\id_{\Pi\ci\us})}) \\ \ci
I_{\upi_1,\us}(E)\end{subarray}}
\ar[ur]^(0.4){\begin{subarray}{l} \upi_2^*(\cE_{(\ut,\eta)})\\
\ci I_{\upi_2,\ut}(E)\end{subarray}} \ar@{-->}[uur]_(0.6){\ga} }\!\!\!\!{}
\end{gathered}
\label{ag8eq2}
\e
Here the morphisms `$\dashra$' are given by $\al=I_{\um,\ut}(E
)^{-1}\ci \um^*(\Phi)\ci I_{\um,\us}(E)$,
$\be=I_{\upi_2,\ut}(E)^{-1}\ci\upi_2^*(\Phi)\ci I_{\upi_2,\us}(E)$
and $\ga=I_{\upi_1,\ut}(E)^{-1} \ci\upi_1^*(\Phi)\ci
I_{\upi_1,\us}(E)$, and as \eq{ag8eq2} commutes we have
$\al=\ga\ci\be$. This motivates:

\begin{dfn} Let $(\uU,\uV,\us,\ut,\uu,\ui,\um)$ be a groupoid in
$\CSch$, with $\us,\ut:\uV\ra\uU$ \'etale, which we write as
$\uV\rra\uU$ for short. Define a {\it quasicoherent sheaf on\/} $\uV\rra\uU$ to be a pair $(E,\Phi)$ where $E$ is a quasicoherent sheaf on $\uU$ and $\Phi:\us^*(E)\ra\ut^*(E)$ is an isomorphism in $\qcoh(\uV)$, such that
\begin{align*}
I_{\um,\ut}(E)^{-1}\ci\um^*(\Phi)\ci I_{\um,\us}(E)=\,
&\bigl(I_{\upi_1,\ut}(E)^{-1}\ci\upi_1^*(\Phi)\ci
I_{\upi_1,\us}(E)\bigr)\ci\\
&\bigl(I_{\upi_2,\ut}(E)^{-1}\ci\upi_2^*(\Phi)\ci
I_{\upi_2,\us}(E)\bigr)
\end{align*}
in morphisms $(\us\ci\um)^*(E)\ra(\ut\ci\um)^*(E)$ in $\qcoh(\uW)$. Define a {\it morphism\/} $\phi:(E,\Phi)\ra(F,\Psi)$ of such sheaves to be a morphism $\phi:E\ra F$ in $\qcoh(\uU)$ such that
$\Psi\ci\us^*(\phi)=\ut^*(\phi)\ci\Phi:\us^*(E)\ra\ut^*(F)$ in $\qcoh(\uV)$. Then quasicoherent sheaves on $\uV\rra\uU$ form an abelian category\I{abelian category} $\qcoh(\uV\rra\uU)$.\G[qcohVU]{$\qcoh(\uV\rra\uU)$}{category of quasicoherent sheaves on a groupoid $\uV\rra\uU$} 

If $\cX$ is a Deligne--Mumford $C^\iy$-stack equivalent to $[\uV\rra\uU]$ with atlas\I{atlas}\I{C-stack@$C^\iy$-stack!atlas} $\Pi:\bar\uU\ra\cX$ then we have a functor $F_\Pi:\qcoh(\cX)\ra\qcoh(\uV\rra\uU)$ defined by $F_\Pi:\cE\mapsto\bigl(\cE(\uU,\Pi),\cE_{(\ut,\eta)}^{-1}\ci\cE_{(\us,\id_{\Pi\ci\us})}\bigr)$ and~$F_\Pi:\phi\mapsto\phi(\uU,\Pi)$.
\label{ag8def2}
\end{dfn}

The next theorem is proved as in Laumon and Moret-Bailly \cite[Prop.~12.4.5]{LaMo} or Olsson \cite[Prop.~4.4]{Olss1}.

\begin{thm} The functor $F_\Pi:\qcoh(\cX)\ra\qcoh(\uV\rra\uU)$ above is an equivalence of categories.
\label{ag8thm1}
\end{thm}

For quotient $C^\iy$-stacks $[\uU/G]$ with $G$ a finite group, so
that $\uV=G\t\uU$, a quasicoherent sheaf $(E,\Phi)$ on $\uV\rra\uU$ is a quasicoherent sheaf $E$ on $\uU$ with a lift $\Phi$ of the $G$-action on $\uU$ up to $E$. That is, $(E,\Phi)$ is a $G$-{\it equivariant quasicoherent sheaf on\/} $\uU$. Hence, if a Deligne--Mumford $C^\iy$-stack $\cX$ is equivalent to a quotient $[\uU/G]$ with $G$ finite, then
$\qcoh(\cX)$ is equivalent to the category $\qcoh^G(\uU)$\G[qcoh(X)c]{$\qcoh^G(\uX)$}{abelian category of $G$-equivariant quasicoherent sheaves on a $C^\iy$-scheme $\uX$ acted on by a finite group $G$} of $G$-equivariant quasicoherent sheaves on~$\uU$.\I{groupoid object|)}\I{category!groupoid object in|)}\I{stack!associated to a groupoid|)}\I{C-stack@$C^\iy$-stack!associated to a groupoid|)}

\subsection{Pullback of sheaves as a weak 2-functor}
\label{ag83}
\I{Deligne--Mumford $C^\iy$-stack!quasicoherent sheaves on!pullback|(}

In Definition \ref{ag5def4}, for a morphism of $C^\iy$-schemes
$\uf:\uX\ra\uY$ we defined a right exact functor\I{functor!right
exact} $\uf^*:\qcoh(\uY)\ra\qcoh(\uX)$. As in Remarks \ref{ag4rem1}(b) and
\ref{ag5rem3}, pullbacks cannot always be made strictly functorial
in $\uf$, that is, we do not have
$\uf^*(\ug^*(\cE))=(\ug\ci\uf)^*(\cE)$ for all $\uf:\uX\ra\uY$,
$\ug:\uY\ra\uZ$ and $\cE\in\qcoh(\uZ)$, but instead we have canonical
isomorphisms~$I_{\uf,\ug}(\cE):(\ug\ci\uf)^*(\cE)\ra\uf^*(\ug^*(\cE))$.

We now generalize this to pullback for sheaves on Deligne--Mumford
$C^\iy$-stacks. The new factor to consider is that we have not only
1-morphisms $f:\cX\ra\cY$, but also 2-morphisms $\eta:f\Ra g$ for
1-morphisms $f,g:\cX\ra\cY$, and we must interpret pullback for
2-morphisms as well as 1-morphisms.

\begin{dfn} Let $f:\cX\ra\cY$ be a 1-morphism of Deligne--Mumford $C^\iy$-stacks, and $\cF\in\qcoh(\cY)$. A {\it pullback\/} of $\cF$ to $\cX$ is $\cE\in\qcoh(\cX)$, together with the following data: if $\uU,\uV$ are $C^\iy$-schemes and $u:\bar\uU\ra\cX$ and $v:\bar\uV\ra\cY$ are \'etale 1-morphisms, then there is a $C^\iy$-scheme $\uW$ and morphisms $\upi_\uU:\uW\ra\uU$, $\upi_\uV:\uW\ra\uV$ giving a 2-Cartesian diagram:\I{2-category!2-Cartesian square}
\e
\begin{gathered}
\xymatrix@C=90pt@R=14pt{ *+[r]{\bar\uW} \ar[r]_(0.3){\bar\upi_\uV}
\ar[d]^{\bar\upi_\uU}
\drtwocell_{}\omit^{}\omit{^\ze} & *+[l]{\bar\uV} \ar[d]_v \\
*+[r]{\bar\uU} \ar[r]^(0.7){f\ci u} & *+[l]{\cY.\!} }
\end{gathered}
\label{ag8eq3}
\e
Then an isomorphism $i(\cF,f,u,v,\ze):\ab\upi^*_\uU\bigl(\cE(\uU,u)\bigr)\ra
\upi^*_\uV\bigl(\cF(\uV,v)\bigr)$ in $\qcoh(\uW)$ should be given,
which is functorial in $(\uU,u)$ in $\cC_\cX$ and $(\uV,v)$ in
$\cC_\cY$ and the 2-isomorphism $\ze$ in \eq{ag8eq3}. We usually
write pullbacks $\cE$ as~$f^*(\cF)$.\G[fEd]{$f^*(\cE)$}{pullback of
quasicoherent sheaf $\cE$ under $f:\cX\ra\cY$}
\label{ag8def3}
\end{dfn}

By a similar proof to Theorem \ref{ag8thm1}, but using descent for
objects and morphisms for quasicoherent sheaves on $C^\iy$-schemes $\uY$ in
the \'etale topology rather than the open cover topology on $\uY$,
we can prove:

\begin{prop} Let\/ $f:\cX\ra\cY$ be a\/ $1$-morphism of Deligne--Mumford\/ $C^\iy$-stacks, and\/ $\cF$ be a quasicoherent sheaf on $\cY$. Then a pullback\/ $f^*(\cF)$ exists in $\qcoh(\cX),$ and is unique up to canonical isomorphism.
\label{ag8prop2}
\end{prop}

From now on we will assume that we have {\it chosen\/} a pullback
$f^*(\cF)$ for all such $f:\cX\ra\cY$ and $\cF$. This could be done
either by some explicit construction of pullbacks, as in the
$C^\iy$-scheme case in \S\ref{ag53}, or by using the Axiom of
Choice.\I{Axiom of Choice} As in Remark \ref{ag5rem3} we cannot
necessarily make these choices functorial in~$f$.

\begin{dfn} Choose pullbacks $f^*(\cF)$ for all 1-morphisms
$f:\cX\ra\cY$ of Deligne--Mumford $C^\iy$-stacks and all
$\cF\in\qcoh(\cY)$, as above.

Let $f:\cX\ra\cY$ be such a 1-morphism, and $\phi:\cE\ra\cF$ be a
morphism in $\qcoh(\cY)$. Then $f^*(\cE),f^*(\cF)\in\qcoh(\cX)$. Define
the {\it pullback morphism\/} $f^*(\phi):f^*(\cE)\ra f^*(\cF)$ to be
the morphism in $\qcoh(\cX)$ characterized as follows. Let
$u:\bar\uU\ra\cX$, $v:\bar\uV\ra\cY$, $\uW,\upi_\uU,\upi_\uV$ be as
in Definition \ref{ag8def3}, with \eq{ag8eq3} 2-Cartesian. Then
the following diagram of morphisms in $\qcoh(\uW)$ commutes:
\begin{equation*}
\xymatrix@C=140pt@R=15pt{ *+[r]{\upi^*_\uU\bigl(f^*(\cE)(\uU,u)\bigr)}
\ar[r]_(0.53){i(\cE,f,u,v,\ze)} \ar[d]^{\pi_\uU^*(f^*(\phi)(\uU,u))} &
*+[l]{\upi^*_\uV\bigl(\cE(\uV,v)\bigr)} \ar[d]_{\pi_\uV^*(\phi(\uV,v))} \\
*+[r]{\upi^*_\uU\bigl(f^*(\cF)(\uU,u)\bigr)} \ar[r]^(0.53){i(\cF,f,u,v,\ze)} &
*+[l]{\upi^*_\uV\bigl(\cF(\uV,v)\bigr).\!} }
\end{equation*}
Using descent for morphisms in $\qcoh(\uY)$ on $C^\iy$-schemes
$\uY$ in the \'etale topology, one can show that there is a unique
morphism $f^*(\phi)$ with this property. This defines a functor $f^*:\qcoh(\cY)\ra\qcoh(\cX)$. 

Let $f:\cX\ra\cY$ and $g:\cY\ra\cZ$ be 1-morphisms of
Deligne--Mumford $C^\iy$-stacks, and $\cE\in\qcoh(\cZ)$. Then $(g\ci
f)^*(\cE)$ and $f^*(g^*(\cE))$ both lie in $\qcoh(\cX)$. One can show
that $f^*(g^*(\cE))$ is a possible pullback of $\cE$ by $g\ci f$.
Thus as in Remark \ref{ag5rem3}, we have a canonical isomorphism
$I_{f,g}(\cE): (g\ci f)^*(\cE)\ra f^*(g^*(\cE))$. This defines a
natural isomorphism of functors~$I_{f,g}:(g\ci f)^*\Ra f^*\ci g^*$.

Let $f,g:\cX\ra\cY$ be 1-morphisms of Deligne--Mumford
$C^\iy$-stacks, $\eta:f\Ra g$ a 2-morphism, and $\cE\in\qcoh(\cY)$.
Then we have $f^*(\cE),g^*(\cE)\in\qcoh(\cX)$. Let $u:\bar\uU\ra\cX$, $v:\bar\uV\ra\cY$, $\uW,\upi_\uU,\upi_\uV$ be as in Definition \ref{ag8def3}. Then as in \eq{ag8eq3} we have 2-Cartesian diagrams\I{2-category!2-Cartesian square}
\begin{equation*}
\xymatrix@C=30pt@R=13pt{ *+[r]{\bar\uW} \ar[rrr]_(0.8){\bar\upi_\uV}
\ar[d]^{\bar\upi_\uU}  &
\drrtwocell_{}\omit^{}\omit{^{\!\!\!\!\!\!\!\!\!\!\!
\!\!\!\!\!\!\!\!\!\!\!\!\!\ze\od(\eta*\id_{u\ci\bar\upi_\uU})
\,\,\,\,\,\,\,\,\,\,\,\,{}}} && *+[l]{\bar\uV} \ar[d]_v & *+[r]{\bar\uW}
\ar[rrr]_(0.7){\bar\upi_\uV} \ar[d]^{\bar\upi_\uU}
& \drtwocell_{}\omit^{}\omit{^\ze} && *+[l]{\bar\uV} \ar[d]_v \\
*+[r]{\bar\uU} \ar[rrr]^(0.8){f\ci u} &&& *+[l]{\cY,\!} & *+[r]{\bar\uU}
\ar[rrr]^(0.7){g\ci u} &&& *+[l]{\cY,\!} }
\end{equation*}
where in $\ze\od(\eta*\id_{u\ci\bar\upi_\uU})$ `$*$' is horizontal
composition\I{2-category!2-morphism!horizontal composition} and
`$\od$' vertical\I{2-category!2-morphism!vertical composition}
composition of 2-morphisms. Thus we have isomorphisms in $\qcoh(\uW)$:
\begin{equation*}
\xymatrix@C=100pt@R=0pt{ \upi^*_\uU\bigl(f^*(\cE)(\uU,u)\bigr)
\ar[dr]^{i(\cE,f,u,v,\ze\od(\eta*\id_{u\ci\bar\upi_\uU}))}
\ar@{..>}[dd] \\
& \upi^*_\uV\bigl(\cE(\uV,v)\bigr). \\
\upi^*_\uU\bigl(g^*(\cE)(\uU,u)\bigr) \ar[ur]_{i(\cE,g,u,v,\ze)}}
\end{equation*}
There is a unique isomorphism `$\dashra$' making this diagram
commute. Taken over all $(\uV,v)$, using descent for morphisms we
can show these isomorphisms are pullbacks of a unique isomorphism
$f^*(\cE)(\uU,u)\ra g^*(\cE)(\uU,u)$, and taken over all $(\uU,u)$
these give an isomorphism $\eta^*(\cE):f^*(\cE)\ra g^*(\cE)$ in $\qcoh(\cX)$. Over all $\cE\in\qcoh(\cY)$, this defines a natural isomorphism~$\eta^*:f^*\Ra g^*$.

If $\cX$ is a Deligne--Mumford $C^\iy$-stack with identity
1-morphism $\id_\cX:\cX\ra\cX$ then for each $\cE\in\qcoh(\cX)$, $\cE$
is a possible pullback $\id_\cX^*(\cE)$, so we have a canonical
isomorphism $\de_\cX(\cE): \id_\cX^*(\cE)\ra\cE$. These define a
natural isomorphism~$\de_\cX:\id_\cX^*\Ra\id_{\qcoh(\cX)}$.
\label{ag8def4}
\end{dfn}

The proof of the next theorem is long but straightforward. Weak $2$-functors are defined in Definition~\ref{agAdef2}.

\begin{thm} Mapping $\cX$ to $\qcoh(\cX)$ for objects\/ $\cX$ in
$\DMCSta,$ and mapping $1$-morphisms\/ $f:\cX\ra\cY$ to\/
$f^*:\qcoh(\cY)\ra\qcoh(\cX),$ and mapping\/ $2$-morphisms\/ $\eta:f\Ra g$
to\/ $\eta^*:f^*\Ra g^*$ for $1$-morphisms $f,g:\cX\ra\cY,$ and the
natural isomorphisms\/ $I_{f,g}:(g\ci f)^*\Ra f^*\ci g^*$ for all\/
$1$-morphisms\/ $f:\cX\ra\cY$ and\/ $g:\cY\ra\cZ$ in $\DMCSta,$
and\/ $\de_\cX$ for all\/ $\cX\in\DMCSta,$ together make up a weak\/ $2$-functor $(\DMCSta)^{\bf op}\ra\mathop{\bf AbCat},$ where $\mathop{\bf AbCat}$ is the\/
$2$-category of abelian categories.\I{abelian category} That is,
they satisfy the conditions:
\begin{itemize}
\setlength{\itemsep}{0pt}
\setlength{\parsep}{0pt}
\item[{\bf(a)}] If\/ $f:\cW\ra\cX,$ $g:\cX\ra\cY,$
$h:\cY\ra\cZ$ are $1$-morphisms in $\DMCSta$ and\/
$\cE\in\qcoh(\cZ)$ then the following diagram commutes
in\/~$\qcoh(\cX):$
\begin{equation*}
\xymatrix@C=120pt@R=15pt{ *+[r]{(h\ci g\ci f)^*(\cE)}
\ar[r]_{I_{f,h\ci g}(\cE)} \ar[d]^{I_{g\ci f,h}(\cE)}
& *+[l]{f^*\bigl((h\ci g)^*(\cE)\bigr)} \ar[d]_{f^*(I_{g,h}(\cE))} \\
*+[r]{(g\ci f)^*\bigl(h^*(\cE)\bigr)} \ar[r]^{I_{f,g}(h^*(\cE))} &
*+[l]{f^*\bigl(g^*(h^*(\cE))\bigr).\!} }
\end{equation*}
\item[{\bf(b)}] If\/ $f:\cX\ra\cY$ is a $1$-morphism in\/
$\DMCSta$ and\/ $\cE\in\qcoh(\cY)$ then the following pairs of
morphisms in\/ $\qcoh(\cX)$ are inverse:
\begin{equation*}
\xymatrix@C=8pt@R=10pt{ {\begin{subarray}{l}\ts f^*(\cE)=\\
\ts (f\!\ci\!\id_\cX)^*(\cE)\end{subarray}}
\ar@/^/[rrr]^{I_{\id_\cX,f}(\cE)} &&& \id_\cX^*(f^*(\cE)),
\ar@/^/[lll]^{\de_\cX(f^*(\cE))} &
{\begin{subarray}{l}\ts f^*(\cE)=\\ \ts (\id_\cY\!\ci\!
f)^*(\cE)\end{subarray}} \ar@/^/[rrr]^{I_{f,\id_\cY}(\cE)} &&&
f^*(\id_\cY^*(\cE)). \ar@/^/[lll]^{f^*(\de_\cY(\cE))} }
\end{equation*}
Also $(\id_f)^*(\id_{\cE})=\id_{f^*(\cE)}:f^*(\cE)\ra f^*(\cE)$.
\item[{\bf(c)}] If\/ $f,g,h:\cX\ra\cY$ are
$1$-morphisms and\/ $\eta:f\Ra g,$ $\ze:g\Ra h$ are
$2$-morphisms in\/ $\DMCSta,$ so that\/ $\ze\od\eta:f\Ra h$ is
the vertical\I{2-category!2-morphism!vertical composition}
composition, and\/ $\cE\in\qcoh(\cY),$ then
\begin{equation*}
\ze^*(\cF)\ci\eta^*(\cE)=(\ze\od\eta)^*(\cE):f^*(\cE)\longra h^*(\cE)
\quad\text{in\/ $\qcoh(\cX)$.}
\end{equation*}
\item[{\bf(d)}] If\/ $f,\ti f:\cX\ra\cY,$ $g,\ti g:\cY\ra\cZ$
are $1$-morphisms and\/ $\eta:f\Ra f',$ $\ze:g\Ra g'$
$2$-morphisms in\/ $\DMCSta,$ so that\/ $\ze*\eta:g\ci f\Ra \ti
g\ci\ti f$ is the horizontal
composition,\I{2-category!2-morphism!horizontal composition}
and\/ $\cE\in\qcoh(\cZ),$ then the following commutes
in\/~$\qcoh(\cX):$
\begin{equation*}
\xymatrix@C=70pt@R=14pt{ *+[r]{(g\ci f)^*(\cE)}
\ar[rr]_{(\ze*\eta)^*(\cE)} \ar[d]^{I_{f,g}(\cE)} &&
*+[l]{(\ti g\ci\ti f)^*(\cE)} \ar[d]_{I_{\ti f,\ti g}(\cE)} \\
*+[r]{f^*(g^*(\cE))} \ar[r]^{\eta^*(g^*(\cE))} & \ti
f^*(g^*(\cE)) \ar[r]^{\ti f^*(\ze^*(\cE))} & *+[l]{\ti
f^*(\ti g^*(\cE)).\!} }
\end{equation*}
\end{itemize}
\label{ag8thm2}
\end{thm}

Using Proposition \ref{ag5prop4} we may prove:

\begin{prop} Let\/ $f:\cX\ra\cY$ be a $1$-morphism of
Deligne--Mumford\/ $C^\iy$-stacks. Then pullback\/
$f^*:\qcoh(\cY)\ra\qcoh(\cX)$ is a right exact functor.\I{functor!right
exact}\I{Deligne--Mumford $C^\iy$-stack!quasicoherent sheaves on!pullback|)}\label{ag8prop3}
\end{prop}

\subsection{\texorpdfstring{Cotangent sheaves of Deligne--Mumford $C^\iy$-stacks}{Cotangent sheaves of Deligne--Mumford C∞-stacks}}
\label{ag84}
\I{Deligne--Mumford $C^\iy$-stack!cotangent sheaf|(}

We now develop the analogue of the ideas of~\S\ref{ag56}.

\begin{dfn} Let $\cX$ be a Deligne--Mumford $C^\iy$-stack. Define
a quasicoherent sheaf $T^*\cX$ on $\cX$ called the {\it cotangent sheaf\/} of $\cX$ by $(T^*\cX)(\uU,u)=T^*\uU$ for all objects $(\uU,u)$ in
$\cC_\cX$ and $(T^*\cX)_{(\uf,\eta)}=\Om_\uf:\uf^*(T^*\uV)\ra T^*\uU$ for all morphisms $(\uf,\eta):(\uU,u)\ra(\uV,v)$ in $\cC_\cX$, where $T^*\uU$ and $\Om_\uf$ are as in \S\ref{ag56}. Here as $\uf:\uU\ra\uV$ is \'etale $\Om_\uf$ is an isomorphism, so $(T^*\cX)_{(\uf,\eta)}$ is an isomorphism in $\qcoh(\uU)$ as required. Also Theorem \ref{ag5thm4}(a) shows that \eq{ag8eq1} commutes for $\cE=T^*\cX$ for all such $(\uf,\eta),(\ug,\ze)$. Hence $T^*\cX$ is a quasicoherent sheaf.

Let $f:\cX\ra\cY$ be a 1-morphism of Deligne--Mumford
$C^\iy$-stacks. Define $\Om_f:f^*(T^*\cY)\ra T^*\cX$ to be the unique morphism in $\qcoh(\cX)$ characterized as follows. Let $u:\bar\uU\ra\cX$, $v:\bar\uV\ra\cY$,
$\uW,\upi_\uU,\upi_\uV$ be as in Definition \ref{ag8def3}, with
\eq{ag8eq3} Cartesian. Then the following diagram in $\qcoh(\uW)$ commutes:
\begin{equation*}
\xymatrix@C=30pt@R=15pt{ *+[r]{\upi^*_\uU\bigl(f^*(T^*\cY)(\uU,u)\bigr)}
\ar[d]^{\pi_\uU^*(\Om_f(\uU,u))} \ar[rrr]_(0.6){i(T^*\cY,f,u,v,\ze)}
&&& \upi^*_\uV\bigl((T^*\cY)(\uV,v)\bigr) \ar@{=}[r] &
*+[l]{\upi^*_\uV(T^*\uV)} \ar[d]_{\Om_{\upi_\uV}} \\
*+[r]{\upi^*_\uU\bigl((T^*\cX)(\uU,u)\bigr)}
\ar[rrr]^(0.57){(T^*\cX)_{(\upi_\uU,\id_{u\ci\upi_\uU})}}
&&& (T^*\cX)(\uW,u\ci\upi_\uU) \ar@{=}[r] & *+[l]{T^*\uW.\!} }
\end{equation*}
This determines $\pi_\uU^*(\Om_f(\uU,u))$ uniquely. Over all
$(\uV,v)$, using descent for morphisms in $\qcoh(\uU)$ on
$C^\iy$-schemes $\uU$ in the \'etale topology, this determines the
morphisms $\Om_f(\uU,u)$, and over all $(\uU,u)$ these
determine~$\Om_f$.
\label{ag8def5}
\end{dfn}

If $\cX$ is an orbifold\I{orbifold!vector bundle} of dimension $n$ then $T^*\cX$ is a vector bundle of rank $n$. Here is the analogue of Theorem~\ref{ag5thm4}:

\begin{thm}{\bf(a)} Let\/ $f:\cX\ra\cY$ and\/ $g:\cY\ra \cZ$ be\/
$1$-morphisms of Deligne--Mumford\/ $C^\iy$-stacks. Then
\e
\Om_{g\ci f}=\Om_f\ci f^*(\Om_g)\ci I_{f,g}(T^*\cZ)
\label{ag8eq4}
\e
as morphisms\/ $(g\ci f)^*(T^*\cZ)\ra T^*\cX$ in\/~$\qcoh(\cX)$.
\smallskip

\noindent{\bf(b)} Let\/ $f,g:\cX\ra\cY$ be\/ $1$-morphisms of
Deligne--Mumford\/ $C^\iy$-stacks and\/ $\eta:f\Ra g$ a
$2$-morphism. Then $\Om_f=\Om_g\ci\eta^*(T^*\cY):f^*(T^*\cY)\ra
T^*\cX$.
\smallskip

\noindent{\bf(c)} Let\/ $\cW,\cX,\cY,\cZ$ be Deligne--Mumford\/ $C^\iy$-stacks in a $2$-Cartesian square
\begin{equation*}
\xymatrix@C=90pt@R=14pt{ *+[r]{\cW} \ar[r]_(0.25)f \ar[d]^e
\drtwocell_{}\omit^{}\omit{^\eta}
 & *+[l]{\cY} \ar[d]_h \\
*+[r]{\cX} \ar[r]^(0.7)g & *+[l]{\cZ} }
\end{equation*}
in $\DMCSta,$ so that\/ $\cW=\cX\t_\cZ\cY$. Then the following is
exact in $\qcoh(\cW)\!:$
\e
\xymatrix@C=13pt{ (g\!\ci\!e)^*(T^*\cZ)
\ar[rrrrrr]^(0.51){\begin{subarray}{l}e^*(\Om_g)\ci I_{e,g}(T^*\cZ)\op\\
-f^*(\Om_h)\ci I_{f,h}(T^*\cZ)\ci\eta^*(T^*\cZ)
\end{subarray}} &&&&&&
{\raisebox{5pt}{$\displaystyle \begin{subarray}{l}\ts e^*(T^*\cX)\op\\
\ts f^*(T^*\cY)\end{subarray}$}} \ar[rr]^(0.55){\Om_e\op \Om_f} &&
T^*\cW \ar[r] & 0.}
\label{ag8eq5}
\e

\label{ag8thm3}
\end{thm}

\begin{proof} For (a), let $u:\bar\uU\ra\cX$, $v:\bar\uV\ra\cY$ and
$w:\bar\uW\ra\cZ$ be \'etale. Then there is a $C^\iy$-scheme $\uV'$
with $\bar\uV{}'=\bar\uV\t_{g\ci v,\cZ,w}\bar\uW$, and fibre product
projections $\upi_\uV:\uV'\ra\uV$, $\upi_\uW:\uV'\ra\uW$. Define
$v'=v\ci\bar\upi_\uV:\bar\uV{}'\ra\cY$. Then $v'$ is \'etale, as $v$
is and $w$ is so $\upi_\uV$ is. Similarly, there is a $C^\iy$-scheme
$\uU'$ with $\bar\uU{}'=\bar\uU\t_{f\ci u,\cY,v'}\bar\uV{}'$, and
fibre product projections $\upi_\uU:\uU'\ra\uU$,
$\upi_{\uV'}:\uU'\ra\uV'$. Define an \'etale 1-morphism
$u'=u\ci\bar\upi_\uU:\bar\uU{}'\ra\cX$. Then we have a 2-commutative
diagram
\begin{equation*}
\xymatrix@C=56pt@R=1pt{ \cX \ar[rr]^f && \cY \ar[rr]^g
&& \cZ \\
& \bar\uU \ar[ul]_(0.2)u \ar[ur] & \ddtwocell_{}\omit^{}\omit{^\eta}
& \bar\uV \ar[ul]_(0.2)v \ar[ur] & \dltwocell_{}\omit^{}\omit{^\ze}
\\ &&&& \bar\uW \ar[uu]^w \\
&&& \bar\uV{}' \ar[uuul]^{v'} \ar[uu]_{\bar\upi_\uV}
\ar[ur]_{\bar\upi_\uW} \\ && \bar\uU{}' \ar[uuuull]^{u'}
\ar[uuul]_{\bar\upi_\uU} \ar[ur]_{\bar\upi_{\uV'}} }
\end{equation*}
with 2-Cartesian squares. On $\uU'$ and $\uV'$ we have commutative
diagrams:
\ea
\begin{gathered}
\xymatrix@C=28pt@R=15pt{ *+[r]{\upi^*_\uU\bigl(f^*(T^*\cY)(\uU,u)\bigr)}
\ar[rrr]^(0.58){i(T^*\cY,f,u,v',\eta)}_(0.58)\cong
\ar@<1.5ex>[d]_\cong^(0.45){(f^*(T^*\cY))_{(\upi_\uU,\id_{u'})}} &&&
\upi_{\uV'}^*\bigl((T^*\cY)(\uV',v')\bigr) \ar@{=}[r] &
*+[l]{\upi_{\uV'}^*(T^*\uV')} \ar[d]_{\Om_{\upi_{\uV'}}} \\
*+[r]{(f^*(T^*\cY))(\uU',u')} \ar[rrr]^(0.58){\Om_f(\uU',u')} &&&
(T^*\cX)(\uU',u') \ar@{=}[r] & *+[l]{T^*\uU',\!} }\!\!\!\!\!{}
\end{gathered}
\label{ag8eq6}\\
\begin{gathered}
\xymatrix@C=28pt@R=15pt{ *+[r]{\upi^*_\uV\bigl(g^*(T^*\cZ)(\uV,v)\bigr)}
\ar[rrr]^(0.58){i(T^*\cZ,g,v,w,\ze)}_(0.58)\cong
\ar@<1.5ex>[d]_\cong^(0.45){(g^*(T^*\cZ))_{(\upi_\uV,\id_{v'})}}
&&& \upi_\uW^*\bigl(T^*\cZ(\uW,w)\bigr) \ar@{=}[r] &
*+[l]{\upi_\uW^*(T^*\uW)} \ar[d]_{\Om_{\upi_\uW}}
\\
*+[r]{(g^*(T^*\cZ))(\uV',v')} \ar[rrr]^(0.58){\Om_g(\uV',v')} &&&
(T^*\cY)(\uV',v') \ar@{=}[r] & *+[l]{T^*\uV'.\!} }\!\!\!\!\!{}
\end{gathered}
\label{ag8eq7}
\ea
Applying $\upi_{\uV'}^*$ to \eq{ag8eq7} we make another
commutative diagram on $\uU'$:
\e
\begin{gathered}
\xymatrix@C=170pt@R=15pt{
*+[r]{\upi_{\uV'}^*\bigl(\upi^*_\uV(g^*(T^*\cZ)(\uV,v))\bigr)}
\ar[r]^(0.57){\upi_{\uV'}^*(i(T^*\cZ,g,v,w,\ze))}_(0.57)\cong
\ar@<1.5ex>[d]_\cong^{\upi_{\uV'}^*((g^*(T^*\cZ))_{(\upi_\uV,\id_{v'})})}
& *+[l]{\upi_{\uV'}^*\bigl(\upi_\uW^*(T^*\uW)\bigr)}
\ar@<-1.5ex>[d]_{\upi_{\uV'}^*(\Om_{\upi_\uW})} \\
*+[r]{\upi_{\uV'}^*\bigl((g^*(T^*\cZ))(\uV',v')\bigr)}
\ar[r]^(0.57){\upi_{\uV'}^*(\Om_g(\uV',v'))}
\ar@<1.5ex>[d]_\cong^{(f^*(g^*(T^*\cZ)))_{(\upi_\uU,\id_{u'})}}
& *+[l]{\upi_{\uV'}^*(T^*\uV')}
\ar@<-1.5ex>[d]^\cong_{(f^*(T^*\cU))_{(\upi_\uU,\id_{u'})}} \\
*+[r]{\bigl(f^*(g^*(T^*\cZ))\bigr)(\uU',u')} \ar[r]^(0.57){(f^*(\Om_g))(\uU',u')}
& *+[l]{\bigl(f^*(T^*\cY)\bigr)(\uU',u').\!} }
\end{gathered}
\label{ag8eq8}
\e
By Theorem \ref{ag5thm4}(a) the following commutes:
\e
\begin{gathered}
\xymatrix@C=140pt@R=15pt{ *+[r]{(\upi_\uW\ci\upi_{\uV'})^*(T^*\uW)}
\ar[r]_(0.55){\Om_{\upi_\uW\ci\upi_{\uV'}}}
\ar@<1.5ex>[d]^{I_{\upi_{\uV'},\upi_\uW}(T^*\uW)}_\cong & *+[l]{T^*\uU'} \\
*+[r]{\upi_{\uV'}^*\bigl(\upi_\uW^*(T^*\uW)\bigr)}
\ar[r]^(0.55){\upi_{\uV'}^*(\Om_{\upi_\uW})} & *+[l]{\upi_{\uV'}^*(T^*\uV').\!}
\ar[u]^{\Om_{\upi_{\uV'}}} }
\end{gathered}
\label{ag8eq9}
\e
Using all this we obtain a commutative diagram on $\uU'$:
\e
\begin{gathered}
\xymatrix@!0@C=75pt@R=35pt{ \bigl((g\ci f)^*(T^*\cZ)\bigr)(\uU',u')
\ar[rrr]_{\Om_{g\ci f}(\uU',u')}
\ar@<-7ex>[ddd]_\cong^{(I_{f,g}(T^*\cZ))(\uU',u')}
\ar@{<->}[dr]^\cong &&& (T^*\cX)(\uU',u') \\
& (\upi_\uW\ci\upi_{\uV'})^*(T^*\uW) \ar[r]
\ar[d]^\cong & T^*\uU' \ar@{=}[ur] \\
& \upi_{\uV'}^*\bigl(\upi_\uW^*(T^*\uW)\bigr)
\ar[r] & \upi_{\uV'}^*(T^*\uV') \ar[u] \\
\bigl(f^*(g^*(T^*\cZ))\bigr)(\uU',u') \ar@{<->}[ur]^(0.4)\cong
\ar[rrr]^{(f^*(\Om_g))(\uU',u')} &&&
\bigl(f^*(T^*\cY)\bigr)(\uU',u'). \ar@{<->}[ul]_(0.4)\cong
\ar@<-3ex>[uuu]^{\Om_f(\uU',u')} }
\end{gathered}\!\!\!\!\!\!\!\!\!\!\!\!\!\!\!\!\!\!{}
\label{ag8eq10}
\e
Here the right hand quadrilateral of \eq{ag8eq10} comes from
\eq{ag8eq6}, the bottom quadrilateral from \eq{ag8eq8}, the
central square is \eq{ag8eq9}, and the remaining two
quadrilaterals are similar. Thus, the outer square of \eq{ag8eq10}
commutes. But this is just \eq{ag8eq4} evaluated at $(\uU',u')$. If
$u:\bar\uU\ra\cX$, $v:\bar\uV\ra\cY$ and $w:\bar\uW\ra\cZ$ are
\'etale atlases then $u':\bar\uU{}'\ra\cX$ is also an \'etale atlas,
and \eq{ag8eq4} evaluated on an atlas implies it in general. This
proves part~(a).

Part (b) is immediate from the definitions. For (c), let
$u:\bar\uU\ra\cX$, $v:\bar\uV\ra\cY$ and $w:\bar\uW\ra\cZ$ be
\'etale. There are $C^\iy$-schemes $\uU',\uV',$ with
$\bar\uU{}'=\bar\uU\t_{g\ci u,\cZ,w}\bar\uW$,
$\bar\uV{}'=\bar\uV\t_{h\ci v,\cZ,w}\bar\uW$, and fibre product
projections $\upi_\uU:\uU'\ra\uU$, $\upi_\uW:\uU'\ra\uW$,
$\upi_\uV:\uV'\ra\uV$, $\upi_\uW:\uV'\ra\uW$. Then $\pi_\uU,\pi_\uV$
are \'etale as $w$ is. Define a $C^\iy$-scheme
$\uT=\uU'\t_{\upi_\uW,\uW,\upi_\uW}\uV'$. The 1-morphisms
$u'\ci\bar\upi_{\uU'}:\bar\uT\ra\cX$ and
$v'\ci\bar\upi_{\uV'}:\bar\uT\ra\cY$ have a natural 2-isomorphism
$g\ci(u'\ci\bar\upi_{\uU'})\Ra h\ci(v'\ci\bar\upi_{\uV'})$
constructed from the 2-isomorphisms in the 2-Cartesian squares
defining $\uU',\uV'$. Thus as $\cW=\cX\t_\cZ\cY$ there is a
1-morphism $t:\bar\uT\ra\cW$, unique up to 2-isomorphism, such that
$u'\ci\bar\upi_{\uU'}\cong e\ci t$ and $v'\ci\bar\upi_{\uV'}\cong
f\ci t$. Also $t$ is \'etale. This gives a 2-commutative diagram
\begin{equation*}
\xymatrix@!0@C=60pt@R=21pt{ & \bar\uT \ar[dl]_(0.6){\bar\upi_{\uU'}}
\ar[ddr]^(0.4){\bar\upi_{\uV'}} \ar[rr]^(0.6){t}
&& \cW \ar[dl]^(0.4){e} \ar[ddr]^{f} \\
\bar\uU{}' \ar[ddr]_{\bar\pi_\uW} \ar[rr]_(0.4){u'} && \cX
\ar[ddr]^(0.35){g} \\ && \bar\uV{}' \ar[rr]^(0.6){v'}
\ar[dl]_(0.4){\bar\upi_\uW} && \cY \ar[dl]^(0.4){h} \\
& \bar\uW \ar[rr]^(0.6){w} && \cZ,}
\end{equation*}
in which the leftmost and rightmost squares are 2-Cartesian.

Applying Theorem \ref{ag5thm4}(b) to the Cartesian square defining
$\uT$ gives an exact sequence in $\qcoh(\uT)$:
\e
\xymatrix@C=12.5pt{ {\raisebox{6pt}{$
\begin{subarray}{l}\ts(\upi_\uW\!\ci\!\upi_{\uU'})^*\\
\ts {}\,\,\, (T^*\uW)\end{subarray}$}}
\ar[rrrrrr]^(0.49){\begin{subarray}{l}\upi_{\uU'}^*(\Om_{\upi_\uW})\ci
I_{\upi_{\uU'},\upi_\uW}(T^*\uW)\op\\
-\upi_{\uV'}^*(\Om_{\upi_\uW})\ci
I_{\upi_{\uV'},\upi_\uW}(T^*\uW)\end{subarray}} &&&&&&
{\raisebox{6pt}{$\begin{subarray}{l}\ts{}\,\,\,
\upi_{\uU'}^*(T^*\uU')\\
\ts\op\upi_{\uV'}^*(T^*\uV')\end{subarray}$}}
\ar[rr]^(0.6){\begin{subarray}{l}\Om_{\upi_{\uU'}}\op\\
\Om_{\upi_{\uV'}}\end{subarray}} && T^*\uT \ar[r] & 0.}\!\!{}
\label{ag8eq11}
\e
By a similar argument to (a), we can use \eq{ag8eq11} to deduce
that \eq{ag8eq5} evaluated at $(\uT,t)$ holds. If
$u:\bar\uU\ra\cX$, $v:\bar\uV\ra\cY$ and $w:\bar\uW\ra\cZ$ are
atlases then $t:\bar\uT\ra\cW$ is an atlas, so this implies
\eq{ag8eq5}, and proves~(c).\I{Deligne--Mumford $C^\iy$-stack!quasicoherent sheaves on|)}\I{Deligne--Mumford $C^\iy$-stack!cotangent sheaf|)}
\end{proof}

\section{\texorpdfstring{Orbifold strata of $C^\iy$-stacks}{Orbifold strata of C∞-stacks}}
\label{ag9}
\I{Deligne--Mumford $C^\iy$-stack!orbifold strata|(}

Let $\cX$ be a Deligne--Mumford $C^\iy$-stack, with topological
space $\cX_\top$. Then each point $[x]\in\cX_\top$ has an isotropy group $\Iso_\cX([x])$,\I{C-stack@$C^\iy$-stack!isotropy group
$\Iso_\cX([x])$} a finite group defined up to isomorphism. For each
finite group $\Ga$ we write $\tcX_{\ci,\top}^\Ga=
\bigl\{[x]\in\cX_\top:\Iso_\cX([x])\cong\Ga\bigr\}$. This is a
locally closed subset of $\cX_\top$, coming from a locally closed
$C^\iy$-substack $\tcX^\Ga_\ci$ of $\cX$ with inclusion $\ti
O{}_\ci^\Ga(\cX):\tcX^\Ga_\ci\ra\cX$, with
\e
\cX_\top=\coprod\nolimits_{\substack{\text{isomorphism classes}\\
\text{of finite groups $\Ga$}}} \tcX_{\ci,\top}^\Ga.
\label{ag9eq1}
\e
One can show that for each $\Ga$, the closure
$\,\ov{\!\tcX}{}_{\ci,\top}^{\,\Ga}$ of $\tcX_{\ci,\top}^\Ga$ in
$\cX_\top$ satisfies
\begin{equation*}
\,\ov{\!\tcX}{}_{\ci,\top}^{\,\Ga}\subseteq\coprod\nolimits_{\substack{
\text{isomorphism classes of finite groups $\De$:}\\
\text{$\Ga$ is isomorphic to a subgroup of
$\De$}}}\tcX_{\ci,\top}^\De.
\end{equation*}
Thus \eq{ag9eq1} is a stratification of $\cX_\top$. The
$\tcX^\Ga_\ci$ are called {\it orbifold strata\/} of~$\cX$.

When $\cX$ is an orbifold,\I{orbifold} as in \S\ref{ag76}, the orbifold strata are manifolds (actually, at the level of $C^\iy$-stacks, the alternative versions $\hcX^\Ga_\ci$ below are manifolds), and are well studied. Orbifold strata of orbifolds come up in areas such as the Atiyah--Singer Index Theorem for orbifolds as in Kawasaki \cite{Kawa}, cobordism of orbifolds as in Druschel \cite{Drus}, String Theory of orbifolds as in Dixon et al.\ \cite{DHVW}, and (quantum) cohomology of orbifolds as in Chen and Ruan~\cite{ChRu1}.

However, very little appears to have been done in considering orbifold strata from the point of view of category theory or stacks, or about orbifold strata of other kinds of Deligne--Mumford stacks. We now define and study orbifold strata of Deligne--Mumford $C^\iy$-stacks. Actually, almost all of \S\ref{ag9} is an exercise in stack theory, not specific to $C^\iy$-stacks. But the author has been unable to find any references on it. 

We will define six variations on $\tcX^\Ga_\ci$ outlined above, Deligne--Mumford
$C^\iy$-stacks written $\cX^\Ga,\tcX^\Ga,\hcX^\Ga$, and open
$C^\iy$-substacks%
\I{C-stack@$C^\iy$-stack!C-substack@$C^\iy$-substack!open}%
\I{C-substack@$C^\iy$-substack!open}
$\cX{}^\Ga_\ci\subseteq\cX^\Ga$, $\tcX^\Ga_\ci\subseteq\tcX^\Ga$,
$\hcX^\Ga_\ci\subseteq \hcX^\Ga$.\G[XGa]{$\cX^\Ga,\tcX^\Ga,\hcX^\Ga,
\cX{}^\Ga_\ci,\tcX^\Ga_\ci,\hcX^\Ga_\ci$}{orbifold strata of a
Deligne--Mumford $C^\iy$-stack $\cX$} The points and
isotropy groups of $\cX^\Ga,\ldots,\hcX^\Ga_\ci$ are given by:
\begin{itemize}
\setlength{\itemsep}{0pt}
\setlength{\parsep}{0pt}
\item[(i)] Points of $\cX^\Ga$ are isomorphism classes
$[x,\rho]$, where $[x]\in\cX_\top$ and
$\rho:\Ga\ra\Iso_\cX([x])$ is an injective morphism, and
$\Iso_{\smash{\cX^\Ga}}([x,\rho])$ is the centralizer of
$\rho(\Ga)$ in $\Iso_\cX([x])$. Points of
$\cX{}^\Ga_\ci\subseteq\cX^\Ga$ are $[x,\rho]$ with $\rho$ an
isomorphism, and $\Iso_{\smash{ \cX{}^\Ga_\ci}}([x,\rho])\cong
C(\Ga)$, the centre of $\Ga$.
\item[(ii)] Points of $\tcX^\Ga$ are pairs $[x,\De]$, where
$[x]\in\cX_\top$ and $\De\subseteq\Iso_\cX([x])$ is isomorphic
to $\Ga$, and $\Iso_{\smash{\tcX^\Ga}}([x,\De])$ is the
normalizer of $\De$ in $\Iso_\cX([x])$. Points of
$\tcX^\Ga_\ci\subseteq\tcX^\Ga$ are $[x,\De]$ with
$\De=\Iso_\cX([x])$, and~$\Iso_{\smash{\tcX^\Ga_\ci}}
([x,\De])\cong\Ga$.
\item[(iii)] Points $[x,\De]$ of $\hcX^\Ga,\hcX^\Ga_\ci$
are the same as for $\tcX^\Ga,\tcX^\Ga_\ci$, but with isotropy
groups $\Iso_{\smash{\hcX^\Ga}}([x,\De])\cong
\Iso_{\smash{\tcX^\Ga}}([x,\De])/\De$
and~$\Iso_{\smash{\hcX^\Ga_\ci}}([x,\De])\cong\{1\}$.
\end{itemize}
There are 1-morphisms $O^\Ga(\cX),\ldots,
\hat\Pi{}^\Ga_\ci(\cX)$\G[OGaXa]{$O^\Ga(\cX),\ti
O^\Ga(\cX),O{}^\Ga_\ci(\cX),\ti O{}^\Ga_\ci(\cX)$}{1-morphisms of
orbifold strata $\cX^\Ga,\ldots,\hcX^\Ga_\ci$ of a Deligne--Mumford
$C^\iy$-stack $\cX$}\G[PiGaXa]{$\ti\Pi^\Ga(\cX),\hat\Pi{}^\Ga(\cX),
\ti\Pi^\Ga_\ci(\cX),\hat\Pi{}^\Ga_\ci(\cX)$}{1-morphisms of orbifold
strata $\cX^\Ga,\ldots,\hcX^\Ga_\ci$ of a Deligne--Mumford
$C^\iy$-stack $\cX$} forming a 2-commutative diagram, where
the columns are inclusions of open $C^\iy$-substacks:%
\I{C-stack@$C^\iy$-stack!C-substack@$C^\iy$-substack!open}%
\I{C-substack@$C^\iy$-substack!open}
\e
\begin{gathered}
\xymatrix@C=55pt@R=6pt{ \cX{}^\Ga_\ci
\ar[rr]^{\ti\Pi{}^\Ga_\ci(\cX)} \ar[dr]_(0.3){O{}^\Ga_\ci(\cX)}
\ar[dd]_\subset \ar@(ul,l)[]_(0.7){\Aut(\Ga)} && \tcX^\Ga_\ci
\ar[r]^{\hat\Pi{}^\Ga_\ci(\cX)} \ar[dl]^(0.3){\ti O^\Ga_\ci(\cX)}
\ar[dd]^\subset  &
*+[r]{\hcX^\Ga_\ci\simeq
\bar{\hat\uX}{}^\Ga_\ci\!\!\!\!\!\!}
\ar@<.5ex>[dd]^\subset \\ & \cX \\
\cX^\Ga \ar[rr]_{\ti\Pi^\Ga(\cX)} \ar[ur]^(0.3){O^\Ga(\cX)}
\ar@(dl,l)[]^(0.7){\Aut(\Ga)} && \tcX^\Ga
\ar[r]_{\hat\Pi{}^\Ga(\cX)} \ar[ul]_(0.3){\ti O^\Ga(\cX)} &
*+[r]{\hcX^\Ga.} }
\end{gathered}
\label{ag9eq2}
\e
Also $\Aut(\Ga)$ acts on $\cX^\Ga,\cX{}^\Ga_\ci$, with
$\tcX^\Ga\simeq[\cX^\Ga/\Aut(\Ga)]$,
$\tcX^\Ga_\ci\simeq[\cX^\Ga_\ci/\Aut(\Ga)]$. Note that there are in general {\it no natural\/ $1$-morphisms\/} from $\hcX^\Ga,\hcX^\Ga_\ci$ to any of $\cX,\cX^\Ga,
\cX^\Ga_\ci,\tcX^\Ga,\tcX^\Ga_\ci$. 

\subsection{\texorpdfstring{The definition of orbifold strata
$\cX^\Ga,\ldots,\hcX^\Ga_\ci$}{The definition of orbifold strata}}
\label{ag91}

We now define the orbifold strata $\cX^\Ga,\ldots,\hcX^\Ga_\ci$ and
study their properties.

\begin{dfn} Let $\cX$ be a Deligne--Mumford
$C^\iy$-stack, and $\Ga$ a finite group. We will explicitly define
another Deligne--Mumford $C^\iy$-stack $\cX^\Ga$. Since $\cX$ is a
stack on the site\I{site} $(\CSch,{\cal J})$, $\cX$ is a category
with a functor $p_\cX:\cX\ra\CSch$ satisfying many conditions. To
define $\cX^\Ga$ we must define another category $\cX^\Ga$ and a
functor~$p_{\smash{\cX^\Ga}}:\cX^\Ga\ra\CSch$.

Define objects of the category $\cX^\Ga$ to be pairs $(A,\rho)$
satisfying:
\begin{itemize}
\setlength{\itemsep}{0pt}
\setlength{\parsep}{0pt}
\item[(a)] $A$ is an object in $\cX$, with $p_\cX(A)=\uU$ for some
object~$\uU\in\CSch$;
\item[(b)] $\rho:\Ga\ra\Aut(A)$ is a group morphism, where $\Aut(A)$
is the group of isomorphisms $a:A\ra A$ in $\cX$, and
$p_\cX\ci\rho(\ga)=\uid_\uU$ for all $\ga\in\Ga$; and
\item[(c)] Let $u$ be a point in $\uU$, and $\uu:\ul{*}\ra\uU$ the
corresponding morphism in $\CSch$. Since $p_\cX:\cX\ra\CSch$ is
a category fibred in groupoids, as in Definition \ref{agAdef5},
there exists a morphism $a_u:A_u\ra A$ in $\cX$ with
$p_\cX(A_u)=\ul{*}$ and $p_\cX(a_u)=\uu$, where $A_u$ is unique
up to isomorphism in~$\cX$.

Having fixed $A_u,a_u$, Definition \ref{agAdef5} also implies
that for each $\ga\in\Ga$ there is a unique isomorphism $\rho_u
(\ga):A_u\ra A_u$ such that $a_u\ci\rho_u(\ga)=\rho(\ga)\ci
a_u:A_u\ra A$, and $p_\cX(\rho_u(\ga))=\uid_{\ul{*}}$. Then
$\rho_u:\Ga\ra\Aut(A_u)$ is a group morphism. We require that
$\rho_u:\Ga\ra\Aut(A_u)$ should be injective for all $u\in\uU$.
This condition is independent of the choice of~$A_u,a_u$.
\end{itemize}

Define morphisms $c:(A,\rho)\ra(B,\si)$ of the category $\cX^\Ga$ to
be morphisms $c:A\ra B$ in $\cX$ satisfying $\si(\ga)\ci
c=c\ci\rho(\ga):A\ra B$ in $\cX$ for all $\ga\in\Ga$. Given
morphisms $c:(A,\rho)\ra(B,\si)$, $d:(B,\si)\ra(C,\tau)$ in
$\cX^\Ga$, define composition $d\ci c:(A,\rho)\ra(C,\tau)$ in
$\cX^\Ga$ to be the composition $d\ci c:A\ra C$ in $\cX$. For each
object $(A,\rho)$ in $\cX^\Ga$, define the identity morphism
$\id_{(A,\rho)}:(A,\rho)\ra(A,\rho)$ in $\cX^\Ga$ to be $\id_A:A\ra
A$ in $\cX$. Define a functor $p_{\smash{\cX^\Ga}}: \cX^\Ga\ra\CSch$
by $p_{\smash{\cX^\Ga}}:(A,\rho)\mapsto \uU=p_\cX(A)$ on objects and
$p_{\smash{\cX^\Ga}}:c\mapsto p_\cX(c)$ on morphisms.

Define $\cX{}^\Ga_\ci$ to be the full subcategory of objects
$(A,\rho)$ in $\cX^\Ga$ such that $\rho_u:\Ga\ra\Aut(A_u)$ in (c)
above is an isomorphism for all $u\in\uU$. Define a functor
$p_{\smash{\cX{}^\Ga_\ci}}=p_\cX\vert_{\smash{\cX{}^\Ga_\ci}}:
\cX{}^\Ga_\ci\ra\CSch$. By Theorem \ref{ag9thm1}(a) below,
$\cX^\Ga$ is a Deligne--Mumford $C^\iy$-stack, and $\cX{}^\Ga_\ci$
is an open $C^\iy$-substack in~$\cX^\Ga$.%
\I{C-stack@$C^\iy$-stack!C-substack@$C^\iy$-substack!open}%
\I{C-substack@$C^\iy$-substack!open}
\label{ag9def1}
\end{dfn}

\begin{dfn} Let $\cX$ be a Deligne--Mumford $C^\iy$-stack, and
$\Ga$ a finite group. Define a category ${\cal P}\tcX^\Ga$ to have
objects pairs $(A,\De)$ satisfying:
\begin{itemize}
\setlength{\itemsep}{0pt}
\setlength{\parsep}{0pt}
\item[(a)] $A$ is an object in $\cX$, with $p_\cX(A)=\uU$ for some
object~$\uU\in\CSch$;
\item[(b)] $\De\subseteq\Aut(A)$ is a subgroup isomorphic to $\Ga$,
where $\Aut(A)$ is the group of isomorphisms $a:A\ra A$ in
$\cX$, and $p_\cX(\de)=\uid_\uU$ for all $\de\in\De$; and
\item[(c)] Let $u$ be a point in $\uU$, and $\uu:\ul{*}\ra\uU$ the
corresponding morphism in $\CSch$. Since $p_\cX:\cX\ra\CSch$ is
a category fibred in groupoids, there exists a morphism
$a_u:A_u\ra A$ in $\cX$ with $p_\cX(A_u)=\ul{*}$ and
$p_\cX(a_u)=\uu$, where $A_u$ is unique up to isomorphism in
$\cX$. For each $\de\in\De$ there is a unique isomorphism
$\de_u:A_u\ra A_u$ such that $a_u\ci\de_u=\de\ci a_u:A_u\ra A$,
and $p_\cX(\de_u)=\uid_{\ul{*}}$. Then $\{\de_u:\de\in\De\}$ is
a subgroup of $\Aut(A_u)$, and $\de\mapsto\de_u$ is a group
morphism. We require that the map $\de\mapsto\de_u$ should be
{\it injective\/} for all $u\in\uU$.
\end{itemize}

Define morphisms $(A,\De)\ra(A',\De')$ of ${\cal P}\tcX^\Ga$ to be
pairs $(c,\io)$, where $c:A\ra A'$ is a morphism in $\cX$ and
$\io:\De\ra\De'$ is a group isomorphism, satisfying $\io(\de)\ci
c=c\ci\de:A\ra A'$ for all $\de\in\De$. Given morphisms
$(c,\io):(A,\De)\ra(A',\De')$, $(c',\io'):(A',\De')\ra(A'',\De'')$
in ${\cal P}\cX^\Ga$, define composition $(c',\io')\ci(c,\io)=(c'\ci
c,\io'\ci\io)$. Define
identities~$\id_{(A,\De)}=(\id_A,\id_\De):(A,\De)\ra(A,\De)$.

Define a functor $p_{\smash{{\cal P}\tcX^\Ga}}:{\cal P}\tcX^\Ga
\ra\CSch$ by $p_{\smash{{\cal P}\tcX^\Ga}}:
(A,\De)\mapsto\uU=p_\cX(A)$ on objects and $p_{\smash{{\cal
P}\tcX^\Ga}}:(c,\io)\mapsto p_\cX(c)$ on morphisms. Define ${\cal
P}\tcX^\Ga_\ci$ to be the full subcategory of objects $(A,\De)$ in
${\cal P}\tcX^\Ga$ with $\{\de_u:\de\in\De\}=\Aut(A_u)$ in (c) above
for all $u\in\uU$. Define a functor $p_{\smash{{\cal
P}\tcX^\Ga_\ci}}= p_{\smash{{\cal P}\tcX^\Ga}}\vert_{\smash{{\cal
P}\tcX^\Ga_\ci}}:{\cal P}\tcX^\Ga_\ci\ra\CSch$.

Although ${\cal P}\tcX^\Ga,{\cal P}\tcX^\Ga_\ci$ are in general not
$C^\iy$-stacks, they are prestacks on the site\I{site} $(\CSch,{\cal
J})$ in the sense of Definition \ref{agAdef6} (that is, morphisms in
${\cal P}\tcX^\Ga,{\cal P}\tcX^\Ga_\ci$ satisfy a sheaf-like
condition over $(\CSch,{\cal J})$, but objects may not). Thus,
${\cal P}\tcX^\Ga,{\cal P}\tcX^\Ga_\ci$ have
stackifications\I{prestack!stackification|(}
$\tcX^\Ga,\tcX^\Ga_\ci$, defined up to equivalence, which are stacks
on the site $(\CSch,{\cal J})$. By Theorem \ref{ag9thm1}(a) below,
$\tcX^\Ga$ is a Deligne--Mumford $C^\iy$-stack, and
$\tcX^\Ga_\ci$ is an open $C^\iy$-substack in~$\tcX^\Ga$.%
\I{C-stack@$C^\iy$-stack!C-substack@$C^\iy$-substack!open}%
\I{C-substack@$C^\iy$-substack!open}

Let $(A,\De),(A',\De')$ be objects in ${\cal P}\tcX^\Ga$. Define a
right action of $\De$ on morphisms $(c,\io):(A,\De)\ra(A',\De')$ in
${\cal P}\tcX^\Ga$ by $(c,\io)\cdot\de=(c\ci\de,\io^\de)$, where
$\io^\de:\De\ra\De'$ maps~$\io^\de:\ep\mapsto\io(\de\ci\ep
\ci\de^{-1})$. If $(c',\io'):(A',\De')\ra(A'',\De'')$ is another
morphism and $\de'\in\De'$, it is easy to show that
\e
\bigl((c',\io')\cdot\de'\bigr)\ci\bigl((c,\io)\cdot
\de\bigr)=\bigl((c',\io')\ci(c,\io)\bigr)\cdot(\io^{-1}(\de')\ci\de).
\label{ag9eq3}
\e

Define a category ${\cal P}\hcX^\Ga$ to have objects $(A,\De)$ as in
${\cal P}\tcX^\Ga$, and to have morphisms $(c,\io)\De:
(A,\De)\ra(A',\De')$ for morphisms $(c,\io):(A,\De)\ra(A',\De')$ in
${\cal P}\tcX^\Ga$, where $(c,\io)\De=\{(c,\io)\cdot\de:
\de\in\De\}$ is the $\De$-orbit of $(c,\io)$. Define composition of
morphisms in ${\cal P}\hcX^\Ga$ by $\bigl((c',\io')
\De'\bigr)\ci\bigl((c,\io)\De\bigr)=\bigl((c',\io')\ci(c,\io)
\bigr)\De$, where $(c',\io')\ci(c,\io)$ is composition of morphisms
in ${\cal P}\tcX^\Ga$. Equation \eq{ag9eq3} shows this is
well-defined. Define identity morphisms
$\id_{(A,\De)}=(\id_A,\id_\De)\De:(A,\De) \ra(A,\De)$ in ${\cal
P}\hcX^\Ga$. Define a functor $p_{\smash{{\cal P}\hcX^\Ga}}:{\cal
P}\hcX^\Ga\ra\CSch$ to map $(A,\De)\mapsto p_\cX(A)$ on objects and
$(c,\io)\De\mapsto p_\cX(c)$ on morphisms.

Define ${\cal P}\hcX^\Ga_\ci$ to be the full subcategory of ${\cal
P}\hcX^\Ga$ whose objects are objects of ${\cal P}\tcX^\Ga_\ci$, and
define $p_{\smash{{\cal P}\hcX {}^\Ga_\ci}}=p_{\smash{{\cal
P}\hcX^\Ga}}\vert_{\smash{{\cal P}\hcX^\Ga_\ci}}:{\cal
P}\hcX^\Ga_\ci\ra\CSch$. Then as for ${\cal P}\hcX^\Ga,{\cal
P}\hcX^\Ga_\ci$ are prestacks on $(\CSch,{\cal J})$, and by Theorem
\ref{ag9thm1}(a) their stackifications $\hcX^\Ga,\hcX^\Ga_\ci$ are
Deligne--Mumford $C^\iy$-stacks. Furthermore, by Theorem
\ref{ag9thm1}(g) below $\hcX^\Ga_\ci$ has trivial isotropy groups,
so by Theorem \ref{ag7thm7} there is a $C^\iy$-scheme
$\hat\uX{}^\Ga_\ci$, unique up to isomorphism, such
that~$\hcX^\Ga_\ci\simeq\bar{\hat\uX} {}^\Ga_\ci$.
\label{ag9def2}
\end{dfn}

Next, we define all the 1-morphisms in~\eq{ag9eq2}.

\begin{dfn} In Definitions \ref{ag9def1} and \ref{ag9def2},
for $\La\in\Aut(\Ga)$ define functors
\begin{gather*}
L^\Ga(\La,\cX):\cX^\Ga\longra\cX^\Ga,\quad
O^\Ga(\cX):\cX^\Ga\longra\cX, \quad {\cal P}\ti O{}^\Ga(\cX):{\cal
P}\tcX^\Ga\longra\cX, \\
{\cal P}\ti\Pi{}^\Ga(\cX):\cX^\Ga\longra {\cal
P}\tcX^\Ga\quad\text{and}\quad {\cal P}\hat\Pi{}^\Ga(\cX):{\cal
P}\tcX^\Ga\longra {\cal P}\hcX^\Ga
\end{gather*}
on objects by
\begin{gather*}
L^\Ga(\La,\cX):(A,\rho)\mapsto (A,\rho\ci\La^{-1}),\;
O^\Ga(\cX):(A,\rho)\mapsto A,\;
{\cal P}\ti O{}^\Ga(\cX):(A,\De)\mapsto A,\\
{\cal P}\ti\Pi{}^\Ga(\cX):(A,\rho)\longmapsto\bigl(A,\rho(\Ga)
\bigr) \quad\text{and}\quad {\cal P}\hat\Pi{}^\Ga(\cX):(A,\De)
\longmapsto(A,\De),
\end{gather*}
and on morphisms by
\begin{gather*}
L^\Ga(\La,\cX):c\longmapsto c,\quad O^\Ga(\cX):c\longmapsto c,\quad
{\cal P}\ti O{}^\Ga(\cX):(c,\io)\longmapsto c, \\
\begin{aligned}
{\cal P}\ti\Pi{}^\Ga(\cX)&:c\mapsto (c,\si\ci\rho^{-1}) &&\text{on
$c:(A,\rho)\ra(B,\si)$, and}\\
{\cal P}\hat\Pi^\Ga(\cX)&:(c,\io)\mapsto(c,\io)\De &&\text{on
$(c,\io):(A,\De)\ra(A',\De')$.}
\end{aligned}
\end{gather*}

It is trivial to check that these are all functors, and commute with
the projections $p_\cX,p_{\smash{\cX^\Ga}}, p_{\smash{\tcX^\Ga}},
p_{\smash{\hcX^\Ga}}$ to $\CSch$. Hence $L^\Ga(\La,\cX),O^\Ga(\cX)$
are 1-morphisms of $C^\iy$-stacks. Note that $L^\Ga(\La,\cX)\ci
L^\Ga(\La',\cX)=L^\Ga(\La\ci\La',\cX)$ and
$L^\Ga(\La^{-1},\cX)=L^\Ga(\La,\cX)^{-1}$ for
$\La,\La'\in\Aut(\Ga)$, so $L^\Ga(-,\cX)$ is an action of
$\Aut(\Ga)$ on $\cX^\Ga$ by 1-isomorphisms.

Now ${\cal P}\ti O{}^\Ga(\cX),{\cal P}\ti\Pi{}^\Ga(\cX),{\cal
P}\hat\Pi{}^\Ga(\cX)$ are 1-morphisms of prestacks, so stackifying
gives 1-morphisms of $C^\iy$-stacks $\ti
O{}^\Ga(\cX):\tcX^\Ga\ra\cX$, $\ti\Pi{}^\Ga(\cX):\cX^\Ga\ra
\tcX^\Ga$, $\hat\Pi{}^\Ga(\cX):\tcX^\Ga\ra\hcX^\Ga$. Define
1-morphisms of $C^\iy$-stacks
\begin{gather*}
L^\Ga_\ci(\La,\cX):\cX^\Ga_\ci\longra\cX^\Ga_\ci,\quad
O^\Ga_\ci(\cX):\cX^\Ga_\ci\longra\cX, \quad \ti O{}^\Ga_\ci(\cX):
\tcX^\Ga_\ci\longra\cX, \\
\ti\Pi{}^\Ga_\ci(\cX):\cX^\Ga_\ci\longra
\tcX^\Ga_\ci\quad\text{and}\quad
\hat\Pi{}^\Ga_\ci(\cX):\tcX^\Ga_\ci\longra\hcX^\Ga_\ci,
\end{gather*}
to be the restrictions of $L^\Ga(\La,\cX),\ldots,\hat\Pi{}^\Ga(\cX)$
to the open $C^\iy$-substacks $\cX^\Ga_\ci,\tcX^\Ga_\ci$. Then
$L^\Ga_\ci(-,\cX)$ is an action of $\Aut(\Ga)$ on $\cX^\Ga_\ci$ by
1-isomorphisms.

It is easy to see that the analogue of \eq{ag9eq2} with prestacks
${\cal P}\tcX^\Ga,\ldots,{\cal P}\hcX^\Ga_\ci$ and prestack
1-morphisms ${\cal P}\ti O{}^\Ga(\cX),\ldots,{\cal
P}\hat\Pi{}^\Ga_\ci(\cX)$ is strictly commutative, i.e.\
2-commutative with identity 2-morphisms. Thus on stackifying,
\eq{ag9eq2} commutes up to canonical 2-isomorphisms.
\label{ag9def3}
\end{dfn}

\begin{dfn} Let the 1-morphisms $O^\Ga(\cX):\cX^\Ga\ra\cX$,
$O^\Ga_\ci(\cX):\cX^\Ga_\ci\ra\cX$ be as in Definition
\ref{ag9def3}. We will define actions of $\Ga$ on
$O^\Ga(\cX),O^\Ga_\ci(\cX)$ by 2-morphisms. For each $\ga\in\Ga$ and
$(A,\rho)\in\cX^\Ga$, define an isomorphism
$E^\Ga(\ga,\cX)(A,\rho):O^\Ga(\cX)(A,\rho)\ra O^\Ga(\cX)(A,\rho)$ in
$\cX$ by $E^\Ga(\ga,\cX)=\rho(\ga):A\ra A$. If
$c:(A,\rho)\ra(B,\si)$ is a morphism in $\cX^\Ga$ then
\begin{equation*}
O^\Ga(\cX)(c)\ci E^\Ga(\ga,\cX)(A,\rho)\!=\!c\ci\rho(\ga)\!=\!
\si(\ga)\ci\rho\!=\!E^\Ga(\ga,\cX)(B,\si)\ci O^\Ga(\cX)(c).
\end{equation*}
Hence $E^\Ga(\ga,\cX):O^\Ga(\cX)\Ra O^\Ga(\cX)$ is a natural
isomorphism of functors. Since $p_\cX(E^\Ga(\ga,\cX)(A,\rho))=
p_\cX(\rho(\ga))=\uid_{p_\cX(A)}$ for all $(A,\rho)$, we have
$p_\cX*E^\Ga(\ga,\cX)=p_{\smash{\cX^\Ga}}$, so $E^\Ga(\ga,\cX):
O^\Ga(\cX)\Ra O^\Ga(\cX)$ is a 2-morphism of $C^\iy$-stacks. Clearly
$E^\Ga(1,\cX)=\id_{\smash{O^\Ga(\cX)}}$ and $E^\Ga(\ga,\cX)\od
E^\Ga(\de,\cX)=E^\Ga(\ga\de,\cX)$\I{2-category!2-morphism!vertical
composition} for all $\ga,\de\in\Ga$, so
$E^\Ga(-,\cX):\Ga\ra\Aut\bigl(O^\Ga(\cX) \bigr)$ is a group
morphism. We define 2-morphisms
$E^\Ga_\ci(\ga,\cX):O^\Ga_\ci(\cX)\Ra O^\Ga_\ci(\cX)$ for
$\ga\in\Ga$ in the same way.
\label{ag9def4}
\end{dfn}

Here are some basic properties of these definitions.

\begin{thm}{\bf(a)} $\cX^\Ga,\tcX^\Ga,\hcX^\Ga$ are
Deligne--Mumford\/ $C^\iy$-stacks, and\/ $\cX{}^\Ga_\ci\subseteq
\cX^\Ga,$ $\tcX^\Ga_\ci\subseteq\tcX^\Ga,$ $\hcX^\Ga_\ci\subseteq
\hcX^\Ga$ are open $C^\iy$-substacks. Also
$\tcX^\Ga\simeq[\cX^\Ga/\Aut(\Ga)]$ and\/
$\tcX^\Ga_\ci\simeq[\cX^\Ga_\ci/\Aut(\Ga)],$ where the
$\Aut(\Ga)$-actions are $L^\Ga(-,\cX)$ and\/~$L^\Ga_\ci(-,\cX)$.%
\I{C-stack@$C^\iy$-stack!C-substack@$C^\iy$-substack!open}%
\I{C-substack@$C^\iy$-substack!open}
\smallskip

\noindent{\bf(b)} If\/ $\cX$ is
separated,\I{C-stack@$C^\iy$-stack!separated} locally
fair,\I{Deligne--Mumford $C^\iy$-stack!locally fair} locally
finitely presented,\I{Deligne--Mumford $C^\iy$-stack!locally
finitely presented} or second countable,\I{Deligne--Mumford
$C^\iy$-stack!second countable} then\/
$\cX^\Ga,\cX{}^\Ga_\ci,\tcX^\Ga,
\tcX^\Ga_\ci,\ab\hcX^\Ga,\ab\hcX^\Ga_\ci$ are separated, locally
fair, locally finitely presented, or second countable, respectively.

If\/ $\cX$ is compact then $\cX^\Ga,\tcX^\Ga,\hcX^\Ga$ are compact.
\smallskip

\noindent{\bf(c)} Points of\/ $\cX^\Ga_\top$ are equivalence classes
$[x,\rho]$ of pairs $(x,\rho),$ where $x:\bar{\ul *}\ra\cX$ is a
$1$-morphism and $\rho:\Ga\ra\Aut(x)$ is an injective group morphism
into the group $\Aut(x)$ of\/ $2$-isomorphisms $\eta:x\Ra x,$ and
pairs $(x,\rho),(x',\rho')$ are equivalent if there exists $\ze:x\Ra
x'$ with\/ $\ze\od\rho(\ga)=\rho'(\ga)\od\ze:x\Ra x'$ for all\/
$\ga\in\Ga$.\I{2-category!2-morphism!vertical composition} They have
isotropy groups\I{C-stack@$C^\iy$-stack!isotropy group
$\Iso_\cX([x])$|(}
\begin{equation*}
\Iso_{\smash{\cX^\Ga}}([x,\rho])=\bigl\{\eta\in\Aut(x):
\rho(\ga)=\eta\rho(\ga)\eta^{-1}\;\>\forall \ga\in\Ga\bigr\}.
\end{equation*}
Points of\/ $\cX^\Ga_{\ci,\top}$ are $[x,\rho]$ with\/
$\rho:\Ga\ra\Aut(x)$ an isomorphism, and have canonical
isomorphisms\/ $\Iso_{\smash{\cX^\Ga_\ci}}([x,\rho])\cong C(\Ga),$
where $C(\Ga)$ is the centre of\/~$\Ga$.
\smallskip

\noindent{\bf(d)} Points of\/ $\tcX^\Ga_\top$ are equivalence
classes $[x,\De]$ of pairs $(x,\De),$ where $x:\bar{\ul *}\ra\cX$ is
a $1$-morphism and\/ $\De\subseteq\Aut(x)$ is a subgroup isomorphic
to $\Ga,$ and pairs $(x,\De),(x',\De')$ are equivalent if there
exists a $2$-isomorphism $\ze:x\Ra x'$ with\/
$\De'=\ze\od\De\od\ze^{-1}$. They have isotropy groups
\begin{equation*}
\Iso_{\smash{\tcX^\Ga}}([x,\De])\cong\bigl\{\eta\in\Aut(x):\De=
\eta\De\eta^{-1}\bigr\}.
\end{equation*}
Points of\/ $\tcX^\Ga_{\ci,\top}$ are $[x,\De]$ with\/
$\De=\Aut(x),$ and have non-canonical
isomorphisms\/~$\Iso_{\smash{\tcX^\Ga_\ci}}([x,\De])\cong\Ga$.
\smallskip

\noindent{\bf(e)} As topological spaces $\hcX^\Ga_\top=
\tcX^\Ga_\top$ and\/ $\hcX^\Ga_{\ci,\top}= \tcX^\Ga_{\ci,\top},$
and\/ $\hat\Pi{}^\Ga(\cX)_\top,\ab \hat\Pi{}^\Ga_\ci(\cX)_\top$ are
the identity maps. For $[x,\De]\in\hcX^\Ga_\top$ we have
\begin{equation*}
\Iso_{\smash{\hcX^\Ga}}([x,\De])\cong\bigl\{\eta\in\Aut(x):\De=
\eta\De\eta^{-1}\bigr\}/\De.
\end{equation*}
Also\/ $\Iso_{\smash{\hcX^\Ga_\ci}}([x,\De])=\{1\}$ for all\/
$[x,\De]\in\hcX^\Ga_{\ci,\top},$ so $\hcX^\Ga_\ci$ is a
$C^\iy$-scheme.\I{C-stack@$C^\iy$-stack!isotropy group
$\Iso_\cX([x])$|)}\I{C-stack@$C^\iy$-stack!is a $C^\iy$-scheme}
\smallskip

\noindent{\bf(f)} $L^\Ga(\La,\cX),\hskip 1pt plus 2pt
L^\Ga_\ci(\La,\cX),\hskip 1pt plus 2pt O^\Ga(\cX),\hskip 1pt plus
2pt O{}^\Ga_\ci(\cX),\hskip 1pt plus 2pt \ti O{}^\Ga(\cX), \hskip
1pt plus 2pt \ti O{}^\Ga_\ci(\cX),\hskip 1pt plus 2pt
\ti\Pi{}^\Ga(\cX),\hskip 1pt plus 2pt \ti\Pi{}^\Ga_\ci(\cX)$ are all
representable, but\/ $\hat\Pi{}^\Ga(\cX),\hat\Pi{}^\Ga_\ci(\cX)$ in general are not
representable.\I{C-stack@$C^\iy$-stack!representable 1-morphism}
\smallskip

\noindent{\bf(g)} $L^\Ga(\La,\cX),\hskip 1pt plus 2pt
L^\Ga_\ci(\La,\cX),\hskip 1pt plus 2pt O^\Ga(\cX),\hskip 1pt plus
2pt \ti O{}^\Ga(\cX),\hskip 1pt plus 2pt \ti\Pi{}^\Ga(\cX),\hskip
1pt plus 2pt \ti\Pi{}^\Ga_\ci(\cX),\hskip 1pt plus 2pt
\hat\Pi{}^\Ga(\cX),\hskip 1pt plus 2pt \hat\Pi{}^\Ga_\ci(\cX)$ are
all proper,\I{C-stack@$C^\iy$-stack!proper 1-morphism} but\/
$O{}^\Ga_\ci(\cX),\ti O{}^\Ga_\ci(\cX)$ in general are not.
\smallskip

\noindent{\bf(h)} $O^\Ga_\ci (\cX)_\top:\cX^\Ga_{\ci,\top}\ra
\cX_\top$ takes $\md{\Aut(\Ga)}\cdot\md{C(\Ga)}/\md{\Ga}$ points
$[x,\rho]$ of\/ $\cX^\Ga_{\ci,\top}$ to each point\/
$[x]\in\cX_\top$ with\/ $\Iso_\cX([x])\cong\Ga$. Also\/ $\ti
O{}^\Ga_\ci (\cX)_\top:\tcX^\Ga_{\ci,\top}\ra \cX_\top$ is a
bijection with the subset of\/ $[x]\in\cX_\top$
with\/~$\Iso_\cX([x])\cong\Ga$.
\label{ag9thm1}
\end{thm}

\begin{proof} For (a), we first prove that $\cX^\Ga$ is a
Deligne--Mumford $C^\iy$-stack. The {\it inertia stack\/}\I{inertia
stack} of $\cX$ is the fibre product\I{C-stack@$C^\iy$-stack!fibre
products} $\cI_\cX=\cX\t_{\De_\cX,\cX\t\cX,\De_\cX}\cX$, where
$\De_\cX:\cX\ra\cX\t\cX$ is the diagonal 1-morphism. There is a
canonical construction of fibre products of stacks. Taking $\cI_\cX$
to be given by this construction, by definition objects of the
category $\cI_\cX$ are triples $(A,B,c)$ where $A,B$ are objects in
$\cX$ with $p_\cX(A)=p_\cX(B)=\uU$ in $\CSch$, and $c:\De_\cX(A)\ra
\De_\cX(B)$ is a morphism in $\cX\t\cX$ with $p_{\cX\t\cX}(c)=
\uid_\uU$. But $\De_\cX(A)=(A,A)$ and $\De_\cX(B)=(B,B)$, so
$c=(c_1,c_2)$ for $c_1,c_2:A\ra B$ morphisms in $\cX$
with~$p_\cX(c_i)=\uid_\uU$.

Thus we may write objects of $\cI_\cX$ as quadruples $(A,B,c,d)$,
where $A,B$ are objects in $\cX$ with $p_\cX(A)=p_\cX(B)=\uU$, and
$c,d:A\ra B$ are isomorphisms in $\cX$ with $p_\cX(c)=p_\cX(d)
=\uid_\uU$. Morphisms $(A,B,c,d)\ra (A',B',c',d')$ in $\cI_\cX$ are
pairs $(a,b)$ with $a:A\ra A'$ and $b:B\ra B'$ morphisms in $\cX$
such that $b\ci c=c'\ci a$ and $b\ci d=d'\ci a$. This forces
$p_\cX(a)=p_\cX(b)$. The functor $p_{\cI_\cX}:\cI_\cX\ra\CSch$ acts
by $p_{\cI_\cX}:(A,B,c,d)\mapsto p_\cX(A)=p_\cX(B)$ on objects and
$p_{\cI_\cX}:(a,b)\mapsto p_\cX(a)=p_\cX(b)$ on morphisms.

Write $i_\cX:\cX\ra\cI_\cX$ for the 1-morphism mapping $A\mapsto
(A,A,\id_A,\id_A)$ on objects and $a\mapsto(a,a)$ on morphisms.
Since $\cX$ is Deligne--Mumford, $i_\cX$ is an equivalence with an
open and closed $C^\iy$-substack $i_\cX(\cX)$ in $\cI_\cX$. Here
$i_\cX(\cX)$ is the subcategory of objects in $\cI_\cX$ isomorphic
to some $(A,A,\id_A,\id_A)$. Thus $i_\cX(\cX)$ is the full
subcategory of objects $(A,B,c,d)$ in $\cI_\cX$ with~$c=d$.

Since $i_\cX(\cX)$ is open and closed in $\cI_\cX$, its complement
$\cJ_\cX=\cI_\cX\sm i_\cX(\cX)$ as a $C^\iy$-stack is also an open
and closed $C^\iy$-substack in $\cI_\cX$. As a subcategory,
$\cJ_\cX$ is not simply the complement of the subcategory
$i_\cX(\cX)$. Instead, $\cJ_\cX$ is the full subcategory of objects
$(A,B,c,d)$ in $\cI_\cX$ satisfying the following condition $(*)$
analogous to Definition~\ref{ag9def1}(c):
\begin{itemize}
\setlength{\itemsep}{0pt}
\setlength{\parsep}{0pt}
\item[$(*)$] Write $\uU=p_\cX(A)=p_\cX(B)$, and let $u\in\uU$,
and $\uu:\ul{*}\ra\uU$ the corresponding morphism in $\CSch$.
Since $p_\cX:\cX\ra\CSch$ is a category fibred in groupoids,
there exist $a_u:A_u\ra A$, $b_u:B_u\ra B$ in $\cX$ with
$p_\cX(A_u)=p_\cX(B_u)=\ul{*}$ and $p_\cX(a_u)=p_\cX(b_u)=\uu$,
and unique isomorphisms $c_u,d_u:A_u\ra B_u$ such that $a_u\ci
c_u=c\ci a_u$ and $a_u\ci d_u=d\ci a_u$, and
$p_\cX(c_u)=p_\cX(d_u)=\uid_{\ul{*}}$. We require that $c_u\ne
d_u$ for all $u\in\uU$.
\end{itemize}

Now form the product $\prod_{\ga\in\Ga}\cX$ of $\md{\Ga}$ copies of
$\cX$, and write $\De_\cX^\Ga: \cX\ra\prod_{\ga\in\Ga}\cX$ for the
diagonal 1-morphism. Consider the $C^\iy$-stack fibre
product\I{C-stack@$C^\iy$-stack!fibre products}
\begin{equation*}
\cY=\cX\t_{\De_\cX^\Ga,\prod_{\ga\in\Ga}\cX,\De_\cX^\Ga}\cX.
\end{equation*}
It is a Deligne--Mumford $C^\iy$-stack by Theorem \ref{ag7thm1}. As
for $\cI_\cX$, we can take objects of $\cY$ to be
$(\md{\Ga}\!+\!2)$-tuples $(A,B,c_\ga:\ga\in\Ga)$, where $A,B$ are
objects in $\cX$ with $p_\cX(A)=p_\cY(B)=\uU$, and $c_\ga:A\ra B$
for $\ga\in\Ga$ are isomorphisms in $\cX$ with
$p_\cX(c_\ga)=\uid_\uU$. Morphisms $(A,B,c_\ga:\ga\in\Ga)\ra
(A',B',c_\ga':\ga\in\Ga)$ in $\cY$ are pairs $(a,b)$ with $a:A\ra
A'$ and $b:B\ra B'$ morphisms in $\cX$ such that $b\ci
c_\ga=c_\ga'\ci a:A\ra B'$ for all $\ga\in\Ga$. The functor
$p_\cY:\cY\ra\CSch$ acts by $p_\cY:(A,B,c_\ga:\ga\in\Ga)\mapsto
p_\cX(A)=p_\cX(B)$ on objects and $p_\cY:(a,b)\mapsto
p_\cX(a)=p_\cX(b)$ on morphisms.

For $\de,\ep\in\Ga$ define $K_{\de,\ep}:\cY\ra \cI_\cX$ to map
$(A,B,c_\ga:\ga\in\Ga)\mapsto(A,B,c_{\de\ep},c_\de\ci c_1^{-1}\ci
c_\ep)$ on objects and $(a,b)\mapsto(a,b)$ on morphisms. It is easy
to show that $K_{\de,\ep}$ is a functor, with $p_{\cI_\cX}\ci
K_{\de,\ep}=p_\cY$. Hence $K_{\de,\ep}:\cY\ra\cI_\cX$ is a
1-morphism of Deligne--Mumford $C^\iy$-stacks. Thus
$K_{\de,\ep}^{-1}\bigl(i_\cX(\cX)\bigr)$ is an open and closed
$C^\iy$-substack in $\cY$, since $i_\cX(\cX)$ is open and closed
in~$\cI_\cX$.

Similarly, for $\de\ne\ep\in\Ga$, define $L_{\de,\ep}:\cY\ra
\cI_\cX$ to map $(A,B,c_\ga:\ga\in\Ga)\mapsto(A,B,c_\de,c_\ep)$ on
objects and $(a,b)\mapsto(a,b)$ on morphisms. Then $L_{\de,\ep}:\cY
\ra\cI_\cX$ is a 1-morphism, so $L_{\de,\ep}^{-1}(\cJ_\cX)$ is an
open and closed $C^\iy$-substack in $\cY$, since $\cJ_\cX$ is open
and closed in $\cI_\cX$. Define
\begin{equation*}
\cY'=\bigcap_{\de,\ep\in\Ga}K_{\de,\ep}^{-1}\bigl(i_\cX(\cX)\bigr)
\cap\bigcap_{\de\ne\ep\in\Ga}L_{\de,\ep}^{-1}\bigl(\cJ_\cX\bigr).
\end{equation*}
Then $\cY'$ is an open and closed $C^\iy$-substack in $\cY$, as it
is a finite intersection of open and closed $C^\iy$-substacks
in~$\cY$.

Define a functor $M:\cX^\Ga\ra\cY'$ to map $M:(A,\rho)\mapsto
(A,A,\rho(\ga):\ga\in\Ga)$ on objects and $M:a\mapsto (a,a)$ on
morphisms. The nontrivial claim here is that if $(A,\rho)$ is an
object in $\cX^\Ga$ then $M\bigl((A,\rho)\bigr)=(A,A,\rho(\ga):
\ga\in\Ga)$ is an object in $\cY'$. The reason for this is that as
$\rho:\Ga\ra\Aut(A)$ is a group morphism, for each $\de,\ep\in\Ga$
we have $\rho(\de\ep)=\rho(\de)\rho(\ep)=\rho(\de)\rho(1)^{-1}
\rho(\ep)$, so $(A,A,\rho(\ga):\ga\in\Ga)$ lies in
$K_{\de,\ep}^{-1}\bigl(i_\cX(\cX)\bigr)$. Also, in Definition
\ref{ag9def1}(c) $\rho_u:\Ga\ra\Aut(A_u)$ is injective, so
$\rho_u(\de)\ne\rho_u(\ep)$ for $\de\ne\ep\in\Ga$. This is
equivalent to condition $(*)$ for $L_{\de,\ep}\bigl((A,A,\rho(\ga):
\ga\in\Ga)\bigr)$, so $(A,A,\rho(\ga):\ga\in\Ga)$ lies
in~$L_{\de,\ep}^{-1}\bigl(\cJ_\cX\bigr)$.

Similarly, define a functor $N:\cY'\ra\cX^\Ga$ to map
$N:(A,B,c_\ga:\ga\in\Ga)\mapsto(A,\rho)$ on objects, where we define
$\rho(\ga)=c_1^{-1}\ci c_\ga$ for $\ga\in\Ga$, and to map
$N:(a,b)\mapsto a$ on morphisms. The nontrivial claim is that if
$(A,B,c_\ga:\ga\in\Ga)$ is an object in $\cY'$ then
$N\bigl((A,B,c_\ga:\ga\in\Ga)\bigr)=(A,\rho)$ is an object in
$\cX^\Ga$. This holds because $(A,B,c_\ga:\ga\in\Ga)\in
K_{\de,\ep}^{-1}(i_\cX(\cX))$ forces $\rho(\de\ep)=
\rho(\de)\rho(\ep)$ for all $\de,\ep$, so $\rho:\Ga\ra\Aut(A)$ is a
group morphism, and $(A,B,c_\ga:\ga\in\Ga)\in
L_{\de,\ep}^{-1}(\cJ_\cX)$ for $\de\ne\ep$ forces
$\rho_u(\de)\ne\rho_u(\ep)$ in Definition \ref{ag9def1}(c), so
$\rho_u$ is injective.

Now $N\ci M=\id_{\smash{\cX^\Ga}}$, and there is a natural
transformation $\eta:M\ci N\Ra\id_{\smash{\cY'}}$ acting by
$\eta:(A,B,c_\ga:\ga\in\Ga)\mapsto(\id_A,c_1)$. So $\cX^\Ga,\cY'$
are equivalent categories. Also $p_{\smash{\cY'}}\ci
M=p_{\smash{\cX^\Ga}}$ and $p_{\smash{\cX^\Ga}}\ci
N=p_{\smash{\cY'}}$. Therefore $M,N$ define equivalences of
$C^\iy$-stacks, so as $\cY'$ is a Deligne--Mumford $C^\iy$-stack,
$\cX^\Ga$ is also a Deligne--Mumford $C^\iy$-stack equivalent to
$\cY'$. This proves the first part of~(a).

To see that $\cX{}^\Ga_\ci$ is an open $C^\iy$-substack%
\I{C-stack@$C^\iy$-stack!C-substack@$C^\iy$-substack!open}%
\I{C-substack@$C^\iy$-substack!open} of $\cX^\Ga$, note that the map
$\cX_\top\ra\N$ mapping $[x]\mapsto\bmd{\Iso_\cX([x])}$ is upper
semicontinuous, so the subset of points $[x]$ in $\cX_\top$ with
$\bmd{\Iso_\cX([x])}\le\md{\Ga}$ is open, and corresponds to an open
$C^\iy$-substack $\cX_{\le\md{\Ga}}$ in $\cX$. But then
$\cX{}^\Ga_\ci\simeq\cX^\Ga\t_{O^\Ga(\cX),\cX,{\rm inc}}
\cX_{\le\md{\Ga}}$, so $\cX{}^\Ga_\ci$ is the open $C^\iy$-substack
in $\cX^\Ga$ corresponding to $\cX_{\le\md{\Ga}}$ in $\cX$, as we
have to prove.

Now $L^\Ga(-,\cX)$ defines an action of the finite group $\Aut(\Ga)$
on the Deligne--Mumford $C^\iy$-stack $\cX^\Ga$ by 1-isomorphisms,
so we may form the quotient $C^\iy$-stack $[\cX^\Ga/\Aut(\Ga)]$,
which is also a Deligne--Mumford $C^\iy$-stack. To define
$[\cX^\Ga/\Aut(\Ga)]$ we first define a prestack $\cX^\Ga/\Aut(\Ga)$
which is the quotient of the category $\cX^\Ga$ by $\Aut(\Ga)$, and
then $[\cX^\Ga/\Aut(\Ga)]$ is its stackification. Since ${\cal
P}\tcX^\Ga$ was defined to be equivalent to $\cX^\Ga/\Aut(\Ga)$, its
stackification $\tcX^\Ga$ is equivalent to $[\cX^\Ga/\Aut(\Ga)]$.
This proves that $\tcX^\Ga$ is a Deligne--Mumford $C^\iy$-stack and
$\tcX^\Ga \simeq[\cX^\Ga/\Aut(\Ga)]$, as in (a). Similarly
$\tcX^\Ga_\ci\subseteq\tcX^\Ga$ is an open $C^\iy$-substack,
and~$\tcX^\Ga_\ci\simeq[\cX^\Ga_\ci/\Aut(\Ga)]$.

To show $\hcX^\Ga$ is Deligne--Mumford, we first observe that ${\cal
P}\hcX^\Ga$ is a prestack,\I{prestack} so $\hcX^\Ga$ is a stack on
$(\CSch,{\cal J})$, and then either note that $\hat\Pi{}^\Ga(\cX):
\tcX^\Ga\ra\hcX^\Ga$ has fibre $[\ul{*}/\Ga]$ and $\tcX^\Ga$ is
Deligne--Mumford, or use the local models for $\hcX^\Ga$ given by
Theorem \ref{ag9thm2}. Then $\hcX^\Ga_\ci\subseteq \hcX^\Ga$ is
open as for $\cX{}^\Ga_\ci,\tcX^\Ga_\ci$. This completes~(a).

For (b), if $\cX$ is separated, locally fair, locally finitely
presented, second countable, or compact, then
$\cY=\cX\t_{\smash{\prod_\ga\cX}}\cX$ is separated, \ldots, compact,
so $\cX^\Ga$ are separated, \ldots, compact as it is equivalent to
an open and closed $C^\iy$-substack $\cY'$ of $\cY$, and
$\cX{}^\Ga_\ci$ is separated, locally fair, locally finitely
presented, or second countable (but not necessarily compact) as it
is open in $\cX^\Ga$. The result for
$\tcX^\Ga,\tcX^\Ga_\ci,\hcX^\Ga,\hcX^\Ga_\ci$ follows as
$\tcX^\Ga\simeq[\cX^\Ga/\Aut(\Ga)]$,
$\tcX^\Ga_\ci\simeq[\cX^\Ga_\ci/\Aut(\Ga)]$, and
$\hcX^\Ga,\hcX^\Ga_\ci$ fibre over $\tcX^\Ga, \tcX^\Ga_\ci$ with
fibre~$[\bar{\ul{*}}/\Ga]$.

For (c), there is a 1-1 correspondence between 1-morphisms
$x:\bar{\ul{*}}\ra\cX$ and objects $A_x$ in $\cX$ with
$p_\cX(A_x)=\ul{*}$, and if $x,y:\bar{\ul{*}}\ra\cX$ correspond to
$A_x,A_y$ in $\cX$ there is a 1-1 correspondence between 2-morphisms
$\eta:x\Ra y$ and morphisms $a_\eta:A_x\ra A_y$ in $\cX$ with
$p_\cX(a_\eta)=\uid_{\ul{*}}$. The same correspondences hold for
$\cX^\Ga$. Thus, each 1-morphism $y:\bar{\ul{*}}\ra\cX^\Ga$
corresponds uniquely to some $(B,\si)$ in $\cX^\Ga$ with
$p_\cX(B)=\ul{*}$, so $B=A_x$ for some unique 1-morphism
$x:\bar{\ul{*}}\ra\cX$, and each $\si(\ga):A_x\ra A_x$ is
$a_{\rho(\ga)}$ for some unique 2-morphism $\rho(\ga):x\Ra x$, and
$\rho:\Ga\ra\Aut(x)$ is a group morphism. Definition \ref{ag9def1}
implies that $\rho$ is injective.

This establishes a 1-1 correspondence between 1-morphisms
$y:\bar{\ul{*}}\ra\cX^\Ga$ and pairs $(x,\rho)$, where
$x:\bar{\ul{*}}\ra\cX$ is a 1-morphism and $\rho:\Ga\ra\Aut(x)$ an
injective group morphism. Similarly, if $y,y':\bar{\ul{*}}
\ra\cX^\Ga$ correspond to $(x,\rho),(x',\rho')$ then 2-morphisms
$\th:y\Ra y'$ correspond to 2-morphisms $\ze:x\Ra x'$ with
$\ze\od\rho(\ga)=\rho'(\ga)\od\ze:x\Ra x'$ for all $\ga\in\Ga$. Also
1-morphisms $y:\bar{\ul{*}}\ra\cX^\Ga_\ci$ correspond to pairs
$(x,\rho)$ with $\rho:\Ga\ra\Aut(x)$ an isomorphism. Part (c) then
follows. Parts (d),(e) come from the definitions of ${\cal
P}\tcX^\Ga,\ldots,{\cal P}\hcX^\Ga_\ci$ in the same way, noting that
stackifying does not change 1-morphisms $\bar{\ul{*}}\ra{\cal
P}\tcX^\Ga$ or their 2-morphisms.

For (f), $L^\Ga(\La,\cX)$ is representable as it is a 1-isomorphism. Suppose $(A,\rho)$ is an object in $\cX^\Ga$ with $p_{\smash{\cX^\Ga}} (A,\rho)=\uU$, so that
$O^\Ga(\cX):(A,\rho)\mapsto A$, and $a:A\ra A'$ is an isomorphism in
$\cX$ with $p_\cX(a)=\uid_\uU$. Then $a:(A,\rho)\ra (A',a\ci\rho\ci
a^{-1})$ is the unique isomorphism in $\cX^\Ga$ with
$O^\Ga(\cX):a\mapsto a$, so $O^\Ga(\cX)$ is representable.
The action ${\cal P}\ti O{}^\Ga(\cX):(c,\io)\mapsto c$ of ${\cal
P}\ti O{}^\Ga(\cX)$ on 1-morphisms is injective, as $c$ determines
$\io$ by $\io(\de)\ci c=c\ci\de$ for $\de\in\De$. This implies that
the stackification $\ti O{}^\Ga(\cX)$ is representable. Then $\ti\Pi{}^\Ga(\cX)$ representable follows from $\ti O{}^\Ga(\cX)\ci\ti\Pi{}^\Ga(\cX)\cong O^\Ga(\cX)$ with $O^\Ga(\cX),\ti O{}^\Ga(\cX)$ representable. Also $L^\Ga_\ci(\La,\cX),
O^\Ga_\ci(\cX),\ti O{}^\Ga_\ci(\cX),\ti\Pi{}^\Ga_\ci(\cX)$ are
representable, as they are restrictions of $L^\Ga(\La,\cX),\ldots,\ti\Pi{}^\Ga(\cX)$ to open $C^\iy$-substacks. The actions of $\hat\Pi{}^\Ga(\cX),\hat
\Pi{}^\Ga_\ci(\cX)$ on isotropy groups have kernels isomorphic to
$\Ga$. So if $\Ga\ne\{1\}$ these actions are not injective, and
$\hat\Pi{}^\Ga(\cX),\hat \Pi{}^\Ga_\ci(\cX)$ are not representable.

For (g), $L^\Ga(\La,\cX),\,L^\Ga_\ci(\La,\cX)$ are 1-isomorphisms,
$\ti\Pi{}^\Ga(\cX),\ti\Pi{}^\Ga_\ci(\cX)$ project to quotients by
$\Aut(\Ga)$, and $\hat\Pi{}^\Ga(\cX),\hat\Pi{}^\Ga_\ci(\cX)$ are
fibrations with fibre $[\bar{\ul{*}}/\Ga]$, so these are all proper.
We can see that $O^\Ga(\cX),\ti O{}^\Ga(\cX)$ are proper, but
$O^\Ga_\ci(\cX),\ti O{}^\Ga_\ci(\cX)$ in general are not, using
Theorem \ref{ag9thm2} and the fact that every Deligne--Mumford
$C^\iy$-stack is locally of the form~$[\uX/G]$.

For (h), if $[x]\in\cX_\top$ with $\Iso_\cX([x])\cong\Ga$, then by
(c) points $[x,\rho]\in\cX^\Ga_{\ci,\top}$ with $O^\Ga_\ci
(\cX)_\top: [x,\rho]\mapsto[x]$ are given by isomorphisms
$\rho:\Ga\ra\Iso_\cX([x])$. There are $\md{\Aut(\Ga)}$ such $\rho$.
If $\rho,\rho'$ are two such isomorphisms, then (c) shows
$[x,\rho]=[x,\rho']$ if and only if $\rho'=\rho^\al$ for some
$\al\in\Ga$, where $\rho^\al:\ga\mapsto\al\ga\al^{-1}$. For
$\al_1,\al_2\in\Ga$, we see that $\rho^{\al_1}= \rho^{\al_2}$ if and
only if $(\al_2^{-1}\al_1)\ga=\ga(\al_2^{-1}\al_1)$ for all
$\ga\in\Ga$, that is, if $\al_2^{-1}\al_1\in C(\Ga)$. Hence, the
$\rho^\al$ for $\al\in\Ga$ realize $\md{\Ga}/\md{C(\Ga)}$ distinct
isomorphisms $\rho':\Ga\ra\Iso_\cX([x])$. So the $\md{\Aut(\Ga)}$
isomorphisms $\rho:\Ga\ra\Iso_\cX([x])$ are identified in groups of
$\md{\Ga}/\md{C(\Ga)}$ to make $\md{\Aut(\Ga)}\cdot\md{C(\Ga)}/
\md{\Ga}$ points $[x,\rho]$ in $\cX^\Ga_{\ci,\top}$. The statement
for $\ti O{}^\Ga_\ci (\cX)_\top$ is immediate as if
$[x,\De]\in\tcX^\Ga_{\ci,\top}$ then $\De=\Aut(x)$, so
$[x,\De]\mapsto[x]$ is a 1-1 correspondence. This completes the
proof of Theorem~\ref{ag9thm1}.
\end{proof}

\begin{ex} Let $\cX$ be a Deligne--Mumford $C^\iy$-stack, and
$\cI_\cX$ the {\it inertia stack\/}\I{inertia stack} of $\cX$, as in
the proof of Theorem \ref{ag9thm1}. Then there is an equivalence
\begin{equation*}
\ts\cI_\cX=\cX\t_{\De_\cX,\cX\t\cX,\De_\cX}\cX\simeq\coprod_{k\ge
1}\cX^{\Z_k}.
\end{equation*}
To see this, note that points of $\cI_\cX$ are equivalence classes
$[x,\eta]$, where $[x]\in\cX_\top$ and $\eta\in\Iso_\cX([x])$. Since
$\cX$ is Deligne--Mumford, $\Iso_\cX([x])$ is a finite group, so
each $\eta\in\Iso_\cX([x])$ has some finite order $k\ge 1$, and
generates an injective morphism $\rho:\Z_k\ra\Iso_\cX([x])$ mapping
$\rho:a\mapsto\eta^a$. We may identify $\cX^{\Z_k}$ with the open
and closed $C^\iy$-substack of $[x,\eta]$ in $\cI_\cX$ for which
$\eta$ has order~$k$.
\label{ag9ex1}
\end{ex}

\subsection{Lifting 1- and 2-morphisms to orbifold strata}
\label{ag92}
\I{Deligne--Mumford $C^\iy$-stack!orbifold strata!lifting 1- and
2-morphisms to|(}

The construction of $\cX^\Ga,\tcX^\Ga,\hcX^\Ga$ extends functorially
to 1- and 2-morphisms.

\begin{dfn} Let $\cX,\cY$ be Deligne--Mumford $C^\iy$-stacks, $\Ga$
a finite group, and $f:\cX\ra\cY$ a representable
1-morphism,\I{C-stack@$C^\iy$-stack!representable 1-morphism|(} so
that $f:\cX\ra\cY$ is a functor with $p_\cY\ci f=p_\cX$. We will
define a representable 1-morphism~$f^\Ga:\cX^\Ga\ra\cY^\Ga$.

On objects $(A,\rho)$ in $\cX^\Ga$, define
$f^\Ga(A,\rho)=(f(A),f\ci\rho)$. We must check that $f^\Ga(A,\rho)$
satisfies Definition \ref{ag9def1}(a)--(c). Parts (a),(b) hold as
$f$ is a functor with $p_\cY\ci f=p_\cX$. For (c), if $u\in\uU$ then
(c) for $(A,\rho)$ shows that $\rho_u:\Ga\ra\Aut(A_u)$ is injective,
so $f\ci\rho_u:\Ga\ra\Aut(f(A_u))$ is injective as $f$ is
representable, and this gives (c) for $\bigl(f(A),f\ci\rho\bigr)$.
On morphisms $c:(A,\rho)\ra(B,\si)$ in $\cX^\Ga$, define
$f^\Ga(c):f^\Ga(A,\rho)\ra f^\Ga(B,\si)$ by~$f^\Ga(c)=f(c):f(A)\ra
f(B)$.

Then $f^\Ga:\cX^\Ga\ra\cY^\Ga$ is a functor, and $p_\cY\ci f=p_\cX$
implies that $p_{\smash{\cY^\Ga}}\ci f^\Ga=p_{\smash{\cX^\Ga}}$.
Hence $f^\Ga:\cX^\Ga\ra\cY^\Ga$ is a 1-morphism of $C^\iy$-stacks.
It is the unique such 1-morphism with $O^\Ga(\cY)\ci f^\Ga=f\ci
O^\Ga(\cX):\cX^\Ga\ra\cY$. Also, $f^\Ga$ is injective on morphisms,
as $f$ is, so $f^\Ga$ is representable.

Now let $f,g:\cX\ra\cY$ be representable, and $\eta:f\Ra g$ be a
2-morphism. Then $f,g:\cX\ra\cY$ are functors, and $\eta:f\Ra g$ is
a natural isomorphism. Define $\eta^\Ga:f^\Ga\Ra g^\Ga$ by taking
the isomorphism $\eta^\Ga(A,\rho):f^\Ga(A,\rho)\ra g^\Ga(A,\rho)$ in
$\cY^\Ga$ for each object $(A,\rho)$ in $\cX^\Ga$ to be the
isomorphism $\eta^\Ga(A,\rho)=\eta(A):f(A)\ra g(A)$ in $\cY$. Then
$\eta^\Ga:f^\Ga\Ra g^\Ga$ is a natural isomorphism of functors, and
hence a 2-morphism in $\DMCSta$. It is the unique such 2-morphism
with~$\id_{O^\Ga(\cY)}*\eta^\Ga=\eta*\id_{O^\Ga(\cX)}$.

Write $\DMCSta^{\bf re}$\G[DMCStare]{$\DMCSta^{\bf re}$}{2-category
of Deligne--Mumford $C^\iy$-stacks with representable 1-morphisms}
for the 2-subcategory of $\DMCSta$ with only representable
1-morphisms. Define $F^\Ga:\DMCSta^{\bf re}\ra\DMCSta^{\bf re}$ by $F^\Ga:\cX\mapsto F^\Ga(\cX)=\cX^\Ga$ on objects, $F^\Ga:f\mapsto F^\Ga(f)=f^\Ga$ on representable
1-morphisms, and $F^\Ga:\eta\mapsto F^\Ga(\eta)=\eta^\Ga$ on
2-morphisms. Then $F^\Ga$ is a strict 2-functor, in the sense of~\S\ref{agA1}.
\label{ag9def5}
\end{dfn}

\begin{dfn} Let $\cX,\cY$ be Deligne--Mumford $C^\iy$-stacks, $\Ga$
a finite group, and $f:\cX\ra\cY$ a representable 1-morphism. Define functors
${\cal P}\ti f{}^\Ga:{\cal P}\tcX^\Ga\ra{\cal P}\ti\cY{}^\Ga$
mapping $(A,\De)\mapsto(f(A),f(\De))$ on objects and
$(c,\io)\mapsto(f(c),f\ci\io\ci f\vert_\De^{-1})$ on morphisms, and
${\cal P}\hat f{}^\Ga:{\cal P}\hcX^\Ga\ra{\cal P}\hat\cY{}^\Ga$
mapping $(A,\De)\mapsto(f(A),f(\De))$ and
$(c,\io)\De\mapsto(f(c),f\ci\io\ci f\vert_\De^{-1})f(\De)$. Then
${\cal P}\ti f{}^\Ga,{\cal P}\hat f{}^\Ga$ are 1-morphisms of
prestacks, so stackifying gives 1-morphisms $\ti
f{}^\Ga:\tcX^\Ga\ra\ti\cY{}^\Ga$ and $\hat
f{}^\Ga:\hcX^\Ga\ra\hat\cY{}^\Ga$.

If $f,g:\cX\ra\cY$ are representable, and $\eta:f\Ra g$ is a
2-morphism, we define ${\cal P}\ti\eta{}^\Ga:{\cal P}\ti
f{}^\Ga\Ra{\cal P}\ti g{}^\Ga$ and ${\cal P}\hat\eta{}^\Ga:{\cal
P}\hat f{}^\Ga\Ra{\cal P}\hat g{}^\Ga$ by ${\cal
P}\ti\eta{}^\Ga:(A,\De)\mapsto (\eta(A),\io^\eta)$, where
$\io^\eta:f(\De)\ra g(\De)$ maps $\io^\eta:f(\de)\mapsto
g(\de)=\eta(A)\ci f(\de)\ci\eta(A)^{-1}$ for $\de\in\De$, and ${\cal
P}\hat\eta{}^\Ga:(A,\De)\mapsto (\eta(A),\io^\eta)f(\De)$. Then
${\cal P}\ti\eta{}^\Ga,{\cal P}\hat\eta{}^\Ga$ are 2-morphisms of
prestacks, so stackifying gives 2-morphisms $\ti\eta{}^\Ga:\ti
f{}^\Ga\Ra\ti g{}^\Ga$ and~$\hat\eta{}^\Ga:\hat f{}^\Ga\Ra\hat
g{}^\Ga$.

As in Definition \ref{ag9def5}, we would like to define 2-functors
\e
\ti F{}^\Ga,\hat F{}^\Ga:\DMCSta^{\bf re}\longra\DMCSta^{\bf re}
\label{ag9eq4}
\e
by $\ti F{}^\Ga(\cX)=\tcX^\Ga$ on objects, $\ti F{}^\Ga(f)=\ti f{}^\Ga$ on 1-morphisms, $\ti F{}^\Ga(\eta)=\ti\eta{}^\Ga$ on 2-morphisms, and so on. But there is a difference. Stackifications of 1-morphisms of prestacks involve arbitrary choices, and are unique only up to 2-isomorphism. Therefore strict equalities of 1-morphisms of prestacks translate, on stackification, to 2-isomorphisms of their stackifications, rather than strict equalities. 

For representable 1-morphisms $f:\cX\ra\cY$, $g:\cY\ra\cZ$, in prestack
1-morphisms we have ${\cal P}\ti g{}^\Ga\ci{\cal P}\ti f{}^\Ga={\cal P}\,\,\,\,\widetilde{\!\!\!\!(g\ci f)\!\!\!\!}\,\,\,\,{}^\Ga$. Thus, stackification gives a 2-isomorphism $\ti F{}^\Ga_{g,f}:\ti g{}^\Ga\ci\ti f{}^\Ga\Ra\,\,\,\,\widetilde{\!\!\!\!(g\ci f)\!\!\!\!}\,\,\,\,{}^\Ga$, which need not be the identity. Similarly, $\cP(\,\widetilde{\!\id_\cX\!}\,)^\Ga=\id_{{\cal P}\tcX^\Ga}:{\cal P}\tcX^\Ga\ra{\cal P}\ti\cX{}^\Ga$, but on stackification we get a 2-isomorphism $\ti F{}^\Ga_\cX:(\,\widetilde{\!\id_\cX\!}\,)^\Ga\Ra\id_{\tcX^\Ga}$, which need not be the identity. Because of this, $\ti F{}^\Ga,\hat F{}^\Ga$ in \eq{ag9eq4} are weak 2-functors rather than strict 2-functors, in the sense of~\S\ref{agA1}.

The 1-morphisms in \eq{ag9eq2} are compatible with $\ti
f{}^\Ga,\hat f{}^\Ga$ by 2-isomorphisms
\begin{align*}
\ti O{}^\Ga(\cY)\!\ci\!\ti
f{}^\Ga\!\cong\! f\!\ci\!\ti O{}^\Ga(\cX),\;\>
\ti\Pi{}^\Ga(\cY)\!\ci\! f^\Ga\!\cong\! \ti f{}^\Ga\!\ci\!\ti\Pi{}^\Ga(\cX),\;\>
\hat\Pi{}^\Ga(\cY)\!\ci\!\ti f{}^\Ga\!\cong\!\hat
f{}^\Ga\!\ci\!\hat\Pi{}^\Ga(\cX),
\end{align*}
which follow by stackifying equalities of 1-morphisms of prestacks.%
\I{C-stack@$C^\iy$-stack!representable 1-morphism|)}
\label{ag9def6}
\end{dfn}

\begin{rem} For $f:\cX\ra\cY$ and $\Ga$ as above, the restriction
$f^\Ga\vert_{\smash{\cX^\Ga_\ci}}$ need not map
$\cX^\Ga_\ci\ra\cY^\Ga_\ci$, but only $\cX^\Ga_\ci\ra\cY^\Ga$,
unless $f$ induces isomorphisms on isotropy groups. Thus we do not
define a 1-morphism $f^\Ga_\ci:\cX^\Ga_\ci\ra\cY^\Ga_\ci$, or a
2-functor $F^\Ga_\ci:\DMCSta^{\bf re}\ra\DMCSta^{\bf re}$. The same
applies for the actions of $f$ on orbifold
strata~$\tcX^\Ga_\ci,\hcX^\Ga_\ci$.\I{prestack!stackification|)}%
\I{Deligne--Mumford $C^\iy$-stack!orbifold strata!lifting 1- and
2-morphisms to|)}
\label{ag9rem1}
\end{rem}

\subsection{\texorpdfstring{Orbifold strata of quotient $C^\iy$-stacks
$[\protect\uX/G]$}{Orbifold strata of quotient C∞-stacks
[X/G]}}
\label{ag93}
\I{Deligne--Mumford $C^\iy$-stack!orbifold
strata!of quotient $C^\iy$-stacks|(}%
\I{C-stack@$C^\iy$-stack!quotients $[\protect\uX/G]$!orbifold strata|(}%
\I{quotient C-stack@quotient $C^\iy$-stack!orbifold strata|(}

The next theorem describes $\cX^\Ga,\ldots,\hat\uX{}^\Ga_\ci$
explicitly when $\cX$ is a quotient $C^\iy$-stack $[\uX/G]$, as in
\S\ref{ag71}. We can prove it by showing the explicit constructions
of Definition \ref{ag7def1} and Definitions
\ref{ag9def1}--\ref{ag9def2} commute up to equivalence.

\begin{thm} Let\/ $\uX$ be a
Hausdorff\/\I{C-scheme@$C^\iy$-scheme!Hausdorff} $C^\iy$-scheme
and\/ $G$ a finite group acting on $\uX$ by isomorphisms, and
write\/ $\cX=[\uX/G]$ for the quotient\/ $C^\iy$-stack, which is a
Deligne--Mumford\/ $C^\iy$-stack. Let\/ $\Ga$ be a finite group.
Then there are equivalences of\/ $C^\iy$-stacks
\ea
\cX^\Ga&\simeq\ts\bigl[\bigl(\coprod_{\text{injective group
morphisms $\rho:\Ga\ra G$}}\uX^{\rho(\Ga)}\bigr)/G\bigr],
\label{ag9eq5}\\
\cX{}^\Ga_\ci&\simeq\ts\bigl[\bigl(\coprod_{\text{injective group
morphisms $\rho:\Ga\ra G$}}\uX^{\rho(\Ga)}_\ci\bigr)/G\bigr],
\label{ag9eq6}\\
\tcX^\Ga&\simeq\ts\bigl[\bigl(\coprod_{\text{subgroups $\De\subseteq
G$: $\De\cong\Ga$}}\uX^\De\bigr)/G\bigr],
\label{ag9eq7}\\
\tcX^\Ga_\ci&\simeq\ts\bigl[\bigl(\coprod_{\text{subgroups
$\De\subseteq G$: $\De\cong\Ga$}}\uX^\De_\ci\bigr)/G\bigr],
\label{ag9eq8}
\ea
where for each subgroup\/ $\De\subseteq G,$ we write\/ $\uX^\De$ for
the closed\/ $C^\iy$-subscheme in\/ $\uX$ fixed by $\De$ in\/ $G,$
and\/ $\uX^\De_\ci$ for the open $C^\iy$-subscheme in $\uX^\De$ of
points in $\uX$ whose stabilizer group in $G$ is exactly\/~$\De$.

Here the action of\/ $G$ on $\coprod_\rho\uX^{\rho(\Ga)}$ in
\eq{ag9eq5} is defined as follows. Let\/ $g\in G$ and\/
$\rho:\Ga\ra G$ be an injective morphism. Define another injective
morphism $\rho^g:\Ga\ra G$ by $\rho^g:\ga\mapsto g\rho(\ga)g^{-1}$.
Then $g(\uX^{\rho(\Ga)})=\uX^{\rho^g(\Ga)},$ as $C^\iy$-subschemes
of\/ $\uX,$ and the action of\/ $g$ on $\coprod_\rho\uX^{\rho(\Ga)}$
maps $\uX^{\rho(\Ga)}\ra \uX^{\rho^g(\Ga)}$ by the restriction of\/
$g:\uX\ra\uX$ to\/ $\uX^{\rho(\Ga)}$. The $G$-actions for
\eq{ag9eq6}--\eq{ag9eq8} are similar.

We can also rewrite equations \eq{ag9eq5}--\eq{ag9eq8} as
\ea
\cX^\Ga&\simeq\coprod_{\begin{subarray}{l}\text{conjugacy classes
$[\rho]$ of injective}\\ \text{group morphisms $\rho:\Ga\ra
G$}\end{subarray}\!\!\!\!\!\!\!\!\!\!\!\!\!\!\!\!\!\!\!
\!\!\!\!\!\!\!\!\!} \bigl[\,\uX^{\rho(\Ga)}/\bigl\{g\in
G:g\rho(\ga)=\rho(\ga)g\;\> \forall\ga\in\Ga\bigr\}\bigr],
\label{ag9eq9}\\
\cX^\Ga_\ci&\simeq\coprod_{\begin{subarray}{l}\text{conjugacy
classes $[\rho]$ of injective}\\ \text{group morphisms $\rho:\Ga\ra
G$}\end{subarray}\!\!\!\!\!\!\!\!\!\!\!\!\!\!\!\!\!\!\!
\!\!\!\!\!\!\!\!\!} \bigl[\,\uX^{\rho(\Ga)}_\ci/\bigl\{g\in
G:g\rho(\ga)=\rho(\ga)g\;\> \forall\ga\in\Ga\bigr\}\bigr],
\label{ag9eq10}\\
\tcX^\Ga&\simeq\coprod_{\text{conjugacy classes $[\De]$ of subgroups
$\De\subseteq G$ with $\De\cong\Ga$} \!\!\!\!\!\!\!\!\!
\!\!\!\!\!\!\!\!\!\!\!\!\!\!\!\!\!\!\!\!\!\!\!\!\!\!\!\!\!\!\!\!
\!\!\!\!\!\!\!\!\!\!\!\!\!\!\!\!\!\!\!\!\!\!\!}
\bigl[\,\uX^\De/\bigl\{g\in G:\De=g\De g^{-1}\bigr\}\bigr],
\label{ag9eq11}\\
\tcX^\Ga_\ci&\simeq\coprod_{\text{conjugacy classes $[\De]$ of
subgroups $\De\subseteq G$ with $\De\cong\Ga$} \!\!\!\!\!\!\!\!\!
\!\!\!\!\!\!\!\!\!\!\!\!\!\!\!\!\!\!\!\!\!\!\!\!\!\!\!\!\!\!\!\!
\!\!\!\!\!\!\!\!\!\!\!\!\!\!\!\!\!\!\!\!\!\!\!}
\bigl[\,\uX^\De_\ci/\bigl\{g\in G:\De=g\De g^{-1}\bigr\}\bigr].
\label{ag9eq12}
\ea
Here morphisms $\rho,\rho':\Ga\ra G$ are conjugate if\/
$\rho'=\rho^g$ for some $g\in G,$ and subgroups $\De,\De'\subseteq
G$ are conjugate if\/ $\De=g\De'g^{-1}$ for some $g\in G$. In
\eq{ag9eq9}--\eq{ag9eq12} we sum over one representative $\rho$ or
$\De$ for each conjugacy class.

In the notation of\/ {\rm\eq{ag9eq11}--\eq{ag9eq12},} there are
equivalences of\/ $C^\iy$-stacks
\ea
\hcX^\Ga&\simeq\coprod_{\text{conjugacy classes $[\De]$ of subgroups
$\De\subseteq G$ with $\De\cong\Ga$} \!\!\!\!\!\!\!\!\!
\!\!\!\!\!\!\!\!\!\!\!\!\!\!\!\!\!\!\!\!\!\!\!\!\!\!\!\!\!\!\!\!
\!\!\!\!\!\!\!\!\!\!\!\!\!\!\!\!\!\!\!\!\!\!\!}
\bigl[\,\uX^\De\big/\bigl(\{g\in G:\De=g\De
g^{-1}\}/\De\bigr)\bigr],
\label{ag9eq13}\\
\hcX^\Ga_\ci&\simeq\coprod_{\text{conjugacy classes $[\De]$ of
subgroups $\De\subseteq G$ with $\De\cong\Ga$} \!\!\!\!\!\!\!\!\!
\!\!\!\!\!\!\!\!\!\!\!\!\!\!\!\!\!\!\!\!\!\!\!\!\!\!\!\!\!\!\!\!
\!\!\!\!\!\!\!\!\!\!\!\!\!\!\!\!\!\!\!\!\!\!\!}
\bigl[\,\uX^\De_\ci\big/\bigl(\{g\in G:\De=g\De
g^{-1}\}/\De\bigr)\bigr].
\label{ag9eq14}
\ea

Under the equivalences {\rm\eq{ag9eq5}--\eq{ag9eq14},} the
$1$-morphisms in \eq{ag9eq2} are identified up to $2$-isomorphism
with\/ $1$-morphisms between quotient\/ $C^\iy$-stacks induced by
natural\/ $C^\iy$-scheme morphisms between
$\coprod_\rho\uX^{\rho(\Ga)},\uX,\ldots.$ For example, the disjoint
union over $\rho$ of the inclusion $\uX^{\rho(\Ga)} \hookra\uX$ is a
$G$-equivariant morphism $\coprod_\rho\uX^{\rho(\Ga)}\ra\uX,$
inducing a $1$-morphism $\smash{[\coprod_\rho\uX^{\rho(\Ga)}/G]\ra
[\uX/G]}$. This is identified with\/ $O^\Ga(\cX):\cX^\Ga\ra\cX$ by\/
\eq{ag9eq5}. Similarly, $\ti\Pi^\Ga(\cX):\cX^\Ga\ra\tcX^\Ga$ is
identified by {\rm\eq{ag9eq5}, \eq{ag9eq7}} with the $1$-morphism
$\smash{[\coprod_\rho\uX^{\rho(\Ga)}/G]\ra[\coprod_\De\uX^\De/G]}$
induced by the $C^\iy$-scheme morphism
$\coprod_\rho\uX^{\rho(\Ga)}\ra \coprod_\De\uX^\De$ mapping
morphisms $\rho$ to subgroups $\De=\rho(\Ga),$ and acting by
$\uid_{\smash{\uX^\De}}: \uX^{\rho(\Ga)}\ra\uX^\De$ for
$\De=\rho(\Ga)$.\I{Deligne--Mumford $C^\iy$-stack!orbifold
strata!of quotient $C^\iy$-stacks|)}%
\I{C-stack@$C^\iy$-stack!quotients $[\protect\uX/G]$!orbifold strata|)}%
\I{quotient C-stack@quotient $C^\iy$-stack!orbifold strata|)}
\label{ag9thm2}
\end{thm}

\subsection{Sheaves on orbifold strata}
\label{ag94}
\I{Deligne--Mumford $C^\iy$-stack!quasicoherent sheaves on|(}\I{Deligne--Mumford $C^\iy$-stack!orbifold strata!sheaves on|(}

Let $\cX$ be a Deligne--Mumford $C^\iy$-stack, $\Ga$ a finite group,
and $\cE\in\qcoh(\cX)$, so that $\cE^\Ga:=O^\Ga(\cX)^*(\cE)\in
\qcoh(\cX^\Ga)$. We will show that there is a natural representation
of $\Ga$ on $\cE^\Ga$, and also the action of $\Aut(\Ga)$ on
$\cX^\Ga$ lifts to $\cE^\Ga$, so that $\Aut(\Ga)\lt\Ga$ acts
equivariantly on~$\cE^\Ga$.

\begin{dfn} Let $\cX$ be a Deligne--Mumford $C^\iy$-stack, and $\Ga$ a
finite group, so that \S\ref{ag91} defines the orbifold stratum
$\cX^\Ga$, a 1-morphism $O^\Ga(\cX):\cX^\Ga\!\ra\!\cX$, an action of
$\Aut(\Ga)$ on $O^\Ga(\cX)$ by 2-isomorphisms
$E^\Ga(\ga,\cX):O^\Ga(\cX)\!\Ra\!O^\Ga(\cX)$, and an action of
$\Aut(\Ga)$ on $\cX^\Ga$ by
1-isomorphisms~$L^\Ga(\La,\cX):\cX^\Ga\!\ra\!\cX^\Ga$.

Suppose $\cE$ is a quasicoherent sheaf on $\cX,$ and write $\cE^\Ga$
for the pullback sheaf $O^\Ga(\cX)^*(\cE)$ in $\qcoh(\cX^\Ga)$.
Using the notation of Definition \ref{ag8def3}, for each
$\ga\in\Ga$ and $\La\in\Aut(\Ga)$ define morphisms
$R^\Ga(\ga,\cE):\cE^\Ga\ra\cE^\Ga$ and $S^\Ga(\La,\cE):
L^\Ga(\La,\cX)^*(\cE^\Ga)\ra\cE^\Ga$ in $\qcoh(\cX^\Ga)$ by
\begin{align*}
R^\Ga(\ga,\cE)&=E^\Ga(\ga,\cX)^*(\cE):O^\Ga(\cX)^*(\cE)\longra
O^\Ga(\cX)^*(\cE)\qquad\text{and}\\
S^\Ga(\La,\cE)&=I_{L^\Ga(\La,\cX),O^\Ga(\cX)}(\cE)^{-1}:
L^\Ga(\La,\cX)^*\ci O^\Ga(\cX)^*(\cE)\longra O^\Ga(\cX)^*(\cE),
\end{align*}
where the definition of $S^\Ga(\La,\cE)$ uses~$O^\Ga(\cX)\ci
L^\Ga(\La,\cX)=O^\Ga(\cX)$.

Since $E^\Ga(1,\cX)=\id_{\smash{O^\Ga(\cX)}}$ and $E^\Ga(\ga,\cX)\od
E^\Ga(\de,\cX)=E^\Ga(\ga\de,\cX)$\I{2-category!2-morphism!vertical
composition} for $\ga,\de\in\Ga$ as in Definition \ref{ag9def4}, we
have
\begin{equation*}
R^\Ga(1,\cE)=\id_{\smash{\cE^\Ga}}\;\>\text{and}\;\>
R^\Ga(\ga,\cE)\ci R^\Ga(\de,\cE)=R^\Ga(\ga\de,\cE)\;\>\text{for all
$\ga,\de\in\Ga$.}
\end{equation*}
Hence $R^\Ga(-,\cE)$ is an action of $\Ga$ on $\cE^\Ga$ by
isomorphisms.

As $L^\Ga(\id_\Ga,\cX)=\id_{\smash{\cX^\Ga}}$ and $L^\Ga(\La,\cX)\ci
L^\Ga(\La,\cX)=L^\Ga(\La\La',\cX)$ for $\La,\La'\in\Aut(\La)$, by
properties of morphisms $I_{*,*}(*)$ we find that
\begin{align*}
S^\Ga(\id_\Ga,\cE)&=\de_{\smash{\cX^\Ga}}(\cE^\Ga):
\id_{\smash{\cX^\Ga}}^*(\cE^\Ga)\longra \cE^\Ga,\;\>\text{and}\\
S^\Ga(\La\La',\cE)&=S^\Ga(\La',\cE)\!\ci\!
L^\Ga(\La',\cX)^*(S^\Ga(\La,\cE))\!\ci\!
I_{L^\Ga(\La',\cX),L^\Ga(\La,\cX)}(\cE^\Ga).
\end{align*}
This means that the $S^\Ga(\La,\cE)$ define a lift of the action of
$\Aut(\Ga)$ on $\cX^\Ga$ to $\cE^\Ga$, that is, $\cE^\Ga$ is an
$\Aut(\Ga)$-equivariant sheaf on~$\cX^\Ga$.

If $\ga\in\Ga$ and $\La\in\Aut(\Ga)$ then noting that $O^\Ga(\cX)\ci
L^\Ga(\La,\cX)=O^\Ga(\cX)$, one can show from Definitions
\ref{ag9def3} and \ref{ag9def4} that
\begin{equation*}
E^\Ga(\La(\ga),\cX)*\id_{L^\Ga(\La,\cX)}=E^\Ga(\ga,\cX):
O^\Ga(\cX)\Longra O^\Ga(\cX).
\end{equation*}
Pulling back $\cE$ by this equation and using properties of the
$I_{*,*}(*)$ we find that
\e
R^\Ga(\ga,\cE)\ci S^\Ga(\La,\cE)=S^\Ga(\La,\cE)\ci
L^\Ga(\La,\cX)^*(R^\Ga(\La(\ga),\cE)).
\label{ag9eq15}
\e
This is a compatibility between the actions of $\Ga$ and $\Aut(\Ga)$
on $\cE^\Ga$. It says that the action of $\Aut(\Ga)$ on $\cX^\Ga$
lifts to an action of $\Aut(\Ga)\lt\Ga$ on~$\cE^\Ga$.

Let $\al:\cE_1\ra\cE_2$ be a morphism in $\qcoh(\cX)$. Then
$\al^\Ga:=O^\Ga(\cX)^*(\al):\cE_1^\Ga\ra\cE_2^\Ga$ is a morphism in
$\qcoh(\cX^\Ga)$. Since $E^\Ga(\ga,\cX)^*:O^\Ga(\cX)^*\Ra
O^\Ga(\cX)^*$ is a natural isomorphism of functors, we see that
\begin{equation*}
\al^\Ga\ci R^\Ga(\ga,\cE_1)=R^\Ga(\ga,\cE_2)\ci\al^\Ga\quad\text{for
$\ga\in\Ga$}.
\end{equation*}
Similarly we find that
\begin{equation*}
\al^\Ga\ci S^\Ga(\La,\cE_1)=S^\Ga(\La,\cE_2)\ci
L^\Ga(\La,\cX)^*(\al^\Ga)\quad\text{for $\La\in\Aut(\Ga)$.}
\end{equation*}
These imply that $R(\ga,-)$ and\/ $S(\La,-)$ are natural
isomorphisms of functors.

Now let $f:\cX\ra\cY$ be a representable 1-morphism of
$C^\iy$-stacks, so that as in \S\ref{ag92} we have
$f^\Ga:\cX^\Ga\ra\cY^\Ga$. Let $\cF\in\qcoh(\cY)$. Then we may form
$f^*(\cF)\in\qcoh(\cX)$ and hence
$f^*(\cF)^\Ga=O^\Ga(\cX)^*(f^*(\cF))\in\qcoh(\cX^\Ga)$, or we may
form $\cF^\Ga=O^\Ga(\cY)^*(\cF)\in\qcoh(\cY^\Ga)$ and hence
$(f^\Ga)^*(\cF^\Ga)\in\qcoh(\cX^\Ga)$. Since $O^\Ga(\cY)\ci
f^\Ga=f\ci O^\Ga(\cX)$, these are related by the canonical
isomorphism
\e
T^\Ga(f,\cF):=I_{f^\Ga,O^\Ga(\cY)}(\cF)\ci I_{O^\Ga(\cX),
f}(\cF)^{-1}: f^*(\cF)^\Ga\longra(f^\Ga)^*(\cF^\Ga).
\label{ag9eq16}
\e
Using properties of $I_{*,*}(*)$, it is easy to show that
\e
(f^\Ga)^*(R^\Ga(\ga,\cF))\ci T^\Ga(f,\cF)=T^\Ga(f,\cF)\ci
R^\Ga(\ga,f^*(\cF)) \quad\text{for $\ga\in\Ga$,}
\label{ag9eq17}
\e
and noting that $f^\Ga\ci L^\Ga(\La,\cX)=L^\Ga(\La,\cY)\ci f^\Ga$,
we also find that
\begin{align*}
T^\Ga(f,\cF)\ci S^\Ga(\La,f^*(\cF))=\,&(f^\Ga)^*(S^\Ga(\La,\cF))\ci
I_{f^\Ga,L^\Ga(\La,\cY)}(\cF^\Ga)\,\ci\\
&I_{L^\Ga(\La,\cX),f^\Ga}(\cF^\Ga)^{-1}\!\ci\!
L^\Ga(\La,\cX)^*(T^\Ga(f,\cF)).
\end{align*}
This shows that the isomorphisms $T^\Ga(f,\cF)$ identify the
$(\Aut(\Ga)\lt\Ga)$-actions on $f^*(\cF)^\Ga$
and~$(f^\Ga)^*(\cF^\Ga)$.

Now let $\cX,\Ga,\cX^\Ga,\cE$ and $\cE^\Ga$ be as above, and write
$R_0,\ldots,R_k$ for the irreducible representations of $\Ga$ over
$\R$ (that is, we choose one representative $R_i$ in each
isomorphism class of irreducible representations), with $R_0=\R$ the
trivial representation. Then since $R^\Ga(-,\cE)$ is an action of
$\Ga$ on $\cE^\Ga$ by isomorphisms, by elementary representation
theory we have a canonical decomposition
\e
\cE^\Ga\cong\ts\bigop_{i=0}^k\cE^\Ga_i\ot R_i\quad\text{for
$\cE_0^\Ga,\ldots,\cE^\Ga_k\in\qcoh(\cX^\Ga)$.}
\label{ag9eq18}
\e
We will be interested in splitting $\cE^\Ga$ into {\it trivial\/}
and {\it nontrivial\/} representations of $\Ga$, denoted by
subscripts `tr' and `nt'. So we write
\e
\cE^\Ga=\cE^\Ga_\tr\op\cE^\Ga_\nt,
\label{ag9eq19}
\e
where $\cE^\Ga_\tr,\cE^\Ga_\nt$ are the subsheaves of $\cE^\Ga$
corresponding to the factors $\cE^\Ga_0\ot R_0$ and
$\bigop_{i=1}^k\cE^\Ga_i\ot R_i$ respectively. Equivalently,
consider $\frac{1}{\md{\Ga}}\sum_{\ga\in\Ga}
R^\Ga(\ga,\cE):\cE^\Ga\ra\cE^\Ga$. It is a projection (its square is
itself), with image $\cE^\Ga_\tr$ and kernel~$\cE^\Ga_\nt$.

If $\Ga$ acts on $R_i$ by $\rho_i:\Ga\ra\Aut(R_i)$, and
$\La\in\Aut(\Ga)$, then $\rho_i\ci\La^{-1}:\Ga\ra\Aut(R_i)$ is also
an irreducible representation of $\Ga$, and so is isomorphic to
$R_{\La(i)}$ for some unique $\La(i)=0,\ldots,k$. This defines an
action of $\Aut(\Ga)$ on $\{0,\ldots,k\}$ by permutations. One can
show using \eq{ag9eq15} that $S^\Ga(\La,\cE)$ acts on the splitting
\eq{ag9eq18} by mapping $L^\Ga(\La,\cX)^*(\cE^\Ga_i\ot R_i)\ra
\cE^\Ga_{\La^{-1}(i)}\ot R_{\La^{-1}(i)}$. Since $\La(0)=0$, it
follows that $S^\Ga(\La,\cE)$ maps $L^\Ga(\La,\cX)^*(\cE^\Ga_\tr)\ra
\cE^\Ga_\tr$ and $L^\Ga(\La,\cX)^*(\cE^\Ga_\nt)\ra \cE^\Ga_\nt$,
that is, $S^\Ga(\La,\cE)$ preserves the splitting~\eq{ag9eq19}.

Equation \eq{ag9eq17} implies that $T^\Ga(f,\cF)$ canonically maps
$f^*(\cF)^\Ga_i\ot R_i\ra(f^\Ga)^*(\cF^\Ga_i\ot R_i)$ in
\eq{ag9eq18} for $f^*(\cF)^\Ga,\cF^\Ga$, and so maps
$f^*(\cF)^\Ga_\tr \ra(f^\Ga)^*(\cF^\Ga_\tr)$ and
$f^*(\cF)^\Ga_\nt\ra (f^\Ga)^*(\cF^\Ga_\nt)$ in~\eq{ag9eq19}.
\label{ag9def7}
\end{dfn}

The next two definitions explain to what extent this generalizes
to~$\tcX^\Ga,\hcX^\Ga$.

\begin{dfn} Let $\cX$ be a Deligne--Mumford $C^\iy$-stack, and $\Ga$ a
finite group, so that \S\ref{ag91} defines the orbifold strata
$\cX^\Ga,\tcX^\Ga$ with $\tcX^\Ga\simeq [\cX^\Ga/\Aut(\Ga)]$, and
1-morphisms $O^\Ga(\cX):\cX{}^\Ga\ra\cX$, $\ti
O^\Ga(\cX):\tcX^\Ga\ra\cX$ and $\ti\Pi{}^\Ga(\cX):\cX{}^\Ga\ra
\tcX^\Ga$ with~$\ti O^\Ga(\cX)\ci\ti\Pi{}^\Ga(\cX)=O^\Ga(\cX)$.

Let us ask: how much of the structure on $\cE{}^\Ga$ in Definition
\ref{ag9def7} descends to $\ti{\cal E}{}^\Ga$? It turns out that
$\ti{\cal E}{}^\Ga$ does not have natural representations of $\Ga$
or $\Aut(\Ga)$, since we do not have actions of $\Ga$ on $\ti
O^\Ga(\cX)$ by 2-isomorphisms or of $\Aut(\Ga)$ on $\tcX^\Ga$ by
1-isomorphisms. In effect, taking the quotient by $\Aut(\Ga)$ in
$\tcX^\Ga\simeq [\cX^\Ga/\Aut(\Ga)]$ destroys both these actions.

However, at least part of the natural decompositions
\eq{ag9eq18}--\eq{ag9eq19} descends to $\ti{\cal E}{}^\Ga$. As in
Definition \ref{ag9def7}, write $R_0,\ldots,R_k$ for the
irreducible representations of $\Ga$, so that $\Aut(\Ga)$ acts on
the indexing set $\{0,\ldots,k\}$. Form the quotient set
$\{0,\ldots,k\}/\Aut(\Ga)$, so that points of
$\{0,\ldots,k\}/\Aut(\Ga)$ are orbits $O$ of $\Aut(\Ga)$ in
$\{0,\ldots,k\}$. Then we may rewrite \eq{ag9eq18} as
\begin{equation*}
\cE^\Ga\cong\ts\bigop_{O\in\{0,\ldots,k\}/\Aut(\Ga)}
\bigl[\bigop_{i\in O} \cE^\Ga_i\ot R_i\bigr].
\end{equation*}

Since $S^\Ga(\La,\cE)$ maps $L^\Ga(\La,\cX)^*(\cE^\Ga_i\ot R_i)\ra
\cE^\Ga_{\La^{-1}(i)}\ot R_{\La^{-1}(i)}$, we see that
\begin{equation*}
\ts S^\Ga(\La,\cE):L^\Ga(\La,\cX)^*\bigl(\bigop_{i\in O} \cE^\Ga_i\ot
R_i\bigr)\longra \bigop_{i\in O} \cE^\Ga_i\ot R_i
\end{equation*}
for each $O\in\{0,\ldots,k\}/\Aut(\Ga)$. Now the $S^\Ga(\La,\cE)$
lift the action of $\Aut(\Ga)$ on $\cX^\Ga$ to $\cE^\Ga$, and
$\ti{\cal E}{}^\Ga$ is essentially the quotient of $\cE^\Ga$ by this
lifted action of $\Aut(\Ga)$ under the equivalence $\tcX^\Ga\simeq
[\cX^\Ga/\Aut(\Ga)]$. Therefore any decomposition of $\cE^\Ga$ which
is invariant under $S^\Ga(\La,\cE)$ for all $\La\in\Aut(\Ga)$
corresponds to a decomposition of $\ti{\cal E}{}^\Ga$. Hence there
is a canonical splitting
\e
\begin{split}
\ti{\cal E}{}^\Ga&=\ts\bigop_{O\in\{0,\ldots,k\}/\Aut(\Ga)}
\ti{\cal E}{}^\Ga_O,\quad\text{where}\\
I_{\ti\Pi{}^\Ga(\cX),\ti O^\Ga(\cX)}(\cE)^{-1}
\bigl[\ti\Pi{}^\Ga(\cX)^*(\ti{\cal E}{}^\Ga_O)\bigr]&\cong
\ts\bigop_{i\in O} \cE^\Ga_i\ot R_i\quad\text{under \eq{ag9eq18}.}
\end{split}
\label{ag9eq20}
\e
As for \eq{ag9eq19} we define the {\it trivial\/} and {\it
nontrivial\/} parts of $\ti{\cal E}{}^\Ga$ by $\ti{\cal
E}{}^\Ga_\tr= \ti{\cal E}{}^\Ga_{\{0\}}$ and $\ti{\cal E}{}^\Ga_\nt=
\bigop_{O\in \{1,\ldots,k\}/\Aut(\Ga)}\ti{\cal E}{}^\Ga_O$. Then
\e
\begin{split}
\ti{\cal E}{}^\Ga=\ti{\cal E}{}^\Ga_\tr\op\ti{\cal
E}{}^\Ga_\nt,\;\>\text{where} \;\> I_{\ti\Pi{}^\Ga(\cX),\ti
O^\Ga(\cX)}(\cE)^{-1}\bigl[
\ti\Pi{}^\Ga(\cX)^*(\ti{\cal E}{}^\Ga_\tr)\bigr]&=\cE{}^\Ga_\tr\\
\text{and}\;\> I_{\ti\Pi{}^\Ga(\cX),\ti O^\Ga(\cX)}(\cE)^{-1}
\bigl[\ti\Pi{}^\Ga(\cX)^*(\ti{\cal
E}{}^\Ga_\nt)\bigr]&=\cE{}^\Ga_\nt.
\end{split}
\label{ag9eq21}
\e

Each point $[x,\De]$ of $\tcX^\Ga_\top$ has isotropy
group\I{C-stack@$C^\iy$-stack!isotropy group $\Iso_\cX([x])$}
$\Iso_{\smash{\tcX^\Ga}}([x,\De])$ with a distinguished subgroup
$\De$ with a noncanonical isomorphism $\De\cong\Ga$. The fibre of
$\ti{\cal E}{}^\Ga$ at $[x,\De]$ is a representation of
$\Iso_{\smash{\tcX^\Ga}}([x,\De])$, and hence a representation of
$\De$. Equation \eq{ag9eq21} corresponds to splitting the fibre of
$\ti{\cal E}{}^\Ga$ at $[x,\De]$ into trivial and nontrivial
representations of $\De$. Equation \eq{ag9eq20} corresponds to
decomposing the fibre of $\ti{\cal E}{}^\Ga$ at $[x,\De]$ into
families of irreducible representations of $\De\cong\Ga$ that are
independent of the choice of isomorphism~$\De\cong\Ga$.

Now let $f:\cX\ra\cY$ be a representable 1-morphism of
$C^\iy$-stacks, so that as in \S\ref{ag92} we have a representable
1-morphism $\ti f{}^\Ga:\tcX^\Ga\ra\ti\cY{}^\Ga$ with $f\ci \ti
O{}^\Ga(\cX)=\ti O{}^\Ga(\cY) \ci\ti f{}^\Ga$. Let
$\cF\in\qcoh(\cY)$, so that $\ti\cF{}^\Ga\in\qcoh(\ti\cY{}^\Ga)$,
$f^*(\cF)\in\qcoh(\cX)$, and $\,\,\,\widetilde{\!\!\!f^*(\cF)\!\!\!}
\,\,\,{}^\Ga\in\qcoh(\tcX^\Ga)$. As for \eq{ag9eq16}, we have a
canonical isomorphism
\begin{equation*}
\ti T{}^\Ga(f,\cF):=I_{\ti f{}^\Ga,\ti O^\Ga(\cY)}(\cF)\ci I_{\ti
O{}^\Ga(\cX), f}(\cF)^{-1}:\,\,\,
\widetilde{\!\!\!f^*(\cF)\!\!\!}\,\,\,{}^\Ga\longra(\ti
f{}^\Ga)^*(\ti\cF{}^\Ga).
\end{equation*}
As for $T^\Ga(f,\cF)$ in Definition \ref{ag9def7}, $\ti
T{}^\Ga(f,\cF)$ maps ${}\,\,\,\widetilde{\!\!\!
f^*(\cF)\!\!\!}\,\,\,{}^\Ga_O\ra(\ti f{}^\Ga)^*(\cF^\Ga_O)$ in
\eq{ag9eq20} for ${}\,\,\,\widetilde{\!\!\!f^*(\cF)\!\!\!}
\,\,\,{}^\Ga,\ti\cF{}^\Ga$, and so maps $\,\,\,\widetilde{
\!\!\!f^*(\cF)\!\!\!}\,\,\,{}^\Ga_\tr \ra(\ti f{}^\Ga)^*
(\ti\cF{}^\Ga_\tr)$ and ${}\,\,\,\widetilde{\!\!\!f^*(\cF)\!\!\!}
\,\,\,{}^\Ga_\nt\ra(\ti f{}^\Ga)^*(\ti\cF{}^\Ga_\nt)$
in~\eq{ag9eq21}.
\label{ag9def8}
\end{dfn}

\begin{dfn} Let $\cX$ be a Deligne--Mumford $C^\iy$-stack, and
$\Ga$ a finite group, so that \S\ref{ag91} defines the orbifold
strata $\tcX^\Ga,\hcX^\Ga$ and 1-morphisms $\ti
O^\Ga(\cX):\tcX^\Ga\ra\cX$ and $\hat\Pi{}^\Ga:\tcX^\Ga\ra \hcX^\Ga$,
where $\hat\Pi{}^\Ga$ is non-representable, with
fibre~$[\bar{\ul{*}}/\Ga]$.

Suppose $\cE$ is a quasicoherent sheaf on $\cX$. Since we have no
1-morphism $\hcX^\Ga\ra\cX$, we cannot pull $\cE$ back to $\hcX^\Ga$
to define $\hat\cE{}^\Ga$ in $\qcoh(\hcX^\Ga)$. But we do have
$\ti{\cal E}{}^\Ga=\ti O^\Ga(\cX)^*(\cE)$ in $\qcoh(\tcX^\Ga)$, with
splitting $\ti{\cal E}{}^\Ga=\ti{\cal E}{}^\Ga_\tr \op\ti{\cal
E}{}^\Ga_\nt$ as in \eq{ag9eq21}, so we can form the pushforward
$\hat\Pi{}^\Ga_*(\ti{\cal E}{}^\Ga)$ in $\qcoh(\hcX^\Ga)$. Now
pushforwards take global sections of a sheaf on the fibres of the
1-morphism. The fibres of $\hat\Pi{}^\Ga$ are $[\bar{\ul{*}}/\Ga]$.
Quasicoherent sheaves on $[\bar{\ul{*}}/\Ga]$ correspond to
$\Ga$-representations, and the global sections correspond to the
trivial ($\Ga$-invariant) part.

As the $\Ga$-invariant part of $\ti{\cal E}{}^\Ga$ is $\ti{\cal
E}{}^\Ga_\tr$, we see that $\hat\Pi{}^\Ga_*(\ti{\cal
E}{}^\Ga_\nt)=0$, that is, $\cE{}^\Ga_\nt$ {\it and\/ $\ti{\cal
E}{}^\Ga_\nt$ do not descend to\/} $\hcX^\Ga$. Define
$\hat\cE{}^\Ga_\tr=\hat\Pi{}^\Ga_* (\ti{\cal E}{}^\Ga_\tr)$ in
$\qcoh(\hcX^\Ga)$. This is the natural analogue of
$\cE{}^\Ga_\tr,\ti{\cal E}{}^\Ga_\tr$ on $\hcX^\Ga$, and has a
canonical isomorphism
\e
(\hat\Pi{}^\Ga)^*(\hat\cE{}^\Ga_\tr)\cong\ti{\cal E}{}^\Ga_\tr.
\label{ag9eq22}
\e

Now let $f:\cX\ra\cY$ be a representable 1-morphism of
$C^\iy$-stacks, so that as in \S\ref{ag92} we have a representable
1-morphism $\ti f{}^\Ga:\tcX^\Ga\ra\ti\cY{}^\Ga$. Then there is a
canonical isomorphism
\begin{equation*}
\hat T{}^\Ga_\tr(f,\cF):
\,\,\,\widehat{\!\!\!f^*(\cF)\!\!\!}\,\,\,{}^\Ga_\tr\longra(\hat
f{}^\Ga)^*(\hat\cF{}^\Ga_\tr),
\end{equation*}
the composition of the natural isomorphism $\hat\Pi{}^\Ga_*\ci (\ti
f{}^\Ga)^*(\ti\cF{}^\Ga_\tr)\ra (\hat f{}^\Ga)^*\ci
\hat\Pi{}^\Ga_*(\ti\cF{}^\Ga_\tr)$ with~$\hat\Pi{}^\Ga_*(\ti
T{}^\Ga(f,\cF)\vert_{\smash{\,\,\,\widetilde{\!\!\!f^*(\cF)\!\!\!}
\,\,\,{}^\Ga_\tr}}\bigr)$.
\label{ag9def9}
\end{dfn}

\subsection{\texorpdfstring{Sheaves on orbifold strata of quotients
{$[\protect\uX/G]$}}{Sheaves on orbifold strata of quotients [X/G]}}
\label{ag95}
\I{quotient C-stack@quotient $C^\iy$-stack!orbifold strata!sheaves
on|(}\I{Deligne--Mumford $C^\iy$-stack!orbifold
strata!of quotient $C^\iy$-stacks|(}%
\I{C-stack@$C^\iy$-stack!quotients $[\protect\uX/G]$!orbifold
strata|(}

In the next theorem we take $\cX=[\uX/G]$, and use the explicit
description of $\cX^\Ga$ in Theorem \ref{ag9thm2} to give an
alternative formula for the action $R^\Ga(-,\cE)$ of $\Ga$ on
$\cE^\Ga$ in Definition \ref{ag9def7}. This then allows us to
understand the splittings \eq{ag9eq18}--\eq{ag9eq22} in terms of
sheaves on $\uX$. The proof is a long but straightforward
consequence of the definitions, and we leave it as an exercise.

\begin{thm} Let\/ $\uX$ be a
Hausdorff\/\I{C-scheme@$C^\iy$-scheme!Hausdorff} $C^\iy$-scheme, $G$
a finite group, $\ur:G\ra\Aut(\uX)$ an action of\/ $G$ on\/ $\uX,$
and\/ $\cX=[\uX/G]$ the quotient Deligne--Mumford\/ $C^\iy$-stack.
Then \eq{ag9eq5} gives an equivalence
$\cX^\Ga\simeq[\coprod_{\text{injective $\rho:\Ga\ra
G$}}\uX^{\rho(\Ga)}/G]$.

Write $\qcoh^G(\uX)$ for the abelian category\I{abelian
category} of\/ $G$-equivariant quasicoherent sheaves on $\uX,$ with
objects pairs $(\cE,\Phi)$ for $\cE\in\qcoh(\uX)$ and\/
$\Phi(g):\ur(g)^*(\cE)\ra\cE$ is an isomorphism in $\qcoh(\uX)$ for
all\/ $g\in G$ satisfying $\Phi(1)=\de_\uX(\cE)$ and
\begin{equation*}
\Phi(g h)=\Phi(h)\ci\ur(h)^*(\Phi(g))\ci I_{\ur(h),\ur(g)}(\cE)
\quad\text{for all\/ $g,h\in G,$}
\end{equation*}
and morphisms $\al:(\cE,\Phi)\ra(\cF,\Psi)$ in $\qcoh^G(\uX)$ are
morphisms $\al:\cE\ra\cF$ in $\qcoh(\uX)$ with\/
$\al\ci\Phi(g)=\Psi(g)\ci\ur(g)^*(\al)$ for all\/~$g\in G$.

Then $\qcoh^G(\uX)$ is isomorphic to $\qcoh(G\t\uX\rra\uX)$ in
Definition\/ {\rm\ref{ag8def2},} so Theorem\/ {\rm\ref{ag8thm1}}
gives an equivalence of categories $F_\Pi:\qcoh(\cX)\ra
\qcoh^G(\uX)$. Using \eq{ag9eq5} we also get an equivalence
$F_\Pi^\Ga:\qcoh(\cX^\Ga)\ra\qcoh^G(\coprod_\rho\uX^{\rho(\Ga)})$.
These categories and functors fit into a $2$-commutative diagram:
\e
\begin{gathered}
\xymatrix@C=130pt@R=11pt{ *+[r]{\qcoh(\cX)} \ar[r]_(0.35){F_\Pi}
\ar[d]^{O^\Ga(\cX)^*}
\drtwocell_{}\omit^{}\omit{^{{}\,\,\,\,\,\,\,\,\,\,\,\,\,\,\,
\,\,\,\,\,\,\,\,\,\,\,\,\,\,\,\,\,\,\,\, N^\Ga(\cX)}} &
*+[l]{\qcoh^G(\uX)} \ar[d]_{\ui^*_\uX}  \\
*+[r]{\qcoh(\cX^\Ga)} \ar[r]^(0.35){F_\Pi^\Ga} &
*+[l]{\qcoh^G(\coprod_\rho\uX^{\rho(\Ga)}),\!} }
\end{gathered}
\label{ag9eq23}
\e
where $\ui_\uX:\coprod_\rho\uX^{\rho(\Ga)}\ra\uX$ is the union over
$\rho$ of the inclusion morphisms $\uX^{\rho(\Ga)}\ra\uX,$ which is
$G$-equivariant and so induces a pullback functor $\ui^*_\uX$ as
shown, and\/ $N^\Ga(\cX)$ is a natural isomorphism of functors.

Let\/ $(E,\Phi)\in\qcoh^G(\uX),$ so that\/ $\ui^*_\uX(E,\Phi)\in
\qcoh^G(\coprod_\rho\uX^{\rho(\Ga)})$. Define $\bar
R{}^\Ga\bigl(\ga,(E,\Phi)\bigr):\ui^*_\uX(E,\Phi)\ra
\ui^*_\uX(E,\Phi)$ in $\qcoh^G(\coprod_\rho\uX^{\rho(\Ga)})$ for\/
$\ga\in\Ga$ such that
\begin{align*}
&\bar R{}^\Ga\bigl(\ga,(E,\Phi)\bigr)\vert_{\uX^{\rho(\Ga)}}:
\ui_\uX\vert_{\uX^{\rho(\Ga)}}^*(E)\longra
\ui_\uX\vert_{\uX^{\rho(\Ga)}}^*(E)\;\>\text{is given by}\\
&\bar R{}^\Ga\bigl(\ga,(E,\Phi)\bigr)\vert_{\uX^{\rho(\Ga)}}=
\ui_\uX\vert_{\uX^{\rho(\Ga)}}^*(\Phi(\rho(\ga^{-1})))\ci
I_{\ui_\uX\vert_{\uX^{\rho(\Ga)}},\ur(\rho(\ga^{-1}))}(\cE)
\end{align*}
for each $\rho,$ noting that\/
$\ur(\rho(\ga^{-1}))\ci\ui_\uX\vert_{\uX^{\rho(\Ga)}}=
\ui_\uX\vert_{\uX^{\rho(\Ga)}}$. Then $\bar
R{}^\Ga\bigl(-,(E,\Phi)\bigr)$ is an action of\/ $\Ga$ on
$\ui_\uX\vert_{\uX^{\rho(\Ga)}}^*(E)$ by isomorphisms. Furthermore,
for each\/ $\cE$ in $\qcoh(\cX)$ and\/ $\ga$ in $\Ga,$ the following
diagram in $\qcoh^G(\coprod_\rho\uX^{\rho(\Ga)})$ commutes:
\begin{equation*}
\xymatrix@C=140pt@R=15pt{ *+[r]{F_\Pi^\Ga(\cE^\Ga)}
\ar[r]_{F_\Pi^\Ga(R^\Ga(\ga,\cE))} \ar[d]^{N^\Ga(\cX)(\cE)} &
*+[l]{F_\Pi^\Ga(\cE^\Ga)} \ar[d]_{N^\Ga(\cX)(\cE)} \\
*+[r]{\ui^*_\uX\ci F_\Pi(\cE)} \ar[r]^{\bar R{}^\Ga(\ga,F_\Pi(\cE))} &
*+[l]{\ui^*_\uX\ci F_\Pi(\cE).\!} }
\end{equation*}
That is, the equivalences of categories $F_\Pi,F_\Pi^\Ga$ in
\eq{ag9eq23} identify the $\Ga$-actions $R^\Ga(-,-)$ on
$O^\Ga(\cX)^*$ and\/ $\bar R{}^\Ga(-,-)$ on $\ui^*_\uX$ by natural
isomorphisms.\I{quotient C-stack@quotient $C^\iy$-stack!orbifold
strata!sheaves on|)}\I{Deligne--Mumford $C^\iy$-stack!orbifold
strata!of quotient $C^\iy$-stacks|)}%
\I{C-stack@$C^\iy$-stack!quotients $[\protect\uX/G]$!orbifold
strata|)}
\label{ag9thm3}
\end{thm}

\subsection{Cotangent sheaves of orbifold strata}
\label{ag96}
\I{Deligne--Mumford $C^\iy$-stack!orbifold strata!cotangent
sheaves|(}

Finally we apply these ideas to write the cotangent sheaves of
$\cX^\Ga,\tcX^\Ga,\hcX^\Ga$ in terms of the pullbacks of $T^*\cX$.
The theorem illustrates the principle that when passing to orbifold
strata, it is often natural to restrict to the trivial parts
$\cE{}^\Ga_\tr,\ti{\cal E}{}^\Ga_\tr,\hat\cE{}^\Ga_\tr$ of the
pullbacks of $\cE$. The nontrivial parts $(T^*\cX)^\Ga_\nt,
\,\,\,\,\, \widetilde{\!\!\!\!\!(T^*\cX)\!\!\!\!\!}
\,\,\,\,\,^\Ga_\nt$ should be interpreted as the {\it conormal
sheaves\/} of $\cX^\Ga,\tcX^\Ga$ in~$\cX$.

\begin{thm} Let\/ $\cX$ be a Deligne--Mumford\/
$C^\iy$-stack and\/ $\Ga$ a finite group, so that\/
{\rm\S\ref{ag91}} defines $O^\Ga(\cX):\cX^\Ga\ra\cX$. As in
Definition\/ {\rm\ref{ag8def5}} we have cotangent sheaves
$T^*\cX,T^*(\cX^\Ga)$ and a morphism
$\Om_{\smash{O^\Ga(\cX)}}:O^\Ga(\cX)^*(T^*\cX)\ab \ra T^*(\cX^\Ga)$
in $\qcoh(\cX^\Ga)$. But\/ $O^\Ga(\cX)^*(T^*\cX)=(T^*\cX)^\Ga,$ so
by \eq{ag9eq19} we have a splitting
$(T^*\cX)^\Ga=(T^*\cX)^\Ga_\tr\op(T^*\cX)^\Ga_\nt$. Then
$\Om_{\smash{O^\Ga(\cX)}}\vert_{\smash{(T^*\cX)^\Ga_\tr}}:
(T^*\cX)^\Ga_\tr\ra T^*(\cX^\Ga)$ is an isomorphism,
and\/~$\Om_{O^\Ga(\cX)}\vert_{(T^*\cX)^\Ga_\nt}=0$.

Similarly, using the $1$-morphism $\ti\O{}^\Ga(\cX):\tcX^\Ga\ra\cX$
and the splitting \eq{ag9eq21} for
${}\,\,\,\,\,\widetilde{\!\!\!\!\!(T^*\cX)\!\!\!\!\!}
\,\,\,\,\,^\Ga$ we find that\/ $\Om_{\smash{\ti\O{}^\Ga(\cX)}}
\vert_{\smash{\,\,\,\,\, \widetilde{\!\!\!\!\!(T^*\cX)\!\!\!\!\!}
\,\,\,\,\, ^\Ga_\tr}}: \,\,\,\,\,\widetilde{\!\!\!\!\!(T^*\cX)
\!\!\!\!\!}\,\,\,\,\,^\Ga_\tr\ra T^*(\tcX^\Ga)$ is an isomorphism,
and\/ $\Om_{\smash{\ti\O{}^\Ga(\cX)}} \vert_{\smash{\,\,\,\,\,
\widetilde{\!\!\!\!\!(T^*\cX)\!\!\!\!\!} \,\,\,\,\,^\Ga_\nt}}=0$.

Also, there is a natural isomorphism $\smash{{}\,\,\,\,\,\widehat{
\!\!\!\!\!(T^*\cX) \!\!\!\!\!}\,\,\,\,\,^\Ga_\tr\cong
T^*(\hcX^\Ga)}$ in\/~$\qcoh(\hcX^\Ga)$.
\label{ag9thm4}
\end{thm}

\begin{proof} All of the claims are local statements on
$\cX^\Ga,\tcX^\Ga,\hcX^\Ga$, that is, it is enough to prove them on
open covers of $\cX^\Ga,\tcX^\Ga,\hcX^\Ga$. As $\cX$ is Deligne--Mumford it is covered by open $C^\iy$-substacks $\cU$ equivalent to $[\uU/G]$ for $\uU$ an affine $C^\iy$-scheme and $G$ a finite group. Then $\cX^\Ga,\tcX^\Ga,\hcX^\Ga$ are covered by
the corresponding $\cU^\Ga,\ti\cU{}^\Ga,\hat\cU{}^\Ga$. Thus it is
sufficient to prove the theorem when $\cX\simeq[\uX/G]$ for $\uX$ an affine $C^\iy$-scheme and $G$ a finite group acting on $\uX$. As the theorem is independent of $\cX$ up to equivalence, we may
take~$\cX=[\uX/G]$.

Thus we can apply Theorems \ref{ag9thm2} and \ref{ag9thm3} to
translate each part of the theorem into statements about
$\uX,\uX^{\rho(\Ga)},\ldots.$ For the first part, using the notation
of Theorem \ref{ag9thm3}, we find that $F_\Pi(T^*\cX)=
(T^*\uX,\Phi)$, where $\Phi(g)=\Om_{\ur(g)}$ for $g\in G$. Similarly
$F_\Pi^\Ga(T^*\cX^\Ga)=\bigl(T^*(
\coprod_\rho\uX^{\rho(\Ga)}),\Phi^\Ga\bigr)$, where
$\Phi^\Ga(g)=\coprod_\rho\Om_{\ur(g)\vert_{\uX^{\rho(\Ga)}}}$, and
$\ur(g)\vert_{\uX^{\rho(\Ga)}}$ maps~$\uX^{\rho(\Ga)}\ra
\uX^{\rho^g(\Ga)}$.

Fix an injective morphism $\rho:\Ga\ra G$, and write
$\ui^\rho_\uX:\uX^{\rho(\Ga)}\ra\uX$ for the inclusion of
$\uX^{\rho(\Ga)}$ as a $C^\iy$-subscheme. Then $(\ui^\rho_\uX)^*
(T^*\uX)=\ui^*_\uX(T^*\uX)\vert_{\smash{\uX^{\rho(\Ga)}}}$ in
$\qcoh(\uX^{\rho(\Ga)})$, and $\Om_{\ui^\rho_\uX}=\Om_{\ui_\uX}
\vert_{\smash{\uX^{\rho(\Ga)}}}$. Theorem \ref{ag9thm3} and
$\Phi(g)=\Om_{\ur(g)}$ show that the $\Ga$-action $\bar
R{}^\Ga\bigl(\ga,(T^*\uX,\Phi)\bigr)$ on $\bigl(\ui^*_\uX(T^*\uX),
\ui^*_\uX(\Phi)\bigr)$ acts on $(\ui^\rho_\uX)^*(T^*\uX)$ by
\begin{equation*}
\bar R{}^\Ga\bigl(\ga,(T^*\uX,\Phi)\bigr)
\vert_{(\ui^\rho_\uX)^*(T^*\uX)}=
(\ui^\rho_\uX)^*(\Om_{\ur(\rho(\ga^{-1}))})\ci
I_{\ui^\rho_\uX,\ur(\rho(\ga^{-1}))}.
\end{equation*}

Let $(\ui^\rho_\uX)^*(T^*\uX)=(\ui^\rho_\uX)^*(T^*\uX)_\tr\op
(\ui^\rho_\uX)^*(T^*\uX)_\nt$ be the decomposition of
$(\ui^\rho_\uX)^*(T^*\uX)$ into trivial and nontrivial
$\Ga$-representations under the action of $\bar
R{}^\Ga(-,(T^*\uX,\Phi))$. Since Theorem \ref{ag9thm3} shows that
the $\Ga$-actions $R{}^\Ga(-,T^*\cX)$ and $\bar R{}^\Ga(-,
(T^*\uX,\Phi))$ are intertwined by $F_\Pi^\Ga$, the splitting into
trivial and nontrivial parts corresponds. As $F_\Pi^\Ga$ is an
equivalence of categories by Theorem \ref{ag8thm1}, the first part
of the theorem is thus equivalent to showing that
\begin{equation*}
\Om_{\ui^\rho_\uX}:(\ui^\rho_\uX)^*(T^*\uX)=
(\ui^\rho_\uX)^*(T^*\uX)_\tr\op(\ui^\rho_\uX)^*(T^*\uX)_\nt
\longra T^*\uX^{\rho(\Ga)}
\end{equation*}
is an isomorphism $(\ui^\rho_\uX)^*(T^*\uX)_\tr\ra
T^*\uX^{\rho(\Ga)}$, and is zero on $(\ui^\rho_\uX)^*(T^*\uX)_\nt$.

To see this, let $x\in\uX^{\rho(\Ga)}\subseteq\uX$, and write $\fC_x$ for the local $C^\iy$-ring $\O_{X,x}$, the stalk of $\O_X$ at $x$. Then the action $\rho$ of $\Ga$ on $\uX$ fixing $x$ induces an action $\phi:\Ga\ra\Aut(\fC_x)$ of $\Ga$ on $\fC_x$. Since each $\phi(\ga):\fC_x\ra\fC_x$ acts on $\fC_x$ as a $C^\iy$-ring isomorphism, it is an $\R$-linear map, so we may split $\fC_x=\fC_{x,\tr}\op\fC_{x,\nt}$ into trivial and nontrivial $\Ga$-representations. Write $(\fC_{x,\nt})$ for the ideal in $\fC_x$ generated by $\fC_{x,\nt}$, and $\fD_x=\fC_x/(\fC_{x,\nt})$ for the quotient
$C^\iy$-ring, with projection $\pi_x:\fC_x\ra\fD_x$. Then
$\O_{X^{\rho(\Ga)},x}\cong\fD_x$ and~$\ui^\rho_{\uX,x}\cong\pi_x:\fC_x\ra\fD_x$.

We have cotangent modules $\Om_{\fC_x},\Om_{\fD_x}$ with morphisms
$\Om_{\pi_x}:\Om_{\fC_x}\ra\Om_{\fD_x}$ and $(\Om_{\pi_x})_*:\Om_{\fC_x}\ot_{\fC_x}\fD_x\ra\Om_{\fD_x}$. In stalks at $x\in\uX^{\rho(\Ga)}\subseteq\uX$ we have
$[T^*\uX]_x\cong\Om_{\fC_x}$, $[T^*\uX^{\rho(\Ga)}]_x\cong\Om_{\fD_x}$, $[(\ui^\rho_\uX)^*(T^*\uX)]_x\cong\Om_{\fC_x}\ot_{\fC_x}\fD_x$ and $[\Om_{\ui^\rho_\uX}]_x\cong(\Om_{\pi_x})_*:\Om_{\fC_x}\ot_{\fC_x}\fD_x\ra\Om_{\fD_x}$. The $\Ga$-action on $\fC_x$ induces one on $\Om_{\fC_x}$, and hence one on $\Om_{\fC_x}\ot_{\fC_x}\fD_x$. Thus
we split into trivial and nontrivial $\Ga$-representations,
$\Om_{\fC_x}\ot_{\fC_x}\fD_x=(\Om_{\fC_x}\ot_{\fC_x}\fD_x)_\tr\op(\Om_{\fC_x}
\ot_{\fC_x}\fD_x)_\nt$. This $\Ga$-action is identified with that on the stalk
$[(\ui^\rho_\uX)^*(T^*\uX)]_x$. Hence $[(\ui^\rho_\uX)^*(T^*\uX)_\tr]_x\cong(\Om_{\fC_x}\ot_{\fC_x}\fD_x)_\tr$ and~$[(\ui^\rho_\uX)^*(T^*\uX)_\nt]_x\cong(\Om_{\fC_x}\ot_{\fC_x}\fD_x)_\nt$.

We have a linear map $\d_{\fC_x}:\fC_x\ra\Om_{\fC_x}$, whose image generates
$\Om_{\fC_x}$ as a $\fC_x$-module. It induces a linear map $\d_{\fC_x}\ot
\pi_x:\fC_x\ra \Om_{\fC_x}\ot_{\fC_x}\fD_x$, whose image generates
$\Om_{\fC_x}\ot_{\fC_x}\fD_x$ as a $\fD_x$-module. As $\d_{\fC_x}\ot\pi_x$ is
$\Ga$-equivariant, it maps $\fC_{x,\tr}$ and $\fC_{x,\nt}$ to
$(\Om_{\fC_x}\ot_{\fC_x}\fD_x)_\tr$ and $(\Om_{\fC_x}\ot_{\fC_x}\fD_x)_\nt$,
respectively. Hence $(\Om_{\fC_x}\ot_{\fC_x}\fD_x)_\tr$ and
$(\Om_{\fC_x}\ot_{\fC_x}\fD_x)_\nt$ are generated as $\fD_x$-modules by
$(\d_{\fC_x}\ot\pi_x)(\fC_{x,\tr})$ and~$(\d_{\fC_x}\ot\pi_x)(\fC_{x,\nt})$.

Since $\fD_x=\fC_x/(\fC_{x,\nt})$, we see that $(\Om_{\pi_x})_*:\Om_{\fC_x}\ot_{\fC_x}\fD_x\ra\Om_{\fD_x}$ is surjective, with kernel generated by $(\d_{\fC_x}\ot\pi_x)\bigl((\fC_{x,\nt})\bigr)$. It is enough to use not the whole ideal $(\fC_{x,\nt})$, but only the generating subspace $\fC_{x,\nt}$. The $\fD_x$-submodule generated by $(\d_{\fC_x}\ot\pi_x)(\fC_{x,\nt})$ is $(\Om_{\fC_x}\ot_{\fC_x}\fD_x)_\nt$. Thus, $(\Om_{\pi_x})_*$ is surjective with kernel $(\Om_{\fC_x}\ot_{\fC_x}\fD_x)_\nt$, so
$(\Om_{\pi_x})_*\vert_{(\Om_{\fC_x}\ot_{\fC_x}\fD_x)_\tr}:(\Om_{\fC_x}\ot_{\fC_x}\fD_x)_\tr\ra\Om_{\fD_x}$ is an isomorphism, and $(\Om_{\pi_x})_*\vert_{(\Om_{\fC_x}\ot_{\fC_x}\fD_x)_\nt}=0$. Therefore $[\Om_{\ui^\rho_\uX}\vert_{\smash{(\ui^\rho_\uX)^*(T^*\uX)_\tr}}]_x:[(\ui^\rho_\uX)^*(T^*\uX)_\tr]_x\ra[T^*\uX^{\rho(\Ga)}]_x$ is an isomorphism, and $[\Om_{\ui^\rho_\uX}\vert_{\smash{(\ui^\rho_\uX)^*(T^*\uX)_\nt}}]_x=0$. As this holds for all $x\in\uX^{\rho(\Ga)}\subseteq\uX$, the first part follows.

For the second part, Theorem \ref{ag8thm3}(a) and $\ti
O^\Ga(\cX)\ci\ti\Pi{}^\Ga(\cX)=O^\Ga(\cX)$ give a commutative
diagram in $\qcoh(\cX^\Ga)$:
\e
\begin{gathered}
\xymatrix@C=150pt@R=16pt{
*+[r]{\begin{subarray}{l}\ts \ti\Pi{}^\Ga(\cX)^*(\ti
O{}^\Ga(\cX)^*(T^*\cX))=\\
\ts\ti\Pi{}^\Ga(\cX)^*(\,\,\,\,\,
\widetilde{\!\!\!\!\!(T^*\cX)\!\!\!\!\!} \,\,\,\,\, ^\Ga_\tr)\op
\ti\Pi{}^\Ga(\cX)^*(\,\,\,\,\,
\widetilde{\!\!\!\!\!(T^*\cX)\!\!\!\!\!} \,\,\,\,\,
^\Ga_\nt)\end{subarray}} \ar[r]^(0.6){\ti\Pi{}^\Ga(\cX)^*(\Om_{\ti
O{}^\Ga(\cX)})}  &
*+[l]{\ti\Pi{}^\Ga(\cX)^*(T^*(\tcX^\Ga))}
\ar[d]_{\Om_{\ti\Pi{}^\Ga(\cX)}}
\\
*+[r]{\begin{subarray}{l}\ts O^\Ga(\cX)^*(T^*\cX)=\\
\ts(T^*\cX)^\Ga_\tr\op(T^*\cX)^\Ga_\nt\end{subarray}}
\ar[u]_{I_{\ti\Pi{}^\Ga(\cX),\ti O^\Ga(\cX)}(\cE)}
\ar[r]^(0.65){\Om_{O^\Ga(\cX)}} & *+[l]{T^*(\cX^\Ga).} }
\end{gathered}
\label{ag9eq24}
\e
As $\ti\Pi^\Ga(\cX)$ is the projection
$\cX^\Ga\ra[\cX^\Ga/\Aut(\Ga)]$, it is \'etale, so
$\Om_{\smash{\ti\Pi^\Ga(\cX)}}$ is an isomorphism. Also
$I_{\ti\Pi{}^\Ga(\cX),\ti O^\Ga(\cX)}(\cE)$ identifies `tr',`nt'
with `tr',`nt' components. Thus \eq{ag9eq24} and the first part
show $\ti\Pi{}^\Ga(\cX)^*(\Om_{\smash{ \ti\O{}^\Ga(\cX)}}
\vert_{\smash{\,\,\,\,\, \widetilde{\!\!\!\!\!(T^*\cX)\!\!\!\!\!}
\,\,\,\,\,^\Ga_\tr}})\!:\!\ti\Pi{}^\Ga(\cX)^*(\,\,\,\,\,\widetilde{
\!\!\!\!\!(T^*\cX)\!\!\!\!\!}\,\,\,\,\,^\Ga_\tr)\ab\ra
\ti\Pi{}^\Ga(\cX)^*(T^*(\tcX^\Ga))$ is an isomorphism, and\/
$\ti\Pi{}^\Ga(\cX)^*(\Om_{\smash{\ti\O{}^\Ga(\cX)}}
\vert_{\smash{\,\,\,\,\, \widetilde{\!\!\!\!\!(T^*\cX)\!\!\!\!\!}
\,\,\,\,\,^\Ga_\nt}})=0$. As $\ti\Pi{}^\Ga(\cX)$ is \'etale and
surjective, the second part of the theorem follows. The third part
is proved by a similar argument
involving~$\hat\Pi{}^\Ga$.\I{Deligne--Mumford $C^\iy$-stack!quasicoherent sheaves on|)}\I{Deligne--Mumford $C^\iy$-stack!orbifold
strata|)}\I{Deligne--Mumford $C^\iy$-stack!orbifold strata!sheaves
on|)}\I{Deligne--Mumford $C^\iy$-stack!orbifold strata!cotangent
sheaves|)}
\end{proof}

\appendix

\section{Background material on stacks}
\label{agA}

Finally we recall some background material on stacks needed in \S\ref{ag6}--\S\ref{ag9}. Readers unfamiliar with stacks are advised to look at an introductory text such as Olsson \cite{Olss2}, Vistoli \cite{Vist}, Gomez \cite{Gome}, or Laumon and Moret-Bailly \cite{LaMo} before reading this section.

Stacks of any kind form a strict 2-category $\bs{\cal C}$, with objects
$\cX,\cY$, 1-morphisms $f,g:\cX\ra\cY$, and 2-morphisms $\eta:f\Ra
g$. So we begin in \S\ref{agA1} with an introduction to 2-categories. Sections \ref{agA2}--\ref{agA5} cover Grothendieck (pre)topologies, sites, prestacks and stacks, descent theory, properties of 1-morphisms of stacks, geometric stacks, and stacks associated to groupoids.

Our principal references were Artin \cite{Arti}, Behrend et al.\ \cite{BEFF}, Gomez
\cite{Gome}, Laumon and Moret-Bailly \cite{LaMo}, Metzler
\cite{Metz}, Noohi \cite{Nooh}, and Olsson \cite{Olss2}. The topological and smooth stacks discussed by Metzler and Noohi are closer to our situation than the stacks in algebraic geometry of
\cite{BEFF,Gome,LaMo}, so we often follow \cite{Metz,Nooh},
particularly in \S\ref{agA5} which is based on Metzler \cite[\S
3]{Metz}. Heinloth \cite{Hein} and Behrend and Xu \cite{BeXu} also
discuss smooth stacks.

\subsection{Introduction to 2-categories}
\label{agA1}
\I{2-category|(}

A good reference on 2-categories for our purposes is Behrend et al.\
\cite[App.~B]{BEFF}, and Borceux \cite[\S 7]{Borc} and Kelly and Street \cite{KeSt} are also helpful.

\begin{dfn} A {\it strict\/} 2-{\it category\/} $\bs{\cal C}$\I{2-category!strict} consists of a proper class of {\it objects\/} $\Obj(\bs{\cal C})$, for all
$X,Y$ in $\Obj(\bs{\cal C})$ a small category $\Hom(X,Y)$, for all
$X$ in $\Obj(\bs{\cal C})$ an object $\id_X$ in $\Hom(X,X)$ called the
{\it identity $1$-morphism}, and for all $X,Y,Z$ in $\Obj(\bs{\cal C})$
a functor
\begin{equation*}
\mu_{X,Y,Z}:\Hom(X,Y)\t\Hom(Y,Z)\longra\Hom(X,Z).
\end{equation*}
These must
satisfy the {\it identity property}, that
\e
\mu_{X,X,Y}(\id_X,-)=\mu_{X,Y,Y}(-,\id_Y)=\id_{\Hom(X,Y)}
\label{agAeq1}
\e
as functors $\Hom(X,Y)\ra\Hom(X,Y)$, and the {\it associativity property}, that
\e
\mu_{W,Y,Z}\ci(\mu_{W,X,Y}\t\id_{\Hom(Y,Z)})
=\mu_{W,X,Z}\ci(\id_{\Hom(W,X)}\t\mu_{X,Y,Z}) 
\label{agAeq2}
\e
for all $W,X,Y,Z$, as functors
\begin{equation*}
\Hom(W,X)\t\Hom(X,Y)\t\Hom(Y,Z)\longra\Hom(W,X).
\end{equation*} 

Objects $f$ of $\Hom(X,Y)$ are called 1-{\it
morphisms},\I{2-category!1-morphism} written $f:X\ra Y$. For
1-morphisms $f,g:X\ra Y$, morphisms $\eta\in \Hom_{\Hom(X,Y)}(f,g)$
are called 2-{\it morphisms},\I{2-category!2-morphism} written
$\eta:f\Ra g$. Thus, a 2-category has objects $X$, and two kinds of
morphisms, 1-morphisms $f:X\ra Y$ between objects, and 2-morphisms
$\eta:f\Ra g$ between 1-morphisms. 

A {\it weak\/ $2$-category},\I{2-category!weak} or {\it bicategory}, is like a strict 2-category, except that the equations of functors \eq{agAeq1}, \eq{agAeq2} are required to hold only up to specified natural isomorphisms, which should themselves satisfy identities. Strict 2-categories are examples of weak 2-categories in which these specified natural isomorphisms are identities. We will not give much detail on weak 2-categories, since the 2-categories of stacks we are interested in are strict.

In many examples, all 2-morphisms are 2-isomorphisms (i.e.\ have an inverse), so that the categories $\Hom(X,Y)$ are groupoids. Such 2-categories are called (2,1)-{\it
categories}.
\label{agAdef1}
\end{dfn}

This is quite a complicated structure. There are three kinds of
composition in a 2-category, satisfying various associativity
relations. If $f:X\ra Y$ and $g:Y\ra Z$ are 1-morphisms then
$\mu_{X,Y,Z}(f,g)$ is the {\it composition of\/
$1$-morphisms},\I{2-category!1-morphism!composition} written $g\ci
f:X\ra Z$. If $f,g,h:X\ra Y$ are 1-morphisms and $\eta:f\Ra g$,
$\ze:g\Ra h$ are 2-morphisms then composition of $\eta,\ze$ in the
category $\Hom(X,Y)$ gives the {\it vertical composition of\/
$2$-morphisms}\I{2-category!2-morphism!vertical composition} of
$\eta,\ze$, written $\ze\od\eta:f\Ra h$, as a diagram
\begin{equation*}
\xymatrix@C=25pt{ X \rruppertwocell^f{\eta} \rrlowertwocell_h{\ze}
\ar[rr]_(0.35)g && Y & \ar@{~>}[r] && X
\rrtwocell^f_h{{}\,\,\,\,\ze\od\eta\!\!\!\!\!} && Y.}
\end{equation*}
And if $f,\ti f:X\ra Y$ and $g,\ti g:Y\ra Z$ are 1-morphisms and
$\eta:f\Ra\ti f$, $\ze:g\Ra\ti g$ are 2-morphisms then
$\mu_{X,Y,Z}(\eta,\ze)$ is the {\it horizontal composition of\/
$2$-morphisms},\I{2-category!2-morphism!horizontal composition}
written $\ze*\eta:g\ci f\Ra\ti g\ci\ti f$, as a diagram
\begin{equation*}
\xymatrix@C=20pt{ X \rrtwocell^f_{\ti f}{\eta} && Y
\rrtwocell^g_{\ti g}{\ze} && Z & \ar@{~>}[r] && X \rrtwocell^{g\ci
f}_{\ti g\ci\ti f}{{}\,\,\,\ze*\eta\!\!\!\!\!} && Z. }
\end{equation*}
There are also two kinds of identity: {\it identity\/
$1$-morphisms\/} $\id_X:X\ra X$ and {\it identity\/
$2$-morphisms\/}~$\id_f:f\Ra f$.

In a strict 2-category $\bs\cC$, composition of 1-morphisms is strictly associative, $(g\ci f)\ci e=g\ci(f\ci e)$, and horizontal composition of 2-morphisms is strictly associative, $(\ze*\eta)*\ep=\ze*(\eta*\ep)$. In a weak 2-category $\bs\cC$, composition of 1-morphisms is associative up to specified 2-isomorphisms. 

A basic example is the {\it $2$-category of categories\/}
$\mathfrak{Cat}$, with objects small categories $\cC$, 1-morphisms
functors $F:\cC\ra\cD$, and 2-morphisms natural transformations
$\eta:F\Ra G$ for functors $F,G:\cC\ra\cD$. Orbifolds naturally form
a 2-category, as do Deligne--Mumford\I{stack} and Artin
stacks in algebraic geometry.

In a 2-category $\bs{\cal C}$, there are three notions of when
objects $X,Y$ in $\bs{\cal C}$ are `the same': {\it equality\/}
$X=Y$, and {\it isomorphism}, that is we have 1-morphisms $f:X\ra
Y$, $g:Y\ra X$ with $g\ci f=\id_X$ and $f\ci g=\id_Y$, and {\it
equivalence},\I{2-category!equivalence in} that is we have
1-morphisms $f:X\ra Y$, $g:Y\ra X$ and 2-isomorphisms $\eta:g\ci
f\Ra\id_X$ and $\ze:f\ci g\Ra\id_Y$. Usually equivalence is the most
useful. For example, isomorphisms are not preserved by equivalences of 2-categories, whereas equivalences are.

Let $\bs{\cal C}$ be a 2-category. The {\it homotopy
category\/}\I{homotopy category} $\Ho(\bs{\cal C})$ of $\bs{\cal C}$
is the category whose objects are objects of $\bs{\cal C}$, and
whose morphisms $[f]:X\ra Y$ are 2-isomorphism classes $[f]$ of
1-morphisms $f:X\ra Y$ in $\bs{\cal C}$. Then equivalences in
$\bs{\cal C}$ become isomorphisms in $\Ho(\bs{\cal C})$,
2-commutative diagrams in $\bs{\cal C}$ become commutative diagrams
in $\Ho(\bs{\cal C})$, and so on.

{\it Commutative diagrams\/} in 2-categories should in general only
commute {\it up to (specified)\/ $2$-isomorphisms}, rather than
strictly. Then we say the diagram 2-{\it commutes}. A simple example
of a commutative diagram in a 2-category $\bs{\cal C}$
is\I{2-category!2-commutative diagram}
\begin{equation*}
\xymatrix@C=50pt@R=8pt{ & Y \ar[dr]^g \ar@{=>}[d]^\eta \\
X \ar[ur]^f \ar[rr]_h && Z, }
\end{equation*}
which means that $X,Y,Z$ are objects of $\bs{\cal C}$, $f:X\ra Y$,
$g:Y\ra Z$ and $h:X\ra Z$ are 1-morphisms in $\bs{\cal C}$, and
$\eta:g\ci f\Ra h$ is a 2-isomorphism.

Next we discuss 2-functors between 2-categories, following Borceux \cite[\S 7.2, \S 7.5]{Borc} and Behrend et al.~\cite[\S B.4]{BEFF}.

\begin{dfn} Let $\bs\cC,\bs\cD$ be strict 2-categories. A {\it strict\/ $2$-functor\/} $F:\bs\cC\ra\bs\cD$ assigns an object $F(X)$ in $\bs\cD$ for each object $X$ in $\bs\cC$, a 1-morphism $F(f):F(X)\ra F(Y)$ in $\bs\cD$ for each 1-morphism $f:X\ra Y$ in $\bs\cC$, and a 2-morphism $F(\eta):F(f)\Ra F(g)$ in $\bs\cD$ for each 2-morphism $\eta:f\Ra g$ in $\bs\cC$, such that $F$ preserves all the structures on $\bs\cC,\bs\cD$, that is,
\ea
F(g\ci f)&=F(g)\ci F(f), & \!F(\id_X)&=\id_{F(X)}, & \!F(\ze*\eta)&=\!F(\ze)\!* \!F(\eta), 
\label{agAeq3}\\
F(\ze\od\eta)&=F(\ze)\od F(\eta), & \!F(\id_f)&=\id_{F(f)}.
\label{agAeq4}
\ea

Now let $\bs\cC,\bs\cD$ be weak 2-categories. Then strict 2-functors $F:\bs\cC\ra\bs\cD$ are not well-behaved. To fix this, we need to relax \eq{agAeq3} to hold only up to specified 2-isomorphisms. A {\it weak\/ $2$-functor\/} (or {\it pseudofunctor\/}) $F:\bs\cC\ra\bs\cD$ assigns an object $F(X)$ in $\bs\cD$ for each object $X$ in $\bs\cC$, a 1-morphism $F(f):F(X)\ra F(Y)$ in $\bs\cD$ for each 1-morphism $f:X\ra Y$ in $\bs\cC$, a 2-morphism $F(\eta):F(f)\Ra F(g)$ in $\bs\cD$ for each 2-morphism $\eta:f\Ra g$ in $\bs\cC$, a 2-isomorphism $F_{g,f}:F(g)\ci F(f)\Ra F(g\ci f)$ in $\bs\cD$ for all 1-morphisms $f:X\ra Y$, $g:Y\ra Z$ in $\bs\cC$, and a 2-isomorphism $F_X:F(\id_X)\Ra \id_{F(X)}$ in $\bs\cD$ for all objects $X$ in $\bs\cC$ such that \eq{agAeq4} holds, and for all $e:W\ra X$, $f:X\ra Y$, $g:Y\ra Z$ in $\bs\cC$ the following diagram of 2-isomorphisms commutes in $\bs\cD$:
\begin{equation*}
\xymatrix@C=87pt@R=15pt{ *+[r]{(F(g)\ci F(f))\ci F(e)} \ar@{=>}[d]^{\al_{F(g),F(f),F(e)}} \ar@{=>}[r]_(0.65){F_{g,f}*\id_{F(e)}} & F(g\ci f)\ci F(e) \ar@{=>}[r]_(0.4){F_{g\ci f,e}} & *+[l]{F((g\ci f)\ci e)} \ar@{=>}[d]_{F(\al_{g,f,e})} \\
*+[r]{F(g)\ci (F(f)\ci F(e))} \ar@{=>}[r]^(0.65){\id_{F(g)}*F_{f,e}} & F(g)\ci F(f\ci e) \ar@{=>}[r]^(0.4){F_{g,f\ci e}} &
*+[l]{F(g\ci (f\ci e)),\!\!} } 
\end{equation*}
and for all 1-morphisms $f:X\ra Y$ in $\bs\cC$, the following commute in $\bs\cD$:
\begin{equation*}
\xymatrix@!0@C=25pt@R=30pt{
*+[r]{F(f)\ci F(\id_X)} \ar@{=>}[rrrrrr]_(0.58){F_{f,\id_X}} \ar@{=>}[d]^{\id_{F(f)}*F_X} &&&&&& *+[l]{F(f\ci\id_X)} \ar@{=>}[d]_{F(\be_f)} & *+[r]{F(\id_Y)\ci F(f)} \ar@{=>}[rrrrrr]_(0.58){F_{\id_Y,f} } \ar@{=>}[d]^{F_Y*\id_{F(f)}} &&&&&& *+[l]{F(\id_Y\ci f)} \ar@{=>}[d]_{F(\ga_f)} \\
*+[r]{F(f)\ci \id_{F(X)}} \ar@{=>}[rrrrrr]^(0.6){\be_{F(f)}} &&&&&& *+[l]{F(f),\!\!} &
*+[r]{\id_{F(Y)}\ci F(f)} \ar@{=>}[rrrrrr]^(0.6){\ga_{F(f)}} &&&&&& *+[l]{F(f),\!\!} } 
\end{equation*}
and if $f,\dot f:X\ra Y$ and $g,\dot g:Y\ra Z$ are 1-morphisms and
$\eta:f\Ra\dot f$, $\ze:g\Ra\dot g$ are 2-morphisms in $\bs\cC$ then the following commutes in $\bs\cD$:
\begin{equation*}
\xymatrix@C=140pt@R=15pt{ *+[r]{F(g)\ci F(f)} \ar@{=>}[d]^{F(\ze)*F(\eta)} \ar@{=>}[r]_{F_{g,f}} & *+[l]{F(g\ci f)} \ar@{=>}[d]_{F(\ze*\eta)} \\
*+[r]{F(\dot g)\ci F(\dot f)} \ar@{=>}[r]^{F_{\dot g,\dot f}} & *+[l]{F(\dot g\ci\dot f).\!\!} } 
\end{equation*}

There are obvious notions of {\it composition\/} $G\ci F$ of strict and weak 2-functors $F:\bs\cC\ra\bs\cD$, $G:\bs\cD\ra\bs\cE$, {\it identity\/ $2$-functors\/} $\id_{\bs\cC}$, and so on.

If $\bs\cC,\bs\cD$ are strict 2-categories, then a strict $2$-functor $F:\bs\cC\ra\bs\cD$ can be made into a weak 2-functor by taking all $F_{g,f},F_X$ to be identity 2-morphisms. 
\label{agAdef2}
\end{dfn}

Here are some well-known facts about 2-categories and 2-functors:
\begin{itemize}
\setlength{\itemsep}{0pt}
\setlength{\parsep}{0pt}
\item[(i)] Every weak 2-category $\bs\cC$ is equivalent as a weak 2-category to a strict 2-category $\bs\cC'$, that is, weak 2-categories can always be strictified.
\item[(ii)] If $\bs\cC,\bs\cD$ are strict 2-categories, and $F:\bs\cC\ra\bs\cD$ is a weak 2-functor, it may not be true that $F$ is 2-naturally isomorphic to a strict 2-functor $F':\bs\cC\ra\bs\cD$. That is, weak 2-functors cannot necessarily be strictified. 

Even if one is working with strict 2-categories, weak 2-functors are often the correct notion of functor between them.
\end{itemize}

We define fibre products\I{2-category!fibre products in|(} in 2-categories, following
\cite[Def.~B.13]{BEFF}.

\begin{dfn} Let $\bs{\cal C}$ be a 2-category and $g:X\ra Z$,
$h:Y\ra Z$ be 1-morphisms in $\bs{\cal C}$. A {\it fibre product\/}
$X\t_ZY$ in $\bs{\cal C}$ consists of an object $W$, 1-morphisms
$e:W\ra X$ and $f:W\ra Y$ and a 2-isomorphism $\eta:g\ci e\Ra h\ci
f$ in $\bs{\cal C}$, so that we have a 2-commutative diagram
\e
\begin{gathered}
\xymatrix@C=100pt@R=14pt{ *+[r]{W} \ar[r]_(0.25){f} \ar[d]^{e}
\drtwocell_{}\omit^{}\omit{^{\eta}}
 & *+[l]{Y} \ar[d]_{h} \\ *+[r]{X} \ar[r]^(0.7){g} & *+[l]{Z} }
\end{gathered}
\label{agAeq5}
\e
with the following universal property: suppose $e':W'\ra X$ and
$f':W'\ra Y$ are 1-morphisms and $\eta':g\ci e'\Ra h\ci f'$ is a
2-isomorphism in $\bs{\cal C}$. Then there should exist a 1-morphism
$b:W'\ra W$ and 2-isomorphisms $\ze:e\ci b\Ra e'$, $\th:f\ci b\Ra
f'$ such that the following diagram of 2-isomorphisms commutes:
\e
\begin{gathered}
\xymatrix@C=90pt@R=15pt{ *+[r]{g\ci e\ci b} \ar@{=>}[r]_{\eta*\id_b}
\ar@{=>}[d]^{{}\id_g*\ze} & *+[l]{h\ci f\ci b} \ar@{=>}[d]_{\id_h*\th{}}
\\ *+[r]{g\ci e'} \ar@{=>}[r]^{\eta'} & *+[l]{h\ci f'.\!} }
\end{gathered}
\label{agAeq6}
\e
Furthermore, if $\ti b,\ti\ze,\ti\th$ are alternative choices of
$b,\ze,\th$ then there should exist a unique 2-isomorphism $\ep:\ti
b\Ra b$ with
\begin{equation*}
\ti\ze=\ze\od(\id_{e}*\ep)\quad\text{and}\quad
\ti\th=\th\od(\id_{f}*\ep).
\end{equation*}
We call such a fibre product diagram \eq{agAeq5} a 2-{\it Cartesian
square}.\I{2-category!2-Cartesian square} If a fibre product
$X\t_ZY$ in $\bs{\cal C}$ exists then it is unique up to equivalence
in~$\bs{\cal C}$.
\label{agAdef3}
\end{dfn}

Orbifolds,\I{orbifold!transverse fibre product} and stacks\I{stack}
in algebraic geometry, form 2-categories, and Definition
\ref{agAdef3} is the right way to define fibre products of orbifolds
or stacks, as in \cite{BEFF}. Given a 2-commutative diagram in a 2-category
\begin{equation*}
\xymatrix@C=80pt@R=15pt{ U \ar[r]_(0.25){f} \ar[d]^{e}
\drtwocell_{}\omit^{}\omit{^{\eta}}
 & W \ar[r]_(0.25){i} \ar[d]^{h}
\drtwocell_{}\omit^{}\omit{^{\ze}}
 & Y \ar[d]_{k} \\ V \ar[r]^(0.7){g} & X \ar[r]^(0.7){j} & Z,\!}
\end{equation*}
if the two small rectangles are 2-Cartesian, then the outer rectangle is too.\I{2-category!fibre products
in|)}\I{2-category|)}

\subsection{Grothendieck topologies, sites, prestacks, and stacks}
\label{agA2}
\I{stack|(}\I{site|(}

Some references for this section are Olsson \cite{Olss2}, Artin \cite{Arti}, Behrend et al.\ \cite{BEFF}, and Laumon and Moret-Bailly~\cite{LaMo}.

\begin{dfn} Let $\cC$ be a category, and $U\in\cC$. A {\it sieve\/} $\cS$ on $U$ is a collection of morphisms $\phi:V\ra U$ in $\cC$ closed under precomposition, that is, if $\phi:V\ra U$ lies in $\cS$ and $\psi:W\ra V$ is a morphism in $\cC$ then $\phi\ci\psi:W\ra U$ lies in $\cS$.

A {\it Grothendieck topology\/}\I{Grothendieck topology\/} on $\cC$ is a collection of distinguished sieves for each object $U\in\cC$ called {\it covering sieves}, satisfying some axioms we will not give. A {\it site\/} $(\cC,\cJ)$ is a category $\cC$ with a Grothendieck
topology~$\cJ$.

It is often convenient to define Grothendieck topologies using Grothendieck pretopologies. A {\it Grothendieck pretopology\/}\I{Grothendieck pretopology} $\cP\cJ$ on $\cC$ is a collection of families $\{\vp_a:U_a\ra U\}_{a\in A}$ of morphisms in $\cC$ called {\it
coverings}, satisfying:
\begin{itemize}
\setlength{\itemsep}{0pt}
\setlength{\parsep}{0pt}
\item[(i)] If $\vp:V\ra U$ is an isomorphism in $\cC$, then
$\{\vp:V\ra U\}$ is a covering;
\item[(ii)] If $\{\vp_a:U_a\ra U\}_{a\in A}$ is a covering, and
$\{\psi_{ab}:V_{ab}\ra U_a\}_{b\in B_a}$ is a covering for all
$a\in A$, then $\{\vp_a\ci\psi_{ab}:V_{ab}\ra U\}_{a\in A,\;
b\in B_a}$ is a covering.
\item[(iii)] If $\{\vp_a:U_a\ra U\}_{a\in A}$ is a covering and
$\psi:V\ra U$ is a morphism in $\cC$ then
$\{\pi_V:U_a\t_{\vp_a,U,\psi}V\ra V\}_{a\in A}$ is a covering,
where the fibre product\I{fibre product}\I{category!fibre
product} $U_a\t_UV$ exists in $\cC$ for all $a\in A$.
\end{itemize}
Each Grothendieck pretopology $\cP\cJ$ has an associated Grothendieck topology $\cJ$, in which a sieve $\cS$ on $U\in\cC$ is a covering sieve in $\cJ$ if and only if it contains a covering $\{\vp_a:U_a\ra U\}_{a\in A}$ in $\cP\cJ$.

A Grothendieck pretopology $\cP\cJ$ gives a notion of {\it open cover\/} of objects in $\cC$. For example, if $\cC$ is the category of topological spaces $\Top$, we could define $\cP\cJ$ to be the collection of families $\{\vp_a:U_a\ra U\}_{a\in A}$ in $\Top$ such that $\vp_a:U_a\ra U$ is a homeomorphism with an open subset $\vp_a(U_a)\subseteq U$ for $a\in A$, with $U=\bigcup_{a\in A}\vp_a(U_a)$, so that $\{\vp_a:U_a\ra U\}_{a\in A}$ is an open cover of~$U$.
\label{agAdef4}
\end{dfn}

\begin{dfn} Let $\cC$ be a category. A {\it category fibred in
groupoids over\/}\I{category!fibred in groupoids} $\cC$ is a functor
$p_\cX:\cX\ra\cC$, where $\cX$ is a category, such that given any
morphism $g:C_1\ra C_2$ in $\cC$ and $X_2\in\cX$ with
$p_\cX(X_2)=C_2$, there exists a morphism $f:X_1\ra X_2$ in $\cX$
with $p_\cX(f)=g$, and given commutative diagrams (on the left) in
$\cX$, in which $g$ is to be determined, and (on the right) in
$\cC$:
\e
\begin{gathered}
\xymatrix@C=25pt@R=10pt{ X_1 \ar@{..>}[rr]_g \ar[dr]_f && X_2
\ar[dl]^h \\ & X_3 }\end{gathered}\quad {\buildrel p_\cX \over
\rightsquigarrow} \quad \begin{gathered} \xymatrix@C=10pt@R=5pt{
p_\cX(X_1) \ar[rr]_{g'} \ar[dr]_{p_\cX(f)} && p_\cX(X_2)
\ar[dl]^{p_\cX(h)} \\ & p_\cX(X_3), }
\end{gathered}
\label{agAeq7}
\e
then there exists a unique morphism $g$ as shown with $p_\cX(g)=g'$
and $f=h\ci g$. Often we refer to $\cX$ as the category fibred in
groupoids (or prestack, or stack, etc.), leaving $p_\cX$ implicit.

If $p_\cX:\cX\ra\cC$ is a category fibred in groupoids and $C$ is an object in $\cC$, the {\it fibre\/} $\cX_C$ is the subcategory of $\cX$ with objects those $X\in\cX$ with $p_\cX(X)=C$, and morphisms those $f:X_1\ra X_2$ with $p_\cX(f)=\id_C:C\ra C$. Then $\cX_C$ is a groupoid (i.e.\ a category with all morphisms isomorphisms).
\label{agAdef5}
\end{dfn}

\begin{dfn} Let $(\cC,\cJ)$ be a site, and $p_\cX:\cX\ra\cC$ be a
category fibred in groupoids over $\cC$. We call $\cX$ a {\it
prestack\/}\I{prestack} if whenever $\{\vp_a:U_a\ra U\}_{a\in A}$ is
a covering family in $\cJ$ and we are given commutative diagrams in
$\cX,\cC$ for all $a,b\in A$, in which $f$ is to be determined:
\e
\begin{gathered}
\xymatrix@R=7pt@C=6pt{ & X_{ab} \ar[dl] \ar@<-1ex>[ddr] \ar[rr] &&
Y_{ab} \ar[dl] \ar[ddr] \\ X_a \ar[ddr]_{x_a} \ar[rr] && Y_a
\ar@<1ex>[ddr]^(0.4){y_a} \\ && X_b \ar[rr] \ar[dl]_(0.4){x_b} &&
Y_b \ar[dl]^(0.4){y_b} \\ & X \ar@{..>}[rr]^(0.7)f && Y}
\end{gathered}
\;\>{\buildrel p_\cX \over \rightsquigarrow}\;\>
\begin{gathered} \xymatrix@R=6pt@C=6pt{ & \!\!U_a\!\t_U\!U_b
\ar[dl]_(0.7){\pi_{U_a}} \ar[ddr]^(0.4){\pi_{U_b}} \ar@{=}[rr] &&
U_a\!\t_U\!U_b\!\! \ar[dl]^{\pi_{U_a}}
\ar[ddr]^(0.6){\pi_{U_b}} \\
U_a \ar[ddr]_{\vp_a} \ar@{=}[rr] && U_a \ar[ddr]^(0.35){\vp_a} \\ &&
U_b \ar@{=}[rr] \ar[dl]^(0.3){\vp_b} && U_b \ar[dl]^(0.4){\vp_b} \\
& U \ar@{=}[rr] && U,}\!\!\!{}
\end{gathered}
\label{agAeq8}
\e
then there exists a unique $f:X\ra Y$ in $\cX$ with $p_\cX(f)=\id_U$
making \eq{agAeq8} commute for all~$a\in A$.

Let $p_\cX:\cX\ra\cC$ be a prestack. We call $\cX$ a {\it
stack\/}\I{stack!definition} if whenever $\{\vp_a:U_a\ra U\}_{a\in
A}$ is a covering family in $\cJ$ and we are given commutative
diagrams in $\cX,\cC$ for all $a,b,c\in A$, with $X_{ab}=X_{ba}$,
$X_{abc}=X_{bac}=X_{acb}$, etc., in which the object $X$ and
morphisms $x_a$ are be determined:
\begin{small}
\begin{equation}
\begin{gathered}
\xymatrix@R=6pt@C=3.5pt{ & X_{abc} \ar[dl] \ar@<-1ex>[ddr] \ar[rr]
&& X_{ac} \ar[dl]^(0.4){x_{ac}} \ar[ddr]^(0.4){x_{ca}} \\
X_{ab} \ar[ddr]_{x_{ba}} \ar[rr]_{x_{ab}} && X_a
\ar@{..>}@<1ex>[ddr]^(0.45){x_a} \\ && X_{bc} \ar[rr]^(0.7){x_{cb}}
\ar[dl]_(0.4){x_{bc}} && X_c \ar@{..>}[dl]^(0.4){x_c} \\ & X_b
\ar@{..>}[rr]^(0.7){x_b} && X}
\end{gathered}
\;{\buildrel p_\cX \over \rightsquigarrow}\>\;\> \begin{gathered}
\xymatrix@!0@R=18pt@C=35pt{ & \!\!\!\!\!U_a\!\t_U\!U_b\!\t_U\!U_c \ar[dl]
\ar[ddr] \ar[rr] && U_a\!\t_U\!U_c\!\!\! \ar[dl]
\ar[ddr] \\
\!\!\!U_a\!\t_U\!U_b\!\! \ar[ddr] \ar[rr] && U_a
\ar@<1ex>[ddr]^{\vp_a} \\ && \!\!\!U_b\!\t_U\!U_c \ar[rr] \ar[dl]
&& U_c \ar[dl]^{\vp_c} \\ & U_b \ar[rr]^{\vp_b} && U,}
\end{gathered}
\label{agAeq9}
\end{equation}\end{small}then there exists $X\in\cX$ and morphisms
$x_a:X_a\ra X$ with $p_\cX(x_a)=\vp_a$ for all $a\in A$, making
\eq{agAeq9} commute for all $a,b,c\in A$.
\label{agAdef6}
\end{dfn}

Thus, in a prestack we have a sheaf-like condition allowing us to
glue morphisms in $\cX$ uniquely over covers in $\cC$; in a
stack we also have a sheaf-like condition allowing us to glue
objects in $\cX$ over covers in~$\cC$.

\begin{dfn} Let $(\cC,\cJ)$ be a site. A 1-{\it
morphism\/}\I{stack!1-morphism}\I{prestack!1-morphism} between
(pre)stacks $\cX,\cY$ on $(\cC,\cJ)$ is a functor
$F:\cX\ra\cY$ with $p_\cY\ci F=p_\cX:\cX\ra\cC$. If $F,G:\cX\ra\cY$
are 1-morphisms, a 2-{\it
morphism\/}\I{stack!2-morphism}\I{prestack!2-morphism} $\eta:F\Ra G$
is an isomorphism of functors with $\id_{p_\cY}*\eta=
\id_{p_\cX}:p_\cY\ci F\Ra p_\cY\ci G$. That is, for all $X\in\cX$ we
are given an isomorphism $\eta(X):F(X)\ra G(X)$ in $\cY$ with
$p_\cY(\eta(X))=\id_{p_\cX(X)}$, such that if $f:X_1\ra X_2$ is a
morphism in $\cX$ then $\eta(X_2)\ci F(f)=G(f)\ci\eta(X_1):F(X_1)\ra
G(X_2)$ in $\cY$. With these definitions, the stacks and prestacks
on $(\cC,\cJ)$ form ({\it strict\/}) 2-{\it
categories},\I{2-category} which we write as
$\Sta_{(\cC,\cJ)}$\G[StaCJ]{$\Sta_{(\cC,\cJ)}$}{2-category of stacks
on a site $(\cC,\cJ)$} and
$\Presta_{(\cC,\cJ)}$.\G[PrestaCJ]{$\Presta_{(\cC,\cJ)}$}{2-category
of prestacks on a site $(\cC,\cJ)$} All 2-morphisms in
$\Sta_{(\cC,\cJ)},\Presta_{(\cC,\cJ)}$ are invertible, that is, are
2-isomorphisms, so $\Sta_{(\cC,\cJ)},\Presta_{(\cC,\cJ)}$ are (2,1)-categories.

A {\it substack\/}\I{stack!substack} $\cY$ of a stack $\cX$ is a
strictly full subcategory $\cY$ in $\cX$ such that
$p_\cY:=p_\cX\vert_\cY:\cY\ra\cC$ is a stack. The inclusion functor
$i_\cY:\cY\hookra\cX$ is then a 1-morphism of stacks.
\label{agAdef7}
\end{dfn}

\begin{dfn} Let $(\cC,\cJ)$ be a site, and $\cX$ a
prestack on $(\cC,\cJ)$, so that $\Sta_{(\cC,\cJ)}$ and
$\Presta_{(\cC,\cJ)}$ are 2-categories. A {\it stack associated
to\/} $\cX$, or {\it stackification of\/}\I{prestack!stackification}
$\cX$, is a stack $\hcX$ with a 1-morphism of prestacks
$i:\cX\ra\hcX$, such that for every stack $\cY$, composition with
$i$ yields an equivalence of categories $\Hom(\hcX,\cY)\,{\buildrel
i^*\over\longra}\,\Hom(\cX,\cY)$.
\label{agAdef8}
\end{dfn}

As in \cite[Lem.~3.2]{LaMo}, every prestack has an associated stack,
just as every presheaf has an associated sheaf.

\begin{prop} For every prestack\/ $\cX$ on $(\cC,\cJ)$
there exists an associated stack\/ $i:\cX\ra\hcX,$ which is unique
up to equivalence in $\Sta_{(\cC,\cJ)}$.
\label{agAprop1}
\end{prop}

There is a natural construction of fibre products in the
2-category~$\Sta_{(\cC,\cJ)}$:

\begin{dfn} Let $(\cC,\cJ)$ be a site, $\cX,\cY,\cZ$ be stacks on
$(\cC,\cJ)$, and $F:\cX\ra\cZ$, $G:\cY\ra\cZ$ be 1-morphisms. Define
a category $\cW$ to have objects $(X,Y,\al)$, where $X\in\cX$,
$Y\in\cY$ and $\al:F(X)\ra G(Y)$ is an isomorphism in $\cZ$ with
$p_\cX(X)=p_\cY(Y)=U$ and $p_\cX(\al)=\id_U$ in $\cC$, and for
objects $(X_1,Y_1,\al_1),(X_2,Y_2,\al_2)$ in $\cW$ a morphism
$(f,g):(X_1,Y_1,\al_1)\ra(X_2,Y_2,\al_2)$ in $\cW$ is a pair of
morphisms $f:X_1\ra X_2$ in $\cX$ and $g:Y_1\ra Y_2$ in $\cY$ with
$p_\cX(f)=p_\cY(g)=\vp:U\ra V$ in $\cC$ and $\al_2\ci
F(f)=G(g)\ci\al_1:F(X_1)\ra G(Y_2)$ in $\cZ$. Then $\cW$ is a stack
over~$(\cC,\cJ)$.

Define 1-morphisms $p_\cW:\cW\ra\cC$ by $p_\cW:(X,Y,\al)\mapsto
p_\cX(X)$ and $p_\cW:(f,g)\mapsto p_\cX(f)$, and $\pi_\cX:\cW\ra\cX$
by $\pi_\cX:(X,Y,\al)\mapsto X$ and $\pi_\cX:(f,g)\mapsto f$, and
$\pi_\cY:\cW\ra\cY$ by $\pi_\cY:(X,Y,\al)\mapsto Y$ and
$\pi_\cY:(f,g)\mapsto g$. Define a 2-morphism $\eta:F\ci\pi_\cX\Ra
G\ci\pi_\cY$ by $\eta(X,Y,\al)=\al$. Then $\cW,\pi_\cX,\pi_\cY,\eta$
is a fibre product\I{stack!fibre product}\I{2-category!fibre
products in} $\cX\t_\cZ\cY$ in $\Sta_{(\cC,\cJ)}$, in the sense of
Definition~\ref{agAdef3}.

The functor $\id_\cC:\cC\ra\cC$ is a {\it terminal object\/} in
$\Sta_{(\cC,\cJ)}$, and may be thought of as a point $*$. {\it
Products\/} $\cX\t\cY$ in $\Sta_{(\cC,\cJ)}$ are fibre products over
$*$. If $\cX$ is a stack, the {\it diagonal\/ $1$-morphism\/} is the
natural 1-morphism $\De_\cX:\cX\ra\cX\t\cX$. The {\it inertia
stack\/}\I{inertia stack} $I_\cX$ of $\cX$ is the fibre product
$\cX\t_{\De_\cX,\cX\t\cX,\De_\cX}\cX$, with natural {\it inertia\/
$1$-morphism\/} $\io_\cX:I_\cX\ra\cX$ from projection to the first
factor of $\cX$. Then we have a 2-Cartesian
diagram\I{2-category!2-Cartesian square} in~$\Sta_{(\cC,\cJ)}$:
\begin{equation*}
\xymatrix@C=90pt@R=14pt{ *+[r]{I_\cX} \ar[r] \ar[d]^{\io_\cX}
\drtwocell_{}\omit^{}\omit{^{}} & *+[l]{\cX} \ar[d]_{\De_\cX} \\
*+[r]{\cX} \ar[r]^(0.25){\De_\cX} & *+[l]{\cX\t\cX.\!} }
\end{equation*}
There is also a natural 1-morphism $\jmath_\cX:\cX\ra I_\cX$ induced
by the 1-morphism $\id_\cX$ from $\cX$ to the two factors $\cX$ in
$I_\cX=\cX\t_{\cX\t\cX}\cX$ and the identity 2-morphism on
$\De_\cX\ci\id_\cX:\cX\ra\cX\t\cX$.
\label{agAdef9}
\end{dfn}

\subsection{Descent theory on a site}
\label{agA3}
\I{descent theory|(}

The theory of descent in algebraic geometry, due to Grothendieck,
says that objects and morphisms over a scheme $U$ can be described
locally on an open cover $\{U_i:i\in I\}$ of $U$. It is described by
Behrend et al.\ \cite[App.~A]{BEFF} and Olsson \cite[\S 4]{Olss2}, and at length by Vistoli
\cite{Vist}. We shall express descent as conditions on a general
site~$(\cC,\cJ)$.

\begin{dfn} Let $(\cC,\cJ)$ be a site. We say that
$(\cC,\cJ)$ {\it has descent for objects\/}\I{site!has descent for
objects} if whenever $\{\vp_a:U_a\ra U\}_{a\in A}$ is a covering in
$\cJ$ and we are given morphisms $f_a:X_a\ra U_a$ in $\cC$ for all
$a\in A$ and isomorphisms $g_{ab}:X_a\t_{\vp_a\ci f_a,U,\vp_b}U_b\ra
X_b\t_{\vp_b\ci f_b,U,\vp_a}U_a$ in $\cC$ for all $a,b\in A$ with
$g_{ab}=g_{\smash{ba}}^{\smash{-1}}$ such that for all $a,b,c\in A$
the following diagram commutes:
\begin{gather*}
\xymatrix@R=30pt@C=-60pt{ {\begin{subarray}{l}\ts(X_a\t_{\vp_a\ci
f_a,U,\vp_b}U_b)\t_{\pi_U,U,\vp_c}U_c\cong \\ \ts(X_a\t_{\vp_a\ci
f_a,U,\vp_c}U_c)\t_{\pi_U,U,\vp_b}U_b\end{subarray}}
\ar@<.5ex>[rr]^{g_{ab}\t\id_{U_c}}
\ar@<.5ex>[dr]^(0.65){g_{ac}\t\id_{U_b}} &&
{\begin{subarray}{l}\ts(X_b\t_{\vp_b\ci
f_b,U,\vp_c}U_c)\t_{\pi_U,U,\vp_a}U_a\cong \\ \ts(X_b\t_{\vp_b\ci
f_b,U,\vp_a}U_a)\t_{\pi_U,U,\vp_c}U_c\end{subarray}}
\ar@<.5ex>[dl]^(0.35){{}\,\,g_{bc}\t\id_{U_a}}
\ar@<.5ex>[ll]^{g_{ba}\t\id_{U_c}} \\
& {\begin{subarray}{l}\ts(X_c\t_{\vp_c\ci
f_c,U,\vp_a}U_a)\t_{\pi_U,U,\vp_b}U_b\cong \\ \ts(X_c\t_{\vp_c\ci
f_c,U,\vp_b}U_b)\t_{\pi_U,U,\vp_a}U_a,\end{subarray}}
\ar@<.5ex>[ul]^(0.65){g_{ca}\t\id_{U_b}}
\ar@<.5ex>[ur]^(0.38){g_{cb}\t\id_{U_a}\,\,{}} }
\end{gather*}
then there exist a morphism $f:X\ra U$ in $\cC$ and isomorphisms
$g_a:X_a\ra X\t_{f,U,\vp_a}U_a$ for all $a\in A$ such that
$f_a=\pi_{U_a}\ci g_a$ and the diagram below commutes for
all~$a,b\in A$:
\begin{equation*}
\xymatrix@C=150pt@R=14pt{*+[r]{X_a\t_{\vp_a\ci f_a,U,\vp_b}U_b}
\ar[r]_(0.4){g_a\t\id_{U_b}} \ar[dd]^{g_{ab}} &
*+[l]{(X\t_{f,U,\vp_a}U_a)\t_{\vp_a\ci \pi_{U_a},U,\vp_b}U_b}
\ar[d]_{\cong} \\
& *+[l]{X\t_{f,U,\pi_U}(U_a\t_{\vp_a,U,\vp_b}U_b)} \ar[d]_{\cong} \\
*+[r]{X_b\t_{\vp_b\ci f_b,U,\vp_a}U_a} &
*+[l]{(X\t_{f,U,\vp_b}U_b)\t_{\vp_b\ci\pi_{U_b},U,\vp_a}U_a.\!}
\ar[l]_(0.6){g_b^{-1}\t\id_{U_a}}}
\end{equation*}
Furthermore $X,f$ should be unique up to canonical isomorphism. Note
that all the fibre products used above exist in $\cC$ by
Definition~\ref{agAdef4}(iii).
\label{agAdef10}
\end{dfn}

\begin{dfn} Let $(\cC,\cJ)$ be a site. We say that
$(\cC,\cJ)$ {\it has descent for morphisms\/}\I{site!has descent for
morphisms} if whenever $\{\vp_a:U_a\ra U\}_{a\in A}$ is a covering
in $\cJ$ and $f:X\ra U$, $g:Y\ra U$ and $h_a:X\t_{f,U,\vp_a}U_a\ra
Y\t_{g,U,\vp_a}U_a$ for all $a\in A$ are morphisms in $\cC$ with
$\pi_{U_a}\ci h_a=\pi_{U_a}$ and for all $a,b\in A$ the following
diagram commutes:
\begin{equation*}
\xymatrix@C=170pt@R=14pt{
*+[r]{(X\t_{f,U,\vp_a}U_a)\t_{\vp_a\ci \pi_{U_a},U,\vp_b}U_b}
\ar[d]^{\cong} \ar[r]_{h_a\t\id_{U_b}} &
*+[l]{(Y\t_{g,U,\vp_a}U_a)\t_{\vp_a\ci \pi_{U_a},U,\vp_b}U_b}
\ar[d]_{\cong} \\
*+[r]{X\t_{f,U,\pi_U}(U_a\t_{\vp_a,U,\vp_b}U_b)}
\ar[d]^{\cong} & *+[l]{Y\t_{g,U,\pi_U}(U_a\t_{\vp_a,U,\vp_b}U_b)}
\ar[d]_{\cong} \\
*+[r]{(X\t_{f,U,\vp_b}U_b)\t_{\vp_b\ci\pi_{U_b},U,\vp_a}U_a}
\ar[r]^{h_b\t\id_{U_a}} & *+[l]{(Y\t_{g,U,\vp_b}U_b)
\t_{\vp_b\ci\pi_{U_b},U,\vp_a}U_a,\!} }
\end{equation*}
then there exists a unique $h:X\ra Y$ in $\cC$ with
$h_a=h\t\id_{U_a}$ for all~$a\in A$.
\label{agAdef11}
\end{dfn}

Then \cite[Prop.s A.12, A.13 \& \S A.6]{BEFF} show that descent
holds for objects and morphisms for affine schemes with the fppf
topology, but for arbitrary schemes with the fppf topology, descent
holds for morphisms and fails for objects.\I{descent theory|)}

\subsection{Properties of 1-morphisms}
\label{agA4}

Objects $V$ in $\cC$ yield stacks $\bar V$ on~$(\cC,\cJ)$.

\begin{dfn} Let $(\cC,\cJ)$ be a site, and $V$ an object
of $\cC$. Define a category $\bar V$ to have objects $(U,\th)$ where
$U\in\cC$ and $\th:U\ra V$ is a morphism in $\cC$, and to have
morphisms $\psi:(U_1,\th_1)\ra (U_2,\th_2)$ where $\psi:U_1\ra U_2$
is a morphism in $\cC$ with $\th_2\ci\psi=\th_1:U_1\ra V$. Define a
functor $p_{\bar V}:\bar V\ra\cC$ by $p_{\bar V}:(U,\th)\mapsto U$
and $p_{\bar V}:\psi\mapsto\psi$. Note that $p_{\bar V}$ is {\it
injective on morphisms}. It is then automatic that $p_{\bar V}:\bar
V\ra\cC$ is a category fibred in groupoids, since in \eq{agAeq7} we
can take $g=g'$. It is also automatic that $p_{\bar V}:\bar V\ra\cC$
is a prestack, since in \eq{agAeq8} we must have
$X_a=Y_a=(U_a,\th_a)$, $x_a=y_a=\vp_a$, $X=Y=(U,\th)$, etc., and the
unique solution for $f$ is~$f=\id_U$.

The site $(\cC,\cJ)$ is called {\it
subcanonical\/}\I{site!subcanonical|(} if $\bar V$ is a stack for
all objects $V\in\cC$. If descent for morphisms holds for
$(\cC,\cJ)$ then $(\cC,\cJ)$ is subcanonical. Most sites used in practice are subcanonical. Suppose $(\cC,\cJ)$ is a subcanonical site. If
$f:V\ra W$ is a morphism in $\cC$, define a 1-morphism $\bar f:\bar
V\ra\bar W$ in $\Sta_{(\cC,\cJ)}$ by $\bar f:(U,\th)\mapsto
(U,f\ci\th)$ and $\bar f:\psi\mapsto\psi$. Then the (2-)functor
$V\mapsto\bar V$, $f\mapsto\bar f$ embeds $\cC$ as a full discrete
2-subcategory of~$\Sta_{(\cC,\cJ)}$.
\label{agAdef12}
\end{dfn}

\begin{dfn} Let $(\cC,\cJ)$ be a subcanonical site. A
stack $\cX$ over $(\cC,\cJ)$ is called {\it
representable\/}\I{stack!representable} if it is equivalent in
$\Sta_{(\cC,\cJ)}$ to a stack of the form $\bar V$ for some
$V\in\cC$. A 1-morphism $F:\cX\ra\cY$ in $\Sta_{(\cC,\cJ)}$ is
called {\it representable\/}\I{stack!1-morphism!representable} if
for all $V\in\cC$ and all 1-morphisms $G:\bar V\ra\cY$, the fibre
product $\cX\t_{F,\cY,G}\bar V$ in $\Sta_{(\cC,\cJ)}$ is a
representable stack.
\label{agAdef13}
\end{dfn}

\begin{rem} For stacks in algebraic geometry, one often takes a
different definition of representable objects and 1-morphisms:
$(\cC,\cJ)$ is a category of schemes with the \'etale topology, but
stacks are called representable if they are equivalent to an {\it
algebraic space\/}\I{algebraic space} rather than a scheme. This is
because schemes are not general enough for some purposes, e.g.\ the
quotient of a scheme by an \'etale equivalence relation may be an
algebraic space but not a scheme.

In our situation, we will have no need to enlarge $C^\iy$-schemes to
some category of `$C^\iy$-algebraic spaces', as $C^\iy$-schemes are
already general enough, e.g.\ the quotient of a locally fair
$C^\iy$-scheme\I{C-scheme@$C^\iy$-scheme!locally fair} by an \'etale
equivalence relation is a locally fair $C^\iy$-scheme. This is
because the natural topology on $C^\iy$-schemes is much finer than
the Zariski\I{Zariski topology} or \'etale topology\I{etale
topology@\'etale topology} on schemes, for instance, affine
$C^\iy$-schemes are always Hausdorff.
\label{agArem1}
\end{rem}

\begin{dfn} Let $(\cC,\cJ)$ be a subcanonical site. Let
$\bs P$ be a property of morphisms in $\cC$. (For instance, if $\cC$
is the category $\Top$ of topological spaces, then $\bs P$ could be
`proper', `open', `surjective', `covering map', \ldots). We say that
$\bs P$ is {\it invariant under base change\/} if for all Cartesian
squares\I{Cartesian square} in $\cC$
\begin{equation*}
\xymatrix@C=90pt@R=15pt{ *+[r]{W}\ar[r]_f \ar[d]^e & *+[l]{Y} \ar[d]_h \\
*+[r]{X} \ar[r]^g & *+[l]{Z,\!} }
\end{equation*}
if $g$ is $\bs P$, then $f$ is $\bs P$. We say that $\bs P$ is {\it
local on the target\/} if whenever $f:U\ra V$ is a morphism in $\cC$
and $\{\vp_a:V_a\ra V\}_{a\in A}$ is a covering in $\cJ$ such that
$\pi_{V_a}:U\t_{f,V,\vp_a}V_a\ra V_a$ is $\bs P$ for all $a\in A$,
then $f$ is~$\bs P$.

Let $\bs P$ be invariant under base change and local in the target,
and let $F:\cX\ra\cY$ be a representable
1-morphism\I{stack!1-morphism!representable} in $\Sta_{(\cC,\cJ)}$.
If $W\in\cC$ and $G:\bar W\ra\cY$ is a 1-morphism then
$\cX\t_{F,\cY,G}\bar W$ is equivalent to $\bar V$ for some
$V\in\cC$, and under this equivalence the 1-morphism $\pi_{\bar
W}:\cX\t_{F,\cY,G}\bar W\ra\bar W$ is 2-isomorphic to $\bar f:\bar
V\ra\bar W$ for some unique morphism $f:V\ra W$ in $\cC$. We say
that $F$ {\it has property\/} $\bs P$ if for all $W\in\cC$ and
1-morphisms $G:\bar W\ra\cY$, the morphism $f:V\ra W$ in $\cC$
corresponding to $\pi_{\bar W}:\cX\t_{F,\cY,G}\bar W\ra\bar W$ has
property~$\bs P$.
\label{agAdef14}
\end{dfn}

We define {\it surjective\/} 1-morphisms without requiring them
representable.

\begin{dfn} Let $(\cC,\cJ)$ be a site, and $F:\cX\ra\cY$ be a
1-morphism in $\Sta_{(\cC,\cJ)}$. We call $F$ {\it
surjective\/}\I{stack!1-morphism!surjective} if whenever $Y\in\cY$
with $p_\cY(Y)=U\in\cC$, there exists a covering $\{\vp_a:U_a\ra
U\}_{a\in A}$ in $\cJ$ such that for all $a\in A$ there exists
$X_a\in\cX$ with $p_\cX(X_a)=U_a$ and a morphism $g_a:F(X_a)\ra Y$
in $\cY$ with~$p_\cY(g_a)=\vp_a$.
\label{agAdef15}
\end{dfn}

Following \cite[Prop.~3.8.1, Lem.~4.3.3 \& Rem.~4.14.1]{LaMo},
\cite[\S 6]{Nooh}, we may prove:

\begin{prop} Let\/ $(\cC,\cJ)$ be a subcanonical site, and
\begin{equation*}
\xymatrix@C=90pt@R=15pt{ *+[r]{\cW} \ar[r]_(0.25){f} \ar[d]^{e}
\drtwocell_{}\omit^{}\omit{^\eta} & *+[l]{\cY} \ar[d]_h \\ 
*+[r]{\cX} \ar[r]^(0.7)g & *+[l]{\cZ} }
\end{equation*}
be a $2$-Cartesian square in $\Sta_{(\cC,\cJ)}$. Let\/ $\bs P$ be a
property of morphisms in $\cC$ which is invariant under base change
and local in the target. Then:
\begin{itemize}
\setlength{\itemsep}{0pt}
\setlength{\parsep}{0pt}
\item[{\bf(a)}] If\/ $h$ is
representable,\I{stack!1-morphism!representable} then\/ $e$ is
representable. If also $h$ is\/ $\bs P,$ then\/ $e$ is\/~$\bs
P$.
\item[{\bf(b)}] If\/ $g$ is surjective, then\/ $f$ is surjective.
\end{itemize}
Now suppose also that\/ $(\cC,\cJ)$ has descent for objects and
morphisms, and that\/ $g$ (and hence\/ $f$) is surjective. Then:
\begin{itemize}
\setlength{\itemsep}{0pt}
\setlength{\parsep}{0pt}
\item[{\bf(c)}] If\/ $e$ is surjective then\/ $h$ is
surjective, and if\/ $e$ is representable, then\/ $h$ is
representable, and if also $e$ is\/ $\bs P,$ then\/ $h$
is\/~$\bs P$.\I{site!subcanonical|)}
\end{itemize}
\label{agAprop2}
\end{prop}

\subsection{Geometric stacks, and stacks associated to groupoids}
\label{agA5}

The 2-category $\Sta_{(\cC,\cJ)}$ of all stacks over a site
$(\cC,\cJ)$ is usually too general to do geometry with. To obtain a
smaller 2-category whose objects have better properties, we impose
extra conditions on a stack~$\cX$:

\begin{dfn} Let $(\cC,\cJ)$ be a site. We call a stack
$\cX$ on $(\cC,\cJ)$ {\it geometric\/}\I{stack!geometric|(} if the
diagonal 1-morphism $\De_\cX:\cX\ra\cX\t\cX$ is
representable,\I{stack!1-morphism!representable} and there exists
$U\in\cC$ and a surjective 1-morphism $\Pi:\bar U\ra\cX$, which we
call an {\it atlas}\I{atlas}\I{stack!atlas} for $\cX$. Write
$\GSta_{(\cC,\cJ)}$\G[GStaCJ]{$\GSta_{(\cC,\cJ)}$}{2-category of
geometric stacks on a site $(\cC,\cJ)$} for the full 2-subcategory
of geometric stacks in $\Sta_{(\cC,\cJ)}$. Here $\De_\cX$
representable implies $\Pi$ is representable.
\label{agAdef16}
\end{dfn}

To obtain nice classes of stacks, one usually requires further
properties $\bs P$ of $\De_\cX$ and $\Pi$. For example, in algebraic
geometry with $(\cC,\cJ)$ schemes with the \'etale topology, we
assume $\De_\cX$ is quasicompact and separated, and $\Pi$ is \'etale
for Deligne--Mumford stacks $\cX$, and $\Pi$ is smooth for Artin
stacks~$\cX$.

The following material is based on Metzler \cite[\S 3.1 \& \S
3.3]{Metz}, Laumon and Moret-Bailly \cite[\S\S 2.4.3, 3.4.3, 3.8,
4.3]{LaMo}, and Lerman~\cite[\S 4.4]{Lerm}.

We can characterize geometric stacks $\cX$ up to equivalence solely
in terms of objects and morphisms in $\cC$, using the idea of {\it
groupoid objects\/} in~$\cC$.

\begin{dfn} A {\it groupoid object\/}\I{groupoid
object|(}\I{category!groupoid object in|(} $(U,V,s,t,u,i,m)$ in a
category $\cC$, or simply {\it groupoid\/} in $\cC$, consists of
objects $U,V$ in $\cC$ and morphisms $s,t:V\ra U$, $u:U\ra V$,
$i:V\ra V$ and $m:V\t_{s,U,t}V\ra V$ satisfying the identities
\begin{gather}
s\ci u=t\ci u=\id_{U},\;\> s\ci i=t,\;\> t\ci
i=s,\;\> s\ci m=s\ci\pi_2,\;\> t\ci m=t\ci\pi_1,
\nonumber\\
\begin{gathered}
m\ci (i\t\id_{V})=u\ci s,\;\> m\ci (\id_{V}\t i)=u\ci t,\\
m\ci (m\t\id_{V})= m\ci(\id_{V}\t m):
V\t_{U}V\t_{U}V\longra V,
\end{gathered}
\label{agAeq10}\\
m\ci(\id_{V}\t u)=m\ci(u\t\id_{V}):V=V\t_{U}U\longra V,
\nonumber
\end{gather}
where we suppose all the fibre products exist.

Groupoids in $\cC$ are so called because a groupoid in $\Sets$ is a
groupoid in the usual sense, that is, a category with invertible
morphisms, where $U$ is the set of {\it objects}, $V$ the set of
{\it morphisms}, $s:V\ra U$ the {\it source\/} of a morphism,
$t:V\ra U$ the {\it target\/} of a morphism, $u:U\ra V$ the {\it
unit\/} taking $X\mapsto\id_X$, $i$ the {\it inverse\/} taking
$f\mapsto f^{-1}$, and $m$ the {\it multiplication\/} taking
$(f,g)\mapsto f\ci g$ when $s(f)=t(g)$. Then \eq{agAeq10} reduces to
the usual axioms for a groupoid.
\label{agAdef17}
\end{dfn}

From a geometric stack with an atlas,\I{atlas} we can construct a
groupoid in~$\cC$.

\begin{dfn} Let $(\cC,\cJ)$ be a subcanonical site, and
suppose $\cX$ is a geometric stack on $(\cC,\cJ)$ with atlas
$\Pi:\bar U\ra\cX$. Then $\bar U\t_{\Pi,\cX,\Pi}\bar U$ is
equivalent to $\bar V$ for some $V\in\cC$ as $\Pi$ is
representable.\I{stack!1-morphism!representable} Hence we can take
$\bar V$ to be the fibre product, and we have a 2-Cartesian
square\I{2-category!2-Cartesian square}
\e
\begin{gathered}
\xymatrix@C=90pt@R=15pt{ *+[r]{\bar V} \ar[r]_(0.25){\bar t} \ar[d]^{\bar
s} \drtwocell_{}\omit^{}\omit{^\eta} & *+[l]{\bar U}
\ar[d]_\Pi \\
*+[r]{\bar U} \ar[r]^(0.7)\Pi & *+[l]{\cX} }
\end{gathered}
\label{agAeq11}
\e
in $\Sta_{(\cC,\cJ)}$. Here as $(\cC,\cJ)$ is subcanonical, any
1-morphism $\bar V\ra\bar U$ in $\Sta_{(\cC,\cJ)}$ is 2-isomorphic
to $\bar f$ for some unique morphism $f:V\ra U$ in $\cC$. Thus we
may write the projections in \eq{agAeq11} as $\bar s,\bar t$ for some
unique $s,t:V\ra U$ in~$\cC$.

By the universal property of fibre products there exists a
1-morphism $H:\bar U\ra\bar V$, unique up to 2-isomorphism, with
$\bar s\ci H\cong\id_{\bar U}\cong \bar t\ci H$. This $H$ is
2-isomorphic to $\bar u:\bar U\ra\bar V$ for some unique morphism
$u:U\ra V$ in $\cC$, and then $s\ci u=t\ci u=\id_U$. Similarly,
exchanging the two factors of $U$ in the fibre product we obtain a
unique morphism $i:V\ra V$ in $\cC$ with $s\ci i=t$ and $t\ci i=s$.
In $\Sta_{(\cC,\cJ)}$ we have equivalences
\begin{equation*}
\overline{V\t_{s,U,t}V}\simeq \bar V\t_{\bar s,\bar U,\bar t}\bar V
\simeq (\bar U\t_\cX\bar U)\t_{\bar U}(\bar U\t_\cX\bar U)
\simeq \bar U\t_\cX\bar U\t_\cX\bar U.
\end{equation*}
Let $m:V\t_{s,U,t}V\ra V$ be the unique morphism in $\cC$ such that
$\bar m$ is 2-isomorphic to the projection
$\overline{V\t_{s,U,t}V}\ra\bar V=\bar U\t_\cX\bar U$ corresponding
to projection to the first and third factors of $\bar U$ in the
final fibre product. It is now not difficult to verify that
$(U,V,s,t,u,i,m)$ is a groupoid in~$\cC$.
\label{agAdef18}
\end{dfn}

Conversely, given a groupoid in $\cC$ we can construct a
stack~$\cX$.

\begin{dfn} Let $(\cC,\cJ)$ be a site with descent for
morphisms, and $(U,\ab V,\ab s,\ab t,\ab u,\ab i,m)$ be a groupoid
in $\cC$. Define a prestack $\cX'$ on $(\cC,\cJ)$ as follows: let
$\cX'$ be the category whose objects are pairs $(T,f)$ where $f:T\ra
U$ is a morphism in $\cC$, and morphisms are $(p,q):(T_1,f_1)\ra
(T_2,f_2)$ where $p:T_1\ra T_2$ and $q:T_1\ra V$ are morphisms in
$\cC$ with $f_1=s\ci q$ and $f_2\ci p=t\ci q$. Given morphisms
$(p_1,q_1):(T_1,f_1)\ra (T_2,f_2)$ and $(p_2,q_2):(T_2,f_2)\ra
(T_3,f_3)$ the composition is $(p_2,q_2)\ci (p_1,q_1)=\bigl(p_2\ci
p_1,m\ci (q_1\t (q_2\ci p_2))\bigr)$, where $q_1\t(q_2\ci
p_2):T_1\ra V\t_{t,U,s}V$ is induced by the morphisms $q_1:T_1\ra V$
and $q_2\ci p_2:T_1\ra V$, which satisfy~$t\ci q_1=f_2\ci
p_1=s\ci(q_2\ci p_2)$.

Define a functor $p_{\cX'}:\cX'\ra\cC$ by $p_{\cX'}:(T,f)\mapsto T$
and $p_{\cX'}:(p,q)\mapsto p$. Using the groupoid axioms \eq{agAeq10}
we can show that $p_{\cX'}:\cX'\ra\cC$ is a category fibred in
groupoids. Since $(\cC,\cJ)$ has descent for morphisms, we can also
show $\cX'$ is a prestack. But in general it is not a stack. Let
$\cX$ be the associated stack from Proposition \ref{agAprop1}. We
call $\cX$ the {\it stack associated to the
groupoid\/}\I{stack!associated to a groupoid|(} $(U,V,s,t,u,i,m)$.
It fits into a natural 2-commutative diagram~\eq{agAeq11}.
\label{agAdef19}
\end{dfn}

Groupoids in $\cC$ are often written $V\rra U$, to emphasize
$s,t:V\ra U$, leaving $u,i,m$ implicit. The associated stack is then
written as~$[V\rra U]$.

Our next theorem is proved by Metzler \cite[Prop.~70]{Metz} when
$(\cC,\cJ)$ is the site of topological spaces with open covers, but
examining the proof shows that all he uses about $(\cC,\cJ)$ is that
fibre products exist in $\cC$ and $(\cC,\cJ)$ has descent for
objects and morphisms. See also Lerman \cite[Prop.~4.31]{Lerm}. If
fibre products may not exist in $\cC$ then one must also require the
morphisms $s,t$ in $(U,V,s,t,u,i,m)$ to be {\it
representable\/}\I{stack!1-morphism!representable} in $\cC$, that
is, for all $f:T\ra U$ in $\cC$ the fibre products $T_{f,U,s}V$ and
$T_{f,U,t}V$ exist in~$\cC$.

\begin{thm} Let\/ $(\cC,\cJ)$ be a site, and suppose that
all fibre products exist in $\cC,$ and that descent for objects and
morphisms holds in\/ $(\cC,\cJ)$. Then the constructions of
Definitions\/ {\rm\ref{agAdef18}, \ref{agAdef19}} are
inverse. That is, if\/ $(U,V,s,t,u,i,m)$ is a groupoid in\/ $\cC$
and\/ $\cX$ is the associated stack, then $\cX$ is a geometric
stack, and the $2$-commutative diagram \eq{agAeq11} is $2$-Cartesian,
and\/ $\Pi$ in \eq{agAeq11} is surjective and so an atlas\I{atlas}
for $\cX,$ and\/ $(U,V,s,t,u,i,m)$ is canonically isomorphic to the
groupoid constructed in Definition\/ {\rm\ref{agAdef18}} from the
atlas\/ $\Pi:\bar U\ra\cX$. Conversely, if\/ $\cX$ is a geometric
stack with atlas\/ $\Pi:\bar U\ra\cX,$ and\/ $(U,V,s,t,u,i,m)$ is
the groupoid in\/ $\cC$ constructed from $\Pi$ in Definition\/
{\rm\ref{agAdef18},} and\/ $\tcX$ is the stack associated to
$(U,V,s,t,u,i,m)$ in Definition\/ {\rm\ref{agAdef19},} then\/ $\cX$
is equivalent to\/ $\tcX$ in\/ $\Sta_{(\cC,\cJ)}$. Thus every
geometric stack is associated to a groupoid.
\label{agAthm}
\end{thm}

In the situation of Theorem \ref{agAthm} we have 2-Cartesian
diagrams\I{2-category!2-Cartesian square}
\e
\begin{gathered}
\xymatrix@!0@C=32pt@R=10pt{ *+[r]{\bar V} \ar[rrrr]_(0.7){\bar t}
\ar[ddd]^{\bar s} &&&& *+[l]{\bar U} \ar[ddd]_\Pi & *+[r]{\bar V}
\ar[rrrr]_(0.7){\Pi\ci\bar s} \ar[ddd]^{\bar s\t\bar t} &&&&
*+[l]{\cX} \ar[ddd]_{\De_\cX}
\\
& \drrtwocell_{}\omit^{}\omit{^{}} &&&&&
\drrtwocell_{}\omit^{}\omit{^{}}
\\ &&&&&&&& \\
*+[r]{\bar U} \ar[rrrr]^(0.7)\Pi &&&& *+[l]{\cX,\!{}} &
*+[r]{\bar U\t\bar U} \ar[rrrr]^(0.65){\Pi\t\Pi} &&&& *+[l]{\cX\t\cX,\!{}}
\\ \\
*+[r]{\bar V\t_{\bar s\t\bar t,\bar U\t\bar U,\De_{\bar U}}\!\bar U}
\ar[rrrr]_(0.75){\Pi\ci\bar s\t \Pi\ci\bar t} \ar[ddd]^{\pi_{\bar
U}} &&&& *+[l]{I_\cX} \ar[ddd]_{\io_\cX} & *+[r]{\bar U}
\ar[ddd]^(0.45){\bar u\t\bar\id_U} \ar[rrrr]_(0.7){\Pi} &&&&
*+[l]{\cX} \ar[ddd]_{\jmath_\cX}
\\
& \drrtwocell_{}\omit^{}\omit{^{}}  &&&&&
\drrtwocell_{}\omit^{}\omit{^{}}
\\ &&&&&&&& \\
*+[r]{\bar U} \ar[rrrr]^(0.7){\Pi} &&&& *+[l]{\cX,\!{}} &
*+[r]{\bar V\!\t_{\bar s\t\bar
t,\bar U\t\bar U,\De_{\bar U}}\!\bar U} \ar[rrrr]^(0.75){\Pi\ci\bar
s\t \Pi\ci\bar t} &&&& *+[l]{I_\cX,\!{}} }
\end{gathered}
\label{agAeq12}
\e
with surjective rows. So from Proposition \ref{agAprop2} we deduce:

\begin{cor} In the situation of Theorem\/ {\rm\ref{agAthm},} let\/
$\bs P$ be a property of morphisms in $\cC$ which is invariant under
base change and local in the target. Then\/ $\bar\Pi:\bar U\ra\cX$
is $\bs P$ if and only if\/ $s:V\ra U$ is $\bs P,$ and\/
$\De_\cX:\cX\ra\cX\t\cX$ is\/ $\bs P$ if and only if\/ $s\times
t:V\ra U\t U$ is\/ $\bs P,$ and\/ $\io_\cX:I_\cX\ra\cX$ is $\bs P$
if and only if\/ $\pi_U:V\t_{s\times t,U\t U,\De_U}U\ra U$ is\/ $\bs
P,$ and\/ $\jmath_\cX:\cX\ra I_\cX$ is $\bs P$ if and only if\/
$u\t\id_U:U\ra V\t_{s\times t,U\t U,\De_U}U$ is\/~$\bs P$.
\label{agAcor1}
\end{cor}

We can describe atlases\I{atlas} for fibre products of geometric
stacks.

\begin{ex} Suppose $(\cC,\cJ)$ is a subcanonical site,
and all fibre products exist in $\cC$. Let
\begin{equation*}
\xymatrix@C=90pt@R=14pt{ *+[r]{\cW} \ar[r]_(0.2){f} \ar[d]^{e}
\drtwocell_{}\omit^{}\omit{^\eta} & *+[l]{\cY} \ar[d]_h \\ *+[r]{\cX} 
\ar[r]^(0.7)g & *+[l]{\cZ} }
\end{equation*}
be a 2-Cartesian diagram\I{2-category!2-Cartesian square} in
$\Sta_{(\cC,\cJ)}$, where $\cX,\cY,\cZ$ are geometric stacks. Let
$\Pi_\cX:\bar U_\cX\ra\cX$ and $\Pi_\cY:\bar U_\cY\ra\cY$ be
atlases. As $\De_\cZ$ is
representable\I{stack!1-morphism!representable} the fibre product
$\bar U_\cX\t_{g\ci\Pi_\cX,\cZ,h\ci\Pi_\cY}\bar U_\cY$ is
represented by an object $U_\cW$ of $\cC$. Then we have a
2-commutative diagram, where we omit 2-morphisms:
\e
\begin{gathered}
\xymatrix@R=7pt@C=27pt{ & \bar U_W:=\bar U_\cX\t_\cZ\bar U_\cY
\ar@/^2pc/@(ur,ur)@{..>}[rrrd]^{\Pi_\cW} \ar[rr]^{\pi_1} \ar[dl]
\ar[ddr] && \cX\t_\cZ\bar U_\cY \ar[dl] \ar[ddl] \ar[dr]^(0.6){\pi_2} \\
\bar U_\cX \ar@<1ex>[ddr] \ar[rr]^{\Pi_\cX} && \cX
\ar[ddl]_g && \cW \ar[ll]^e \ar[ddl]^f \\
&& \bar U_\cY \ar[dl] \ar[dr]^{\Pi_\cY} \\
& \cZ && \cY \ar[ll]_h }
\end{gathered}
\label{agAeq13}
\e
Here the five squares in \eq{agAeq13} are 2-Cartesian. Define $\Pi_\cW=\pi_2\ci\pi_1:\bar U_\cW\ra\cW$, where $\pi_1,\pi_2$ are as in \eq{agAeq13}. Proposition
\ref{agAprop2}(a),(b) imply that $\pi_1,\pi_2$ are representable and
surjective, since $\Pi_\cX,\Pi_\cY$ are. Hence
$\Pi_\cW=\pi_2\ci\pi_1$ is also representable and surjective, so
$\cW$ is a geometric stack, and $\Pi_\cW$ is an atlas for $\cW$. In
the same way, if $\bs P$ is a property of morphisms in $\cC$ which
is invariant under base change and local in the target and closed
under compositions, and $\Pi_\cX,\Pi_\cY$ are $\bs P$, then
$\Pi_\cW$ is~$\bs P$.

Now let $\bar V_\cW=\bar U_\cW\t_\cW\bar U_\cW$ and complete to a
groupoid $(U_\cW,V_\cW,\ab s_\cW,\ab t_\cW,\ab u_\cW,\ab i_\cW,\ab
m_\cW)$ in $\cC$ as above, with $\cW\simeq[V_\cW\rra U_\cW]$, and do
the same for $\cX,\cY$. Then by a diagram chase similar to
\eq{agAeq13} we can show that
\e
\bar V_\cW\cong \bar V_\cX\t_\cZ\bar V_\cY\quad\text{and}\quad
V_\cW\cong (U_\cW\t_{U_\cX}V_\cX)\t_{U_\cY}V_\cY.
\label{agAeq14}
\e
\label{agAex}
\end{ex}

\begin{cor} Suppose\/ $(\cC,\cJ)$ is a subcanonical site,
and all fibre products exist in\/ $\cC$. Then the\/
$2$-subcategory\/ $\GSta_{(\cC,\cJ)}$ of geometric stacks is closed
under fibre products in\/~$\Sta_{(\cC,\cJ)}$.\I{stack|)}%
\I{site|)}\I{stack!geometric|)}\I{groupoid
object|)}\I{category!groupoid object in|)}\I{stack!associated to a
groupoid|)}
\label{agAcor2}
\end{cor}

%%%%%%%%%%%%%%%%%%%%%%%%%%%%%%%%%%%%%%%%%%%%%%%%%%%%%%%%%%%%%%%%%%%%%%%%
%%%%%%%%%%%%%%%%%%%%%%%%%  Bibliography  %%%%%%%%%%%%%%%%%%%%%%%%%%%%%%%
%%%%%%%%%%%%%%%%%%%%%%%%%%%%%%%%%%%%%%%%%%%%%%%%%%%%%%%%%%%%%%%%%%%%%%%%

\clearpage
\printnomenclature[1.3cm]
\clearpage
\addcontentsline{toc}{section}{Index}
\printindex

\end{document}